\documentclass[reqno]{amsart}

\usepackage[english]{babel}
\usepackage{amsmath}
\usepackage{amsthm}
\usepackage{amsfonts}
\usepackage{amssymb}
\usepackage{anysize}
\usepackage{hyperref}
\usepackage{appendix}
\usepackage{mathtools}
\usepackage{graphicx}
\usepackage{pgfplots}
\usepackage{subcaption} 
\usepackage{enumitem}
\usepackage{dsfont}
\usepackage{stmaryrd} 
\usepackage{comment} 
\usepackage[dvipsnames]{xcolor}

\makeatletter
\@namedef{subjclassname@2020}{\textup{2020} Mathematics Subject Classification}
\makeatother

\newtheorem{definition}{Definition}[section]
\newtheorem{lemma}[definition]{Lemma}
\newtheorem{theorem}[definition]{Theorem}
\newtheorem{corollary}[definition]{Corollary}
\newtheorem{proposition}[definition]{Proposition}
\newtheorem{example}[definition]{Example}
\newtheorem{remark}[definition]{Remark}
\newtheorem{assumption}{Assumption}

\DeclareMathOperator{\divop}{div}

\DeclareMathOperator{\diam}{diam}

\DeclareMathOperator{\sgn}{sgn}

\DeclareMathOperator*{\esssup}{ess~sup}

\allowdisplaybreaks 

\numberwithin{equation}{section}

\title[]{Mean field limit of non-exchangeable interacting diffusions on co-evolutionary networks}

\author[Juli\'an Cabrera-Nyst]{Juli\'an Cabrera-Nyst}
\address{Departamento de Matem\'atica Aplicada and Research Unit ``Modeling Nature'' (MNat), Facultad de Ciencias, Universidad de Granada, 18071 Granada, Spain}
\email{jcabreranyst@ugr.es}

\author[David Poyato]{David Poyato}
\address{Departamento de Matem\'atica Aplicada and Research Unit ``Modeling Nature'' (MNat), Facultad de Ciencias, Universidad de Granada, 18071 Granada, Spain}
\email{davidpoyato@ugr.es}

\begin{document}

\date{\today}


\thanks{\textbf{Acknowledgment.} JC and DP have received support from the State Research Agency of the Spanish Ministry
of Science and FEDER-EU, grant  PID2022-137228OB-I00 (MICIU/AEI/10.13039/501100011033). JC has received support
from the predoctoral grant PREP2022-000019, funded by MICIU/AEI/10.13039/501100011033 and by ESF+.
DP has received support by Consejer\'ia de Universidad, Investigaci\'on e Innovaci\'on \& ERDF/EU Andalusia Program, grant C-EXP-265-UGR23, and by Consejer\'ia de Universidad, Investigaci\'on e Innovaci\'on (Junta de Andaluc\'ia), Grant QUAL21-011 of Modeling Nature Research Unit.}

\begin{abstract}
	Systems in which the network structure and particle states co-evolve in mutual influence are increasingly recognized as essential for modeling complex adaptive systems. However, traditional models of non-exchangeable interacting particle systems frequently assume a fixed network topology, a simplification that fails to capture the dynamical nature of many real-world phenomena.  In this paper, we rigorously establish the mean-field limit for systems of non-exchangeable interacting diffusions on co-evolutionary networks. The primary analytical challenge arises from the coupling between the network dynamics and the agents' states, which induces non-Markovian dynamics where the system's evolution depends on its entire history. Consequently, the macroscopic limit is not governed by a classical partial differential equation, but rather by a coupled system of path-dependent McKean-Vlasov SDEs for the particles' states together with a transport equation for the distribution of the weights. A further difficulty stems from the non-linear weight dynamics, which requires an adequate choice for the limiting network structure. To overcome the structural limitations of classical graphon theory, we employ the framework of \(\mathcal{K}\)-graphons (Lovász and Szegedy, 2010), also termed probability-graphons (Abraham, Delmas, and Weibel, 2025). To the best of our knowledge, this is the first example in the literature in which the natural topology of probability-graphons is used in the context of mean-field limits, providing a natural and rigorous framework that is fully compatible with non-linear network adaptivity.

\end{abstract}

\maketitle
\tableofcontents

\section{Introduction}\label{sec:introduction}
Interacting particle systems provide a fundamental mathematical framework for modeling and analyzing the collective behavior
of large ensembles of agents. Classically, these systems have been studied under the assumption of exchangeability, 
where agents are indistinguishable and their interactions are mediated by a homogeneous, all-to-all topology.
However, many real-world phenomena are fundamentally non-exchangeable, featuring distinct agent identities and
heterogeneous interaction topologies. 
While the literature of non-exchangeable multi-agent systems has grown substantially in recent years \cite{AP-26,APP-24-arxiv,BCW-23,BCN-24,CM-19,JPS-24,JSZ-24-arxiv,Lucon-20},
a ubiquitous simplification in these non-exchangeable formulations is that the
underlying interaction topology remains either static or evolves according to a predetermined, external rule. While
analytically tractable, this static assumption represents a significant limitation
when modeling complex systems characterized by co-evolutionary dynamics, where the agents and their underlying interaction
network mutually shape one another. 

This co-evolutionary interplay is particularly relevant in fields such as social dynamics
\cite{LLP-25,PKI-23}, epidemiology \cite{SS-08,TS-14}, and neuroscience \cite{BSS-20,HNP-16,KYSN-17}
(for a more detailed overview of the literature on adaptive networks, see \cite{BGKKY-23}). 
In these systems, the network structure is not fixed but evolves in response to the states of the agents,
leading to a rich interplay between the dynamics of the agents and the evolution of the network. 

Based on these motivations, and as a prototypical system featuring the above properties, we will consider the following general system of interacting diffusions on co-evolutionary networks:
\begin{equation}\label{eq:multi-agent-discrete}
\begin{aligned}
&dX_i^N(t)=\sum_{j=1}^N w_{ij}^N(t)\,K(X_i^N(t),X_j^N(t))\,dt+\sqrt{2\nu}\,dW_i^N(t),\\
&\frac{d}{dt}w_{ij}^N(t)=\frac{1}{N}\Gamma(X_i^N(t),X_j^N(t),Nw_{ij}^N(t)),\\
& X_i^N(0)=X_{i,0}^N, \quad  w_{ij}^N(0)=w_{ij,0}^N,
\end{aligned}
\end{equation}
for any \(i,j\in \{1,\ldots,N\}\). Above, \(X_i^N(t)\in \mathbb{R}^d\) represents the state of the \(i\)-th agent at time \(t\), \(w^N_{ij}(t) \in \mathbb{R}\) is the interaction weight that particle \(j\) exerts on particle \(i\), and \(W_i^N\) are independent copies of a \(d\)-dimensional standard Wiener process. 
The initial states \(X_{i,0}^N\) and initial weights \(w_{ij,0}^N\) are random variables, and are assumed to be independent of the Wiener processes. The spatial interaction between agents is
mediated by the kernel \(K\), while the plasticity function \(\Gamma\) governs the adaptive co-evolutionary dynamics of the
network weights. It is important to note that the state variables \(X_i^N\) evolve according to Stochastic Differential Equations
(SDEs), whereas the network weights \(w_{ij}^N\) are governed by Ordinary Differential Equations (ODEs) with random initial data and random inputs (the agent states).

The aim of this paper is to rigorously identify the mean-field limit of the 
system \eqref{eq:multi-agent-discrete} as \(N\to \infty\). This question has been addressed in the literature in some
particular cases \cite{A-26-arxiv,GKX-25,T-25-arxiv,Zhou-25-arxiv} which we discuss later in this introduction but, to the best of our knowledge, the general case of
non-linear plasticity function \(\Gamma\) has not been addressed yet in this
level of generality due to fundamental limitations of classical graphon theory as dense graph limits.
As we can see by choosing particular forms of the plasticity function \(\Gamma\), 
the system \eqref{eq:multi-agent-discrete} generalises a wide variety of models already considered in the literature,
including the classical exchangeable multi-agent systems and also the case of non-exchangeable multi-agent systems with static weights. We also emphasize that \eqref{eq:multi-agent-discrete} is only a simple enough prototype of system with co-evolutionary weights, but more complex versions can be treated with the same techniques developed in this work, as we discuss later in Section \ref{sec:conclusions}.

Before formulating our main result, we will review below the
mean-field limit results existing in the literature for these
particular cases, and we will highlight the differences and main challenges with our general case of co-evolutionary weights.

\subsection{Review of previous literature} \phantom{=} \newline

\(\diamond \) \textbf{Exchangeable setting:} \(\Gamma=0\) and homogeneous \(w_{ij,0}^N=\frac{1}{N}\).\\
Here, the interaction weights are uniform and static. Assuming the initial states \(X_{i,0}^N\) are independent and identically distributed (i.i.d.), the system exhibits a law of large numbers phenomenon as \(N\to \infty\) in which the empirical measure associated with the particle states,
\begin{equation} \label{eq:empirical-measure}
m^N_t=\frac{1}{N}\sum_{i=1}^{N} \delta_{X_i^N(t)},
\end{equation}
converges almost surely to a deterministic probability measure \(m_t\). Structurally, the interaction that each agent experiences is mediated entirely through this empirical measure:
\begin{equation} \label{eq:interaction-empirical-measure}
\frac{1}{N}\sum_{j=1}^{N} K(X_i^N(t),X_j^N(t))=\int_{\mathbb{R}^d} K(X_i^N(t),y) dm_t^N(y).
\end{equation}
Because of this mean-field interaction structure, all particles \(X^N_i\) behave asymptotically as independent copies of a
single effective non-linear process, and the empirical measures \(m^N_t\) converge to a limiting measure \(m_t\) identified as the time-marginal law of the solution to the following McKean-Vlasov SDE:
\begin{equation} \label{eq:McKean-Vlasov-exchangeable-case}
\begin{aligned}
&dX(t)=\int_{\mathbb{R}^d} K(X(t),y)dm_t(y)\,dt+\sqrt{2\nu}\,dW(t),  \\
&X(0)=X_0\sim m_0, \quad m_t={\rm Law}(X(t)),
\end{aligned}
\end{equation}
where \(m_0 \in \mathcal{P}(\mathbb{R}^d)\) is the common distribution of the initial data \(X_{i,0}^N\).
Furthermore, because this limiting process is Markovian, it can be recast analytically in terms of a PDE using Itô's formula to show that
the marginal law \(m_t\) evolves according to the corresponding non-linear Fokker-Planck equation:
\begin{equation} \label{eq:multi-agent-macroscopic-PDE-exchangeable}
\partial_t m_t(x)+\divop_x\left(m_t(x)\int_{\mathbb{R}^d} K(x,y)\,dm_t(y)\right)=\nu \Delta_x m_t(x).
\end{equation}
This formulation highlights the fundamental exchangeability of the system: because the interaction depends entirely on the empirical measure, the dynamics are completely unaffected by any relabeling of the agents, guaranteeing that the joint law of the system is invariant under permutations. 

The standard methodology to establish this convergence is \emph{propagation of chaos}, which leverages the system's
exchangeability to prove rigorously that the joint law of any finite number of particles converges to the product
of the marginal limits \(m_t\). Following Sznitman's classical approach \cite{S-91}, this property is typically proven
via a synchronous coupling argument, working for Lipschitz \(K\). By constructing independent copies of the limit process
\eqref{eq:McKean-Vlasov-exchangeable-case} driven by the same Wiener processes as the microscopic system, one can show that
the distance between the trajectories of the interacting particles and the independent copies of the limit process vanishes as \(N \to \infty\). For comprehensive overviews of the mean-field limit for exchangeable systems, we refer to the seminal works \cite{BH-77,Dobrushin-79,MJ-98,N-84,S-91} and the reviews \cite{CD-22,CD-22-second,Jabin-14}.

\medskip

\(\diamond\) \textbf{Non-exchangeable setting with static weights:} \(\Gamma=0\) and heterogeneous \(w_{ij,0}^N\).\\
In this regime, the interaction topology is static, {\it i.e.}, \(w^N_{ij}(t)=w_{ij,0}^N\) and generally heterogeneous. Such heterogeneity
of the interaction weights explicitly breaks the indistinguishability of the agents in the system, rendering
the particles strictly non-exchangeable. For this reason, the agents are no longer expected to become statiscally 
identical and therefore their collective
influence cannot be accurately captured by a single, globally averaged distribution. Instead, the interaction is
naturally mediated by the \emph{extended empirical measure}, defined  as 
\begin{equation}\label{eq:extended-empirical-measure}
\mu^{N,\xi}_t=\sum_{i=1}^N \mathds{1}_{I_i^N}(\xi) \delta_{X_i^N(t)}, \quad \xi \in [0,1],
\end{equation}
where \(I_i^N:=\left[\frac{i-1}{N},\frac{i}{N} \right)\). This object is a parametrized family of probability measures 
where the continuous variable \(\xi\in [0,1]\) acts as a macroscopic label. By indexing the individual measures 
with \(\xi\), the framework is in position to effectively preserve the distinct identities of the agents in the limit.
The classical empirical
measure \eqref{eq:empirical-measure} is then naturally recovered by integrating over the label space:
\[
\int_{0}^{1} \mu^{N,\xi}_t \,d\xi = \frac{1}{N}\sum_{i=1}^N \delta_{X_i^N(t)}=m_t^N.
\]
Under this heterogeneous configuration, the interaction experienced by each particle is governed jointly
by the extended empirical measure and the initial weights \((w^N_{ij,0})_{1\leq i,j\leq N}\).
To treat these discrete weights analytically, one introduces a piecewise-constant function which serves as a continuous
version of the adjacency matrix, and takes the following form: 
\begin{equation}
	\label{eq:continuous-graphon}
	w^N_0(\xi,\xi')=\sum_{i,j=1}^N Nw_{ij,0}^N \mathds{1}_{I_i^N\times I_j^N}(\xi,\xi'), \quad \xi,\xi'\in [0,1].
\end{equation}
Using this representation, the discrete interaction drift for any particle associated with the macroscopic
label \(\xi \in I_i^N\) (representing the discrete label \(i\)) can be rewritten as:
\[
\frac{1}{N}\sum_{j=1}^{N} Nw_{ij,0}^N K(X_i^N(t),X_j^N(t))=\int_0^1 \int_{\mathbb{R}^d} w_0^N(\xi,\xi') K(X_i^N(t),y) d\mu_t^{N,\xi'}(y) \,d\xi'.
\]
This integral formulation reveals that passing to the \(N \to \infty\) limit requires establishing the asymptotic
behavior of two distinct objects: the extended empirical measure \(\mu_t^{N}\) and the network structure
\(w_0^N\). Consequently, rigorous identification of the macroscopic dynamics necessitates the application
of large graph limit theory. 

The theory of graphons \cite{LS-06} is the most prominent framework used for this purpose in the literature. Under classical graphon convergence,
the corresponding mean-field limit is described as a heterogeneous, Vlasov-Fokker-Planck type PDE, where the macroscopic
interactions are weighted by a limiting graphon \cite{BCW-23,BCN-24,CM-19,DGL-16,K-VM-18,Lucon-20,PT-22-arxiv}.
Beyond classical graphons, alternative 
graph limit theories have also been leveraged to capture more complex static topologies, including graphings \cite{BS-01}, graphops \cite{BS-22} or $s$-graphons \cite{KLS-19}, and even hypergraph limits in the context of higher-order interactions. Those have been recently exploited in the literature of mean-field limits, see \cite{APP-24-arxiv,CM-19,GK-22,K-VM-18,JPS-24,KX-22,KX-25,LRW-23,ORS-20}.

In this non-exchangeable setting, the classical propagation of chaos property breaks down: due to the intrinsic heterogeneity
of the network topology, even if the particles are initially independent and identically distributed (i.i.d.),
they do not retain this property in the asymptotic limit as \(N\to \infty\). Instead, one can still prove a weaker
form of propagation of chaos, which is  the so-called \emph{propagation of independence} (see \cite{JPS-24}): the processes describing any 
finite number of particles become asymptotically  independent as the number of particles goes to infinity, but they do
not converge to independent copies of a single process.
Intead, they converge to distinct, independent limit processes that are parametrized by their macroscopic labels \(\xi\in [0,1]\).
Consequently, the (random) extended empirical measure \((\mu^{N,\xi}_t)_{\xi \in [0,1]}\) converges to a deterministic
family of probability measures \((\mu_t^\xi)_{\xi \in [0,1]}\) representing the laws of the solution to the following system of McKean-Vlasov SDEs:
\begin{equation}\label{eq:multi-agent-macroscopic-SDE-non-exchangeable-static-weights}
\begin{aligned}
	&dX(t,\xi)= \int_{0}^{1} \int_{\mathbb{R}^d} w_0(\xi,\xi') K(X(t,\xi),y) \, d\mu_t^{\xi'}(y) \, d\xi' \, dt + \sqrt{2\nu}\,dW(t,\xi), \quad t\in [0,T],\,\xi\in [0,1],\\
	&X(0,\xi)=X_0(\xi)\sim \mu_0^{\xi}, \quad \mu_t^\xi={\rm Law}(X(t,\xi)).
\end{aligned}
\end{equation}
Above, \((W(\cdot,\xi))_{\xi\in [0,1]}\) is a family of standard Wiener processes, and the macroscopic network topology is encoded by the limiting graphon \(w_0\in L^\infty([0,1]^2)\) which, under suitable assumptions on the graphs sequence \cite{LS-06}, arises as the graph limit
of the scaled initial weights \(Nw_{ij,0}^N\). Since the limiting McKean-Vlasov system of SDEs
\eqref{eq:multi-agent-macroscopic-SDE-non-exchangeable-static-weights} is
a system of Markovian Itô diffusions, their laws
\(\mu_t^\xi\) satisfy a standard forward Kolmogorov equation in distributional sense, namely
\begin{equation}\label{eq:multi-agent-macroscopic-PDE-non-exchangeable-static-weights} 
	\partial_t \mu_t^\xi(x)+\divop_x\left(\mu_t^\xi(x)\int_0^1 \int_{\mathbb{R}^d}w_0(\xi,\xi')\,K(x,y)\,d\mu_t^{\xi'}(y)\,d\xi'\right)=\nu\Delta_x \mu_t^\xi(x).
\end{equation}
In this case, the system cannot be reduced into a single, global Fokker-Planck equation by simply integrating over the label
space \(\xi\in [0,1]\) since the graphon \(w_0\) is generally heterogeneous and therefore the full family of measures \((\mu_t^\xi)_{\xi\in [0,1]}\) is required to fully describe the macroscopic dynamics.

\medskip

	\(\diamond\) \textbf{Non-exchangeable setting with co-evolutionary weights}.\\
In the general case in which we do not suppose any particular form for the
	 plasticity function \(\Gamma\), the system \eqref{eq:multi-agent-discrete} has a completely different structure
	due to the coupling between the dynamics of the states and the weights, which turns the system into a
	non-Markovian process. This can be seen by first solving the second ODE in \eqref{eq:multi-agent-discrete}
	and plugging it into the first SDE: 
\begin{equation}\label{eq:multi-agent-discrete-solved-weights}
\begin{aligned}
&dX_i^N(t)=\frac{1}{N}\sum_{j=1}^N \Phi_t(X_{i,[0,t]}^N,X_{j,[0,t]}^N,Nw_{ij,0}^N) \, K(X_i^N(t),X_j^N(t))\,dt+\sqrt{2\nu}\,dW_i^N(t),\quad \\
& X_i^N(0)=X_{i,0}^N.
\end{aligned}
\end{equation}
Above, \(\Phi_t\) represents the flow map of the evolution ODE prescribed by the plasticity function \(\Gamma\). 
Given deterministic paths \(\gamma,\tilde{\gamma} \in \mathcal{C}^d_{T}:= C([0,T],\mathbb{R}^d)\), then
\(w(t)=\Phi_t(\gamma_{[0,t]},\tilde{\gamma}_{[0,t]},w_0)\) solves the ODE  
\begin{align*}
&w'(t)=\Gamma(\gamma(t),\tilde{\gamma}(t),w(t)),\\
&w(0)=w_0.
\end{align*}
Then, 
\(Nw^N_{ij}(t)=\Phi_t(X_{i,[0,t]}^N,X_{j,[0,t]}^N,Nw_{ij,0}^N)\) is a stochastic process solving the ODE pathwise. We emphasize that the flow map \(\Phi_t\) does not only depend 
on the local values \(X_i^N(t)\) and \(X_j^N(t)\) of the agents, but actually on the full past of both trajectories
up to time \(t\). Of course, this introduces an infinite-memory effect which turns
\eqref{eq:multi-agent-discrete-solved-weights} into a non-Markovian process. For this reason, identifying the
evolution of the law of the process in terms of a Fokker-Planck type PDE is not possible, and this impedes
identifying the mean-field limit of \eqref{eq:multi-agent-discrete} as \(N\to \infty\) in the classical sense.
This inherent path-dependence dictates that the interaction experienced by each particle is now mediated
through the \emph{extended path-dependent empirical measure}, which is defined for a solution
\((X_i^N)_{1\leq i\leq N}\) of the system \eqref{eq:multi-agent-discrete} as
\begin{equation}\label{eq:extended-path-dependent-empirical-measure}
\mu^{N,\xi}_{[0,t]}=\sum_{i=1}^N \mathds{1}_{I_i^N}(\xi) \delta_{X_{i,[0,t]}^N}, \quad t \in [0,T], \quad \xi \in [0,1].
\end{equation}
Using this notation, the interaction drift experienced by the particle associated with the macroscopic label
\(\xi \in I_i^N\) (representing the discrete label \(i\)) can be recast in terms of these extended measures:
\[
\begin{aligned}
&\frac{1}{N}\sum_{j=1}^{N} \Phi_t(X_{i,[0,t]}^N,X_{j,[0,t]}^N,Nw_{ij,0}^N) K(X_i^N(t),X_j^N(t)) \\ 
&\quad = \int_0^1 \int_{\mathcal{C}_t^d} \Phi_t(X_{i,[0,t]}^N,\gamma,w_0^N(\xi,\xi')) K(X_i^N(t),\gamma(t)) \, d\mu_{[0,t]}^{N,\xi'}(\gamma) \, d\xi',
\end{aligned}
\]
where \(w^N_0\) is again defined as in \eqref{eq:continuous-graphon}. 

Identifying the mean-field limit of this generalized system entails profound analytical challenges.
Beyond the inherent non-Markovian structure of the system, the non-linear nature of the flow map
\(\Phi_t\) of the plasticity function \(\Gamma\) introduces severe topological obstructions when characterizing the limiting network structure.
Specifically, if one attempts to use compactness of the space of graphons under the standard cut-metric \cite[Theorem 5.1]{LS-07}
to establish the convergence of a subsequence of the empirical graphon
\(w_0^N \to w_0\), passing to the limit in the flow maps fails, as the linear nature of the cut-norm is incompatible with non-linear
transformations. If, instead, one enforces a stronger convergence (such as uniform or \(L^p\) convergence)
to safely pass the limit through \(\Phi_t\), this requires assuming stringent well-prepared initial weights,
a highly restrictive assumption which is not satisfied in many relevant cases.

To bypass the structural limitations of classical graphon theory, our approach employs for the first time in the context
of mean-field limits the formalism of \(\mathcal{K}\)-graphons \cite{LS-10-arxiv}, also termed probability-graphons \cite{ADW-25}.
Objects of this type have recently been considered in \cite{AP-26} for the derivation of a simpler continuum limit equation, 
although their analysis did not exploit the natural topology of the space.
Endowed with its natural topology, this space of graph limits admits a precise characterization of relatively compact sets and
inherently accommodates non-linear structural operations. Crucially, leveraging these compactness arguments entirely
circumvents the need for strictly well-prepared initial weights. By simply ensuring that the initial network
configurations reside within a relatively compact set, one can rigorously pass to the limit via standard
subsequence extraction. This profound relaxation allows the framework to treat fully random initial topologies
for the network weights,
as well as independent (though intrinsically non-identically distributed) initial data for both the particle states.

	Only a handful of works have already addressed the mean-field limit of particle systems featuring co-evolutionary
	dynamics between particle states and interaction weights, and they do it in special regimes. For an affine plasticity function in the weight variable,
	the mean-field limit was first described in \cite{GKX-25} as the solution to a fixed-point equation for the
	characteristics of a generalized Vlasov equation, tracing back to the seminal ideas of Dobrushin \cite{Dobrushin-79}
	and Neunzert \cite{N-84} for mean-field limits of deterministic particle systems via stability estimates.
	In \cite{T-25-arxiv}, these ideas were extended to accommodate more general plasticity functions, relying on
	certain discretizations of the prescribed limiting network structure. To see how our framework cover these cases,
	we refer to the noiseless case \(\nu=0\) in Section \ref{subsec:special-cases}.
    
    In \cite{A-26-arxiv}, a mean-field
	limit description via a Vlasov-type equation was derived for a broader class of plasticity rules that allows
	for mean-field interactions among the weights. In this framework, provided that the plasticity function is independent
	of the target particle's state, {\it i.e.}, \(\Gamma(x,y,w)=\Gamma(y,w)\), it is shown that one can derive a closed PDE
	for an extended probability law that incorporates the weights to the agents' states as an additional variable.
    
    Closer to our approach
	are the contributions \cite{CD-25-arxiv,Zhou-25-arxiv}, which describe the mean-field limit in terms of the law
	of a limiting path-dependent McKean-Vlasov process. In \cite{CD-25-arxiv}, the mean-field limit is obtained as
	a corollary of a broader framework which exploits the specific additive noise structure. In this case,
    the particle system is
	realized as the fixed point of a suitable map enjoying strong regularity properties which ensures that a mean-field description governed by a McKean-Vlasov process can be rigorously derived provided that suitable assumptions on the limiting network structure are considered concerning regularity in one of the label variable.
	Finally, in \cite{Zhou-25-arxiv}, the mean-field limit is established for systems governed by affine plasticity 
    functions. In their approach, the asymptotic network structure is modeled via classical graphons rather than
    probability-graphons, and the convergence argument relies on a dual sampling lemma linking the limiting graphon 
    to the initial distribution of the states. Our present work naturally recovers these insights as a special 
    case, recontextualizing the limit within the probabilistic formalism of 
    probability-graphons, which directly accommodates fully non-linear plasticity rules.


\subsection{Our main result}

To formalize our mathematical setting, we impose the following assumptions throughout the rest of the paper:

\begin{assumption}[On the interaction kernel]\label{assump:interaction-kernel}
~
\begin{enumerate}[label=(\roman*)]
\item {\bf (Bounded-Lipschitz interaction kernel)}:\\
Assume that there exist $B_K,L_K>0$ such that \(K:\mathbb{R}^d\times \mathbb{R}^d\to \mathbb{R}^d\) satisfies  
\begin{align*}
|K(x,y)|&\leq B_K,\\
|K(x_1,y_1)-K(x_2,y_2)|&\leq L_K(|x_1-x_2|+|y_1-y_2|),
\end{align*}
for all $x,x_1,x_2,y,y_1,y_2\in \mathbb{R}^d$.
\end{enumerate}
\end{assumption}

\begin{assumption}[On the plasticity function]\label{assump:plasticity-function}
~
\begin{enumerate}[label=(\roman*)]
\setcounter{enumi}{1}
\item {\bf (Bounded-Lipschitz plasticity function)}:\\
Assume that there exist $B_\Gamma,L_\Gamma>0$ such that
$\Gamma: \mathbb{R}^d\times \mathbb{R}^d\times \mathbb{R}\to \mathbb{R}$ satisfies  
\begin{align*}\label{eq:plasticity-function-hypothesis-1}
&|\Gamma(x,y,w)|\leq B_\Gamma\left(1+|w|\right),\\
&|\Gamma(x_1,y_1,w_1)-\Gamma(x_2,y_2,w_2)|\leq L_\Gamma\left(|x_1-x_2|+|y_1-y_2|+|w_1-w_2|\right),
\end{align*}
for all $x,x_1,x_2,y,y_1,y_2\in \mathbb{R}^d$ and all $w,w_1,w_2\in \mathbb{R}$.
\end{enumerate}
\end{assumption}

\begin{assumption}[On the coupling weights]\label{assump:weights}
~
\begin{enumerate}[label=(\roman*)]
\setcounter{enumi}{2}
\item {\bf (Scaling of the initial coupling weights)}:\\
Assume that there exists $W>0$ such that
\begin{equation} \label{eq:hypothesis-weight}
\sup_{N\in \mathbb{N}} \max_{1\leq i,j\leq N} N\vert w_{ij,0}^N\vert   \leq W,
\end{equation}
almost surely.
\end{enumerate}
\end{assumption}

\begin{assumption}[On the initial data]
~
\begin{enumerate}[label=(\roman*)]\label{assump:initial-data}
\setcounter{enumi}{3}
\item {\bf (Scaling of the initial data)}:\\
Assume that $X_{i,0}^N$ satisfies
\begin{equation}\label{eq:hypothesis-initial-data}
\sup_{N\in \mathbb{N}}\frac{1}{N}\sum_{i=1}^{N}\mathbb{E}|X_{i,0}^N|^{2+\delta}<\infty,
\end{equation}
for some $\delta>0$.
\end{enumerate}
\end{assumption}

Under these assumptions, and guided by the insights from the previous  subsection, we expect that the limit as \(N\to \infty\) of the system of interacting diffusions
\eqref{eq:multi-agent-discrete} is represented, in some sense to be clarified below in Theorem \ref{theo:main}, 
by the following continuum system of non-Markovian McKean-Vlasov SDEs:
\begin{align}
&d X(t,\xi)=\int_0^1\int_{\mathcal{C}^d_t} w_{\gamma}(t,\xi,\xi')  \,K(X(t,\xi),\gamma(t))\,d\mu_{[0,t]}^{\xi'}(\gamma)\,d\xi' dt+ \sqrt{2\nu}\,dW(t,\xi),\quad t\in [0,T],\,\xi\in [0,1],\nonumber\\
&\partial_s q^{\xi,\xi'}_{\gamma}(s,w)+\partial_w(\Gamma(X(s,\xi),\gamma(s),w)q^{\xi,\xi'}_{\gamma}(s,w))=0,\quad s\in [0,t],\,\xi,\xi'\in [0,1],\,\gamma\in \mathcal{C}^d_t,\label{eq:multi-agent-macroscopic-SDE}\\
&X(0,\xi)\sim \mu_0^{\xi},\quad q^{\xi,\xi'}_\gamma(0)=q^{\xi,\xi'}_0,\quad w_{\gamma}(t,\xi,\xi')=\int_{\mathbb{R}} w \, q^{\xi,\xi'}_{\gamma}(t,dw), \quad \xi,\xi'\in [0,1],\nonumber
\end{align}
where  \(\mu^{\xi}_{[0,t]}:={\rm Law}(X(\cdot,\xi)_{[0,t]})\) is the pathwise law of each agent, \((W(\cdot,\xi))_{\xi \in [0,1]}\)
is a continuum family of standard Wiener processes, and \(\mathcal{C}^d_t:=C([0,t],\mathbb{R}^d)\) is endowed with the uniform norm.

Structurally, \eqref{eq:multi-agent-macroscopic-SDE} constitutes a system of path-dependent McKean-Vlasov SDEs for the particle
states, in which the evolution of the state \(X(t,\xi)\) of an agent with label \(\xi\) explicitly depends on the family of
pathwise laws \((\mu_{[0,t]}^{\xi'})_{\xi' \in [0,1]} \subset \mathcal{P}(\mathcal{C}^d_t)\) of all other agents. This interaction
is modulated by the dynamical random coupling weights \(w_\gamma(t,\xi,\xi') \in \mathbb{R}\), which are realized as the
average of the random 
probability measures \(q^{\xi,\xi'}_\gamma(t, dw) \in \mathcal{P}(\mathbb{R})\). Each realization of this stochastic measure, in turn, 
evolves deterministically according to a continuity equation driven by the non-linear plasticity function \(\Gamma\), and departing 
from initial data \(q^{\xi,\xi'}_0 \in \mathcal{P}(\mathbb{R})\). Such initial data corresponds to a probability-graphon,
which characterizes the graph limit of the initial weighted random graphs encoding the initial network topology. 
Representing the network structure via probability measures over the weight space generalizes classical 
graphon theory \cite{LS-06}, similarly as young measures generalize functions, yielding a robust topological setting inherently
suited for weighted networks governed by non-linear co-evolutionary dynamics.

Conceptually, the measure \(q^{\xi,\xi'}_{\gamma}(t,dw)\) represents the probability law of the directed weight associated to an influencing agent \(\xi'\) and a target agent \(\xi\), conditioned to agent $\xi$ having realized the random spacial trajectory \(X(\cdot,\xi)_{[0,t]}\), and agent \(\xi'\) having realized a spatial trajectory \(\gamma \in \mathcal{C}^d_t\). 
The continuity equation governs how this evolving distribution
is dynamically transported over time by the plasticity rule \(\Gamma\), evaluated jointly along the specific interacting
paths \(X(\cdot,\xi)_{[0,t]}\) and \(\gamma\). This can be easily seen by solving the transport equation along the characteristics, which yields the explicit push-forward representation
\begin{equation} \label{eq:flow-map-formulation-weight-distribution-conditioned}
q^{\xi,\xi'}_{\gamma}(t,dw)=\Phi_t(X(\cdot,\xi)_{[0,t]},\gamma,\cdot)_{\#} q^{\xi,\xi'}_0(dw).
\end{equation}
	Taking the expectation with respect to the underlying randomness implicit through the process \(X(\cdot,\xi)_{[0,t]}\) and 
        integrating with respect to the pathwise law \(\mu^{\xi'}_{[0,t]}\) of the process $X(\cdot,\xi')_{[0,t]}$, we obtain the
	deterministic time-dependent probability-graphon \(q_t: [0,1]^2 \to \mathcal{P}(\mathbb{R})\) defined as
	\begin{equation} \label{eq:limiting-probability-graphon-flow-map-formulation-intro}
		q^{\xi,\xi'}_t= \int_{\mathcal{C}_t^d} \mathbb{E}[q^{\xi,\xi'}_{\gamma}(t)] \, d\mu^{\xi'}_{[0,t]}(\gamma) = \Phi_t(\cdot,\cdot,\cdot)_{\#}(\mu^{\xi}_{[0,t]}\otimes \mu^{\xi'}_{[0,t]}\otimes q_0^{\xi,\xi'}), \quad t \in [0,T],
	\end{equation}
	which will be shown to characterize the limiting network structure in the mean-field limit.
Stated informally, our main mean-field limit result reads as follows:

\begin{theorem}\label{theo:main}
	Assume that the interaction kernel \(K\) and the plasticity function \(\Gamma\) satisfy Assumptions
	\ref{assump:interaction-kernel} and \ref{assump:plasticity-function}. Let 
	\((X_{i}^N,w_{ij}^N)_{1\leq i,j\leq N}\) be the unique strong
	solution to the discrete system \eqref{eq:multi-agent-discrete}, starting from independent initial states and
	weights \((X_{i,0}^N,w_{ij,0}^N)_{1\leq i,j\leq N}\) that satisfy Assumptions \ref{assump:weights} and
	\ref{assump:initial-data}. Define the (random) extended empirical measure for the states and empirical probability-graphon for the weights:
	\[
	\begin{aligned}
	\mu^{N,\xi}_{[0,t]} &:= \sum_{i=1}^{N} \mathds{1}_{I_i^N}(\xi) \delta_{X_{i,[0,t]}^N}, \quad && t\in [0,T],\,\xi\in [0,1], \\ 
	q^{N,\xi,\xi'}_{t} &:= \sum_{i,j=1}^N \mathds{1}_{I_i^N\times I_j^N}(\xi,\xi') \delta_{Nw_{ij}^N(t)}, \quad && t\in [0,T],\,\xi,\xi'\in [0,1],
	\end{aligned}
	\]
    Then, up to the extraction of a subsequence and a suitable relabeling of the agents' indices,  $(\mu^{N,\xi}_{[0,t]})_{\xi\in [0,1]}$ and $(q^{N,\xi,\xi'}_t)_{\xi,\xi'\in [0,1]}$ respectively converge, in a suitable sense, to the deterministic family of probability measures
    \((\mu_{[0,t]}^{\xi})_{\xi \in [0,1]} \subset \mathcal{P}(\mathcal{C}^d_t)\) corresponding to the pathwise law of the solution $X(\cdot,\xi)_{[0,t]}$ of the McKean-Vlasov SDEs
    \(\eqref{eq:multi-agent-macroscopic-SDE}_{1}\), and to the deterministic family of
    probability-graphons \((q_t^{\xi,\xi'})_{\xi,\xi'\in [0,1]}\subset \mathcal{P}(\mathbb{R})\), obtained via the relation \eqref{eq:limiting-probability-graphon-flow-map-formulation-intro} from the family \(q^{\xi,\xi'}_{\gamma}\) solving the transport equation \(\eqref{eq:multi-agent-macroscopic-SDE}_2\).
\end{theorem}

A precise, and rigorous statement of this result can be found in Theorem \ref{theo:main-complete}.
The proof of this theorem relies on a suitable {\it propagation of independence} for the multi-agent system \eqref{eq:multi-agent-discrete}, refined {\it stability estimates} of the solutions to the system of path-dependent McKean-Vlasov SDEs \eqref{eq:multi-agent-macroscopic-SDE} with respect to appropriate 
metrics on the joint space of nodes (using fibered measures for the initial distribution of the states) and edges (using probability-graphons for the initial distribution of the weights), and a new {\it compactness resut} of such joint space of nodes and edges, which we call {\it node-edge probability-graphons}. We refer to Section \ref{sec:preliminaries} for more details about the precise functional setting.  

We remark that our framework
allows us to accommodate general independent, non-identically distributed initial states, alongside 
dense sequences of random initial weights. This represents a novel aspect of our work even in the context of static weights, as it derives rigorously the mean-field limit of \eqref{eq:multi-agent-discrete} for a strictly broader class of initial data than those previously considered in the literature. Additionally, since we combine stability estimates and compactness arguments, we can rigorously pass to the limit circumventing the need for any restrictive well-preparedness assumptions on the initial states and weights, contrarily to previous results.

The rest of the paper is organized as follows. 

In Section \ref{sec:preliminaries}, we introduce the functional
setting required to address the mean-field limit. First, we introduce the {\it fibered probability measures} to treat the laws of the agents' states as node decorations and, second, we motivate the need of introducing {\it probability-graphons} to treat edge decorations. We make a comparative study of the various topologies that can be considered upon these spaces, and we set our choice of working distances as those imposing a balanced set of restrictions to ensure relative compactness and stability in law for the macroscopic system \eqref{eq:multi-agent-macroscopic-SDE}. Since node and edge operate simultaneously in the multi-agent system, we define the {\it node-edge probability-graphons} as the product space of fibered probability measures and probability-graphons endowed with a suitable metric structure that is compatible with simultaneous rearrangements of the macroscopic labels on nodes and edges. As mentioned above, the results developed in this section are fundamental as they allow us to obtain the compactness necessary to pass to the limit in the discrete system \eqref{eq:multi-agent-discrete} without imposing any restrictive well-preparedness assumptions on the initial data. In turns, this framework enables us to consider general independent initial data for both the states and weights, which need not be identically distributed. Besides our use, we believe that this section could be of independent interest in the community of large graph limits.

The rest of the sections focus on the proof of Theorem \ref{theo:main}. By introducing an intermediate
system with independent components,  we demonstrate in Section \ref{sec:propagation-of-independence} that the independence assumed for the discrete initial data
propagates for later times asymptotically as \(N \to \infty\).
In Section \ref{sec:well-posedness} we establish the well-posedness of the macroscopic system \eqref{eq:multi-agent-macroscopic-SDE}. 
In Section \ref{sec:stability-estimate}, we derive a stability estimate 
for the macroscopic system with respect to the initial fibered law and probability-graphon.
In Section \ref{sec:main-theorem}, we state and rigorously prove our main result, Theorem \ref{theo:main-complete}.
Finally, in Section \ref{sec:conclusions}, we highlight the flexibility of our main result by showcasing a wide variety of special cases treated in the literature. We also demonstrate how our techniques easily extend to systems with more complex interaction structures, such as those with multiplicative noise in the agent's equation ({\it e.g.}, modulated by an additional co-evolutionary network), or also bounded additive noise in the weights' equations.
In Appendix \ref{appendix:SDEs} we summarize some fundamental concepts of stochastic processes and path-dependent SDEs, which we will use in this paper.

\medskip

\textbf{Notation:}\\
Throughout the paper, we employ the following notation. For any \(N\in\mathbb{N}\), we denote \(\llbracket 1,N\rrbracket=\{1,\ldots,N\}\). The space of continuous functions \(C([0,t];\mathbb{R}^d) \) endowed with the uniform norm is denoted by \(\mathcal{C}^d_t\). Fixing a finite time horizon \(T>0\), for any \(\gamma,\bar{\gamma} \in \mathcal{C}^d_T\), we define the uniform norm up to time \(t\) as
\[
\Vert \gamma- \bar{\gamma} \Vert_{\ast, t}:= \max_{s\in [0,t]} \vert \gamma(s)-\bar{\gamma}(s)\vert, \quad t\in [0,T].
\]
The restriction of a path \(\gamma\) to the interval \([0,t]\) is denoted by \(\gamma_{[0,t]}\). 

Given a Polish space \(\mathcal{X}\), let \(\mathcal{M}_+(\mathcal{X})\) denote the space of finite, non-negative measures on \((\mathcal{X},\mathcal{B}(\mathcal{X}))\), where \(\mathcal{B}(\mathcal{X})\) is the Borel \(\sigma\)-algebra of \(\mathcal{X}\). We denote by \(\mathcal{P}(\mathcal{X})\) the subspace of Borel probability measures on \(\mathcal{X}\), and by \(\mathcal{P}_p(\mathcal{X})\) the subset of probability measures with finite \(p\)-th order moments. Given \(\mu \in \mathcal{P}(\mathcal{X})\), another Polish space \(\mathcal{Y}\), and a measurable map \(f:\mathcal{X}\to \mathcal{Y}\), the push-forward measure \(f_{\#}\mu \in \mathcal{P}(\mathcal{Y})\) is defined as 
\(f_{\#}\mu(B)=\mu(f^{-1}(B))\) for every \( B\in \mathcal{B}(\mathcal{Y}).\)
Equivalently, it satisfies
\[
\int_{\mathcal{X}} \varphi(f(x))\,d\mu(x)=\int_{\mathcal{Y}} \varphi(y)\,d(f_{\#}\mu)(y), 
\]
for every continuous and bounded function \(\varphi \in C_b(\mathcal{Y})\). The space \(\mathcal{P}(\mathcal{X})\) is endowed 
with the narrow topology, which is characterized sequentially as follows: a sequence \((\mu_n)_{n\in \mathbb{N}}\subset \mathcal{P}(\mathcal{X})\) converges narrowly to \(\mu \in \mathcal{P}(\mathcal{X})\) if 
\[\lim_{n\to \infty} \int_{\mathcal{X}} \varphi(x)\,d\mu_n(x)=\int_{\mathcal{X}} \varphi(x)\,d\mu(x), \quad \forall \varphi \in C_b(\mathcal{X}).\] 
The narrow topology is metrized by the bounded-Lipschitz distance \(d_{\mathrm{BL}}\) (see, for instance \cite{Dudley-02}), which for any two probability measures \(\mu, \nu \in \mathcal{P}(\mathcal{X})\) is defined as 
\[d_{\mathrm{BL}}(\mu,\nu):=\sup_{\Vert \varphi\Vert_{\mathrm{BL}} \leq 1} \int_{\mathcal{X}} \varphi(x)\,d(\mu-\nu)(x),\]
where \(\mathrm{BL}(\mathcal{X})\) is the space of bounded-Lipschitz functions on \(\mathcal{X}\) endowed with the norm \(\Vert \varphi \Vert_{\mathrm{BL}}=\max\{\Vert \varphi \Vert_{\infty}, [\varphi]_{\textrm{Lip}}\}\). We will denote by \(\mathrm{BL}_1(\mathcal{X})\) the space of functions \(\varphi \in \mathrm{BL}(\mathcal{X})\) such that \(\Vert \varphi \Vert_{\mathrm{BL}}\leq 1\).

The space \(\mathcal{P}_p(\mathcal{X})\) is endowed with the \(p\)-Wasserstein distance, defined as
\[W_p(\mu,\nu):=\left( \inf_{\pi \in \Pi(\mu,\nu)} \int_{\mathcal{X}\times \mathcal{X}} d(x,y)^p \,d\pi(x,y) \right)^{1/p},\]
where \(\Pi(\mu,\nu)\) is the set of transport plans between \(\mu\) and \(\nu\); that is, the set of probability measures on \(\mathcal{X}\times \mathcal{X}\) with marginals \(\mu\) and \(\nu\), respectively. The above infimum is strictly achieved, and any minimizer is called an optimal transport plan between \(\mu\) and \(\nu\). The space of optimal transport plans between \(\mu\) and \(\nu\) is denoted by \(\Pi_o(\mu,\nu)\). For a comprehensive treatment of optimal transport theory and the analytical properties of Wasserstein spaces, we refer to the monographs \cite{AGS-05,Santambrogio-15,Villani-09}.

Given a continuous stochastic process \(X:[0,T]\times \Omega\to \mathbb{R}^d\) adapted to some filtered probability space \((\Omega, \mathcal{F}, (\mathcal{F}_t)_{t \in [0,T]}, \mathbb{P})\), note that the map \(t\mapsto X(t,\omega)\) is continuous for \(\mathbb{P}\)-a.e. \(\omega \in \Omega\). Consequently, we can identify (and we shall often do) the stochastic process with a random variable taking values in the space of continuous functions, \(X_{[0,t]}:\Omega \to \mathcal{C}_t^d\). The law of \(X_{[0,t]}\) viewed in this functional space is referred to as its \emph{pathwise law}, denoted by \(\mathrm{Law}(X_{[0,t]})=X_{[0,t]\#}\mathbb{P}\in \mathcal{P}(\mathcal{C}_t^d)\).
For a comprehensive introduction to the foundational theory of stochastic processes, we refer the reader to \cite{Oksendal-03} and also Appendix \ref{appendix:SDEs}. 

Finally, for any \(1\leq p\leq \infty\) we denote by \(L^p(\Omega,\mathcal{X})\) the set of measurable maps (or equivalently random variables) \(f:\Omega \to \mathcal{X}\) such that 
\(\mathbb{E}[d_\mathcal{X}^p(f,x_0)]<\infty\), where \(x_0 \in \mathcal{X}\) is any fixed reference point. The set \(L^p(\Omega,\mathcal{X})\) is a metric space endowed with the distance
\[
d_{L^p(\Omega,\mathcal{X})}(f,g) := \left(\int_{\Omega} d_{\mathcal{X}}^p(f(\omega),g(\omega)) \,d\mathbb{P}(\omega)\right)^{1/p}\equiv\left(\mathbb{E}\left[ d_{\mathcal{X}}^p(f,g) \right]\right)^{1/p}.
\]
For a stochastic process regarded as a random variable \(X:\Omega\to \mathcal{C}_T^d\) valued in the Banach space \(\mathcal{X}=\mathcal{C}_T^d\), the above metric is induced by a norm, more precisely,
\[\Vert X \Vert_{L^p(\Omega,\mathcal{C}_T^d)}:=\left(\int_{\Omega}\max_{t\in [0,T]}|X(t,\omega)|^p\,d\mathbb{P}(\omega)\right)^{1/p}\equiv\left( \mathbb{E} \Vert X \Vert^p_{\ast,T} \right)^{1/p}.\]

\section{Preliminaries} \label{sec:preliminaries}
In this section, we introduce the functional setting required to address the convergence of \eqref{eq:multi-agent-discrete} to \eqref{eq:multi-agent-macroscopic-SDE}.
As explained in the introduction, we 
must handle two distinct components in our mean-field limit: the particle states and the coupling weights.

Regarding the particle states, the extended empirical measures
\((\mu^{N,\xi}_{[0,T]})_{\xi \in [0,1]}\) will be shown to converge
towards the family of pathwise laws
\((\mu^{\xi}_{[0,T]})_{\xi \in [0,1]} \subset \mathcal{P}(\mathcal{C}_T^d)\) of the processes solving the limiting equation 
\eqref{eq:multi-agent-macroscopic-SDE}. Therefore, the first subsection explores the topological properties of the space of such families of probability measures, which can be identified with the space of \emph{fibered probability measures} \cite{PP-23}.

For the network structure, the empirical probability-graphons \((q^{N,\xi,\xi'}_t)_{\xi,\xi' \in [0,1]}\) will be shown to 
converge to a family \((q^{\xi,\xi'}_t)_{\xi,\xi' \in [0,1]}\subset \mathcal{P}(\mathbb{R})\) that will capture the limiting 
network structure. We rely on the theory of \emph{probability-graphons} \cite{ADW-25}, which we review in
the second subsection.

Finally, in the third subsection, viewing the particle states as the nodes of the network
and the weights as the edges, we integrate both frameworks to introduce the space of \emph{node-edge probability-graphons}.
We equip this space with a suitable topology that, upon identifying objects up to common measure-preserving
transformations on the nodes and edges, provides a joint compactness criterion crucial for the mean-field limit result.
This criterion extends the compactness results for probability-graphons obtained in \cite[Section 5]{ADW-25} to the node-edge setting and is of independent interest.

\subsection{Fibered probability measures}
The family of pathwise laws of the processes solving the limiting equation \eqref{eq:multi-agent-macroscopic-SDE} will require
a suitable measurability property, which is captured by the following definition. We refer to \cite{PP-23} for further details.
\subsubsection{(Labeled) fibered probability measures}
\begin{definition}[Borel family of probability measures]
Given a measurable space \((\Omega,\mathcal{A})\) and a Polish space \(\mathcal{X}\), we say that a family \((\mu^{\omega})_{\omega \in \Omega}\subset \mathcal{P}(\mathcal{X})\) is a Borel family of probability measures if the map
\[\omega \in \Omega\mapsto \mu^{\omega}(B) \in [0,1]\] is measurable for every Borel set \(B\subset \mathcal{X}\).
\end{definition}
The following result gives several characterization of Borel families of probability measures. 
\begin{proposition}[{\cite[Proposition 2.14]{PP-23}}] \label{prop:characterization-random-prob-measures}
	Let \((\Omega,\mathcal{A})\) be a measurable space and \(\mathcal{X}\) be a Polish space. Given a family \((\mu^{\omega})_{\omega \in \Omega}\subset \mathcal{P}(\mathcal{X})\), the following are equivalent:
	\begin{enumerate} [label=(\roman*)] 
		\item \((\mu^{\omega})_{\omega \in \Omega}\) is a Borel family of probability measures.
		\item The map \[\omega \in \Omega\mapsto \int_{\mathcal{X}} \varphi(x)d\mu^{\omega}(x)\in \mathbb{R}\] is measurable for every bounded and measurable \(\varphi:\mathcal{X}\to \mathbb{R}\).
		\item The map \(\omega \in \Omega\mapsto \mu^{\omega}\in \mathcal{P}(\mathcal{X})\) is a random probability measure, {\it i.e.}, it is measurable when \(\mathcal{P}(\mathcal{X})\) is endowed with the narrow topology.
	\end{enumerate}
\end{proposition}
    In other words, the notion of a Borel family of probability measures is equivalent to that of a Young measure, a random probability measure or a Markov transition kernel.
    
    Another equivalent way to capture the measurability property of a Borel family of probability measures in a more compact form
is through
the following notion.
\begin{definition}[Fibered probability measures] \label{def:fibered-probability-measures}
	Let \(\Omega\) and \(\mathcal{X}\) be Polish spaces, we 
	define the space of fibered probability measures on \(\mathcal{X}\times \Omega\) with base
	measure \(\nu\in \mathcal{P}(\Omega)\) as 	
	\[
		\mathcal{P}_{\nu}(\mathcal{X}\times \Omega)=\{\mu \in \mathcal{P}(\mathcal{X}\times \Omega): \pi_{\Omega \#}\mu=\nu\},
	\]
	where \(\pi_{\Omega}:\mathcal{X}\times \Omega\to \Omega\) is the projection on the second component,
	\(\pi_{\Omega}(x,\omega)=\omega\).
\end{definition}
The classical disintegration theorem \cite[Theorem 5.3.1]{AGS-05} allows us to identify every fibered probability measure
with its Borel family of disintegrated measures.
\begin{theorem}[Disintegration theorem] \label{theorem:disintegration}
	Let \(\Omega\)  and \(\mathcal{X}\) be Polish spaces, and let \(\nu \in \mathcal{P}(\Omega)\). Consider any
	\(\mu \in \mathcal{P}_{\nu}(\mathcal{X}\times \Omega)\). Then, there exists a unique (up to \(\nu\)-a.e. equality)
	Borel family \((\mu^{\omega})_{\omega \in \Omega}\subset \mathcal{P}(\mathcal{X})\) such that 
     	\[
     		\int_{\mathcal{X}\times \Omega} \varphi(x,\omega)d\mu(x,\omega)=\int_{\Omega} \int_{\mathcal{X}} \varphi(x,\omega)d\mu^{\omega}(x)d\nu(\omega),
     	\]
	for every bounded and measurable function \(\varphi:\mathcal{X}\times \Omega\to \mathbb{R}\). Conversely, given any Borel
	family \((\mu^{\omega})_{\omega \in \Omega}\subset \mathcal{P}(\mathcal{X})\), we can associate a unique fibered probability
	measure \(\mu \in \mathcal{P}_{\nu}(\mathcal{X}\times \Omega)\) so that the above formula holds true, and it will be denoted
	by \(\mu(x,\omega)=\mu^{\omega}(x)\otimes d\nu(\omega)\). 
\end{theorem}

\begin{remark}[On the choice of the label space \(\Omega\) and the reference measure \(\nu\)]
	While the previous definitions hold for a general Polish space \(\Omega\), from this point forward we identify
	the label space \(\Omega\)  with the unit interval \([0,1]\) equipped with the Borel \(\sigma\)-algebra
	and the Lebesgue measure \(\nu=d\xi_{[0,1]}\). This choice is standard in the theory of graphons
	\cite{Lovasz-12}, where the unit interval serves as the canonical domain to parametrize a continuum of nodes.
	Accordingly, we replace the abstract variable \(\omega\) with the macroscopic label \(\xi\) in our subsequent analysis.
	The family of disintegrations of a fibered probability measure \(\mu \in \mathcal{P}_{\nu}(\mathcal{X}\times [0,1])\)
	will thus be denoted by \((\mu^{\xi})_{\xi \in [0,1]}\).
\end{remark}
The space \(\mathcal{P}_{\nu}(\mathcal{X}\times [0,1])\) can be endowed with several topologies. Since it is a closed subset of
\(\mathcal{P}(\mathcal{X}\times [0,1])\) under the narrow topology, it naturally inherits this topology, which is metrized by
the bounded-Lipschitz distance. For any two fibered probability measures 
\(\mu,\bar{\mu}\in \mathcal{P}_{\nu}(\mathcal{X}\times [0,1])\), this distance can be expressed in terms
of the disintegrated measures as follows:
\[
	d_{\text{BL}}(\mu,\bar{\mu})=\sup_{\varphi \in \text{BL}_1(\mathcal{X}\times [0,1])}\int_{0}^{1}\int_{\mathcal{X}} \varphi(x,\xi)d(\mu^{\xi}-\bar{\mu}^{\xi})(x)d\xi.
\]
Regarding a fibered probability measure \(\mu \in \mathcal{P}_{\nu}(\mathcal{X}\times [0,1])\) as a random probability measure
\(\xi \in [0,1] \mapsto \mu^{\xi} \in \mathcal{P}(\mathcal{X})\), this suggests considering the following family of
fibered distances.
 \begin{definition}
	Given any two fibered probability measures \(\mu,\bar{\mu} \in \mathcal{P}_{\nu}(\mathcal{X}\times [0,1])\), we define the following fibered distances between \(\mu\) and \(\bar{\mu}\): 
  	\[
  		\begin{aligned}
  		&d_{\text{BL},p}(\mu,\bar{\mu}):=\left( \int_{0}^{1} d_{\text{BL}}(\mu^{\xi},\bar{\mu}^{\xi})^p d\xi \right)^{1/p}, \quad 1\leq p<\infty, \\ 
  		&d_{\text{BL},\infty}(\mu,\bar{\mu}):= \esssup_{\xi \in [0,1]} d_{\text{BL}}(\mu^{\xi},\bar{\mu}^{\xi}).
  		\end{aligned}  
  	\]
\end{definition}
Since the bounded-Lipschitz distance between any two probability measures is bounded by 2, the distances \(d_{\text{BL},p}\)
are always finite, and for \(1\leq p <\infty\), they induce the same topology on \(\mathcal{P}_{\nu}(\mathcal{X}\times [0,1])\).
It is straightforward to verify that the narrow topology is coarser than the topology induced by these fibered distances. In
particular, we have the following hierarchy of distances:
\[
	d_{\text{BL}}(\mu,\bar{\mu})\leq d_{\text{BL},p}(\mu,\bar{\mu})\leq d_{\text{BL},q}(\mu,\bar{\mu}) \leq d_{\text{BL},\infty}(\mu,\bar{\mu}), \quad 1\leq p\leq q<\infty.
\]
For a proof of the first inequality, see \cite[Proposition 2.19]{APP-24-arxiv}. Furthermore, the Rademacher sequence
(see \cite[Example 3.15]{PP-23}) provides a standard example demonstrating that the narrow topology on
\(\mathcal{P}_{\nu}(\mathcal{X}\times [0,1])\) is strictly coarser than the topology induced by the fibered
distances \(d_{\text{BL},p}\).

Another interesting distance arises from the cut norm, which is widely used in the theory of graph limits \cite{Lovasz-12}.
The next subsection will reveal the utility of this choice for probability-graphons, as it naturally extends the classical 
compactness theorems of graphons to the probability-graphon framework \cite[Section 5]{ADW-25}. Motivated by this,
we define a cut distance for fibered probability measures as follows.
\begin{definition}[Bounded-Lipschitz cut distance]
	Given any two fibered probability measures \(\mu,\bar{\mu} \in \mathcal{P}_{\nu}(\mathcal{X}\times [0,1])\), we define
	the bounded-Lipschitz cut distance as 
	\begin{equation} \label{eq:bounded-Lipschitz-cut-fibered-probability-measures}
		d_{\text{BL},\square}(\mu,\bar{\mu}):= \sup_{S\subset [0,1]} d_{\text{BL}} \left( \int_{S} \mu^{\xi} d\xi, \int_{S} \bar{\mu}^{\xi} d\xi \right), 
	\end{equation}
	where the supremum is taken over all measurable sets \(S\subset [0,1]\), and 
	\(\int_{S}\mu^{\xi}d\xi \in \mathcal{M}_+(\mathcal{X})\) denotes the finite positive measure defined
	by 
	\[
		\int_{\mathcal{X}} \varphi(x)d\left( \int_{S} \mu^{\xi} d\xi \right)(x)=\int_{S} \int_{\mathcal{X}} \varphi(x)d\mu^{\xi}(x)d\xi
	\]
	for every bounded measurable function \(\varphi:\mathcal{X}\to \mathbb{R}\). 
\end{definition}

The previous definition is a natural generalization of the classical cut norm on the space \(L^1([0,1];\mathbb{R})\),
which is given by
\[
	\Vert f \Vert_{\square}=\sup_{S\subset [0,1]} \left\vert \int_{S} f(\xi)d\xi \right\vert.
\] 
It is a standard fact that, since we are in the unit interval, this norm is equivalent to the \(L^1\) norm, satisfying
\[
	\Vert f \Vert_{\square}\leq \Vert f\Vert_{L^1}\leq 2 \Vert f \Vert_{\square}.
\] 
Thus, by interchanging the supremum over \(\Phi \in \text{BL}_1(\mathcal{X})\) with the supremum over measurable sets
\(S\subset [0,1]\) in \eqref{eq:bounded-Lipschitz-cut-fibered-probability-measures}, and applying the above equivalence,
we obtain
\begin{equation} \label{eq:equivalence-cut-distance-Pettis-norm}
	d_{\text{BL},\square}(\mu,\bar{\mu})\leq  \sup_{\Phi \in \text{BL}_1(\mathcal{X})} \Vert \Phi_{\mu}-\Phi_{\bar{\mu}} \Vert_{L^1} = \sup_{\Phi \in \text{BL}_1(\mathcal{X})} \int_{0}^{1} \left\vert  \int_{\mathcal{X}} \Phi(x)d(\mu^{\xi}-\bar{\mu}^{\xi})(x) \right\vert d\xi \leq 2 d_{\text{BL},\square}(\mu,\bar{\mu}),
\end{equation} 
where, for any \(\Phi \in \text{BL}_1(\mathcal{X})\), we define the function \(\Phi_{\mu}:[0,1]\to \mathbb{R}\) as
\[
	\Phi_{\mu}(\xi):= \int_{\mathcal{X}} \Phi(x)d\mu^{\xi}(x), \quad \xi \in [0,1].
\]
Furthermore, since the supremum of an integral is bounded above by the integral of the supremum,
it follows that for any \(\mu,\bar{\mu} \in \mathcal{P}_{\nu}(\mathcal{X}\times [0,1])\),
\[
	d_{\text{BL},\square}(\mu,\bar{\mu})\leq d_{\text{BL},1}(\mu,\bar{\mu}).
\]
This inequality implies that the topology induced by the bounded-Lipschitz cut distance is a priori coarser
than the one induced by \(d_{\text{BL},1}\). While a coarser topology is favorable for establishing compactness results, 
as is the case for classical graphons, we demonstrate in Corollary \ref{corollary:d-BL-cut-equivalent-d-BL-1} that these
two distances actually induce equivalent topologies on \(\mathcal{P}_{\nu}(\mathcal{X}\times [0,1])\).

\begin{remark}[Relating \(d_{\text{BL},1}\) and \(d_{\text{BL},\square}\) to the Bochner and Pettis norms]
	Using the theory of integration of vector-valued maps, we can see that the distances \(d_{\text{BL},1}\) and
	\(d_{\text{BL},\square}\) are induced, respectively, by the classical Bochner and Pettis norms in the space
	\(L^1([0,1];\text{BL}(\mathcal{X})^*)\). To see this, recall that the space of probability measures \(\mathcal{P}(\mathcal{X})\)
	can be seen as a subset of the dual space \(\text{BL}(\mathcal{X})^*\) via the embedding
	\(\mu \in \mathcal{P}(\mathcal{X})\mapsto F_{\mu}\in \text{BL}(\mathcal{X})^*\) given by  
	\[
	F_{\mu}(\varphi)= \int_{\mathcal{X}} \varphi(x)d\mu(x), \quad \varphi \in \text{BL}(\mathcal{X}).
	\]    
	Thus, a random probability measure \(\mu:[0,1]\to \mathcal{P}(\mathcal{X})\) can be viewed as a vector-valued map with values
	in \(\text{BL}(\mathcal{X})^*\). Since its range is contained in the separable space
	\(\mathcal{P}(\mathcal{X})\subset \text{BL}(\mathcal{X})^*\), the Pettis measurability theorem \cite[Theorem 2.2]{DU-77}
	ensures that the map is strongly measurable. Moreover, random probability measures are Bochner integrable since they take
	values in the bounded unit ball of \(\text{BL}(\mathcal{X})^*\). From this perspective, we can view
	\(\mathcal{P}_{\nu}(\mathcal{X}\times [0,1])\) as a subset of \(L^1([0,1];\text{BL}(\mathcal{X})^*)\). Recalling the
	definitions of the Bochner \(L^1\) and Pettis norms for a measurable map \(f:[0,1]\to \text{BL}(\mathcal{X})^*\),
	\[
		\Vert f \Vert_{L^1}=\int_{0}^{1} \Vert f(\xi) \Vert_{\text{BL}(\mathcal{X})^*} d\xi, \quad \Vert f \Vert_{\text{Pettis}}=\sup_{\Phi \in \text{BL}_1} \int_{0}^{1} \left\vert \langle f(\xi),\Phi \rangle \right\vert d\xi,
	\]
	we have that for every \(\mu,\bar{\mu} \in \mathcal{P}_{\nu}(\mathcal{X}\times [0,1])\),
	\[
		d_{\text{BL},1}(\mu,\bar{\mu})= \Vert \mu-\bar{\mu} \Vert_{L^1}, \quad \quad \sup_{\Phi \in \text{BL}_1 } \Vert \Phi_{\mu}-\Phi_{\bar{\mu}} \Vert_{L^1} = \Vert \mu-\bar{\mu} \Vert_{\text{Pettis}}.
	\]
	In a general infinite-dimensional Banach space \(X\), the Pettis norm is strictly weaker than the Bochner norm in
	\(L^1([0,1];X^*)\) (see \cite[Example 1]{DG-93}). However, we will see that as a consequence of Lemma
	\ref{lemma:equivalence-bounded-Lipschitz-cut-and-BL,p}, these two norms induce two distances in the space
	of fibered probability measures that are topologically equivalent.
\end{remark}

Since fibered probability measures share a fixed marginal on the label space, the tightness of a subset
\(\mathcal{K}\subset \mathcal{P}_{\nu}(\mathcal{X}\times [0,1])\) depends entirely on the tightness of its marginals on
the state space. To formalize this, we introduce the averaging map 
\(M: \mathcal{P}_{\nu}(\mathcal{X}\times [0,1]) \to \mathcal{P}(\mathcal{X})\), which associates every fibered measure
\(\mu\) with its first marginal, denoted by \(M_{\mu} := \pi_{\mathcal{X} \#}\mu\).
In terms of its disintegration, this averaged measure is explicitly given by
\[
	M_{\mu} = \int_{0}^{1} \mu^{\xi}d\xi.
\]
\begin{proposition} \label{prop:equivalence-tightness}
	Let \(\mathcal{K} \subset \mathcal{P}_{\nu}(\mathcal{X}\times [0,1])\). The following statements are equivalent:
	\begin{enumerate} [label=(\roman*)]
		\item The set \(\mathcal{K}\) is tight in \(\mathcal{P}(\mathcal{X}\times [0,1])\).
		\item The set of averaged measures \(M_{\mathcal{K}}:=\{M_{\mu} : \mu \in \mathcal{K}\} \subset \mathcal{P}(\mathcal{X})\) is tight.
		\item For every \(\varepsilon >0\), there exists a compact set \(K_{\varepsilon} \subset \mathcal{X}\) such that
		\[
			\nu(\{\xi \in [0,1]: \mu^{\xi}(\mathcal{X}\setminus K_{\varepsilon}) \geq \varepsilon\}) \leq \varepsilon, \quad \forall \mu \in \mathcal{K}.
		\]
	\end{enumerate}
\end{proposition}
\begin{proof}
	\(\diamond\) \textsc{Step 1}: Equivalence between \((i)\) and \((ii)\).\\ 
	If \(M_{\mathcal{K}}\) is tight, then given any \(\varepsilon>0\) there exists a compact set \(K_{\varepsilon}\subset \mathcal{X}\)
	such that \(M_{\mu}(\mathcal{X}\setminus K_{\varepsilon})\leq \varepsilon\) for every \(\mu \in \mathcal{K}\). Thus, for every
	\(\mu \in \mathcal{K}\),
	\[
		\mu((\mathcal{X}\setminus K_{\varepsilon})\times [0,1])=\int_{0}^{1} \mu^{\xi}(\mathcal{X}\setminus K_{\varepsilon})d\xi = M_{\mu}(\mathcal{X}\setminus K_{\varepsilon})\leq \varepsilon,
	\]
	and therefore \(\mathcal{K}\) is tight as a subset of \(\mathcal{P}(\mathcal{X}\times [0,1])\). Conversely, if \(\mathcal{K}\) is
	tight as a subset of \(\mathcal{P}(\mathcal{X}\times [0,1])\), then given any \(\varepsilon>0\) there exists a compact set
	\(K_{\varepsilon}\subset \mathcal{X}\times [0,1]\) such that 
	\(\mu((\mathcal{X}\times [0,1])\setminus K_{\varepsilon})\leq \varepsilon\) for every \(\mu \in \mathcal{K}\).
	Since the projection of a compact set is compact, the set \(\pi_{\mathcal{X}}(K_{\varepsilon})\subset \mathcal{X}\) is compact,
	and since \((\mathcal{X}\backslash \pi_{\mathcal{X}}(K_{\varepsilon}))\times [0,1]\subset (\mathcal{X}\times [0,1])\backslash K_{\varepsilon}\), for every \(\mu \in \mathcal{K}\) we have that
	\[
		M_{\mu}(\mathcal{X}\setminus \pi_{\mathcal{X}}(K_{\varepsilon}))\leq  \mu((\mathcal{X}\times [0,1])\setminus K_{\varepsilon})\leq \varepsilon,
	\]
	thus, \(M_{\mathcal{K}}\) is tight in \(\mathcal{P}(\mathcal{X})\).

	\medskip

	\(\diamond\) \textsc{Step 2}: Equivalence between \((ii)\) and \((iii)\).\\
	Suppose that \(M_{\mathcal{K}}\) is tight. Then, for every \(\varepsilon >0\), there exists a compact set 
	\(K_{\varepsilon} \subset \mathcal{X}\) such that \(M_{\mu}(\mathcal{X}\setminus K_{\varepsilon}) < \varepsilon^2\)
	for every \(\mu \in \mathcal{K}\). Thus, we have that 
	\[
		\varepsilon^2\geq \int_{0}^{1} \mu^{\xi} (\mathcal{X}\setminus K_{\varepsilon}) d\xi \geq \int_{A_{\mu}} \mu^{\xi} (\mathcal{X}\setminus K_{\varepsilon}) d\xi \geq \nu(A_{\mu}) \varepsilon,
	\]
	where \(A_{\mu}:=\{\xi \in [0,1]: \mu^{\xi}(\mathcal{X}\setminus K_{\varepsilon}) \geq \varepsilon\}\). This shows that
	\(\nu(A_{\mu}) \leq \varepsilon\). Conversely, suppose that for every \(\varepsilon >0\), there exists a compact set
	\(K_{\varepsilon} \subset \mathcal{X}\) such that
	\[
		\nu(\{\xi \in [0,1]: \mu^{\xi}(\mathcal{X}\setminus K_{\varepsilon}) \geq \varepsilon\}) \leq \varepsilon, \quad \forall \mu \in \mathcal{K}.
	\]
	Then,
	\[
		M_{\mu}(\mathcal{X}\setminus K_{\varepsilon}) = \int_{0}^{1} \mu^{\xi}(\mathcal{X}\setminus K_{\varepsilon}) d\xi \leq \varepsilon + \nu(\{\xi \in [0,1]: \mu^{\xi}(\mathcal{X}\setminus K_{\varepsilon}) \geq \varepsilon\}) \leq 2\varepsilon\, \quad \forall \mu \in \mathcal{K},
	\]
	for every \(\mu \in \mathcal{K}\), which shows that \(M_{\mathcal{K}}\) is tight. 
\end{proof}
The following lemma shows that, on tight subsets of \(\mathcal{P}_{\nu}(\mathcal{X}\times [0,1])\), 
the distance \(d_{\text{BL},1}\) can be controlled by \(d_{\text{BL},\square}\) up to an arbitrarily
small additive error \(\varepsilon\), at the cost of an \(\varepsilon\)-dependent multiplicative constant
that blows up as \(\varepsilon\) goes to zero. 
\begin{lemma} \label{lemma:equivalence-bounded-Lipschitz-cut-and-BL,p}
	Let \(\mathcal{A}\) and \(\mathcal{B}\) be tight subsets of \(\mathcal{P}_{\nu}(\mathcal{X}\times [0,1])\). 
	Then, for every \(\varepsilon>0\), there exists a constant \(\mathcal{N}_{\varepsilon} > 0\) such that for all \(\mu \in \mathcal{A}\)
	and \(\bar{\mu} \in \mathcal{B}\), 
	\[
		d_{\text{BL},1}(\mu,\bar{\mu})\leq 2\mathcal{N}_{\varepsilon} d_{\text{BL},\square}(\mu,\bar{\mu}) + \varepsilon.
	\]
\end{lemma}
\begin{proof}
	Since the sets \(\mathcal{A}\) and \(\mathcal{B}\) are tight, given any \(\varepsilon>0\), there exists a compact
	set \(K_{\varepsilon}\subset \mathcal{X}\) such that 
	\[
		\int_{0}^{1} \mu^{\xi}(\mathcal{X}\setminus K_{\varepsilon}) d\xi < \frac{\varepsilon}{8}, \quad  \forall \mu \in \mathcal{A}\cup \mathcal{B}.
	\]
	Consider the family of restricted test functions
	\(\mathcal{F}_{K_\varepsilon} = \{ \Phi\vert_{K_\varepsilon} : \Phi \in \text{BL}_1(\mathcal{X}) \}\). Since
	\(\mathcal{F}_{K_\varepsilon}\) is a uniformly bounded and equicontinuous family of functions, the Arzelà-Ascoli
	theorem guarantees that it is relatively compact in the uniform topology of \(C(K_\varepsilon)\). Specifically,
	we can extract a finite \((\varepsilon/4)\)-net 
	\(\{\psi_1,\dots,\psi_{\mathcal{N}_{\varepsilon}}\} \subset \mathcal{F}_{K_\varepsilon}\). Because each \(\psi_i\) belongs to
	\(\mathcal{F}_{K_\varepsilon}\), there must exist functions \(\Phi_i \in \text{BL}_1(\mathcal{X})\) such that
	\(\Phi_i\vert_{K_\varepsilon} = \psi_i\). This guarantees that for any \(\Phi \in \text{BL}_1(\mathcal{X})\),
	there is an index \(i \in \llbracket 1,\mathcal{N}_{\varepsilon}\rrbracket\) such that
	\[
		\sup_{x \in K_\varepsilon} \vert \Phi(x) - \Phi_i(x)\vert \leq \varepsilon/4.
	\]
	Now, for any \(\mu \in \mathcal{A}\), \(\bar{\mu}\in \mathcal{B}\), any \(\xi \in [0,1]\), and any test function
	\(\Phi \in \text{BL}_1(\mathcal{X})\), we compare the integral of \(\Phi\) with its \((\varepsilon/4)\)-close
	approximation \(\Phi_i\) over the full space \(\mathcal{X}\). By splitting the difference over
	\(K_\varepsilon\) and its complement, and noting that \(\Vert\Phi - \Phi_i\Vert_\infty \leq 2\), we get
	\[
		\begin{aligned}
			\left\vert \int_{\mathcal{X}} \Phi \, d(\mu^{\xi}-\bar{\mu}^{\xi}) \right\vert &\leq \left\vert \int_{\mathcal{X}} \Phi_i \,d(\mu^{\xi}-\bar{\mu}^{\xi}) \right\vert + \left\vert \int_{K_\varepsilon} (\Phi-\Phi_i) \,d(\mu^{\xi}-\bar{\mu}^{\xi}) \right\vert + \left\vert \int_{\mathcal{X} \setminus K_\varepsilon} (\Phi-\Phi_i) \,d(\mu^{\xi}-\bar{\mu}^{\xi}) \right\vert \\
			&\leq \left\vert \int_{\mathcal{X}} \Phi_i \,d(\mu^{\xi}-\bar{\mu}^{\xi}) \right\vert + \frac{\varepsilon}{4}\left(\mu^\xi(K_\varepsilon) + \bar{\mu}^\xi(K_\varepsilon)\right) + 2\left(\mu^\xi(\mathcal{X} \setminus K_\varepsilon) + \bar{\mu}^\xi(\mathcal{X} \setminus K_\varepsilon)\right).
		\end{aligned}
	\]
	Since \(\mu^\xi\) and \(\bar{\mu}^\xi\) are probability measures, their maximum possible mass on \(K_\varepsilon\) is 1.
	Furthermore, because taking the supremum over all \(\Phi\) simply means finding the worst-case match among our finite net,
	we can bound the supremum for a fixed \(\xi\) by taking the maximum over all \(i\):
	\[
		\sup_{\Phi \in \text{BL}_1(\mathcal{X})} \left\vert \int_{\mathcal{X}} \Phi \,d(\mu^{\xi}-\bar{\mu}^{\xi}) \right\vert \leq \max_{i \in \llbracket 1,\mathcal{N}_{\varepsilon} \rrbracket} \left\vert \int_{\mathcal{X}} \Phi_i \,d(\mu^{\xi}-\bar{\mu}^{\xi}) \right\vert + \frac{\varepsilon}{2} + 2\mu^\xi(\mathcal{X} \setminus K_\varepsilon) + 2\bar{\mu}^\xi(\mathcal{X} \setminus K_\varepsilon).
	\]
	We now integrate this estimate over \(\xi \in [0,1]\), and due to the tightness condition, we get  
	\[
		d_{\text{BL},1}(\mu,\bar{\mu}) \leq \int_{0}^{1} \max_{i \in \llbracket 1,\mathcal{N}_{\varepsilon} \rrbracket} \left\vert \int_{\mathcal{X}} \Phi_i(x) \,d(\mu^{\xi}-\bar{\mu}^{\xi})(x) \right\vert d\xi + \varepsilon. 
	\]
	Finally, we bound the maximum by the sum of all elements and use the equivalence \eqref{eq:equivalence-cut-distance-Pettis-norm}
	to bound the integral of each term by \(2d_{\text{BL},\square}(\mu,\bar{\mu})\)
	\[
	\begin{aligned}
		\int_{0}^{1} \max_{i \in \llbracket 1,\mathcal{N}_{\varepsilon} \rrbracket} \left\vert \int_{\mathcal{X}} \Phi_i \,d(\mu^{\xi}-\bar{\mu}^{\xi}) \right\vert d\xi &\leq \sum_{i=1}^{\mathcal{N}_{\varepsilon}} \int_{0}^{1} \left\vert \int_{\mathcal{X}} \Phi_i(x) \,d(\mu^{\xi}-\bar{\mu}^{\xi})(x) \right\vert d\xi \\
		&\leq \sum_{i=1}^{\mathcal{N}_{\varepsilon}} \sup_{\Phi \in \text{BL}_1(\mathcal{X})} \int_{0}^{1} \left\vert \int_{\mathcal{X}} \Phi(x) \,d(\mu^{\xi}-\bar{\mu}^{\xi})(x) \right\vert d\xi \\
		&\leq 2\mathcal{N}_{\varepsilon} d_{\text{BL},\square}(\mu,\bar{\mu}),
	\end{aligned}
	\]
	which completes the proof.
\end{proof}

\begin{corollary} \label{corollary:d-BL-cut-equivalent-d-BL-1}
	The distances \(d_{\text{BL},\square}\) and \(d_{\text{BL},1}\) induce the same topology on
	\(\mathcal{P}_{\nu}(\mathcal{X}\times [0,1])\).
\end{corollary}
\begin{proof}
	Since \(d_{\text{BL},\square}\) is bounded above by \(d_{\text{BL},1}\), it suffices to show that if a sequence
	\((\mu_n)_{n\in \mathbb{N}} \subset \mathcal{P}_{\nu}(\mathcal{X}\times [0,1])\) converges to \(\mu\) in the 
	bounded-Lipschitz cut distance, it also converges in the \(d_{\text{BL},1}\) distance. By taking \(S=[0,1]\) in 
	\eqref{eq:bounded-Lipschitz-cut-fibered-probability-measures}, we observe that for any
	\(\mu,\bar{\mu} \in \mathcal{P}_{\nu}(\mathcal{X}\times [0,1])\),
	\[
		d_{\text{BL}}(M_{\mu},M_{\bar{\mu}}) \leq d_{\text{BL},\square}(\mu,\bar{\mu}).
	\] 
	Consequently, the convergence \(\mu_n \to \mu\) in \(d_{\text{BL},\square}\) implies the narrow convergence of the averaged
	measures \(M_{\mu_n} \to M_{\mu}\) in \(\mathcal{P}(\mathcal{X})\). Thus, by Prokhorov's theorem, 
	the set \(\{M_{\mu_n} : n \in \mathbb{N}\}\) is
	tight in \(\mathcal{P}(\mathcal{X})\), and by Lemma \ref{lemma:equivalence-bounded-Lipschitz-cut-and-BL,p}, it follows that
	for every \(\varepsilon > 0\), there exists a constant \(\mathcal{N}_{\varepsilon}\) such that
	\[
		d_{\text{BL},1}(\mu_n,\mu) \leq 2\mathcal{N}_{\varepsilon} d_{\text{BL},\square}(\mu_n, \mu) + \varepsilon
	\]
	for all \(n \in \mathbb{N}\). Taking the limit as \(n \to \infty\) yields
	\[
		\limsup_{n\to \infty} d_{\text{BL},1}(\mu_n,\mu) \leq \varepsilon.
	\]
	Since \(\varepsilon > 0\) is arbitrary, we conclude that \(\lim_{n \to \infty} d_{\text{BL},1}(\mu_n,\mu) = 0\).
\end{proof}
\begin{remark}
	We have established the following relation between the different topologies on the space of fibered probability measures
	\(\mathcal{P}_{\nu}(\mathcal{X}\times [0,1])\):
	\begin{equation} \label{eq:topologies-fibered-measures}
		\tau_{\text{narrow}} \subsetneq \tau_{d_{\text{BL},\square}} = \tau_{d_{\text{BL},1}} \subsetneq \tau_{d_{\text{BL},\infty}}.
	\end{equation}
	Since \(d_{\text{BL},\square}\) and \(d_{\text{BL},1}\) are topologically equivalent, from now on we will work with
	the \(d_{\text{BL},1}\) distance, which is easier to handle.
	
	Like in the classical theory of real-valued functions,
	it is important to note that the equivalence between these two distances is a special feature of the geometry of the unit
	interval. Once we pass to the unit square in the next subsection, we will see that they are no longer equivalent on 
	\(\mathcal{P}_{\nu\otimes \nu}(\mathcal{X}\times [0,1]^2)\).
	In this space (which naturally identifies with the space of probability-graphons introduced later), 
	the bounded-Lipschitz cut distance regains its critical importance as the correct metric for obtaining
	compactness results, just as in the classical theory of graphons.
\end{remark}
\subsubsection{ (Unlabeled) fibered probability measures.}
A cornerstone result in graph limit theory is that the space of graphons becomes compact only after passing
to a quotient space that identifies graphons up to measure-preserving transformations. This quotient formalizes
the idea that the labeling of nodes is irrelevant. Because the parameter \(\xi \in [0,1]\) in our framework
similarly represents an abstract labeling index, obtaining robust compactness properties necessitates
an analogous quotient structure for fibered probability measures. Consequently, we refer to
\(d_{\text{BL},p}\) as the \emph{labeled} bounded-Lipschitz \(p\) distance, and introduce its \emph{unlabeled} equivalent
as follows.
\begin{definition}
	For any two fibered probability measures \(\mu,\bar{\mu} \in \mathcal{P}_{\nu}(\mathcal{X}\times [0,1])\), we define
	the \emph{unlabeled} version of the distance \(d_{\text{BL},p}\) as 
	\begin{equation} \label{eq:unlabeled-delta-BL-p}
		\delta_{\text{BL},p}(\mu,\bar{\mu}):= \inf_{\varphi \in S_{[0,1]}} d_{\text{BL},p}(\mu,\bar{\mu}^{\varphi}),
	\end{equation}
	where \(S_{[0,1]}\) is the set of bijective measure-preserving maps from \([0,1]\) to itself, and \(\bar{\mu}^{\varphi}\)
	is the fibered probability measure rearranged by \(\varphi\), that is, \(\bar{\mu}^{\varphi,\xi}=\bar{\mu}^{\varphi(\xi)}\).
\end{definition}
By construction, the unlabeled distance \(\delta_{\text{BL},p}\) is only a pseudometric on 
\(\mathcal{P}_{\nu}(\mathcal{X}\times [0,1])\), as it can vanish for distinct fibered measures.
For instance, the distance between any \(\mu \in \mathcal{P}_{\nu}(\mathcal{X}\times [0,1])\) and
its rearrangement \(\mu^{\varphi}\) is zero for all \(\varphi \in S_{[0,1]}\). However, it induces a true metric on the
quotient space obtained by identifying weakly isomorphic fibered probability measures.
\begin{definition}[Weak isomorphism]
	Two fibered probability measures \(\mu,\bar{\mu} \in \mathcal{P}_{\nu}(\mathcal{X}\times [0,1])\) are said to be
	\emph{weakly isomorphic} if there exist two measure-preserving maps (not necessarily bijective) 
	\(\varphi,\psi \in \bar{S}_{[0,1]}\) such that \(\mu^{\varphi}=\bar{\mu}^{\psi}\).
\end{definition}
We denote by \(\widetilde{\mathcal{P}}_{\nu}(\mathcal{X}\times [0,1])\) the corresponding quotient space
modulo weak isomorphism. Adapting the arguments of \cite[Theorem 8.13]{Lovasz-12} to our setting, the following
result ensures that for any \(p\in [1,\infty)\), the infimum in \eqref{eq:unlabeled-delta-BL-p} can be replaced by
a minimum over pairs of measure-preserving maps, and consequently, that the unlabeled distances \(\delta_{\text{BL},p}\)
induce strict metrics on the quotient space \(\widetilde{\mathcal{P}}_{\nu}(\mathcal{X}\times [0,1])\). 
\begin{proposition} \label{prop:minima-versus-infima-unlabeled-distance}
	Consider a distance \(d\) in \(\mathcal{P}_{\nu}(\mathcal{X}\times [0,1])\) which is:
	\begin{enumerate}[label=(\roman*)]
		\item Invariant under rearrangements: for every \(\mu,\bar{\mu} \in \mathcal{P}_{\nu}(\mathcal{X}\times [0,1])\) and every \(\varphi \in S_{[0,1]}\), we have that \(d(\mu,\bar{\mu})=d(\mu^{\varphi},\bar{\mu}^{\varphi})\). 
		\item Smooth: for every \((\mu_n)_{n\in\mathbb{N}}\), \(\mu\) in \(\mathcal{P}_{\nu}(\mathcal{X}\times [0,1])\), if \(\mu_n^\xi\to \mu^{\xi}\) narrowly for a.e. \(\xi \in [0,1]\), then \(d(\mu_n,\mu)\to 0\).   
	\end{enumerate}
	Then, the unlabeled version of \(d\) can be equivalently reformulated as
	\[
		\delta(\mu,\bar{\mu})=\inf_{\varphi \in S_{[0,1]}} d(\mu,\bar{\mu}^{\varphi})=\min_{\varphi, \psi \in \bar{S}_{[0,1]}} d(\mu^{\varphi},\bar{\mu}^{\psi}).
	\]
\end{proposition}
\begin{corollary}
	The unlabeled distance \(\delta_{\text{BL},p}\) is a distance on the quotient space
	\(\widetilde{\mathcal{P}}_{\nu}(\mathcal{X}\times [0,1])\) for every \(p\in [1,\infty)\).
\end{corollary}
\begin{remark}
	The distance \(d_{\text{BL},\infty}\) is not smooth, thus, the previous result does not apply 
	to the unlabeled version of this distance, and neither to the unlabeled version of the distance
	\(d_{\text{BL}}\) because it fails the invariance condition.
\end{remark}
\begin{remark}
	From this point on, we will work with the distance \(d_{\text{BL},1}\) and its unlabeled version \(\delta_{\text{BL},1}\), 
	although the results that we will obtain also hold for \(\delta_{\text{BL},p}\) for any \(p\in [1,\infty)\). 
	Recall that these distances are topologically equivalent, so the choice of \(p\) does not affect the topology of the space. 
\end{remark}
As we will see, a major advantage of working in the quotient space
\((\widetilde{\mathcal{P}}_{\nu}(\mathcal{X}\times [0,1]), \delta_{\text{BL},1})\) lies in the characterization
of its relatively compact sets. Following the classical approach in graph limit theory, the key step is to obtain
a weak regularity lemma. This result enables us to approximate any fibered probability measure within a tight set by a 
finite step fibered measure, where the number of steps depends only on the desired approximation error and not on the
specific measure.

\begin{definition}[Stepping operator]
	For any fibered probability measure \(\mu \in \mathcal{P}_{\nu}(\mathcal{X}\times [0,1])\) and any measurable partition \(\mathcal{P}=\{S_1,\ldots,S_k\}\) of \([0,1]\), we define the step function \(\mu_{\mathcal{P}}\) as 
	\[
	\mu_{\mathcal{P}}^{\xi}= \sum_{i=1}^{k}\mathds{1}_{S_i}(\xi) \mu_{S_i} , \quad \textrm{for } \textrm{a.e. } \xi \in [0,1],
	\]
	where \(\mu_{S_i}=\frac{1}{\nu(S_i)}\int_{S_i} \mu^{\xi} d\xi\in \mathcal{P}(\mathcal{X})\) is the average probability
 	measure over the set \(S_i\). If \(\nu(S_i)=0\), we set \(\mu_{S_i}\) to be an arbitrary probability measure in \(\mathcal{P}(\mathcal{X})\).
\end{definition}
\begin{lemma}
	The stepping operator is non-expansive with respect to the distance \(d_{\text{BL},1}\), {\it i.e.}, for every
	\(\mu, \bar{\mu} \in \mathcal{P}_{\nu}(\mathcal{X}\times [0,1])\) and every partition \(\mathcal{P}\) of \([0,1]\),
	we have
	\[
	d_{\text{BL},1}(\mu_{\mathcal{P}},\bar{\mu}_{\mathcal{P}}) \leq d_{\text{BL},1}(\mu,\bar{\mu}).
	\]
\end{lemma}
\begin{proof}
	Let \(\mathcal{P}=\{S_1,\ldots,S_k\}\) be a finite measurable partition of \([0,1]\). Without loss of generality, we may assume that \(\nu(S_j)>0\) for all \(j\). For a.e. \(\xi \in S_j\), the convexity of the bounded Lipschitz distance yields
	\[
	d_{\text{BL}}(\mu_{\mathcal{P}}^{\xi},\bar{\mu}_{\mathcal{P}}^{\xi}) = d_{\text{BL}}(\mu_{S_j},\bar{\mu}_{S_j}) \leq \frac{1}{\nu(S_j)} \int_{S_j} d_{\text{BL}}(\mu^{\xi'},\bar{\mu}^{\xi'}) d\xi'.
	\]
	Taking the integral with respect to \(\xi\) over \(S_j\), and noting that the right-hand side is constant with respect to \(\xi\), we obtain:
	\[
	\int_{S_j} d_{\text{BL}}(\mu_{\mathcal{P}}^{\xi},\bar{\mu}_{\mathcal{P}}^{\xi}) d\xi \leq \int_{S_j} d_{\text{BL}}(\mu^{\xi'},\bar{\mu}^{\xi'}) d\xi'.
	\]
	Finally, summing over all partition sets \(j=1,\ldots,k\), we conclude
	\[
	d_{\text{BL},1}(\mu_{\mathcal{P}},\bar{\mu}_{\mathcal{P}}) = \sum_{j=1}^{k} \int_{S_j} d_{\text{BL}}(\mu_{\mathcal{P}}^{\xi},\bar{\mu}_{\mathcal{P}}^{\xi}) d\xi \leq \sum_{j=1}^{k} \int_{S_j} d_{\text{BL}}(\mu^{\xi'},\bar{\mu}^{\xi'}) d\xi' = d_{\text{BL},1}(\mu,\bar{\mu}),
	\]
	which completes the proof.
\end{proof}

\begin{lemma}[{Weak regularity lemma for \(\mathcal{P}_{\nu}(\mathcal{X}\times [0,1])\)}] 
	\label{lemma:regularity-fibered-probability-measures}
	Let \(\mathcal{K}\subset \mathcal{P}_{\nu}(\mathcal{X}\times [0,1])\) be a tight set of fibered probability measures.
	Then, for every \(\varepsilon > 0\), there exists an integer \(m \in \mathbb{N}\) such that for any measure 
	\(\mu \in \mathcal{K}\) and any  partition \(\mathcal{Q}\) of \([0,1]\), there exists a refinement
	\(\mathcal{P}\) of \(\mathcal{Q}\) satisfying
	\[
		\vert \mathcal{P} \vert \leq m \vert \mathcal{Q}\vert , \quad \textrm{and} \quad d_{\text{BL},1} (\mu, \mu_{\mathcal{P}}) < \varepsilon.
	\]
\end{lemma}
\begin{proof}
	The proof is a straightforward adaptation of the weak regularity lemma for probability-graphons
	\cite[Proposition 4.13]{ADW-25}. The key difference is that we are working with fibered probability measures 
	in the unit interval, allowing us to leverage the simpler 1D geometry of the domain \([0,1]\) instead of \([0,1]^2\).

	\(\diamond\) \textsc{Step 1}: Suppose first that \(\mathcal{X}\) is compact. \\ 
	In this case, the space \((\mathcal{P}(\mathcal{X}),d_{\text{BL}})\) is compact. Thus, for every \(\varepsilon>0\), 
	there exists an integer \(n_\varepsilon \in \mathbb{N}\) and a finite set of measures
	\(\{\mu_1,\ldots,\mu_{n_\varepsilon}\} \subset \mathcal{P}(\mathcal{X})\) such that
	\(\mathcal{P}(\mathcal{X})=\bigcup_{i=1}^{n_\varepsilon} B_{\varepsilon}(\mu_i)\),
	where \(B_{\varepsilon}(\mu_i)\) denotes the open ball of radius \(\varepsilon\) centered at \(\mu_i\)
	in the \(d_{\text{BL}}\) distance. We can construct a finite measurable partition
	\(\{A_1,\ldots,A_{n_\varepsilon}\}\) of \(\mathcal{P}(\mathcal{X})\) by setting
	\(A_i := B_{\varepsilon}(\mu_i)\setminus \bigcup_{j<i}B_{\varepsilon}(\mu_j)\). For any 
	\(\mu \in \mathcal{P}_{\nu}(\mathcal{X}\times [0,1])\), this induces a measurable partition
	\(\mathcal{P}=\{B_1,\ldots,B_{n_\varepsilon}\}\) of \([0,1]\) given by \(B_i := \{\xi \in [0,1]: \mu^{\xi} \in A_i\}\).
	Next, we define the \(\{\mu_1,\ldots,\mu_{n_\varepsilon}\}\)-valued fibered probability measure \(\bar{\mu}\) as
	\[
		\bar{\mu}^{\xi}=\sum_{i=1}^{n(\varepsilon)} \mathds{1}_{B_i}(\xi) \mu_i, \quad \xi \in [0,1],
	\]
	which satisfies that \(d_{\text{BL}}(\mu^{\xi},\bar{\mu}^{\xi})< \varepsilon\) for all \(\xi \in [0,1]\). 
	Consequently, \(d_{\text{BL},1}(\mu,\bar{\mu}) < \varepsilon\). Because the stepping operator is non-expansive with respect
	to \(d_{\text{BL},1}\), and noting that \(\bar{\mu}_{\mathcal{P}} = \bar{\mu}\), the triangle inequality yields
	\[
		d_{\text{BL},1}(\mu,\mu_{\mathcal{P}}) \leq d_{\text{BL},1}(\mu,\bar{\mu}_{\mathcal{P}}) + d_{\text{BL},1}(\bar{\mu}_{\mathcal{P}},\mu_{\mathcal{P}}) \leq 2\varepsilon.
	\]
	If a partition \(\mathcal{Q}\) of \([0,1]\) was given in advance, we simply consider the common refinement of \(\mathcal{P}\)
	and \(\mathcal{Q}\) to conclude the proof in the compact case.

	\medskip

	\(\diamond\) {\sc Step 2}: Assume now that \(\mathcal{X}\) is not necessarily compact.\\
	Given \(\varepsilon>0\), since \(\mathcal{K}\) is tight, by Proposition \ref{prop:equivalence-tightness} there exists
	a compact set \(K_{\varepsilon} \subset \mathcal{X}\) such that
	\[
		\nu(A_{\mu})\leq\varepsilon \quad \textrm{ and  } \int_{0}^{1}\mu^{\xi}(\mathcal{X}\setminus K_{\varepsilon}) d\xi \leq \varepsilon, \quad \forall \mu \in \mathcal{K},
	\]
	where \(A_{\mu}=\{\xi \in [0,1]: \mu^{\xi}(\mathcal{X}\setminus K_{\varepsilon}) \geq \varepsilon\}.\)
	We construct an auxiliary fibered probability measure \(\bar{\mu}\) by projecting the mass onto \(K_{\varepsilon}\):
	\[
		\bar{\mu}^{\xi} = 
		\begin{cases} 
			\eta, & \text{if } \xi \in A_{\mu}, \\
			\frac{1}{\mu^{\xi}(K_{\varepsilon})} \mu^{\xi}\vert_{K_{\varepsilon}}, & \text{if } \xi \notin A_{\mu},
		\end{cases}
	\]
	where \(\eta \in \mathcal{P}(K_{\varepsilon})\) is an arbitrary fixed probability measure,
	and \(\mu^{\xi}\vert_{K_{\varepsilon}}\) denotes the explicit restriction of the measure \(\mu^{\xi}\) to the compact
	set \(K_{\varepsilon}\), defined for any Borel set \(B\) as \((\mu^{\xi}\vert_{K_{\varepsilon}})(B) := \mu^{\xi}(B \cap K_{\varepsilon})\). 
	A direct computation shows that
	\[
		d_{\text{BL},1}(\mu,\bar{\mu}) = \int_{A_{\mu}} d_{\text{BL}}(\mu^{\xi},\eta) d\xi + \int_{A_{\mu}^c} d_{\text{BL}}\left(\mu^{\xi},\frac{1}{\mu^{\xi}(K_{\varepsilon})} \mu^{\xi}\vert_{K_{\varepsilon}}\right) d\xi \leq 2\nu(A_{\mu}) + 2\int_{A^c_{\mu}} \mu^{\xi}(\mathcal{X}\setminus K_{\varepsilon}) d\xi \leq 4\varepsilon.
	\]
	Then, for every finite measurable partition \(\mathcal{P}\) of \([0,1]\), using the non-expansive property \(d_{\text{BL},1}(\bar{\mu}_{\mathcal{P}},\mu_{\mathcal{P}}) \leq d_{\text{BL},1}(\bar{\mu},\mu) \leq 4\varepsilon\), we obtain
	\[
		d_{\text{BL},1}(\mu,\mu_{\mathcal{P}}) \leq d_{\text{BL},1}(\mu,\bar{\mu}) + d_{\text{BL},1}(\bar{\mu},\bar{\mu}_{\mathcal{P}}) + d_{\text{BL},1}(\bar{\mu}_{\mathcal{P}},\mu_{\mathcal{P}}) \leq 8\varepsilon + d_{\text{BL},1}(\bar{\mu},\bar{\mu}_{\mathcal{P}}).
	\]
	Since \(\bar{\mu}\) is a fibered probability measure whose fibers are supported on the compact set \(K_{\varepsilon}\),
	we can apply the result of \textsc{Step 1} to bound \(d_{\text{BL},1}(\bar{\mu},\bar{\mu}_{\mathcal{P}})\) and conclude
	the proof.
\end{proof}
Building on the preceding lemma, we take the first step toward characterizing the relatively compact sets in
\((\widetilde{\mathcal{P}}_{\nu}(\mathcal{X}\times [0,1]),\delta_{\text{BL},1})\). The proof of the upcoming
proposition is based on the classical approach used to establish the compactness of the graphon space \cite[Theorem 5.1]{LS-07},
and the generalization to probability-graphons in \cite[Lemma 8.1]{ADW-25}, which adapts the martingale convergence argument
to the infinite dimensional setting of probability measures.
\begin{proposition} \label{prop:tightness-in-average-implies-relative-compactness}
	If \(\mathcal{K}\subset \widetilde{\mathcal{P}}_{\nu}(\mathcal{X}\times [0,1])\) is tight, then it is relatively compact in \((\widetilde{\mathcal{P}}_{\nu}(\mathcal{X}\times [0,1]),\delta_{\text{BL},1})\).
\end{proposition}
\begin{proof}
	Let \((\mu_n)_{n\in\mathbb{N}}\) be a sequence in \(\mathcal{K}\). We will prove that there exists a subsequence
	\((\mu_{n_k})_{k\in \mathbb{N}}\) and \(\mu \in \mathcal{P}_{\nu}(\mathcal{X}\times [0,1])\) such that
	\[
		\lim_{k\to \infty}\delta_{\text{BL},1}(\mu_{n_k},\mu) = 0.
	\]

	\(\diamond\) \textsc{Step 1}: Application of the weak regularity lemma. \\ 
	By Lemma \ref{lemma:regularity-fibered-probability-measures}, for every \(k\in \mathbb{N}\) there exists a partition \(\mathcal{P}_{n,k}\) of \([0,1]\) such that:
	\begin{enumerate}[label=(\roman*)]
		\item \(d_{\text{BL},1}(\mu_n,\mu_{n,k}) \leq 1/k\), 
		\item \(\diam(\mathcal{P}_{n,k})\leq 2^{-k}\) and \(|\mathcal{P}_{n,k}| = m_k\), for some integer \(m_k \in \mathbb{N}\) independent of \(n\),
		\item \(\mathcal{P}_{n,k+1}\) is a refinement of \(\mathcal{P}_{n,k}\).
	\end{enumerate}
	Here, \(\mu_{n,k}\) denotes the stepped fibered probability measure associated with \(\mu_n\) and the partition \(\mathcal{P}_{n,k}\), {\it i.e.}, \(\mu_{n,k} = (\mu_n)_{\mathcal{P}_{n,k}}\).

\medskip

	\(\diamond\) \textsc{Step 2}: Interval partitions. \\
	Without loss of generality, we may assume that the partitions \(\mathcal{P}_{n,k}\) consist of intervals. 
	This follows from a standard reordering argument, see \cite[Lemma 8.2]{ADW-25}.

\medskip

	\(\diamond\) \textsc{Step 3}: Convergence of the step measures. \\
	There exists a subsequence \((\mu_{n_{\ell}})_{\ell\in \mathbb{N}}\) of \((\mu_n)_{n\in \mathbb{N}}\) such that for every
	\(k\in \mathbb{N}\), the sequence \((\mu_{n_{\ell},k})_{\ell\in \mathbb{N}}\) converges weakly pointwise almost everywhere
	to some stepped fibered probability measure \(\bar{\mu}_k\) with \(m_k\) steps. 
	To construct this, let us write \(\mathcal{P}_{n,k}=\{S_{n,k,i}: 1\leq i\leq m_k\}\). Fix \(i \in \{1, \dots, m_k\}\). 
	Since the sequence of interval lengths \((\nu(S_{n,k,i}))_{n\in \mathbb{N}}\) takes values in the compact set \([0,1]\),
	we can extract a subsequence \((n_{\ell})_{\ell\in \mathbb{N}}\) such that \(\nu(S_{n_{\ell},k,i})\) converges to some
	\(s_{k,i} \in [0,1]\). Through a standard diagonal argument, we may assume this convergence holds simultaneously for all
	\(k\) and all \(i\). Let \(\mathcal{P}_k=\{S_{k,i}: 1\leq i\leq m_k\}\) be a partition of \([0,1]\) consisting of
	consecutive intervals of length \(s_{k,i}\) (if \(s_{k,i}=0\), we simply set \(S_{k,i}=\emptyset\)).
	Because \(\mathcal{P}_{n,k+1}\) is a refinement of \(\mathcal{P}_{n,k}\) for every \(n\), it follows that \(\mathcal{P}_{k+1}\)
	is a refinement of \(\mathcal{P}_k\). Consider any interval where \(s_{k,i} > 0\). For \(\ell\) sufficiently large,
 	\(\nu(S_{n_{\ell},k,i})\) is uniformly bounded away from zero (say, \(\nu(S_{n_{\ell},k,i}) \geq c > 0\)).
	Due to the tightness of \((\mu_n)_{n\in \mathbb{N}}\), for every \(\varepsilon >0\) there exists a compact set
	\(K_{\varepsilon} \subset \mathcal{X}\) such that
	\[
		\int_{0}^{1}\mu^{\xi}_n(\mathcal{X}\setminus K_{\varepsilon}) d\xi < \varepsilon c, \quad \forall n\in \mathbb{N}.
	\]
	Consequently,
	\[
		\frac{1}{\nu(S_{n_{\ell},k,i})} \int_{S_{n_{\ell},k,i}} \mu^{\xi}_{n_{\ell}}(\mathcal{X}\setminus K_{\varepsilon}) d\xi < \varepsilon
	\]
	for all sufficiently large \(\ell\), proving that the sequence of averaged probability measures on \(S_{n_{\ell},k,i}\) is tight. By Prokhorov's theorem, we can extract a further subsequence (not relabeled) such that these averages converge weakly to some probability measure \(\mu_{k,i} \in \mathcal{P}(\mathcal{X})\) as \(\ell \to \infty\). A final diagonal argument ensures this convergence holds for all \(k\) and \(i\). Defining the limit stepped fibered measure \(\bar{\mu}_k\) as
	\[
		\bar{\mu}_k^{\xi} = \sum_{i=1}^{m_k} \mathds{1}_{S_{k,i}}(\xi) \mu_{k,i}, \quad \text{for } \xi \in [0,1],
	\]
	concludes this step.

	\medskip

	\(\diamond \) \textsc{Step 4}: Martingale convergence to the limit measure.\\
	We now show that a subsequence of \((\bar{\mu}_k)_{k\in \mathbb{N}}\) converges weakly pointwise almost everywhere to some
	fibered measure \(\mu\). First, we verify that the sequence \((\bar{\mu}_k)_{k\in\mathbb{N}}\) is tight.
	Given \(\varepsilon >0\), let \(K_{\varepsilon} \subset \mathcal{X}\) be the compact set associated with the tightness of
	\(\mathcal{K}\). The average measure of \(\bar{\mu}_k\) is \(M_{\bar{\mu}_k} = \sum_{i=1}^{m_k} \nu(S_{k,i}) \mu_{k,i}\).
	Furthermore, \(M_{\mu_{n,k}} = \int_{0}^{1} \mu_n^{\xi} d\xi = M_{\mu_n}\). Because \(\mu_{n_{\ell},k} \to \bar{\mu}_k\)
	weakly pointwise a.e., \(M_{\mu_{n_{\ell},k}}\) converges weakly to \(M_{\bar{\mu}_k}\). By the Portmanteau theorem,
	\[
		M_{\bar{\mu}_k}(\mathcal{X}\setminus K_{\varepsilon}) \leq \liminf_{\ell \to \infty} M_{\mu_{n_{\ell},k}}(\mathcal{X}\setminus K_{\varepsilon}) = \liminf_{\ell \to \infty} M_{\mu_{n_{\ell}}}(\mathcal{X}\setminus K_{\varepsilon}) < \varepsilon.
	\]
	This proves the sequence is tight. Consequently, the joint measures
	\(\bar{\mu}_k(x,\xi) = \bar{\mu}_k^{\xi}(x) \otimes d\xi\) are tight in \(\mathcal{P}(\mathcal{X}\times [0,1])\).
	By Prokhorov's theorem, there exists a subsequence
	\((\bar{\mu}_{k_j})_{j\in \mathbb{N}}\) converging weakly to some \(\mu \in \mathcal{P}(\mathcal{X}\times [0,1])\).
	Since \(\mathcal{P}_{\nu}(\mathcal{X}\times [0,1])\) is closed under narrow convergence, 
	\(\mu \in \mathcal{P}_{\nu}(\mathcal{X}\times [0,1])\). To prove that \(\bar{\mu}^{\xi}_{k_j} \to \mu^{\xi}\) weakly
	for almost every \(\xi\), we leverage the martingale property \(\bar{\mu}_k = (\bar{\mu}_{m})_{\mathcal{P}_k}\) for all
	\(m \geq k\). Let \(Y\) be a random variable distributed uniformly on \([0,1]\). Choose a countable convergence-determining
	class \(\{f_i\}_{i\in\mathbb{N}} \subset C_b(\mathcal{X})\) for the narrow topology of 
	\(\mathcal{P}(\mathcal{X})\). For any such \(f\), the sequence of random variables
	\(\xi \mapsto \bar{\mu}^{\xi}_{k_j}(f) := \int_{\mathcal{X}} f(x) d\bar{\mu}^{\xi}_{k_j}(x)\) forms a bounded martingale
	with respect to the filtration generated by the partitions \(\mathcal{P}_{k_j}\). By the martingale convergence theorem,
	it converges almost surely to a random variable \(Z_f\). For any \(g\in C_b([0,1])\), we have
	\[
		\begin{aligned}
		\int_{0}^{1} g(\xi) \int_{\mathcal{X}} f(x) d\mu^{\xi}(x) d\xi &= \lim_{j\to \infty} \int_{\mathcal{X}\times [0,1]} g(\xi) f(x) d(\bar{\mu}^{\xi}_{k_j}(x) \otimes d\xi) \\   
		&= \lim_{j\to \infty} \mathbb{E}[g(Y)\bar{\mu}^{Y}_{k_j}(f)] = \int_{0}^{1} g(\xi) Z_f(\xi) d\xi.
		\end{aligned}
	\]
	Therefore, \(\int_{\mathcal{X}} f(x) d\mu^{\xi}(x) = Z_f(\xi) = \lim_{j\to \infty} \int_{\mathcal{X}} f(x) d\bar{\mu}^{\xi}_{k_j}(x)\)
	for a.e. \(\xi \in [0,1]\). Taking the countable union of null sets over all \(f_i\) concludes that
	\(\bar{\mu}^{\xi}_{k_j}\) converges weakly to \(\mu^{\xi}\) for almost every \(\xi \in [0,1]\).

	\medskip

	\(\diamond \) \textsc{Step 5:} Final estimate. \\ 
	Take any \(\varepsilon >0\), by point \((i)\) in \textsc{Step 1}, and \textsc{Step 3}, we can ensure that for
	\(\ell\) sufficiently large we have that \(\delta_{\text{BL},1}(\mu_{n_{\ell}},\mu_{n_{\ell},k}) < \varepsilon\)
	and \(\delta_{\text{BL},1}(\mu_{n_{\ell},k},\bar{\mu}_k) < \varepsilon\) for every \(k\in \mathbb{N}\).
 	Moreover, by \textsc{Step 4}, by taking \(j\) sufficiently large we have that
	\(\delta_{\text{BL},1}(\bar{\mu}_{k_j},\mu) < \varepsilon\).
  	Thus, combining the previous estimates with the triangular inequality, we have that for every \(\ell\) sufficiently large,
	\[
		\delta_{\text{BL},1}(\mu_{n_{\ell}},\mu) \leq \delta_{\text{BL},1}(\mu_{n_{\ell}},\mu_{n_{\ell},k_j}) + \delta_{\text{BL},1}(\mu_{n_{\ell},k_j},\bar{\mu}_{k_j}) + \delta_{\text{BL},1}(\bar{\mu}_{k_j},\mu) < 3\varepsilon,
	\]
	which concludes the proof.
\end{proof} 
\begin{theorem} \label{thm:compactness-fibered-probability-measures}
	The space \((\widetilde{\mathcal{P}}_{\nu}(\mathcal{X}\times [0,1]),\delta_{\text{BL},1})\) is Polish. Moreover we have:
	\begin{enumerate}
		\item[(i)] A subset \(\mathcal{K}\subset (\widetilde{\mathcal{P}}_{\nu}(\mathcal{X}\times [0,1]),\delta_{\text{BL},1})\) is relatively compact if and only if it is tight.
		\item[(ii)] The space \((\widetilde{\mathcal{P}}_{\nu}(\mathcal{X}\times [0,1]),\delta_{\text{BL},1})\) is compact if and only if \(\mathcal{X}\) is compact.
	\end{enumerate}
\end{theorem}
\begin{proof}
	\(\diamond\) \textsc{Step 1}: Characterization of relative compactness.\\ 
	If \(\mathcal{K}\) is tight, then Proposition \ref{prop:tightness-in-average-implies-relative-compactness} ensures
	that \(\mathcal{K}\) is relatively compact. Conversely, suppose that \(\mathcal{K}\) is relatively compact. Consider
	the  map \(M:(\widetilde{\mathcal{P}}_{\nu}(\mathcal{X}\times [0,1]),\delta_{\text{BL},1}) \to (\mathcal{P}(\mathcal{X}),d_{\text{BL}})\) given by \(\mu \mapsto M_{\mu}\). 
	This map is well-defined because if \(\mu\) and \(\bar{\mu}\) are weakly isomorphic, their average measures coincide
 	({\it i.e.}, \(M_{\mu}=M_{\bar{\mu}}\)). Moreover, the inequality \(d_{\text{BL}}(M_{\mu},M_{\bar{\mu}}) \leq d_{\text{BL},1}(\mu,\bar{\mu})\) holds for all \(\mu,\bar{\mu} \in \mathcal{P}_{\nu}(\mathcal{X}\times [0,1])\), which implies
	\(d_{\text{BL}}(M_{\mu},M_{\bar{\mu}}) \leq \delta_{\text{BL},1}(\mu,\bar{\mu})\) by taking the infimum over relabelings. 
	This Lipschitz continuity ensures that the image \(M_{\mathcal{K}}=\{M_{\mu}: \mu \in \mathcal{K}\}\) is relatively compact
	in \((\mathcal{P}(\mathcal{X}),d_{\text{BL}})\). By Prokhorov's theorem, \(M_{\mathcal{K}}\) is tight, and thus, by Proposition
	\ref{prop:equivalence-tightness}, \(\mathcal{K}\) itself is tight.

	\medskip

	\(\diamond\) \textsc{Step 2}: Compactness of the total space.\\ 
	If \(\mathcal{X}\) is compact, then \(\widetilde{\mathcal{P}}_{\nu}(\mathcal{X}\times [0,1])\) is trivially tight,
	and therefore, compact. Conversely, if \(\widetilde{\mathcal{P}}_{\nu}(\mathcal{X}\times [0,1])\) is compact, its image under the
	continuous map \(M\) must be compact in \((\mathcal{P}(\mathcal{X}),d_{\text{BL}})\). Since \(M\) is surjective, we conclude
	that \(\mathcal{P}(\mathcal{X})\) is compact. By Prokhorov's theorem, this implies that \(\mathcal{X}\) is compact.

	\medskip

	\(\diamond\) \textsc{Step 3}: Polish property.\\ 
	The separability of the space follows by the same arguments used in \cite[Proposition 4.6]{ADW-25}. 
	To establish completeness, let \((\mu_n)_{n\in \mathbb{N}}\) be a Cauchy sequence in 
	\((\widetilde{\mathcal{P}}_{\nu}(\mathcal{X}\times [0,1]),\delta_{\text{BL},1})\). Because the map \(M\) is \(1\)-Lipschitz,
	the sequence of averaged measures \((M_{\mu_n})_{n\in \mathbb{N}}\) is a Cauchy sequence in
	\((\mathcal{P}(\mathcal{X}),d_{\text{BL}})\). Since \(\mathcal{P}(\mathcal{X})\) is complete, there exists a
	limit measure \(m\in \mathcal{P}(\mathcal{X})\) such that \(d_{\text{BL}}(M_{\mu_n},m)\to 0\). The convergence
	of \((M_{\mu_n})_{n\in\mathbb{N}}\) implies that this set is relatively compact, and hence tight in
	\(\mathcal{P}(\mathcal{X})\). By Proposition \ref{prop:equivalence-tightness}, the original sequence 
	\(\{\mu_n: n\in \mathbb{N}\}\) is uniformly tight. Applying the result of \textsc{Step 1}, the sequence is relatively
	compact, meaning there exists a subsequence \((\mu_{n_k})_{k\in \mathbb{N}}\) and a limit 
	\(\mu \in \mathcal{P}_{\nu}(\mathcal{X}\times [0,1])\) such that \(\delta_{\text{BL},1}(\mu_{n_k},\mu)\to 0\).
	Because \((\mu_n)_{n\in \mathbb{N}}\) is a Cauchy sequence with a convergent subsequence, the entire 
	sequence must converge to the same limit, {\it i.e.}, \(\delta_{\text{BL},1}(\mu_n,\mu)\to 0\). This proves completeness,
	and thus the space is Polish.
\end{proof}
To highlight the advantage of working in the unlabeled space
\((\widetilde{\mathcal{P}}_{\nu}(\mathcal{X}\times [0,1]), \delta_{\text{BL},1})\), the following proposition
characterizes relative compactness in the original space \((\mathcal{P}_{\nu}(\mathcal{X}\times [0,1]), d_{\text{BL},1})\).
The proof adapts classical techniques from \cite[Theorem 4.7.28]{Bogachev-07}, following the adaptation of these
methods to fibered Wasserstein spaces presented in \cite[Theorem 3.3]{BP-25-arxiv}.

\begin{proposition}[{Relative compactness in \((\mathcal{P}_{\nu}(\mathcal{X}\times [0,1]),d_{\text{BL},1})\)}]
	A subset \(\mathcal{K}\subset (\mathcal{P}_{\nu}(\mathcal{X}\times [0,1]),d_{\text{BL},1})\) is relatively compact if and only if the following properties hold:
	\begin{enumerate}[label=(\roman*)]
		\item \(\mathcal{K}\) is tight. 
		\item For every \(\varepsilon > 0\), there exists \(m_{\varepsilon}\in \mathbb{N}\) and a partition \(\mathcal{P}_{\varepsilon}=\{A_1,\ldots,A_{m_{\varepsilon}}\}\) of \([0,1]\) such that 
		\begin{equation} \label{eq:strong-regularity-condition}
			\sup_{\mu \in \mathcal{K}} d_{\text{BL},1}(\mu,\mu_{\mathcal{P}_{\varepsilon}}) < \varepsilon.
		\end{equation}
	\end{enumerate}
\end{proposition}

\begin{proof}
	Suppose that \(\mathcal{K}\) is relatively compact in \((\mathcal{P}_{\nu}(\mathcal{X}\times [0,1]),d_{\text{BL},1})\). It follows
	that it is also relatively compact in the quotient space
	\((\widetilde{\mathcal{P}}_{\nu}(\mathcal{X}\times [0,1]),\delta_{\text{BL},1})\).
	By Theorem \ref{thm:compactness-fibered-probability-measures}, \(\mathcal{K}\) must be tight. To prove the strong
	regularity condition, fix \(\varepsilon >0\). Because \(\mathcal{K}\) is relatively compact,
	it is totally bounded, hence, there exists a finite \(\varepsilon\)-net
	\(\{\mu_1,\ldots,\mu_{\mathcal{N}_{\varepsilon}}\}\) for
	\(\mathcal{K}\) with respect to \(d_{\text{BL},1}\). By Lemma \ref{lemma:regularity-fibered-probability-measures},
	for each \(i\in \llbracket 1,\mathcal{N}_{\varepsilon}\rrbracket\) there exists a partition \(\mathcal{P}_i\) of \([0,1]\) such that
	\(d_{\text{BL},1}(\mu_i,(\mu_i)_{\mathcal{P}_i}) < \varepsilon\). Let \(\mathcal{P}_{\varepsilon}\) be the common
	refinement of the partitions \(\mathcal{P}_1,\ldots,\mathcal{P}_{\mathcal{N}_{\varepsilon}}\). For any \(\mu \in \mathcal{K}\), there exists some 
	\(i\in \llbracket 1,\mathcal{N}_{\varepsilon}\rrbracket\) such that \(d_{\text{BL},1}(\mu,\mu_i) < \varepsilon\). Using the triangle inequality, we can bound
	the distance to its step measure over \(\mathcal{P}_{\varepsilon}\) as follows:
	\[
		d_{\text{BL},1}(\mu,\mu_{\mathcal{P}_{\varepsilon}}) \leq d_{\text{BL},1}(\mu,\mu_i) + d_{\text{BL},1}(\mu_i,(\mu_i)_{\mathcal{P}_{\varepsilon}}) + d_{\text{BL},1}((\mu_i)_{\mathcal{P}_{\varepsilon}},\mu_{\mathcal{P}_{\varepsilon}}).
	\]
	Because \(\mathcal{P}_{\varepsilon}\) is a refinement of \(\mathcal{P}_i\), \(d_{\text{BL},1}(\mu_i,(\mu_i)_{\mathcal{P}_{\varepsilon}}) \leq d_{\text{BL},1}(\mu_i,(\mu_i)_{\mathcal{P}_i}) < \varepsilon\). 
	Furthermore, by the non-expansive property of the stepping operator, 
	\(d_{\text{BL},1}((\mu_i)_{\mathcal{P}_{\varepsilon}},\mu_{\mathcal{P}_{\varepsilon}}) \leq d_{\text{BL},1}(\mu_i,\mu) < \varepsilon\).
	Therefore, \(d_{\text{BL},1}(\mu,\mu_{\mathcal{P}_{\varepsilon}}) < 3\varepsilon\), which establishes
	\eqref{eq:strong-regularity-condition}.

	Conversely, assume that \(\mathcal{K}\) satisfies conditions (i) and (ii). Because the space is complete, it suffices to show
	that \(\mathcal{K}\) is totally bounded. Fix \(\varepsilon >0\). By condition (ii), there exists a partition
	\(\mathcal{P}_{\varepsilon}=\{A_1, \dots, A_{m_{\varepsilon}}\}\) such that 
	\(d_{\text{BL},1}(\mu,\mu_{\mathcal{P}_{\varepsilon}}) < \varepsilon\) for every \(\mu \in \mathcal{K}\). The
	tightness of \(\mathcal{K}\) ensures that for each partition block \(i\in \llbracket 1, m_{\varepsilon}\rrbracket\), the set of
	averaged measures
	\[
		\mathcal{A}_i = \left\{ \mu_{A_i} = \frac{1}{\nu(A_i)}\int_{A_i} \mu^{\xi}d\xi : \mu \in \mathcal{K} \right\} \subset \mathcal{P}(\mathcal{X})
	\]
	is tight. By Prokhorov's theorem, each \(\mathcal{A}_i\) is relatively compact in \((\mathcal{P}(\mathcal{X}),d_{\text{BL}})\). 
	Thus, we can extract a finite \(\varepsilon\)-net \(\{\eta_{i,1},\ldots,\eta_{i,k_i}\}\) for each \(\mathcal{A}_i\).
	We now construct a global finite net for \(\mathcal{K}\). Let \(\mathcal{N}\) be the finite set of all step fibered
	measures \(\tilde{\mu}\) over the partition \(\mathcal{P}_{\varepsilon}\) whose value on the block \(A_i\) belongs to the
	net \(\{\eta_{i,1},\ldots,\eta_{i,k_i}\}\). For any \(\mu \in \mathcal{K}\), its step measure
	\(\mu_{\mathcal{P}_{\varepsilon}}\) has fiber values \(\mu_{A_i}\) on each \(A_i\). By definition of the nets, we
	can choose an index \(j_i\) such that \(d_{\text{BL}}(\mu_{A_i}, \eta_{i,j_i}) < \varepsilon\) for each \(i\). Let
	\(\tilde{\mu} \in \mathcal{N}\) be the step measure constructed from these specific \(\eta_{i,j_i}\). The distance
	between the two step measures is 
	\[
		d_{\text{BL},1}(\mu_{\mathcal{P}_{\varepsilon}},\tilde{\mu}) = \sum_{i=1}^{m_{\varepsilon}} \nu(A_i) d_{\text{BL}}(\mu_{A_i}, \eta_{i,j_i}) < \sum_{i=1}^{m_{\varepsilon}} \nu(A_i) \varepsilon = \varepsilon.
	\]
	Finally, by the triangle inequality, the distance from the original measure \(\mu\) to the net element \(\tilde{\mu}\) is
	\[
		d_{\text{BL},1}(\mu,\tilde{\mu}) \leq d_{\text{BL},1}(\mu,\mu_{\mathcal{P}_{\varepsilon}}) + d_{\text{BL},1}(\mu_{\mathcal{P}_{\varepsilon}},\tilde{\mu}) < \varepsilon + \varepsilon = 2\varepsilon.
	\]
	This proves that \(\mathcal{N}\) serves as a finite \(2\varepsilon\)-net for \(\mathcal{K}\). Hence, \(\mathcal{K}\)
	is totally bounded, and consequently, relatively compact.
\end{proof}

\begin{remark}
	The previous characterization reveals a fundamental difference between the relatively compact sets
	in the unlabeled space \((\widetilde{\mathcal{P}}_{\nu}(\mathcal{X}\times [0,1]),\delta_{\text{BL},1})\) and those 
	in the original space \((\mathcal{P}_{\nu}(\mathcal{X}\times [0,1]),d_{\text{BL},1})\). While tightness is sufficient
	for relative compactness in the unlabeled quotient space, the original labeled space strictly requires an additional
	strong regularity condition. By Lemma \ref{lemma:regularity-fibered-probability-measures}, any tight family satisfies
	a \emph{weak} regularity property: every fibered measure can be approximated by a step measure with a bounded number
	of steps, but the optimal partition may vary from measure to measure. On the other hand, the \emph{strong} regularity
	condition \eqref{eq:strong-regularity-condition} requires the existence of a single, universal partition that
	simultaneously approximates all measures in the family, which is a much more stringent requirement.
\end{remark}

\begin{example}[{A sequence relatively compact for \(\delta_{\text{BL},1}\) but not for \(d_{\text{BL},1}\)}]
	Let \(\mathcal{X} = [-1, 1]\). Consider the Rademacher sequence of fibered probability measures
	\((\mu_n)_{n\in \mathbb{N}}\) defined by \(\mu_n^{\xi} = \delta_{u_n(\xi)}\) for a.e. \(\xi \in [0,1]\),
	where \(u_n(\xi)=\sgn(\sin(2^n \pi \xi))\). This sequence is tight since its average measures are constant: 
	\(M_{\mu_n} = \frac{1}{2}\delta_{-1} + \frac{1}{2}\delta_1\) for all \(n\in \mathbb{N}\). Therefore, by Theorem
	\ref{thm:compactness-fibered-probability-measures}, it is relatively compact in
	\((\widetilde{\mathcal{P}}_{\nu}(\mathcal{X}\times [0,1]),\delta_{\text{BL},1})\).
	Indeed, since exactly half of the domain maps to \(\delta_{1}\) and the other half to \(\delta_{-1}\), for every
	\(n \in \mathbb{N}\) there exist measure-preserving bijections \(\varphi_n \in S_{[0,1]}\) such that 
	\[
		\mu_n^{\varphi_n(\xi)} = \delta_{1}\mathds{1}_{[0,1/2]}(\xi) + \delta_{-1}\mathds{1}_{(1/2,1]}(\xi).
	\]
	Thus, in the quotient space \(\widetilde{\mathcal{P}}_{\nu}(\mathcal{X}\times [0,1])\), the sequence is constant
	and trivially converges.
	However, this sequence is not relatively compact in the labeled space \((\mathcal{P}_{\nu}(\mathcal{X}\times [0,1]), d_{\text{BL},1})\). 
	The sequence \((\mu_n)_{n\in \mathbb{N}}\) converges narrowly ({\it i.e.}, in \(d_{\text{BL}}\)) to the constant fibered measure
	\(\mu\) defined by \(\mu^{\xi} = \frac{1}{2}\delta_{-1} + \frac{1}{2}\delta_1\) for all \(\xi \in [0,1]\). If the sequence
	were relatively compact in \(d_{\text{BL},1}\), since \(d_{\text{BL}} \leq d_{\text{BL},1}\), it would have a subsequence
	converging to \(\mu\) in \(d_{\text{BL},1}\). Yet, a direct computation shows that for every \(n \in \mathbb{N}\),
	\(d_{\text{BL},1}(\mu_n,\mu) = 1.\)
\end{example}

\subsubsection{ (Labeled) fibered Wasserstein spaces}
When the fibers of a probability measure have finite \(p\)-th moments, we can introduce a new class of metrics by first
evaluating the \(p\)-Wasserstein distance between corresponding fibers, and subsequently taking the \(L^q\) norm of these
distances over the label space \([0,1]\).
\begin{definition}
	We define the fibered Wasserstein space of order \(p,q \in [1,\infty)\) as 
	\[
		\mathcal{P}_{p,q,\nu}(\mathcal{X}\times [0,1]) := \left\{\mu \in \mathcal{P}_{\nu}(\mathcal{X}\times [0,1]) : \int_{0}^{1}W_p(\mu^{\xi},\delta_{x_0})^q d\xi < \infty\right\},
	\]
	where \(x_0\) is an arbitrary fixed point in \(\mathcal{X}\). For any \(\mu,\bar{\mu} \in \mathcal{P}_{p,q,\nu}(\mathcal{X}\times [0,1])\), we define the distance 
	\[
		d_{W_p,q}(\mu,\bar{\mu}) := \left( \int_{0}^{1} W_p(\mu^{\xi},\bar{\mu}^{\xi})^q d\xi \right)^{1/q}.
	\]
\end{definition}

One can also define the space of fibered probability measures with finite \(p\)-th moments as
\[
	\mathcal{P}_{p,\nu}(\mathcal{X}\times [0,1]) := \mathcal{P}_{\nu}(\mathcal{X}\times [0,1]) \cap \mathcal{P}_p(\mathcal{X}\times [0,1]).
\]
It is straightforward to verify that \(\mathcal{P}_{p,\nu}(\mathcal{X}\times [0,1])= \mathcal{P}_{p,p,\nu}(\mathcal{X}\times [0,1]) \). Indeed, equipping the product space \(\mathcal{X}\times [0,1]\) with the standard product metric, a measure \(\mu\) belongs to \(\mathcal{P}_{p,\nu}(\mathcal{X}\times [0,1])\) if and only if its \(p\)-th moment with respect to a reference point \((x_0, 0)\) is finite:
\[
	\int_{0}^{1} \int_{\mathcal{X}} \Big( d_{\mathcal{X}}(x, x_0)^p + \vert \xi\vert^p \Big) d\mu^{\xi}(x) d\xi < \infty.
\]
Because the base space \([0,1]\) is bounded, the distance term \(\vert \xi \vert ^p \leq 1\) is trivially integrable for any probability measure. Therefore, the finiteness condition reduces to 
\[
	\int_{0}^{1} \int_{\mathcal{X}} d_{\mathcal{X}}(x, x_0)^p d\mu^{\xi}(x) d\xi = \int_{0}^{1} W_p(\mu^{\xi},\delta_{x_0})^p d\xi < \infty,
\]
which is precisely the condition for \(\mu\) to belong to \(\mathcal{P}_{p,p,\nu}(\mathcal{X}\times [0,1])\). 

The relation between the classical \(p\)-Wasserstein distance and the fibered Wasserstein distance \(d_{W_p,p}\) is established in the following lemma. 

\begin{lemma} \label{lemma:Wassertein-weaker-than-fibered-Wassertein}
	For any \(\mu,\bar{\mu} \in \mathcal{P}_{p,p,\nu}(\mathcal{X}\times [0,1])\), the following inequality holds:
	\[
		W_p(\mu,\bar{\mu}) \leq d_{W_p,p}(\mu,\bar{\mu}).
	\]
\end{lemma}

\begin{proof}
	For almost every \(\xi \in [0,1]\), let \(\pi^\xi \in \Pi_{o}(\mu^\xi, \bar{\mu}^\xi)\) be an optimal transport
	plan between the fibers \(\mu^\xi\) and \(\bar{\mu}^\xi\). By standard measurable selection theorems,
	the family of optimal plans \(\{\pi^\xi\}_{\xi \in [0,1]}\) can be chosen such that the map
	\(\xi \mapsto \pi^\xi\) is measurable.
	
	We now construct the following coupling \(\pi \in \mathcal{P}\big((\mathcal{X}\times [0,1])^2\big)\) between \(\mu\) and
	\(\bar{\mu}\)  
	\[
		\pi(x,\xi,y,\zeta) := \pi^\xi(x,y)\otimes \delta_{\xi}(d\zeta)\otimes d\xi.
	\]
	It is easy to check that \(\pi \in \Pi(\mu,\bar{\mu})\). Since the classical \(p\)-Wasserstein distance
	is defined as the infimum over all valid couplings, its \(p\)-th power is bounded above by the cost
	associated with our specific coupling \(\pi\):
	\[
		\begin{aligned}
			W_p^p(\mu,\bar{\mu}) &\leq \int_{(\mathcal{X}\times [0,1])^2} \Big( d_{\mathcal{X}}(x,y)^p + \vert \xi - \zeta \vert^p \Big) d\pi(x,\xi,y,\zeta) 
			= \int_{0}^{1} \int_{\mathcal{X}\times\mathcal{X}} \Big( d_{\mathcal{X}}(x,y)^p + 0 \Big) d\pi^\xi(x,y) d\xi \\
			&= \int_{0}^{1} W_p^p(\mu^\xi, \bar{\mu}^\xi) d\xi = d_{W_p,p}^p(\mu,\bar{\mu}).
		\end{aligned}
	\]
	Taking the \(1/p\)-th root on both sides concludes the proof.
\end{proof}

Following the classical characterization of convergence in Wasserstein spaces as narrow convergence plus convergence of moments or 
uniform integrability of the moments (see, {\it e.g.}, \cite[Theorem 6.9]{Villani-09}), convergence with respect to
\(d_{W_p,q}\) can be characterized by convergence under \(d_{\text{BL},1}\) alongside a suitable uniform integrability
condition on the moments of the fibers.
\begin{proposition} \label{prop:characterization-convergence-d_Wp_q}
	Let \((\mu_n)_{n\in \mathbb{N}}\) and \(\mu\) be in \(\mathcal{P}_{p,q,\nu}(\mathcal{X}\times [0,1])\). Then, the following are equivalent:
	\begin{enumerate}[label=(\roman*)]
		\item \(d_{W_p,q}(\mu_n,\mu)\to 0\).
		\item \(d_{\text{BL},1}(\mu_n,\mu)\to 0\) and \(\int_{0}^{1} \vert W_p(\mu_n^{\xi},\delta_{x_0}) - W_p(\mu^{\xi},\delta_{x_0}) \vert^q d\xi \to 0\).
		\item \(d_{\text{BL},1}(\mu_n,\mu)\to 0\) and
		\[
		\lim_{R\to \infty} \sup_{n\in \mathbb{N}} \int_{0}^{1} \left( \int_{\mathcal{X}\setminus B_R(x_0)} d(x,x_0)^p d\mu_n^{\xi}(x)\right)^{q/p} d\xi = 0.
		\]
	\end{enumerate}
\end{proposition}

\begin{proof}
	\(\diamond\) \textsc{Step 1}: Equivalence of (i) and (ii).\\
	Assume that \(d_{W_p,q}(\mu_n,\mu)\to 0\). Observe that for a.e. \(\xi \in [0,1]\), the following inequalities hold:
	\[
	\begin{aligned}
		&d_{\text{BL}}(\mu_n^\xi,\mu^{\xi})\leq W_{p}(\mu_n^{\xi},\mu^{\xi}), \\
		&\left\vert W_p(\mu_n^{\xi},\delta_{x_0}) -W_p(\mu^{\xi},\delta_{x_0}) \right\vert \leq W_p(\mu_n^{\xi},\mu^{\xi}).
	\end{aligned}
	\]
	Taking the \(L^q\) norm over \(\xi\) in both inequalities immediately yields the two limits in condition (ii). Conversely, suppose that
	(ii) holds. Given any subsequence \((\mu_{n_k})_{k\in \mathbb{N}}\) of \((\mu_n)_{n\in \mathbb{N}}\), the conditions in (ii)
	allow us to extract a further subsequence (not relabeled) such that \(d_{\text{BL}}(\mu_{n_k}^{\xi},\mu^{\xi})\to 0\) and
	\(W_p(\mu_{n_k}^{\xi},\delta_{x_0})\to W_p(\mu^{\xi},\delta_{x_0})\) for a.e. \(\xi \in [0,1]\).
	By the classical characterization
	of convergence in the \(p\)-Wasserstein space, we conclude that \(W_p(\mu_{n_k}^{\xi},\mu^{\xi}) \to 0\) for a.e.
	\(\xi \in [0,1]\). Since \(W_p(\mu_{n_k}^{\xi},\mu^{\xi})^q \leq 2^{q-1} \big(W_p(\mu_{n_k}^{\xi},\delta_{x_0})^q + W_p(\mu^{\xi},\delta_{x_0})^q \big)\),
	by the dominated convergence theorem we have that
	\[
	\lim_{k \to \infty} d_{W_p,q}(\mu_{n_k},\mu) = \lim_{k \to \infty} \int_{0}^{1} W_p(\mu_{n_k}^{\xi},\mu^{\xi})^q d\xi = 0.
	\]
	Since every subsequence contains a further subsequence converging to zero, the entire sequence converges,
	yielding \(d_{W_p,q}(\mu_n,\mu)\to 0\).

	\medskip

	\(\diamond\) \textsc{Step 2}: (ii) \(\Rightarrow\) (iii). \\ 
	We proceed by contradiction. Suppose that (iii) fails. Then, there exist some \(\varepsilon >0\) and subsequences
	\((\mu_{n_k})_{k\in \mathbb{N}}\) and \((R_k)_{k\in \mathbb{N}}\) with \(R_k \to \infty\) such that
	\[
		\int_{0}^{1} \left( \int_{\mathcal{X}\setminus B_{R_k}(x_0)} d(x,x_0)^p d\mu_{n_k}^{\xi}(x)\right)^{q/p} d\xi \geq \varepsilon, \quad \forall k\in \mathbb{N}.
	\]
	Define the sequences of functions
	\[
		f_k(\xi):=\left( \int_{\mathcal{X}\setminus B_{R_k}(x_0)} d(x,x_0)^p d\mu_{n_k}^{\xi}(x)\right)^{q/p}\leq W^q_p(\mu_{n_k}^{\xi},\delta_{x_0}):=g_k(\xi).
	\]
	Since \((g_k)_{k\in \mathbb{N}}\) converges to \(g(\xi):=W^q_p(\mu^{\xi},\delta_{x_0})\) in \(L^1\), it is uniformly
	integrable. Thus, also the dominated sequence \((f_k)_{k\in \mathbb{N}}\) is uniformly integrable. 
	Furthermore, because \(d_{\text{BL},1}(\mu_{n_k}, \mu) \to 0\) and \(\Vert g_{n_k}-g\Vert_{L^1}\to 0\), 
	we can extract a  further subsequence (not relabeled) such that 
	\(d_{\text{BL}}(\mu_{n_k}^{\xi},\mu^{\xi}) \to 0\) and \(g_{n_k}(\xi)\to g(\xi)\) simultaneously for a.e. 
	\(\xi \in [0,1]\). Standard \(p\)-Wasserstein convergence implies that 
	\[
		\lim_{k \to \infty} f_{n_k}(\xi) = 0 \quad \text{for } \text{a.e. } \xi \in [0,1].
	\]
	Since \((f_{n_k})_{k\in \mathbb{N}}\) is uniformly integrable, Vitali's convergence theorem ensures that
	\[
		\lim_{k \to \infty} \int_{0}^{1} f_{n_k}(\xi) d\xi = 0,
	\]
	which contradicts our assumption that the integral is bounded below by \(\varepsilon\).

	\medskip

	\(\diamond\) \textsc{Step 3}: \((iii)\) \(\Rightarrow\) \((ii)\).\\
	Integrating the following inequality with respect to a probability measure \(\eta\) on \(\mathcal{P}_p(\mathcal{X})\)  
	\[
		d(x,x_0)^p\leq \min(d(x,x_0)^p, R^p) + d(x,x_0)^p \mathds{1}_{\mathcal{X}\setminus B_R(x_0)}(x),
	\]
	and using the inequality \((a+b)^{1/p} \leq a^{1/p}+b^{1/p}\) for \(a,b\geq 0\), we have that
	\[
		W_p(\eta,\delta_{x_0}) \leq \left( \int_{\mathcal{X}} \min(d(x,x_0)^p, R^p) d\eta(x) \right)^{1/p} + \left( \int_{\mathcal{X}\setminus B_{R}(x_0)} d(x,x_0)^p d\eta(x) \right)^{1/p} := M_R(\eta) + m_R(\eta).
	\]
	Since \(W_p(\eta,\delta_{x_0})\geq M_R(\eta)\), the previous inequality also implies that
	\[
		\vert W_p(\eta,\delta_{x_0})-M_R(\eta)\vert \leq m_R(\eta).
	\]
	Thus, applying the triangle inequality, we deduce that
	\[
		\begin{aligned}
			\vert W_p(\mu_n^{\xi},\delta_{x_0}) - W_p(\mu^{\xi},\delta_{x_0}) \vert &\leq \vert W_p(\mu_n^{\xi},\delta_{x_0}) - M_R(\mu_n^{\xi}) \vert + \vert M_R(\mu_n^{\xi}) - M_R(\mu^{\xi}) \vert + \vert M_R(\mu^{\xi}) - W_p(\mu^{\xi},\delta_{x_0}) \vert \\
			&\leq m_R(\mu_n^{\xi}) + \vert M_R(\mu_n^{\xi}) - M_R(\mu^{\xi}) \vert + m_R(\mu^{\xi}).
		\end{aligned}
	\]
	Taking the \(L^q\) norm over \(\xi\) yields
	\[
		\Vert W_p(\mu_n^{\cdot},\delta_{x_0}) - W_p(\mu^{\cdot},\delta_{x_0}) \Vert_{L^q} \leq   \Vert m_R(\mu^{\cdot}_n) \Vert_{L^q} + \Vert M_R(\mu^{\cdot}_n) - M_R(\mu^{\cdot}) \Vert_{L^q} + \Vert m_R(\mu^{\cdot}) \Vert_{L^q}.
	\]
	We claim that the middle term vanishes as \(n \to \infty\). Take any subsequence \((\mu_{n_k})_{k\in \mathbb{N}}\).
	Since \(d_{\text{BL},1}(\mu_{n_k},\mu)\to 0\), there exists a further subsequence (not relabeled) such that
	\(d_{\text{BL}}(\mu_{n_k}^{\xi},\mu^{\xi}) \to 0\) for a.e. \(\xi \in [0,1]\). Since the map
	\(x\mapsto \min(d(x,x_0)^p,R^p)\) is bounded and continuous, narrow convergence implies that 
	\(\left\vert M_R(\mu_{n_{k}}^\xi) - M_R(\mu^\xi) \right\vert \to 0\) almost everywhere. Since this difference is
	uniformly bounded by \(2R\), the dominated convergence theorem ensures \(\Vert M_R(\mu^{\cdot}_{n_k}) - M_R(\mu^{\cdot}) \Vert_{L^q} \to 0\).
	Since this is valid for any subsequence, the full sequence limit is established. As a result, we arrive at the bound:
	\[
		\limsup_{n\to\infty} \Vert W_p(\mu_n^{\cdot},\delta_{x_0}) - W_p(\mu^{\cdot},\delta_{x_0}) \Vert_{L^q} \leq  \sup_{n\in \mathbb{N}} \Vert m_R(\mu^{\cdot}_n) \Vert_{L^q} +  \Vert m_R(\mu^{\cdot}) \Vert_{L^q}.
	\]
	Finally, taking the limit as \(R\to \infty\) and invoking the uniform bound on the tail of the \(p\)-th moments in
	condition (iii), both terms on the right-hand side vanish, concluding the proof.
\end{proof}

\subsubsection{(Unlabeled) fibered Wasserstein spaces}

We are now ready to introduce the unlabeled version of the space \(\mathcal{P}_{p,q,\nu}(\mathcal{X}\times [0,1])\).
In this setting, relatively compact sets can be characterized by combining our compactness results for the quotient
space \((\widetilde{\mathcal{P}}_{\nu}(\mathcal{X}\times [0,1]),\delta_{\text{BL},1})\) with the preceding characterization
of convergence in the fibered Wasserstein metric.

\begin{definition}
	The unlabeled fibered Wasserstein distance \(\delta_{W_p,q}\) between any two fibered probability measures
	\(\mu,\bar{\mu} \in \mathcal{P}_{p,q,\nu}(\mathcal{X}\times [0,1])\) is defined as
	\[
		\delta_{W_p,q}(\mu,\bar{\mu}) := \inf_{\varphi \in S_{[0,1]}} d_{W_p,q}(\mu,\bar{\mu}^{\varphi}).
	\]
\end{definition}
Similar to the bounded-Lipschitz case, \(\delta_{W_p,q}\) is merely a pseudo-metric on
\(\mathcal{P}_{p,q,\nu}(\mathcal{X}\times [0,1])\), and we can rewrite it as a minimum:
\[
	\delta_{W_p,q}(\mu,\bar{\mu}) = \min_{\varphi,\psi \in \bar{S}_{[0,1]}} d_{W_p,q}(\mu^{\varphi},\bar{\mu}^{\psi}).
\]
Therefore, \(\delta_{W_p,q}\) defines a valid distance on the quotient space 
\(\widetilde{\mathcal{P}}_{p,q,\nu}(\mathcal{X}\times [0,1])\) where weakly isomorphic fibered measures are identified.
This naturally yields the inclusion \(\widetilde{\mathcal{P}}_{p,q,\nu}(\mathcal{X}\times [0,1]) \subset \widetilde{\mathcal{P}}_{\nu}(\mathcal{X}\times [0,1])\). Furthermore, the metric domination \(d_{\text{BL}} \leq W_p\) guarantees that
\[ 
	\delta_{\text{BL},1}(\mu,\bar{\mu}) \leq \delta_{W_p,q}(\mu,\bar{\mu}).
\]
This implies that the topology induced by \(\delta_{W_p,q}\) is finer than the topology induced by \(\delta_{\text{BL},1}\).
Building on Proposition \ref{prop:characterization-convergence-d_Wp_q}, we can now characterize convergence
under \(\delta_{W_p,q}\) entirely in terms of the weaker unlabeled distance \(\delta_{\text{BL},1}\),
alongside the uniform integrability of the moments of the fibers.
\begin{proposition} \label{prop:characterization-convergence-delta_W_pq}
	Let \((\mu_n)_{n\in \mathbb{N}}\) and \(\mu\) be in \(\widetilde{\mathcal{P}}_{p,q,\nu}(\mathcal{X}\times [0,1])\). Then, the following are equivalent:
	\begin{enumerate}[label=(\roman*)]
		\item \(\delta_{W_p,q}(\mu_n,\mu)\to 0\).
		\item \(\delta_{\text{BL},1}(\mu_n,\mu)\to 0\) and \(\displaystyle \lim_{R\to \infty} \sup_{n\in \mathbb{N}} \int_{0}^{1} \left( \int_{\mathcal{X}\setminus B_R(x_0)} d(x,x_0)^p d\mu_n^{\xi}(x)\right)^{q/p} d\xi = 0\).
	\end{enumerate}
\end{proposition}

\begin{proof}
	By definition, \(\delta_{W_p,q}(\mu_n,\mu) \to 0\) if and only if there exists a sequence of bijective 
	measure-preserving maps \((\varphi_n)_{n\in \mathbb{N}}\subset S_{[0,1]}\) such that
	\(d_{W_p,q}(\mu_n^{\varphi_n},\mu)\to 0\). 
	Applying Proposition \ref{prop:characterization-convergence-d_Wp_q}-(iii), this convergence occurs if and only if
	\(d_{\text{BL},1}(\mu_n^{\varphi_n},\mu)\to 0\) and 
	\[
		\lim_{R\to \infty} \sup_{n\in \mathbb{N}} \int_{0}^{1} \left( \int_{\mathcal{X}\setminus B_R(x_0)} d(x,x_0)^p d\mu_n^{\varphi_n(\xi)}(x)\right)^{q/p} d\xi = 0.
	\]
	The existence of a sequence \((\varphi_n)_{n\in\mathbb{N}}\) satisfying \(d_{\text{BL},1}(\mu_n^{\varphi_n},\mu)\to 0\)
	is precisely equivalent to \(\delta_{\text{BL},1}(\mu_n,\mu)\to 0\), and because each \(\varphi_n\) is a measure-preserving map on \([0,1]\), we have
	\[
		\int_{0}^{1} \left( \int_{\mathcal{X}\setminus B_R(x_0)} d(x,x_0)^p d\mu_n^{\varphi_n(\xi)}(x)\right)^{q/p} d\xi = \int_{0}^{1} \left( \int_{\mathcal{X}\setminus B_R(x_0)} d(x,x_0)^p d\mu_n^{\xi}(x)\right)^{q/p} d\xi,
	\] 
	which gives us the desired characterization.
\end{proof}

\begin{theorem}
	The space \((\widetilde{\mathcal{P}}_{p,q,\nu}(\mathcal{X}\times [0,1]),\delta_{W_p,q})\) is Polish. Moreover we have:
	\begin{enumerate} [label=(\roman*)]
		\item A subset \(\mathcal{K}\subset (\widetilde{\mathcal{P}}_{p,q,\nu}(\mathcal{X}\times [0,1]),\delta_{W_p,q})\) is relatively compact if and only if it is tight and  
		\begin{equation} \label{eq:uniform-integrability-fibered-measures}
		 	\lim_{R\to \infty} \sup_{\mu \in \mathcal{K}} \int_{0}^{1} \left( \int_{\mathcal{X}\setminus B_R(x_0)} d(x,x_0)^p d\mu^{\xi}(x)\right)^{q/p} d\xi = 0.
		\end{equation}
		\item The space \((\widetilde{\mathcal{P}}_{p,q,\nu}(\mathcal{X}\times [0,1]),\delta_{W_p,q})\) is compact if and only if \(\mathcal{X}\) is compact.
	\end{enumerate}
\end{theorem}

\begin{proof}
	\(\diamond\) \textsc{Step 1}: Proof of \((i)\).\\ 
	Assume that \(\mathcal{K}\subset (\widetilde{\mathcal{P}}_{p,q,\nu}(\mathcal{X}\times [0,1]),\delta_{W_p,q})\) is tight
	and satisfies \eqref{eq:uniform-integrability-fibered-measures}. Take any sequence \((\mu_n)_{n\in \mathbb{N}}\subset \mathcal{K}\).
	Since \(\mathcal{K}\) is tight, it is relatively compact in \((\widetilde{\mathcal{P}}_{\nu}(\mathcal{X}\times [0,1]),\delta_{\text{BL},1})\). 
	Thus, there exists a subsequence \((\mu_{n_k})_{k\in \mathbb{N}}\) converging to some limit
	\(\mu \in \widetilde{\mathcal{P}}_{\nu}(\mathcal{X}\times [0,1])\) in the \(\delta_{\text{BL},1}\) distance.
	Since \((\mu_n)_{n\in\mathbb{N}}\) satisfies the uniform integrability condition
	\eqref{eq:uniform-integrability-fibered-measures}, the limit measure \(\mu\) must belong to
	\(\widetilde{\mathcal{P}}_{p,q,\nu}(\mathcal{X}\times [0,1])\). Furthermore, the characterization
	of convergence established in Proposition \ref{prop:characterization-convergence-delta_W_pq} guarantees that \(\delta_{W_p,q}(\mu_{n_k},\mu)\to 0\). 
	Conversely, assume that \(\mathcal{K}\) is relatively compact in \((\widetilde{\mathcal{P}}_{p,q,\nu}(\mathcal{X}\times [0,1]),\delta_{W_p,q})\). 
	Because the topology induced by \(\delta_{W_p,q}\) is finer than that of \(\delta_{\text{BL},1}\), \(\mathcal{K}\) is
	also relatively compact in \((\widetilde{\mathcal{P}}_{\nu}(\mathcal{X}\times [0,1]),\delta_{\text{BL},1})\),
	which implies that \(\mathcal{K}\) is tight. We argue by contradiction to show that
	\eqref{eq:uniform-integrability-fibered-measures} holds. If it does not, there exists some \(\varepsilon>0\)
	and sequences \((R_n)_{n\in\mathbb{N}}\) and \((\mu_n)_{n\in\mathbb{N}}\subset \mathcal{K}\) such that
	\(R_n \to \infty\) and
	\[
		\int_{0}^{1} \left( \int_{\mathcal{X}\setminus B_{R_n}(x_0)} d(x,x_0)^p d\mu_n^{\xi}(x)\right)^{q/p} d\xi \geq \varepsilon, \quad \forall n\in \mathbb{N}.
	\]
	Due to the relative compactness of \(\mathcal{K}\) in
	\((\widetilde{\mathcal{P}}_{p,q,\nu}(\mathcal{X}\times [0,1]),\delta_{W_p,q})\), there exists a subsequence
	(not relabeled) such that \(\delta_{W_p,q}(\mu_n,\mu)\to 0\) for some
	\(\mu \in \widetilde{\mathcal{P}}_{p,q,\nu}(\mathcal{X}\times [0,1])\). However, applying the characterization
	of convergence in \(\delta_{W_p,q}\) in Proposition \ref{prop:characterization-convergence-delta_W_pq} to this
	convergent subsequence directly contradicts our assumption.
	
	\medskip

	\(\diamond\) \textsc{Step 2}: Proof of \((ii)\).\\ 
	If \(\mathcal{X}\) is compact, then the uniform integrability condition
	\eqref{eq:uniform-integrability-fibered-measures} and the tightness are automatically satisfied for the entire
	space \(\widetilde{\mathcal{P}}_{p,q,\nu}(\mathcal{X}\times [0,1])\). Thus, by point (i), the space is compact. 
	Conversely, suppose that \((\widetilde{\mathcal{P}}_{p,q,\nu}(\mathcal{X}\times [0,1]),\delta_{W_p,q})\) is compact.
	The Wasserstein space \(\mathcal{P}_p(\mathcal{X})\) can be viewed as a closed subset of
	\(\widetilde{\mathcal{P}}_{p,q,\nu}(\mathcal{X}\times [0,1])\) via the natural embedding \(\eta \in \mathcal{P}_p(\mathcal{X}) \mapsto \eta \otimes d\xi \in \widetilde{\mathcal{P}}_{p,q,\nu}(\mathcal{X}\times [0,1])\).
	Since closed subsets of compact spaces are compact, \((\mathcal{P}_p(\mathcal{X}),W_p)\) is compact, which 
	implies that \(\mathcal{X}\) is compact.

	\medskip 

	\(\diamond\) \textsc{Step 3}: Proof of the Polish property.\\
	Separability follows from the fact that step fibered measures taking values in a countable dense subset of finitely supported measures on \(\mathcal{X}\) form a dense countable set, much like in the \(\delta_{\text{BL},1}\) case.	
	For completeness, let \((\mu_n)_{n\in\mathbb{N}}\) be a Cauchy sequence in 
	\((\widetilde{\mathcal{P}}_{p,q,\nu}(\mathcal{X}\times [0,1]),\delta_{W_p,q})\). Since 
	\(\delta_{\text{BL},1} \leq \delta_{W_p,q}\), the sequence is also Cauchy in \(\delta_{\text{BL},1}\).
	Since \((\widetilde{\mathcal{P}}_{\nu}(\mathcal{X}\times [0,1]),\delta_{\text{BL},1})\) is complete, 
	there exists a limit \(\mu \in \widetilde{\mathcal{P}}_{\nu}(\mathcal{X}\times [0,1])\) such that
	\(\delta_{\text{BL},1}(\mu_n,\mu) \to 0\). Furthermore, the Cauchy property in \(\delta_{W_p,q}\)
	implies that the sequence is bounded in \(p\)-th moments and satisfies the uniform integrability condition
	\eqref{eq:uniform-integrability-fibered-measures}. This ensures \(\mu \in \widetilde{\mathcal{P}}_{p,q,\nu}(\mathcal{X}\times [0,1])\).
	By Proposition \ref{prop:characterization-convergence-delta_W_pq}, the convergence
	\(\delta_{\text{BL},1}(\mu_n,\mu) \to 0\) together with the uniform integrability forces \(\delta_{W_p,q}(\mu_n,\mu) \to 0\). Thus, the space is complete, and therefore Polish.
\end{proof}
Similar to the classical case, a bound on higher-order moments guarantees both tightness and uniform integrability of
the fiber-wise moments, which together suffice for relative compactness in
\((\widetilde{\mathcal{P}}_{p,q,\nu}(\mathcal{X}\times [0,1]),\delta_{W_p,q})\).

\begin{proposition} \label{prop:higher-order-moments-implies-compactness}
	Let \(p,q \geq 1\). If a subset \(\mathcal{K}\subset (\widetilde{\mathcal{P}}_{p,q,\nu}(\mathcal{X}\times [0,1]),\delta_{W_p,q})\) satisfies that for some constants \(\varepsilon>0\) and \(M\geq 0\),
	\[
	\sup_{\mu \in \mathcal{K}} \int_{0}^{1} \left( \int_{\mathcal{X}} d(x,x_0)^{p+\varepsilon} d\mu^{\xi}(x)\right)^{q/p} d\xi \leq M,
	\]
	then \(\mathcal{K}\) is relatively compact in \((\widetilde{\mathcal{P}}_{p,q,\nu}(\mathcal{X}\times [0,1]),\delta_{W_p,q})\).
\end{proposition}

\begin{proof}
	By the characterization of relatively compact sets in \(\delta_{W_p,q}\), it suffices to show that
	\(\mathcal{K}\) satisfies the uniform integrability condition \eqref{eq:uniform-integrability-fibered-measures}
	and is tight. For the uniform integrability, take \(R>0\). For any \(\mu \in \mathcal{K}\), since \(d(x,x_0) \geq R\) on the set \(\mathcal{X}\setminus B_R(x_0)\), we have \(d(x,x_0)^p \leq R^{-\varepsilon} d(x,x_0)^{p+\varepsilon}\). Thus,
	\[
		\begin{aligned}
			\int_{0}^{1} \left( \int_{\mathcal{X}\setminus B_R(x_0)} d(x,x_0)^p d\mu^{\xi}(x)\right)^{q/p} d\xi &\leq \int_{0}^{1} \left( \int_{\mathcal{X}\setminus B_R(x_0)} \frac{d(x,x_0)^{p+\varepsilon}}{R^{\varepsilon}} d\mu^{\xi}(x)\right)^{q/p} d\xi \\   
			&\leq \frac{1}{R^{\varepsilon q/p}} \int_{0}^{1} \left( \int_{\mathcal{X}} d(x,x_0)^{p+\varepsilon} d\mu^{\xi}(x)\right)^{q/p} d\xi \leq \frac{M}{R^{\varepsilon q/p}}.
		\end{aligned}
	\]
	As \(R \to \infty\), this bound vanishes uniformly for all \(\mu \in \mathcal{K}\), proving that the uniform
	integrability condition holds. To prove uniform tightness, we bound the tail of the average measures,
	\(M_{\mu}(\mathcal{X}\setminus B_R(x_0)) = \int_0^1 \mu^\xi(\mathcal{X}\setminus B_R(x_0)) d\xi\). 
	By Markov's inequality, the tail of each fiber satisfies 
	\[
		\mu^{\xi}(\mathcal{X}\setminus B_R(x_0)) \leq \frac{1}{R^{p+\varepsilon}} \int_{\mathcal{X}\setminus B_R(x_0)} d(x,x_0)^{p+\varepsilon} d\mu^{\xi}(x) := \frac{\varphi(\xi)}{R^{p+\varepsilon}}.
	\]
	To bound the integral over \(\xi\), we consider two cases depending on the ratio \(q/p\).
	If \(q/p \geq 1\), by Jensen's inequality applied to the convex function \(t \mapsto t^{q/p}\), we have
	\[
		\int_{0}^{1} \mu^{\xi}(\mathcal{X}\setminus B_R(x_0)) d\xi \leq \left( \int_{0}^{1} \left(\frac{\varphi(\xi)}{R^{p+\varepsilon}}\right)^{q/p} d\xi \right)^{p/q} \leq \frac{M^{p/q}}{R^{p+\varepsilon}}.
	\]
	If \(q/p < 1\), since \(\mu^{\xi}\) is a probability measure, \(\mu^{\xi}(\mathcal{X}\setminus B_R(x_0)) \leq \mu^{\xi}(\mathcal{X}\setminus B_R(x_0))^{q/p}\). Integrating this yields
	\[
		\int_{0}^{1} \mu^{\xi}(\mathcal{X}\setminus B_R(x_0)) d\xi \leq \int_{0}^{1} \mu^{\xi}(\mathcal{X}\setminus B_R(x_0))^{q/p} d\xi \leq \int_{0}^{1} \left(\frac{\varphi(\xi)}{R^{p+\varepsilon}}\right)^{q/p} d\xi \leq \frac{M}{R^{(p+\varepsilon)q/p}}.
	\]
	In both cases, \(\sup_{\mu \in \mathcal{K}} M_{\mu}(\mathcal{X}\setminus B_R(x_0)) \to 0\) as \(R \to \infty\).
	This ensures that \(\mathcal{K}\) is tight, which concludes the proof.
\end{proof}

\subsection{Graphons vs probability-graphons\label{sec:limits-of-dense-weighted-graphs}} 

The study of graph limits has been a central topic in graph theory and combinatorics for many years. 
Since the seminal work of Lovász and Szegedy \cite{LS-06}, \emph{graphons} have been used to model limits of
dense graph sequences, providing a powerful framework to bridge discrete graph theory and continuous analysis.
In this section, we review some properties of the graphon space. We will see that since the natural metric
in this space, the cut distance, is not well-behaved under non-linear transformations, it becomes necessary to work
in the space of \emph{probability-graphons}. The notion of convergence in this enriched space is strong enough
to handle non-linearities, which is strictly required in our context due to the generally non-linear nature of
the plasticity function \(\Gamma\).

\subsubsection{The space of graphons and the cut distance} 

\begin{definition}
	Given a constant \(W \geq 0\), we define the set of graphons with uniformly bounded weights in \([-W,W]\) as 
	\[
		\mathcal{G}_W := \left\{w \in L^{\infty}([0,1]^2) : \|w\|_{L^{\infty}} \leq W\right\}.
	\]
	For any two graphons \(w,\bar{w} \in \mathcal{G}_W\), the labeled cut distance is defined as 
	\[
		d_{\square}(w,\bar{w}) := \sup_{S,T\subset [0,1]} \left\vert \iint_{S\times T} (w(\xi,\xi')-\bar{w}(\xi,\xi')) d\xi d\xi' \right\vert,
	\]
	where the supremum is taken over all measurable sets \(S,T\subset [0,1]\). The unlabeled cut distance is defined as 
	\[
		\delta_{\square}(w,\bar{w}) := \inf_{\varphi \in S_{[0,1]}} d_{\square}(w,\bar{w}^{\varphi}),
	\]
	where \(\bar{w}^{\varphi}\) represents the rearrangement of \(\bar{w}\) by \(\varphi\), that is,
	\(\bar{w}^{\varphi}(\xi,\xi') = \bar{w}(\varphi(\xi),\varphi(\xi'))\).
\end{definition}

\begin{remark}
	Although graphons were originally defined as symmetric measurable functions, the classical proof of compactness
	\cite[Theorem 5.1]{LS-07} remains unchanged when this requirement is relaxed. Therefore, we omit the symmetry
	assumption to accommodate the limits of large directed graphs.
\end{remark}
It is analytically convenient to reformulate the labeled cut distance in terms of an operator norm.
To do so, we associate to each graphon \(w\in \mathcal{G}_W\) the bounded linear operator
\(T^w: L^{\infty}([0,1])\to L^1([0,1])\) defined for any \(\psi \in L^{\infty}([0,1])\) by
\[
	T^w[\psi](\xi) = \int_0^1 w(\xi,\xi')\psi(\xi')d\xi',\quad \text{for a.e. } \xi\in [0,1].
\]
It turns out that the operator norm induced by \(T^w\), given by
\begin{equation} \label{eq:operator-norm-cut-distance}
	\Vert T \Vert_{\infty \to 1}= \sup_{ \Vert \psi \Vert_{\infty} \leq 1} \int_{0}^{1} \left\vert \int_{0}^{1}\psi(\xi') w(\xi,\xi')d\xi' \right\vert d\xi,
\end{equation}
is equivalent to the labeled cut distance. Since we can approximate every bounded measurable function by characteristic
functions on measurable sets, the following bounds hold:
\begin{equation} \label{eq:equivalence-cut-distance-operator-norm}
	d_{\square}(w,\bar{w})\leq \Vert T^w-T^{\bar{w}} \Vert_{\infty \to 1} \leq 4d_{\square}(w,\bar{w}).
\end{equation}
Because any two graphons that are rearrangements of each other have an unlabeled cut distance of zero,
\(\delta_{\square}\) is only a pseudometric on \(\mathcal{G}_W\). It becomes a true metric on the quotient space
\(\widetilde{\mathcal{G}}_W\), obtained by identifying two graphons \(w,\bar{w}\in \mathcal{G}_W\) if
\(\delta_{\square}(w,\bar{w})=0\), which precisely identifies graphons that are weakly isomorphic.
A fundamental result in the theory of dense graph limits states that
\((\widetilde{\mathcal{G}}_{W},\delta_{\square})\) is a compact metric space \cite[Theorem 5.1]{LS-07}.
Moreover, let \(G_W\) denote the set of all weighted graphs \(G=(V,E,w)\) with vertex set \(V=\llbracket 1,N \rrbracket\)
and weights in \([-W,W]\). To every graph \(G \in G_W\), we can associate a step graphon \(w_G \in \mathcal{G}_W\)
defined as 
\[
	w_G(\xi,\xi') = \sum_{i,j=1}^N w_{ij} \mathds{1}_{I_i^N\times I_j^N} (\xi,\xi'), \quad \xi,\xi'\in [0,1],
\]
where \(I_i^N:=[\frac{i-1}{N},\frac{i}{N})\). Denoting the natural embedding by \(i:G_W\to \mathcal{G}_W\) via \(G\mapsto i(G):=w_{G}\),
a direct consequence of the weak regularity lemma for graphons \cite[Lemma 3.1]{LS-07} is that \(i(G_W)\) is dense
in \(\mathcal{G}_W\). In other words, the space of finite weighted graphs is dense in the space of graphons with 
respect to the cut distance. 

\begin{remark}[The limit is non-trivial only for dense graph sequences]
	As the name of the section indicates, graphons are used to model limits of dense graph sequences. A sequence of weighted graphs \((G^N=(V^N,E^N,w^N))_{N\in\mathbb{N}}\) with \(V^N =\llbracket 1,N \rrbracket\) is called dense if its weighted edge density stays bounded away from zero as \(N\to \infty\), {\it i.e.}, if
	\[
		\liminf_{N\to \infty} \frac{1}{N^2} \sum_{i,j=1}^{N}\vert w^N_{ij}\vert > 0.
	\]
	If the sequence consists of simple graphs (where \(w^N_{ij}\in \{0,1\}\)), this condition is equivalent to the
	number of edges growing proportionally to \(N^2\). If the sequence is not dense, its limit in the cut distance
	is trivially the zero graphon. This follows because the associated graphon \(w_{G^N}\) satisfies
	\[
		\Vert w_{G^N} \Vert_{L^1} = \frac{1}{N^2}\sum_{i,j=1}^{N} \vert w^N_{ij}\vert.
	\]
	Therefore, if the sequence is sparse, \(\Vert w_{G^N} \Vert_{L^1}\to 0\), which immediately implies \(d_{\square}(w_{G^N},0)\to 0\). 
\end{remark}
Convergence in the unlabeled cut distance is equivalent to \emph{left-convergence}, offering a more combinatorial perspective
through homomorphism densities. If \(F\) is a simple graph and \(w\) is a graphon, the homomorphism density of \(F\) in \(w\),
denoted by \(t(F,w)\), represents the probability that a random map from \(V(F)\) to \([0,1]\) preserves the adjacency structure.
It is defined as:
\begin{equation} \label{eq:homomorphism-density-graphons}
	t(F,w) = \int_{[0,1]^{\vert V(F)\vert}} \prod_{(i,j)\in E(F)} w(\xi_i,\xi_j) d\xi_1\ldots d\xi_{\vert V(F)\vert}.
\end{equation}
For a weighted graph \(G\), its homomorphism density is simply defined as \(t(F,G) := t(F,w_G)\). If \(G\) is a simple
graph, \(t(F,G)\) precisely computes the proportion of graph homomorphisms from \(F\) to \(G\) among all possible maps
from \(V(F)\) to \(V(G)\). 

A sequence of graphons \((w^N)_{N\in\mathbb{N}}\subset \mathcal{G}_W\) is said to be left-convergent to a graphon
\(w\in \mathcal{G}_W\) if \(t(F,w^N)\to t(F,w)\) for every simple graph \(F\). A celebrated result 
\cite[Theorem 11.5]{Lovasz-12} establishes that this is strictly equivalent to convergence in the
unlabeled cut distance:
\begin{equation} \label{eq:homomorphism-density-equivalent-cut-distance}
	\delta_{\square}(w^N,w)\to 0 \iff t(F,w^N)\to t(F,w) \quad \forall \; F \text{ simple graph.}
\end{equation}

Although graphons might seem to be an ideal modeling choice for the limiting network structure in mean-field limits
of non-exchangeable systems, convergence with respect to the cut distance lacks a property that is crucial in our context.
Specifically, the cut distance is not well-behaved under non-linear transformations of the graphon, as was proven
in \cite[Proposition 5.1]{HJ-25}. 

\begin{proposition}
	Let \(\Phi \in L^{\infty}([-W,W])\). The map 
	\[
		(\mathcal{G}_W, d_{\square}) \to (\mathcal{G}_{\| \Phi \|_{L^{\infty}}}, d_{\square})
	\]
	defined by \(w \mapsto \Phi_w\), where \(\Phi_w(\xi,\xi')=\Phi(w(\xi,\xi'))\), is continuous if and only if
	\(\Phi\) is an affine linear function.
\end{proposition}
This is a serious issue since in many applications, the graphon is not only used to model the limit of a sequence of graphs,
but also to model the evolution of the system, and therefore we need to consider non-linear transformations of the graphon,
as is the case in our context where the plasticity function \(\Gamma\) is in general non-linear.

\begin{example}[Graphons and non-linear transformations] \label{ex:non-linear-transformation-graphon}
	Consider a sequence of Erdős–Rényi random graphs \(G^N=G(N,p)\) on \(N\) vertices with connection probability \(p=1/2\).
	It is a classical result (see, {\it e.g.}, \cite[Theorem 4.4.2]{Zhao-23}) that the associated step graphons \(w_{G^N}\)
	converge almost surely in the cut distance to the constant graphon \(w \equiv 1/2\). 
	
	Now, consider the non-linear transformation \(\Phi(x)=x^2\). Since the entries of the adjacency matrices of \(G^N\) are
	strictly zeros and ones, squaring them leaves the graphs entirely unchanged. Thus, the sequence of transformed graphons
	\(\Phi(w_{G^N}) = w_{G^N}\) still converges almost surely to the constant graphon \(1/2\). Conversely, applying the
	non-linear transformation directly to the limit graphon yields \(\Phi(w) \equiv 1/4\), highlighting the failure
	of continuity.
\end{example}
It turns out that the solution to this problem is to consider the space of \emph{probability-graphons},
which has a strengthened notion of convergence that is well-behaved under non-linear transformations.

\subsubsection{The space of  probability-graphons and the cut distance}

Probability-graphons were first introduced by Lovász and Szegedy in \cite{LS-10-arxiv} under the name of
\(\mathcal{K}\)-graphons, as they were interested in the limits of edge-decorated graphs taking values in a compact,
separable topological space \(\mathcal{K}\). They recognized that by strengthening the notion of left-convergence with
test graphs \(F\) whose edges are decorated by bounded continuous functions in \(C_b(\mathcal{K};\mathbb{R})\),
the appropriate limiting object becomes a measurable map \(q:[0,1]^2\to \mathcal{P}(\mathcal{K})\). For such an object,
the homomorphism density of a \(C_b(\mathcal{K};\mathbb{R})\)-decorated graph \(F\) in \(q\) is defined as
\begin{equation} \label{eq:homomorphism-density-probability-graphons}
	t(F,q) = \int_{[0,1]^{\vert V(F)\vert}} \prod_{(i,j)\in E(F)} \int_{\mathcal{K}} F_{ij}(w) dq^{\xi_i,\xi_j}(w) d\xi_1\ldots d\xi_{\vert V(F)\vert}.
\end{equation}  
Accordingly, a sequence of \(\mathcal{K}\)-graphons \((q^N)_{N\in\mathbb{N}}\) is said to be left-convergent
to a \(\mathcal{K}\)-graphon \(q\) if \(t(F,q^N)\to t(F,q)\) for every \(C_b(\mathcal{K};\mathbb{R})\)-decorated test graph \(F\).
\begin{remark}
	This notion of left-convergence is significantly stronger than classical left-convergence for graphons, as the test
	class is vastly expanded to include continuous edge decorations mapping from \(\mathcal{K}\) to \(\mathbb{R}\).
	Naturally, if we restrict the test graphs \(F\) such that every edge is decorated by the identity function
	\(F_{ij}(w) = w\) (assuming \(\mathcal{K} \subset \mathbb{R}\) is a compact interval), the inner integral
	simply computes the expected edge weight. In doing so, we perfectly recover the standard notion of left-convergence
	for classical graphons. 
\end{remark}

While Lovász and Szegedy established this generalized left-convergence in \cite{LS-10-arxiv}, they did not equip the
space of \(\mathcal{K}\)-graphons with a cut-like metric. Consequently, they did not provide a metric characterization of 
homomorphism convergence similar to the equivalence in \eqref{eq:homomorphism-density-equivalent-cut-distance}. The first
rigorous definition of a cut-like metric for \(\mathcal{K}\)-graphons along with the proof that convergence in this metric
is equivalent to left-convergence was established in \cite[Theorem 2.12]{APST-23-arxiv} for the specific case where
\(\mathcal{K}\) is the interval \([-1,1]\). This framework was recently generalized in \cite{ADW-25} to probability-graphons
over any general Polish space \(\mathcal{X}\). Furthermore, the dual concept of right-convergence has also been successfully
extended to the space of probability-graphons in \cite{Zucal-24-arxiv}, where it was proven that left and right convergence
remain fundamentally equivalent in this enriched setting.

From now on, we adopt the term \emph{probability-graphon}, as it describes more accurately the mathematical nature of the
object representing the limiting network structure. Since there are numerous ways to metrize the narrow topology on the
space of probability measures \(\mathcal{P}(\mathcal{Y})\), there are correspondingly many ways to define a cut-like metric
on the space of probability-graphons. However, as shown in \cite[Section 5.2]{ADW-25}, as long as the underlying metric on
the measure space satisfies certain natural properties, all resulting cut distances induce the exact same topology.
Therefore, we choose to work with the bounded-Lipschitz distance, which is the most analytically convenient for our framework.
\begin{definition}[Probability-graphons and bounded-Lipschitz cut distance]
	The space of probability-graphons over a Polish space \(\mathcal{Y}\) is defined as 
	\[
		\mathcal{PG}(\mathcal{Y}):=\{q:[0,1]^2\to \mathcal{P}(\mathcal{Y}) \text{ measurable}\},
	\]  
	where \(\mathcal{P}(\mathcal{Y})\) is equipped with the narrow topology, and as usual we identify 
	two probability-graphons if they are equal almost everywhere. For any two probability-graphons 
	\(q,\bar{q} \in \mathcal{PG}(\mathcal{Y})\), we define the labeled bounded-Lipschitz cut distance as 
	\begin{equation} \label{eq:labeled-bounded-lipschitz-cut-distance-prob-graphons}
		d_{\text{BL},\square}(q,\bar{q}):= \sup_{S,T\subset [0,1]} d_{\text{BL}}\left( \iint_{S\times T} q^{\xi,\xi'}d\xi d\xi', \iint_{S\times T} \bar{q}^{\xi,\xi'}d\xi d\xi' \right). 
	\end{equation}
\end{definition}
\begin{remark}[Probability-graphons as fibered probability measures]
	Recalling the general definition of a fibered probability measure in Definition
	\ref{def:fibered-probability-measures}, a probability-graphon \(q\in \mathcal{PG}(\mathcal{Y})\) can
 	be naturally identified with a fibered probability measure in \(\mathcal{P}_{\nu\otimes \nu}(\mathcal{Y}\times [0,1]^2)\)
	via its family of disintegrations, where \(\nu\) denotes the Lebesgue measure on \([0,1]\). To ensure consistency with the
	preceding sections, we will adopt this fibered measure perspective and notation throughout the remainder of the paper:
	\[
		\mathcal{PG}(\mathcal{Y}) \equiv \mathcal{P}_{\nu\otimes \nu}(\mathcal{Y}\times [0,1]^2).
	\]
\end{remark}
To every probability-graphon \(q \in \mathcal{P}_{\nu\otimes \nu}(\mathcal{Y}\times [0,1]^2)\) and every bounded
measurable function \(\Phi:\mathcal{Y}\to \mathbb{R}\), we can associate the classical graphon
\(\Phi_q \in \mathcal{G}_{\Vert  \Phi \Vert_{L^\infty}}\) defined as
\[
	\Phi_q(\xi,\xi') = \int_{\mathcal{X}} \Phi(y)dq^{\xi,\xi'}(y),\quad \text{for a.e. } \xi,\xi'\in [0,1].
\]
Transformations of this type, which map probability-graphons to classical graphons, are highly useful for
establishing a structural relationship between the two spaces. In particular, in
\eqref{eq:labeled-bounded-lipschitz-cut-distance-prob-graphons}, we can interchange the supremum over the measurable
sets \(S,T\subset [0,1]\) with the supremum over the test functions \(\Phi \in \text{BL}_1(\mathcal{X})\)
coming from the distance \(d_{\text{BL}}\). This yields the following equivalent formulation for the labeled
cut distance:
\begin{equation} \label{eq:cut-distance-prob-graphons-reformulation}
	d_{\text{BL},\square}(q,\bar{q}) = \sup_{ \Phi \in \text{BL}_1(\mathcal{X})} d_{\square}(\Phi_q,\Phi_{\bar{q}}).
\end{equation}
Using the equivalence \eqref{eq:equivalence-cut-distance-operator-norm} between the labeled cut distance and
the operator norm associated with the adjacency operator of the graphon, we obtain the bounds:
\[
	d_{\text{BL},\square}(q,\bar{q}) \leq \sup_{\Phi \in \text{BL}_1(\mathcal{X})} \Vert T^{\Phi_q}-T^{\Phi_{\bar{q}}} \Vert_{\infty\to 1} \leq 4 d_{\text{BL},\square}(q,\bar{q}).
\]
For our analytical purposes, the intermediate expression in the previous chain of inequalities provides
the most convenient functional formulation to identify the bounded-Lipschitz cut distance between two probability-graphons:
\begin{equation} \label{eq:cut-distance-prob-graphons-operator-norm}
	\sup_{\Phi \in \text{BL}_1(\mathcal{X})} \Vert T^{\Phi_q}-T^{\Phi_{\bar{q}}} \Vert_{\infty\to 1} = \sup_{\Phi \in \text{BL}_1(\mathcal{Y})} \sup_{\Vert \psi \Vert_{\infty}\leq 1} \int_{0}^{1} \left\vert \int_{0}^{1} \psi(\xi') \int_{\mathcal{Y}} \Phi(y) d(q^{\xi,\xi'}-\bar{q}^{\xi,\xi'})(y)d\xi' \right\vert d\xi.
\end{equation}

\begin{remark}[\(d_{\text{BL},\square}\) is stronger than \(d_{\square}\) on the space of graphons]
	The space of classical graphons \(\mathcal{G}_W\) can be naturally embedded into the space of probability-graphons \(\mathcal{P}_{\nu\otimes \nu}([-W,W]\times [0,1]^2)\) by identifying each graphon \(w\in \mathcal{G}_W\) with the Dirac probability-graphon \(\delta_w\), defined fiber-wise as \(\delta_w^{\xi,\xi'} = \delta_{w(\xi,\xi')}\). 
	For any two graphons \(w,\bar{w} \in \mathcal{G}_W\), applying \eqref{eq:cut-distance-prob-graphons-reformulation} yields:
	\[
		d_{\text{BL},\square}(\delta_w,\delta_{\bar{w}}) = \sup_{\Phi \in \text{BL}_1([-W,W])} d_{\square}(\Phi_w,\Phi_{\bar{w}}).
	\]
	By taking \(\Phi:[-W,W]\to \mathbb{R}\) to be the identity function \(\Phi(w)=w\)
	(assuming \(W\leq 1\) so that \(\Phi \in \text{BL}_1([-W,W])\); if \(W>1\), we simply scale it by \(1/W\)),
	we obtain the bound:
	\[
		d_{\square}(w,\bar{w})\leq \max\{1,W\} d_{\text{BL},\square}(\delta_w,\delta_{\bar{w}}).
	\]
	Consequently, if a sequence of graphons \((w_n)_{n\in\mathbb{N}}\subset \mathcal{G}_W\) satisfies
	\[
		d_{\text{BL},\square}(\delta_{w_n},q) = \sup_{\Phi \in \text{BL}_1([-W,W])} d_{\square}(\Phi_{w_n},\Phi_q)\to 0,
	\]
	for some probability-graphon \(q\in \mathcal{P}_{\nu\otimes \nu}([-W,W]\times [0,1]^2)\), then the sequence
	\((w_n)_{n\in\mathbb{N}}\) converges in the classical cut distance to the expected graphon
	\(\mathbb{E}_q\in \mathcal{G}_W\), {\it i.e.}, \(d_{\square}(w_n,\mathbb{E}_q)\to 0\), where:
	\[
		\mathbb{E}_q(\xi,\xi'):=\int_{-W}^{W} w \, dq^{\xi,\xi'}(w), \quad \text{for a.e. } \xi,\xi'\in [0,1].
	\]
	Importantly, this convergence is not limited to the first moment. For any \(k\in \mathbb{N}\), the map 
	\(w \mapsto w^k\) is Lipschitz on the compact interval \([-W,W]\), so upon suitable scaling, it belongs
	to the test class \(\text{BL}_1([-W,W])\). Thus, the sequence of power graphons
	\((w_n^k)_{n\in\mathbb{N}}\) necessarily converges
	in the cut distance to the \(k\)-th moment graphon derived from \(q\), {\it i.e.},
	\(d_{\square}(w_n^k,\mathbb{E}_q^k)\to 0\), where:
	\[
		\mathbb{E}_q^k(\xi,\xi'):= \int_{-W}^{W} w^k \, dq^{\xi,\xi'}(w), \quad \text{for a.e. } \xi,\xi'\in [0,1].
	\]
	In general, however, the limit \(q\) will not be a graphon, {\it i.e.}, it will not be of the form \(\delta_{w}\)
	for some graphon \(w\), but a more general probability-graphon, as illustrated in the following example.
\end{remark}

\begin{example}[Continuation of Example \ref{ex:non-linear-transformation-graphon}]
	Consider again the sequence of Erdős–Rényi random graphs \(G^N=G(N,1/2)\) with associated graphons \(w_{G^N}\),
	which almost surely converge in the classical cut distance to the constant graphon \(w \equiv 1/2\). If we now
	consider the sequence of probability-graphons defined as \(q^N=\delta_{w_{G^N}}\), this sequence converges
	almost surely in the labeled bounded-Lipschitz cut distance \(d_{\text{BL},\square}\) to the constant probability-graphon
	\(q\) defined as 
	\[
		q^{\xi,\xi'} = \frac{1}{2}\delta_0 + \frac{1}{2}\delta_1, \quad \text{for a.e. } \xi,\xi'\in [0,1].
	\]
	To verify this, we use the equivalent formulation of the labeled cut distance given
	in \eqref{eq:cut-distance-prob-graphons-reformulation}. For every test function \(\Phi \in \text{BL}_1([-1,1])\),
	we have that
	\[
		d_{\square}(\Phi_{q^N},\Phi_q) = d_{\square}(\Phi_{w_{G^N}},\Phi_q).
	\]
	Since the graphons \(w_{G^N}\) strictly take values in \(\{0,1\}\), we can evaluate \(\Phi\) precisely as
	\[
		\Phi_{w_{G^N}} = (\Phi(1)-\Phi(0))w_{G^N} + \Phi(0).
	\]
	Similarly, for the proposed limit \(q\), the associated classical graphon \(\Phi_q\) evaluates to
	\[
		\Phi_q(\xi,\xi') = \frac{1}{2}\Phi(0) + \frac{1}{2}\Phi(1) = (\Phi(1)-\Phi(0))\frac{1}{2} + \Phi(0).
	\]
	Subtracting these expressions and taking the classical cut distance, we find:
	\[
		d_{\square}(\Phi_{q^N},\Phi_q) = \vert\Phi(1)-\Phi(0)\vert d_{\square}(w_{G^N},1/2) \leq d_{\square}(w_{G^N},1/2),
	\]
	where the inequality holds because \(\Phi\) is \(1\)-Lipschitz. Since \(d_{\square}(w_{G^N},1/2) \to 0\) almost surely,
	taking the supremum over \(\Phi \in \text{BL}_1([-1,1])\) confirms that \(d_{\text{BL},\square}(q^N,q)\to 0\) almost
	surely. Furthermore, this example illustrates the strict inclusion of the topologies,
	\(\tau_{d_{\text{BL},\square}} \subsetneq \tau_{d_{\text{BL},1}}\). While the sequence converges to \(q\) 
	in the labeled cut distance, it does not converge in the \(d_{\text{BL},1}\) distance. Indeed, for any point
	\((\xi,\xi')\in [0,1]^2\), the fiber \(q^{N,\xi,\xi'}\) is either \(\delta_0\) or \(\delta_1\), whereas
	\(q^{\xi,\xi'} = \frac{1}{2}\delta_0 + \frac{1}{2}\delta_1\). The bounded-Lipschitz distance between these
	measures is strictly positive and constant almost everywhere, meaning \(d_{\text{BL},1}(q^N, q) \not\to 0\).
	
	In conclusion, while non-linear transformations behave poorly with respect to the standard cut distance for
	graphons (as demonstrated in Example \ref{ex:non-linear-transformation-graphon}), transitioning to the enriched
	space of probability-graphons ensures that the sequence converges to a limit that correctly captures non-linear
	transformations. Indeed, if we reconsider the non-linear transformation \(\Phi(w)=w^2\), the associated graphon
	derived from the limit probability-graphon \(q\) is:
	\[
		\Phi_q(\xi,\xi') = \int_{-1}^{1} w^2 \, dq^{\xi,\xi'}(w) = \frac{1}{2}(0)^2 + \frac{1}{2}(1)^2 = \frac{1}{2},\quad \text{for a.e. } \xi,\xi'\in [0,1],
	\]
	which perfectly matches the limit of the squared sequence of graphons, {\it i.e.}, \(d_{\square}(w^2_{G^N},\Phi_q)\to 0\).
\end{example}

Before transitioning to the unlabeled space of probability-graphons, we first clarify the relationships among the various
topologies on \(\mathcal{P}_{\nu \otimes \nu}(\mathcal{Y}\times [0,1]^2)\). Recall that for the space of fibered probability
measures \(\mathcal{P}_{\nu}(\mathcal{X}\times [0,1])\) we established the topological relations in
\eqref{eq:topologies-fibered-measures}.
Now, in this two-dimensional case, the equivalence between \(d_{\text{BL},\square}\) and \(d_{\text{BL},1}\) breaks down, yielding
the following strict chain of inclusions:
\begin{equation} \label{eq:inclusion-topologies-prob-graphons}
	\tau_{\text{narrow}} \subsetneq \tau_{d_{\text{BL},\square}} \subsetneq \tau_{d_{\text{BL},1}} \subsetneq \tau_{d_{\text{BL},\infty}}. 
\end{equation}
As in the case of \(\mathcal{P}_{\nu}(\mathcal{X}\times [0,1])\), and with analogous definitions for the distances
\(d_{\text{BL},1}\) and \(d_{\text{BL},\infty}\) in \(\mathcal{P}_{\nu\otimes\nu}(\mathcal{Y}\times[0,1]^2)\),
we have the following inequalities
\[
	d_{\text{BL}} \leq d_{\text{BL},1} \leq d_{\text{BL},\infty},
\]
where we recall that \(d_{\text{BL}}\) is the standard bounded-Lipschitz distance in \(\mathcal{P}(\mathcal{Y}\times [0,1]^2)\).
For any two probability-graphons \(q,\bar{q} \in \mathcal{P}_{\nu \otimes \nu}(\mathcal{Y}\times [0,1]^2)\), this distance 
takes the form
\[
	d_{\text{BL}}(q,\bar{q}) = \sup_{\Phi \in \text{BL}_1(\mathcal{Y}\times [0,1]^2)} \left\vert \int_{[0,1]^2} \int_{\mathcal{Y}} \Phi(y,\xi,\xi') \,d(q^{\xi,\xi'}-\bar{q}^{\xi,\xi'})(y)\,d\xi \,d\xi' \right\vert.
\]
Furthermore, the following relation between the cut distance and the \(L^1\) distance holds
\[
	d_{\text{BL},\square} \leq d_{\text{BL},1}.
\]
The only remaining relation to establish is the one between \(d_{\text{BL}}\) and \(d_{\text{BL},\square}\). While
no general inequality holds between these two distances, their induced topologies can be rigorously compared via
the following estimate.
\begin{proposition}
	Let \(q, \bar{q} \in \mathcal{P}_{\nu\otimes \nu}(\mathcal{Y}\times [0,1]^2)\).
	Then, for every \(m \in \mathbb{N}\),
	\[
		d_{\text{BL}}(q,\bar{q})\leq 4m^2 d_{\text{BL},\square}(q,\bar{q}) + \frac{2\sqrt{2}}{m}.
	\]
\end{proposition}
\begin{proof}
 	Let \(m\in \mathbb{N}\) and partition the interval \([0,1]\) into
	\(m\) sub-intervals \(I^m_j = [\frac{j-1}{m}, \frac{j}{m})\) for \(j = 1, \dots, m\). This partitions the square \([0,1]^2\) into \(m^2\) grid squares of side length \(1/m\). Choose arbitrary points \(\xi_j \in I^m_j\) and \(\xi'_k \in I^m_k\). For any test function \(\Phi \in \text{BL}_1(\mathcal{Y}\times [0,1]^2)\), define its step-function approximation on the grid as
	\[
		\tilde{\Phi}(y,\xi,\xi') = \sum_{j=1}^{m} \sum_{k=1}^{m} \Phi(y,\xi_j,\xi'_k) \mathds{1}_{I^m_j}(\xi) \mathds{1}_{I^m_k}(\xi').
	\]
	Because \(\Phi\) is \(1\)-Lipschitz in all variables, the approximation error is bounded by the maximum Euclidean distance within any square \(I^m_j \times I^m_k\). Since the spatial variable \(y\) is fixed, this is exactly the length of the diagonal of the grid square:
	\[
		\Vert \Phi - \tilde{\Phi} \Vert_{\infty} \leq \frac{\sqrt{2}}{m}.
	\]
	Thus, the integration error is bounded by
	\[
		\left\vert \int_{[0,1]^2} \int_{\mathcal{Y}} (\Phi(y,\xi,\xi') - \tilde{\Phi}(y,\xi,\xi')) \,d(q^{\xi,\xi'}-\bar{q}^{\xi,\xi'})(y)\,d\xi \,d\xi' \right\vert \leq 2 \frac{\sqrt{2}}{m}.
	\]
	Next, we evaluate the integral of the discretized function \(\tilde{\Phi}\). Letting \(\Phi_{j,k}(y) := \Phi(y,\xi_j,\xi'_k)\), we note that \(\Phi_{j,k} \in \text{BL}_1(\mathcal{Y})\), which yields
	\[
		\begin{aligned}
		&\left\vert \int_{[0,1]^2} \int_{\mathcal{Y}} \tilde{\Phi}(y,\xi,\xi') \,d(q^{\xi,\xi'}-\bar{q}^{\xi,\xi'})(y)\,d\xi \,d\xi' \right\vert \\
        &\qquad\leq \sum_{j=1}^{m} \sum_{k=1}^{m} \int_{0}^{1} \mathds{1}_{I^m_j}(\xi) \left\vert \int_{0}^{1} \mathds{1}_{I^m_k}(\xi') \int_{\mathcal{Y}} \Phi_{j,k}(y) \,d(q^{\xi,\xi'}-\bar{q}^{\xi,\xi'})(y) \,d\xi' \right\vert d\xi \\
		&\qquad\leq \sum_{j=1}^{m} \sum_{k=1}^{m} \int_{0}^{1} \left\vert \int_{0}^{1} \mathds{1}_{I^m_k}(\xi') \int_{\mathcal{Y}} \Phi_{j,k}(y) \,d(q^{\xi,\xi'}-\bar{q}^{\xi,\xi'})(y) \,d\xi' \right\vert d\xi 
		\leq 4 m^2 d_{\text{BL},\square}(q,\bar{q}),
		\end{aligned}
	\]
	recognizing that the inner integral is bounded by the operator norm formulation \eqref{eq:cut-distance-prob-graphons-operator-norm}, since \(\Vert \mathds{1}_{I^m_k} \Vert_{\infty} \leq 1\) and \(\Phi_{j,k} \in \text{BL}_1(\mathcal{Y})\). 
	Combining this with the uniform approximation error yields
	\[
		\left\vert \int_{[0,1]^2} \int_{\mathcal{Y}} \Phi(y,\xi,\xi') \,d(q^{\xi,\xi'}-\bar{q}^{\xi,\xi'})(y)\,d\xi \,d\xi' \right\vert \leq 4m^2 d_{\text{BL},\square}(q,\bar{q}) + \frac{2\sqrt{2}}{m}.
	\]
	Since \(\Phi \in \text{BL}_1(\mathcal{Y}\times [0,1]^2)\) was arbitrary, taking the supremum over all such test functions gives the desired inequality.   
\end{proof}
\begin{corollary} \label{cor:d-BL-cut-stronger-narrow-topology}
	The topology induced by \(d_{\text{BL},\square}\) is finer than the narrow topology on \(\mathcal{P}_{\nu\otimes \nu}(\mathcal{Y}\times [0,1]^2)\).
\end{corollary}

\begin{remark}
	The only missing relation among the topologies in \(\mathcal{P}_{\nu\otimes\nu}(\mathcal{Y}\times [0,1]^2)\)
	is the strict inclusion \(\tau_{\text{narrow}}\subsetneq\tau_{d_{\text{BL},\square}}\). This can be demonstrated using
	the 2D Rademacher sequence. Consider the sequence of probability-graphons defined as
	\(q^{N,\xi,\xi'} = \delta_{u_N(\xi,\xi')}\), where
	\[
		u_N(\xi,\xi') = \sgn(\sin(2^N\pi \xi)\sin(2^N\pi \xi')).
	\] 
	With the same technique used in the 1D case, we can show that \(d_{\text{BL}}(q^N,q)\to 0\), where \(q\)
	is the constant probability-graphon defined as \(q^{\xi,\xi'} = \frac{1}{2}\delta_{-1} + \frac{1}{2}\delta_1\).
	However, the labeled bounded-Lipschitz cut distance does not converge to zero. To see this, testing in
	\eqref{eq:cut-distance-prob-graphons-reformulation} against the identity map \(\Phi(y)=y \in \text{BL}_1([-1,1])\) yields
	\(d_{\text{BL},\square}(q^N, q) \geq d_\square(u_N, 0)\). Since the supremum in the cut distance allows us to choose
	\(N\)-dependent measurable sets, we can perfectly align them with the positive oscillations by setting
	\(S_N = \{\xi \in [0,1] : \sin(2^N\pi \xi) > 0\}\) and \(T_N = \{\xi' \in [0,1] : \sin(2^N\pi \xi') > 0\}\). Both sets have measure \(1/2\), and \(u_N \equiv 1\) on \(S_N \times T_N\). Integrating over this product set gives exactly \(1/4\), establishing that \(d_{\text{BL},\square}(q^N,q) \geq 1/4\) for all \(N\).
\end{remark}
Having clarified the topological inclusions in \(\mathcal{P}_{\nu\otimes \nu}(\mathcal{Y}\times [0,1]^2)\) established in
\eqref{eq:inclusion-topologies-prob-graphons}, we can now proceed to introduce the unlabeled space of probability-graphons. In this
setting, the unlabeled bounded-Lipschitz cut distance, defined as 
\[
	\delta_{\text{BL},\square}(q,\bar{q}) := \inf_{\varphi \in S_{[0,1]}} d_{\text{BL},\square}(q,\bar{q}^{\varphi}),
\]
provides the correct metric structure to obtain compactness properties. As is standard, 
this is merely a pseudometric on \(\mathcal{P}_{\nu\otimes\nu}(\mathcal{Y}\times [0,1]^2)\). We denote by
\(\widetilde{\mathcal{P}}_{\nu\otimes\nu}(\mathcal{Y}\times [0,1]^2)\) the quotient space of probability-graphons, 
obtained by identifying \(q,\bar{q} \in \mathcal{P}_{\nu\otimes \nu}(\mathcal{Y}\times [0,1]^2)\) whenever
\(\delta_{\text{BL},\square}(q,\bar{q})=0\). This equivalence relation precisely identifies two probability-graphons
if they are weakly isomorphic, {\it i.e.}, if there exist measure-preserving maps \(\varphi,\psi \in \bar{S}_{[0,1]}\) such that
\(q^{\varphi} = \bar{q}^{\psi}\). This identification is justified by the fact that the unlabeled cut distance
can be equivalently reformulated as a minimum over these maps \cite[Proposition 3.18]{ADW-25}:
\[
	\delta_{\text{BL},\square}(q,\bar{q}) = \min_{\varphi,\psi \in \bar{S}_{[0,1]}} d_{\text{BL},\square}(q^{\varphi},\bar{q}^{\psi}).
\]  
While many classical results for graphons have been successfully generalized to this setting in \cite{ADW-25}, we will
focus on those that lead to the characterization of relatively compact sets. The primary tool in this regard is
the weak regularity lemma for tight subsets of \(\mathcal{P}_{\nu\otimes\nu}(\mathcal{Y}\times [0,1]^2)\).

\begin{lemma}[{\cite[Proposition 4.13]{ADW-25}}] 
	\label{lemma:regularity-probability-graphons}
	Let \(\mathcal{K}\subset \mathcal{P}_{\nu\otimes\nu}(\mathcal{Y}\times [0,1]^2)\) be a tight set of probability-graphons.
	Then, for every \(\varepsilon > 0\), there exists an integer \(m \in \mathbb{N}\) such that for any probability-graphon \(q \in \mathcal{K}\) and any partition \(\mathcal{Q}\) of \([0,1]\), there exists a refinement \(\mathcal{P}\) of \(\mathcal{Q}\) satisfying
	\[
		\vert \mathcal{P} \vert \leq m \vert \mathcal{Q}\vert , \quad \text{and} \quad d_{\text{BL},\square} (q, q_{\mathcal{P}}) < \varepsilon.
	\]
\end{lemma}
As in the one-dimensional case, this weak regularity lemma allows us to characterize the relatively compact subsets of
the quotient space \((\widetilde{\mathcal{P}}_{\nu\otimes\nu}(\mathcal{Y}\times [0,1]^2),\delta_{\text{BL},\square})\). 
This is formalized in the following theorem, which synthesizes the results of Theorem 5.1, Proposition 5.2, and Theorem 5.10
in \cite{ADW-25}.

\begin{theorem} \label{thm:compactness-probability-graphons}
	The space \((\widetilde{\mathcal{P}}_{\nu\otimes\nu}(\mathcal{Y}\times [0,1]^2),\delta_{\text{BL}, \square})\)
	is a Polish space. Moreover, the following properties hold: 
	\begin{enumerate}[label=(\roman*)]
		\item A subset \(\mathcal{K}\subset (\widetilde{\mathcal{P}}_{\nu\otimes\nu}(\mathcal{Y}\times [0,1]^2),\delta_{\text{BL},\square})\) is relatively compact if and only if it is tight. 
		\item The space \((\widetilde{\mathcal{P}}_{\nu\otimes\nu}(\mathcal{Y}\times [0,1]^2),\delta_{\text{BL},\square})\) is compact if and only if \(\mathcal{Y}\) is compact.
	\end{enumerate}  
\end{theorem}

\begin{remark}[The limit is non-trivial only for dense graph sequences] \label{remark:dense-graphs}
	As in the case of classical graphons, the space of probability-graphons is designed to model the limits of dense
	graph sequences. In the space \(\mathcal{Y}\), we can assume the existence of a distinguished point
	\(0\in \mathcal{Y}\), which represents the absence of a connection between two nodes. 
	A sequence of edge-decorated weighted graphs \((G^N=(V^N,E^N,w^N))_{N\in\mathbb{N}}\) with decorations 
	taking values in \(\mathcal{Y}\) is considered dense if:   
	\[
		\liminf_{N\to \infty} \frac{1}{N^2}\sum_{i,j=1}^{N} d_{\mathcal{Y}}(w^N_{ij},0) > 0,
	\]
	meaning the average distance of the edge weights to the no connection state remains bounded away from zero.
	When \(\mathcal{Y}=[-W,W]\), this definition reduces perfectly to the classical density condition for weighted graphs.
	
	If the sequence of graphs is sparse ({\it i.e.}, not dense), the limit in the labeled bounded-Lipschitz cut distance is 
	always the trivial zero probability-graphon, defined fiber-wise as \(q_0^{\xi,\xi'}=\delta_0\) for a.e.
	\(\xi,\xi'\in [0,1]\). This is straightforward to verify, as we have the following bound:  
	\[
		d_{\text{BL},\square}(q_{G^N},q_0) \leq d_{\text{BL},1}(q_{G^N},q_0) \leq \frac{1}{N^2}\sum_{i,j=1}^{N} d_{\mathcal{Y}}(w^N_{ij},0),
	\]
	which converges to zero if the sequence is not dense.
\end{remark}
\subsection{Node-edge probability-graphons}
When modeling graphs where the node features and edge weights take values in Polish spaces
\(\mathcal{X}\) and \(\mathcal{Y}\) respectively, these elements must be reordered jointly
to preserve the structural identity of the network. Consequently, their asymptotic limits
cannot be treated as separate entities. To properly couple the edge limits, modeled by probability-graphons,
and the node limits, modeled by fibered probability measures, we introduce the unified space of node-edge
probability-graphons in this section. While this joint modeling has been established for classical
graphons \cite{Levie-23}, we extend it here to the framework of probability-graphons, formalizing 
the approach suggested in \cite[Remark 1.7]{ADW-25}.

\begin{definition} 
	The space of node-edge probability-graphons with node features in \(\mathcal{X}\) and edge weights in
	\(\mathcal{Y}\) is defined as
	\[
		\mathcal{N}(\mathcal{X})\mathcal{E}(\mathcal{Y}) := \mathcal{P}_{\nu}(\mathcal{X}\times[0,1]) \times \mathcal{P}_{\nu\otimes\nu}(\mathcal{Y}\times[0,1]^2).
	\]
	We also define the corresponding subset with node features in a fibered Wasserstein space of order \(p,q\in [1,\infty)\): 
	\[
		\mathcal{N}_{p,q}(\mathcal{X})\mathcal{E}(\mathcal{Y}) := \mathcal{P}_{p,q,\nu}(\mathcal{X}\times[0,1]) \times \mathcal{P}_{\nu\otimes\nu}(\mathcal{Y}\times[0,1]^2).
	\] 
	For any pair \((\mu,q),(\bar{\mu},\bar{q})\in \mathcal{N}(\mathcal{X})\mathcal{E}(\mathcal{Y})\), we define the coupled labeled and unlabeled distances as follows: 
	\[
		\begin{aligned}
			d_{(\text{BL},1),(\text{BL},\square)}((\mu,q),(\bar{\mu},\bar{q})) &= d_{\text{BL},1}(\mu,\bar{\mu}) + d_{\text{BL},\square}(q,\bar{q}), \\ 
			\delta_{(\text{BL},1),(\text{BL},\square)}((\mu,q),(\bar{\mu},\bar{q})) &= \inf_{\varphi \in S_{[0,1]}} d_{(\text{BL},1),(\text{BL},\square)}((\mu,q),(\bar{\mu},\bar{q})^{\varphi}),
		\end{aligned}
	\]
	where \((\bar{\mu},\bar{q})^{\varphi}\) denotes the joint rearrangement of the node-edge probability-graphon
	by \(\varphi\), {\it i.e.}, \((\bar{\mu},\bar{q})^{\varphi} = (\bar{\mu}^{\varphi},\bar{q}^{\varphi})\). 
	
	Similarly, for any pair \((\mu,q),(\bar{\mu},\bar{q})\in \mathcal{N}_{p,q}(\mathcal{X})\mathcal{E}(\mathcal{Y})\),
	we define the corresponding distances:
	\[
		\begin{aligned}
		d_{(W_p,q),(\text{BL},\square)}((\mu,q),(\bar{\mu},\bar{q})) &= d_{W_p,q}(\mu,\bar{\mu}) + d_{\text{BL},\square}(q,\bar{q}), \\ 
		\delta_{(W_p,q),(\text{BL},\square)}((\mu,q),(\bar{\mu},\bar{q})) &= \inf_{\varphi \in S_{[0,1]}} d_{(W_p,q),(\text{BL},\square)}((\mu,q),(\bar{\mu},\bar{q})^{\varphi}).
		\end{aligned}
	\]
\end{definition}
Applying the techniques from \cite[Theorem 8.13]{Lovasz-12}, the unlabeled coupled distances defined
above can be equivalently reformulated as minimizations over measure-preserving maps:
\[
	\begin{aligned}
		&\delta_{(\text{BL},p),(\text{BL},\square)}((\mu,q),(\bar{\mu},\bar{q})) = \min_{\varphi,\psi \in \bar{S}_{[0,1]}} d_{(\text{BL},p),(\text{BL},\square)}((\mu,q)^{\psi},(\bar{\mu},\bar{q})^{\varphi}), \\
		&\delta_{(W_p,q),(\text{BL},\square)}((\mu,q),(\bar{\mu},\bar{q})) = \min_{\varphi,\psi \in \bar{S}_{[0,1]}} d_{(W_p,q),(\text{BL},\square)}((\mu,q)^{\psi},(\bar{\mu},\bar{q})^{\varphi}).
	\end{aligned}
\]
Consequently, these are genuine distances on the quotient space \(\widetilde{\mathcal{N}(\mathcal{X})\mathcal{E}(\mathcal{Y})}\),
which identifies two node-edge probability-graphons if they are weakly isomorphic, {\it i.e.}, if there exist measure-preserving maps
\(\varphi,\psi\in \bar{S}_{[0,1]}\) such that \((\mu,q)^{\psi} = (\bar{\mu},\bar{q})^{\varphi}\).
\begin{remark}[The necessity of joint coupling]
	The space of node-edge probability-graphons naturally extends the space of probability-graphons by ensuring that both the
	node and edge limits remain structurally coupled. Indeed, one expects that once a specific labeling of the nodes is fixed,
	the underlying edge structure is strictly bound to it.
	This requirement forces us to apply the exact same rearrangement to both the node and edge structures simultaneously
	when defining the unlabeled distances, making it evident that:
	\[
		\widetilde{\mathcal{N}(\mathcal{X})\mathcal{E}(\mathcal{Y})} \neq \widetilde{\mathcal{P}_{\nu}}(\mathcal{X}\times [0,1]) \times \widetilde{\mathcal{P}_{\nu\otimes\nu}}(\mathcal{Y}\times [0,1]^2).
	\]
	Indeed, we only have the strict inclusion \(\widetilde{\mathcal{N}(\mathcal{X})\mathcal{E}(\mathcal{Y})} \subsetneq \widetilde{\mathcal{P}_{\nu}}(\mathcal{X}\times [0,1]) \times \widetilde{\mathcal{P}_{\nu\otimes\nu}}(\mathcal{Y}\times [0,1]^2)\),
	reflecting the fact that joint isomorphism of the node and edge structures implies separate isomorphism, but not conversely.
	
	A simple illustrative example of this phenomenon is the following. Consider the pair of node-edge probability-graphons
	\((\mu,q),(\bar{\mu},q)\in \mathcal{N}(\mathbb{R})\mathcal{E}(\mathbb{R})\) defined fiber-wise as:
	\[
		\begin{aligned}
			\mu^{\xi} &= \delta_{1}\mathds{1}_{[0,1/2]}(\xi) + \delta_{-1}\mathds{1}_{(1/2,1]}(\xi), \\ 
			\bar{\mu}^{\xi} &= \delta_{-1}\mathds{1}_{[0,1/2]}(\xi) + \delta_{1}\mathds{1}_{(1/2,1]}(\xi), \\ 
			q^{\xi,\xi'} &= \delta_{2}\mathds{1}_{(1/2,1]\times [0,1/2]}(\xi,\xi') + \delta_{0}\mathds{1}_{[0,1]^2 \setminus ((1/2,1]\times [0,1/2])}(\xi,\xi').
		\end{aligned}
	\]
	Clearly, the isolated node fibered measures \(\mu\) and \(\bar{\mu}\) are weakly isomorphic (they are identical
	up to a rearrangement \(\psi\) that swaps the two halves of the unit interval). However, the coupled pairs
	\((\mu,q)\) and \((\bar{\mu},q)\) are not weakly isomorphic in the node-edge space. If we apply the swapping
	rearrangement \(\psi\) to perfectly align the node structures (\(\mu^\psi = \bar{\mu}\)), this same map simultaneously
	acts on the edge structure, transposing the support of \(q\) to \([0,1/2]\times (1/2,1]\). Since the original
	\(q\) was strongly asymmetric, we end up with two different edge structures, confirming that the pairs
	represent fundamentally distinct networks.
\end{remark}

Building on the properties of fibered probability measures and probability-graphons established in the preceding sections,
we can now characterize the relatively compact sets in the unified space of node-edge probability-graphons. 
Let us denote by
\[
	\pi_{\mathcal{N}}: \mathcal{N}(\mathcal{X})\mathcal{E}(\mathcal{Y}) \to \mathcal{P}_{\nu}(\mathcal{X}\times [0,1]), \quad (\mu,q) \mapsto \mu,
\]
the projection of the node-edge probability-graphons to their node features, and by
\[
	\pi_{\mathcal{E}}: \mathcal{N}(\mathcal{X})\mathcal{E}(\mathcal{Y}) \to \mathcal{P}_{\nu\otimes\nu}(\mathcal{Y}\times [0,1]^2), \quad (\mu,q) \mapsto q,
\]
the projection to their edge components. These projections induce well-defined maps on the quotient space
\(\widetilde{\mathcal{N}(\mathcal{X})\mathcal{E}(\mathcal{Y})}\), mapping to the respective quotient spaces of fibered
probability measures and probability-graphons. Furthermore, they are continuous with respect to both the corresponding
labeled and unlabeled distances.

The following weak regularity lemma for node-edge probability-graphons is a direct consequence of the corresponding weak
regularity lemmas for fibered probability measures and probability-graphons.
It is proven by applying these results first to the node component to secure an intermediate refinement,
and then to the edge component, using the obtained partition as the starting point.
\begin{lemma}[{Weak regularity lemma for \(\mathcal{N}(\mathcal{X})\mathcal{E}(\mathcal{Y})\)}] \label{lemma:regularity-node-edge}
	Let \(\mathcal{K} \subset \mathcal{N}(\mathcal{X})\mathcal{E}(\mathcal{Y})\) be a subset such that its projections \(\pi_{\mathcal{N}}(\mathcal{K})\) and \(\pi_{\mathcal{E}}(\mathcal{K})\) are tight. Then, for every \(\varepsilon > 0\), there exists an integer \(m \in \mathbb{N}\) such that for any node-edge probability-graphon \((\mu,q) \in \mathcal{K}\) and any partition \(\mathcal{Q}\) of \([0,1]\), there exists a refinement \(\mathcal{P}\) of \(\mathcal{Q}\) satisfying
	\[
		\vert \mathcal{P} \vert \leq m \vert \mathcal{Q}\vert, \quad d_{\text{BL},1}(\mu,\mu_{\mathcal{P}}) < \varepsilon, \quad \text{and} \quad d_{\text{BL},\square}(q,q_{\mathcal{P}}) < \varepsilon.
	\]
\end{lemma}

\begin{theorem} \label{thm:compactness-node-edge-probability-graphons}
	The spaces \((\widetilde{\mathcal{N}(\mathcal{X})\mathcal{E}(\mathcal{Y})},\delta_{(\text{BL},1),(\text{BL},\square)})\) and \((\widetilde{\mathcal{N}_{p,q}(\mathcal{X})\mathcal{E}(\mathcal{Y})},\delta_{(W_p,q),(\text{BL},\square)})\) are Polish. Moreover, the following properties hold: 
	\begin{enumerate}[label=(\roman*)]
		\item A subset \(\mathcal{K}\subset (\widetilde{\mathcal{N}(\mathcal{X})\mathcal{E}(\mathcal{Y})},\delta_{(\text{BL},1),(\text{BL},\square)})\) is relatively compact if and only if its projections \(\pi_{\mathcal{N}}(\mathcal{K})\) and \(\pi_{\mathcal{E}}(\mathcal{K})\) are tight.
		\item A subset \(\mathcal{K}\subset (\widetilde{\mathcal{N}_{p,q}(\mathcal{X})\mathcal{E}(\mathcal{Y})},\delta_{(W_p,q),(\text{BL},\square)})\) is relatively compact if and only if \(\pi_{\mathcal{N}}(\mathcal{K})\) is relatively compact in \((\widetilde{\mathcal{P}_{p,q,\nu}}(\mathcal{X}\times [0,1]),\delta_{W_p,q})\) and \(\pi_{\mathcal{E}}(\mathcal{K})\) is tight.
	\end{enumerate}
\end{theorem}

\begin{proof}
	\(\diamond\) \textsc{Step 1}: We first show that if \(\mathcal{K}\subset \mathcal{N}(\mathcal{X})\mathcal{E}(\mathcal{Y})\) have node and edge projections that are tight, then it is relatively compact.\\
	Let \((\mu_n,q_n)_{n\in\mathbb{N}}\) be a sequence of node-edge probability-graphons in \(\mathcal{K}\). Applying Lemma \ref{lemma:regularity-node-edge}, we can ensure that for every \(k\in \mathbb{N}\), there exists a partition \(\mathcal{P}_{n,k}\) of \([0,1]\) such that:
	\begin{enumerate}[label=(\roman*)]
		\item \(d_{\text{BL},1}(\mu_n,\mu_{n,k}) + d_{\text{BL},\square}(q_n,q_{n,k}) < 1/k\), 
		\item \(\diam(\mathcal{P}_{n,k}) < 2^{-k}\) and \(\vert \mathcal{P}_{n,k} \vert = m_k\) for some integer \(m_k \in \mathbb{N}\) independent of \(n\),
		\item \(\mathcal{P}_{n,k+1}\) is a refinement of \(\mathcal{P}_{n,k}\).
	\end{enumerate}
	Applying \cite[Lemma 8.2]{ADW-25}, there exists a sequence of measure-preserving maps \(\varphi_n \in S_{[0,1]}\) that
	rearranges these partitions into intervals. Since applying this joint rearrangement \(\varphi_n\) to both the node and
	edge components simply translates the sequence within its equivalence class in the quotient space
	\(\widetilde{\mathcal{N}(\mathcal{X})\mathcal{E}(\mathcal{Y})}\), we can assume without loss of generality
	that the partitions \(\mathcal{P}_{n,k}\) are composed of intervals. Since the projection
	\(\pi_{\mathcal{N}}(\mathcal{K})\) is tight, we can follow \textsc{Step} 3 to \textsc{Step} 5 in the proof of
	Proposition \ref{prop:tightness-in-average-implies-relative-compactness} to obtain a subsequence
	\((\mu_{n_{\ell}})_{\ell\in\mathbb{N}}\) and a limit fibered probability measure
	\(\mu\in \mathcal{P}_{\nu}(\mathcal{X}\times [0,1])\) such that \(d_{\text{BL},1}(\mu_{n_{\ell}},\mu)\to 0\)
	as \(\ell\to \infty\). 
	
	Now, taking the corresponding subsequence of probability-graphons \((q_{n_\ell})_{\ell\in\mathbb{N}}\), we use the fact that
	\(\pi_{\mathcal{E}}(\mathcal{K})\) is also tight. Applying \textsc{Step} 3 to \textsc{Step} 5 from the proof of
	\cite[Lemma 8.1]{ADW-25}, we can extract a further subsequence \((q_{n_{\ell_m}})_{m\in\mathbb{N}}\) and find a
	limit probability-graphon \(q\in \mathcal{P}_{\nu\otimes\nu}(\mathcal{Y}\times [0,1]^2)\) such that
	\(d_{\text{BL},\square}(q_{n_{\ell_m}},q)\to 0\) as \(m\to \infty\). 

	Consequently, we have shown that there exists a joint subsequence
	\((\mu_{n_{\ell_m}},q_{n_{\ell_m}})_{m\in\mathbb{N}}\) converging to a node-edge probability-graphon
	\((\mu,q)\in \mathcal{N}(\mathcal{X})\mathcal{E}(\mathcal{Y})\) such that
	\(\delta_{(\text{BL},1),(\text{BL},\square)}((\mu_{n_{\ell_m}},q_{n_{\ell_m}}),(\mu,q))\to 0\) as \(m\to \infty\),
	establishing relative compactness.

	\(\diamond\) \textsc{Step 2}: We now show the converse: if
	\(\mathcal{K}\subset \mathcal{N}(\mathcal{X})\mathcal{E}(\mathcal{Y})\) is relatively compact, then its node and edge projections are tight.\\
	This implication is a direct consequence of the continuity of the respective projections. Indeed, if
	\(\mathcal{K}\) is relatively compact, its continuous projections
	\(\pi_{\mathcal{N}}(\mathcal{K})\) and \(\pi_{\mathcal{E}}(\mathcal{K})\) must also be relatively compact
	in their respective spaces, which directly implies they are tight.
	
	The remainder of the theorem follows naturally. For point (ii), the arguments are identical, with the only
	distinction being the requirement of compactness in the stronger fibered Wasserstein topology for the node
	component. Finally, the Polish property of the quotient spaces is established using the exact same separability
	and completeness arguments previously discussed for fibered probability measures.
\end{proof}

\section{Propagation of independence} \label{sec:propagation-of-independence}
The aim of this section is to introduce an intermediate system of stochastic differential equations with independent components,
which will be used to prove the propagation of independence for the non-exchangeable system \eqref{eq:multi-agent-discrete} as
\(N\to \infty\). This approach was pioneered by Sznitman \cite{S-91} for exchangeable interacting diffusions, and later extended
to the non-exchangeable setting in \cite{APP-24-arxiv,JPS-24,JSZ-24-arxiv,JZ-25}. In \cite{Lacker-26}, such an
intermediate construction was termed the \emph{independent projection} of the interacting diffusion system.

To construct this system, we first need to reveal the dependency of the evolving weights as a 
non-anticipative functional of the particle states. To do so, we will use the flow map associated with the
weight dynamics, whose properties are summarized in the following proposition. We omit the proof 
as it follows directly from the assumptions on the plasticity function \(\Gamma\) and the kernel \(K\).

\begin{proposition} \label{prop:properties-Phi-Pi}
	Under Assumption \ref{assump:plasticity-function} on the plasticity function \(\Gamma\), there exists a unique continuous map \(\Phi: [0,T]\times \mathcal{C}^d_T \times \mathcal{C}^d_T \times \mathbb{R} \to \mathbb{R}\) such that for any \(\gamma_1,\gamma_2 \in \mathcal{C}^d_T\) and any \(w \in \mathbb{R}\), the function \(t\mapsto \Phi(t,\gamma_1,\gamma_2,w)\) is the unique solution to the initial value problem:
	\[
		\begin{aligned}
		&\frac{d}{dt}w(t) = \Gamma(\gamma_1(t),\gamma_2(t),w(t)), \quad t\in [0,T],\\
		&w(0) = w.
		\end{aligned}
	\]
	Furthermore, the flow map \(\Phi\) satisfies the following properties:
	\begin{enumerate}[label=(\roman*)]
		\item \textbf{(\(\Phi\) is non-anticipative):} There exists a family of functionals \((\Phi_t)_{t\in [0,T]}\), with \(\Phi_t: \mathcal{C}^d_t\times \mathcal{C}^d_t\times \mathbb{R} \to \mathbb{R}\), such that 
		\[
			\Phi(t,\gamma,\bar{\gamma},w) = \Phi_t(\gamma_{[0,t]},\bar{\gamma}_{[0,t]},w).
		\]
		\item \textbf{(\(\Phi\) is Lipschitz):}  
		\[
			\vert \Phi(t,\gamma_1,\bar{\gamma}_1,w_1) - \Phi(t,\gamma_2,\bar{\gamma}_2,w_2) \vert \leq L_{\Phi}(t) \left(\Vert \gamma_1-\gamma_2\Vert_{\ast,t} + \Vert \bar{\gamma}_1-\bar{\gamma}_2\Vert_{\ast,t} + \vert w_1-w_2\vert \right),
		\]
		with \(L_{\Phi}(t) = e^{L_{\Gamma}t}\left(1+tL_{\Gamma}\right)\).
		\item \textbf{(\(\Phi\) has linear growth in the weight variable):} 
		\begin{equation}\label{eq:bound-Phi}
		\vert \Phi(t,\gamma,\bar{\gamma},w)\vert \leq B_{\Phi}(t)(1+ \vert w\vert), 
		\end{equation}
		with \(B_{\Phi}(t) = (1+B_{\Gamma}t)e^{L_{\Gamma}t}\).
	\end{enumerate}
	
	Let us denote the product of the flow map and the kernel \(K\) by \(\Pi: [0,T]\times \mathcal{C}_T^d\times \mathcal{C}_T^d\times \mathbb{R} \to \mathbb{R}^d\), defined as 
	\[
		(t,\gamma,\bar{\gamma},w) \mapsto \Phi(t,\gamma,\bar{\gamma},w)K(\gamma(t),\bar{\gamma}(t)).
	\] 
	If the kernel \(K\) satisfies Assumption \ref{assump:interaction-kernel}, then the map \(\Pi\) inherits the
	following properties:
	\begin{enumerate}[label=(\roman*), resume]
		\item \textbf{(\(\Pi\) is Lipschitz):}  
		\begin{equation} \label{eq:lipschitz-property-Pi}
			\vert \Pi(t,\gamma_1,\bar{\gamma}_1,w_1) -\Pi(t,\gamma_2,\bar{\gamma}_2,w_2)\vert \leq L_{\Pi}(t)(1+\vert w_1\vert) \left( \Vert \gamma_1-\gamma_2 \Vert_{\ast,t}+\Vert \bar{\gamma}_1-\bar{\gamma}_2 \Vert_{\ast,t} +\vert w_1-w_2\vert\right), 
		\end{equation}
		with \(L_{\Pi}(t) = (B_K L_{\Phi}(t)+B_{\Phi}(t)L_K)\).
		\item \textbf{(\(\Pi\) has linear growth in the weight variable):}  
		\begin{equation} \label{eq:bound-Pi}
		\vert \Pi(t,\gamma,\bar{\gamma},w)\vert \leq B_{\Pi}(t)(1+ \vert w\vert), 
		\end{equation}
		with \(B_{\Pi}(t) = B_{\Phi}(t)B_K\).
	\end{enumerate}   
\end{proposition}
With this notation, the coupled system of SDEs \eqref{eq:multi-agent-discrete} for the state processes 
\((X_i^N)_{1\leq i\leq N}\) and the weights \((w^N_{ij})_{1\leq i,j\leq N}\) can be reformulated as a
closed system of path-dependent SDEs strictly in terms of the state variables. This is achieved by solving the
weight dynamics \(\eqref{eq:multi-agent-discrete}_2\) using the flow map defined in the preceding proposition,
and substituting this relation into the particle state equation \(\eqref{eq:multi-agent-discrete}_1\):
\begin{equation} \label{eq:multi-agent-discrete-closed-states}
	\begin{aligned}
		dX_i^N(t) &= \frac{1}{N}\sum_{j=1}^N \Phi_t(X_{i,[0,t]}^N,X_{j,[0,t]}^N,Nw_{ij,0}^N)\,K(X_i^N(t),X_j^N(t))\,dt + \sqrt{2\nu}\,dW_i^N(t),\quad t\in [0,T],\\
		X_i^N(0) &= X_{i,0}^N.
	\end{aligned}
\end{equation}
The systems \eqref{eq:multi-agent-discrete} and \eqref{eq:multi-agent-discrete-closed-states} are mathematically equivalent. Given a solution to the full interacting SDE, the weights admit the exact representation
\begin{equation} \label{eq:weights-solved}
	w^N_{ij}(t) = \frac{1}{N}  \Phi_t(X_{i,[0,t]}^N,X_{j,[0,t]}^N,Nw_{ij,0}^N),
\end{equation}
which immediately recovers the closed state equation. Conversely, solving for the states first and subsequently defining the weights via \eqref{eq:weights-solved} perfectly reconstructs a solution to the original coupled system.

Well-posedness is guaranteed by Theorem \ref{th:well-posedness-path-dependent-SDEs-additive-noise} for path-dependent SDEs,
as the map \(\Pi\) defined in Proposition \ref{prop:properties-Phi-Pi} satisfies the required Lipschitz properties.

\begin{lemma} \label{lemma:well-posedness-multi-agent-system-solved-weights}
	Under Assumptions \ref{assump:interaction-kernel} and \ref{assump:plasticity-function}, consider initial weights \((w_{ij,0}^N)_{1\leq i,j\leq N} \subset L^{\infty}(\Omega,\mathbb{R})\) and initial states \((X_{i,0}^N)_{1\leq i\leq N} \subset L^p(\Omega,\mathbb{R}^d)\) for some \(p\in[1,\infty)\). Provided that these initial random variables are independent of the driving standard Wiener processes \((W_i^N)_{1\leq i\leq N}\), there exists a unique strong solution \((X_i^N)_{1\leq i\leq N} \in L^p(\Omega,\mathcal{C}_T^d)\) to the system of path-dependent SDEs \eqref{eq:multi-agent-discrete-closed-states}.
\end{lemma}

Having reformulated the original system as \eqref{eq:multi-agent-discrete-closed-states}, we now introduce the intermediate
system of SDEs. This system is defined by replacing the interaction term in \eqref{eq:multi-agent-discrete-closed-states}
with its expected value, conditioned on the natural filtration generated by each individual particle.
\begin{lemma}[Intermediate system] \label{lemma:intermediate-system}
	Consider the following nonlinear system of SDEs for the unknown processes \((\bar{X}_i^N)_{1\leq i\leq N}\):
	\begin{align}
		d\bar X_i^N(t) &= \frac{1}{N}\sum_{j=1}^N\int_{\mathcal{C}^d_t}\int_{-W}^{W}\Phi_t(\bar X_{i,[0,t]}^N,\gamma,w) K(\bar X_i^N(t),\gamma(t)) \, dq_{ij,0}^N(w) \, d\bar \mu^{N,j}_{[0,t]}(\gamma) \, dt + \sqrt{2\nu}\,dW_i^N(t),\quad t\in [0,T], \nonumber\\
		\bar X_i^N(0) &= X_{i,0}^N, \quad \bar \mu^{N,i}_{[0,t]} := \mathrm{Law}(\bar X_{i,[0,t]}^N) \in \mathcal{P}(\mathcal{C}^d_t),\label{eq:multi-agent-discrete-independent-projection-solved-weights}
	\end{align}
	where \(q_{ij,0}^N := \mathrm{Law}(Nw_{ij,0}^N) \in \mathcal{P}(\mathbb{R})\). Under Assumptions \ref{assump:interaction-kernel}, \ref{assump:plasticity-function}, and \ref{assump:weights}, consider initial states \((X_{i,0}^N)_{1\leq i\leq N} \subset L^p(\Omega,\mathbb{R}^d)\) for some \(p\in[1,\infty)\). Provided that these initial states are independent of the driving standard Wiener processes \((W_i^N)_{1\leq i\leq N}\), there exists a unique strong solution \((\bar{X}_i^N)_{1\leq i\leq N} \subset L^p(\Omega,\mathcal{C}_T^d)\) to the system \eqref{eq:multi-agent-discrete-independent-projection-solved-weights}. Furthermore, if the initial states \((X_{i,0}^N)_{1\leq i\leq N}\) are independent, then the resulting solution processes \((\bar{X}_i^N)_{1\leq i\leq N}\) are also independent.
\end{lemma}

\begin{proof}
	The system \eqref{eq:multi-agent-discrete-independent-projection-solved-weights} is a McKean-Vlasov system
	of path-dependent SDEs that can be written in the following form:
	\[
	\begin{aligned}
		d\bar{X}^N(t) &= b(t,\bar{X}^N,\bar{\mu}^N) \,dt + \sqrt{2\nu} \,dW^N(t), \\
		\bar{X}^N(0) &= X^N_{0}, \quad \bar{\mu}^N = \mathrm{Law}(\bar{X}^N) \in \mathcal{P}(\mathcal{C}_T^{dN}),
	\end{aligned}
	\]
	with \(\bar{X}^N=(\bar{X}^N_1,\ldots,\bar{X}^N_N)\), \(\bar{X}^N_0=(X^N_{1,0},\ldots,X^N_{N,0})\), and \(W^N=(W^N_1,\ldots,W^N_N)\). The drift vector is \(b(t,\gamma,\mu)=(b_i(t,\gamma,\mu))_{1\leq i\leq N}\) for any \(\gamma=(\gamma_1,\ldots,\gamma_N)\in \mathcal{C}_T^{dN}\), where each component \(b_i:[0,T]\times \mathcal{C}_T^{dN}\times \mathcal{P}(\mathcal{C}_T^{dN})\to \mathbb{R}^d\) is given by
	\[
		b_i(t,\gamma,\mu) = \frac{1}{N}\sum_{j=1}^{N} \int_{\mathcal{C}^d_t}\int_{-W}^{W} \Pi_t(\gamma_{i,[0,t]},\tilde{\gamma},w) \,dq_{ij,0}^N(w) \,d\mu^{j}_{[0,t]}(\tilde{\gamma}).
	\] 
	Here, \(\mu^{j}=e_{j\#} \mu\) is the pushforward measure via the projection \(e_j: \mathcal{C}_T^{dN}\to \mathcal{C}_T^d\), \((\gamma_1,\ldots,\gamma_j,\ldots,\gamma_N)\mapsto \gamma_j\), and we recall that \(\Pi_t(\gamma,\tilde{\gamma},w)=\Phi_t(\gamma,\tilde{\gamma},w)K(\gamma(t),\tilde{\gamma}(t))\).
	We will show that \(b\) satisfies the requirements of Theorem
	\ref{th:well-posedness-non-Markovian-McKean-Vlasov-additive-noise}, which establishes the well-posedness of the
	system \eqref{eq:multi-agent-discrete-independent-projection-solved-weights}. 
	
	The pathwise Lipschitz property can be easily derived using the Lipschitz properties of \(\Pi\) listed in Proposition \ref{prop:properties-Phi-Pi}. For any \(\gamma,\bar{\gamma} \in \mathcal{C}_T^{dN}\) and \(\mu,\nu \in \mathcal{P}(\mathcal{C}_T^{dN})\), applying the triangle inequality yields:
	\[
	\begin{aligned}
		\vert b(t,\gamma,\mu) - b(t,\bar{\gamma},\nu)\vert &\leq \frac{1}{N} \sum_{i=1}^{N} \sum_{j=1}^{N} \int_{\mathcal{C}_t^d}\int_{-W}^{W} \left\vert \Pi_t(\gamma_{i,[0,t]},\tilde{\gamma},w)- \Pi_t(\bar{\gamma}_{i,[0,t]},\tilde{\gamma},w)\right\vert \,dq_{ij,0}^N(w) \,d\mu^{j}_{[0,t]}(\tilde{\gamma}) \\ 
		&\quad + \frac{1}{N} \sum_{i=1}^{N} \sum_{j=1}^{N}\left\vert \int_{\mathcal{C}_t^d}\int_{-W}^{W} \Pi_t(\bar{\gamma}_{i,[0,t]},\tilde{\gamma},w) \,dq_{ij,0}^N(w) \,d(\mu^{j}_{[0,t]}-\nu^{j}_{[0,t]})(\tilde{\gamma})\right\vert \\
		&:= I_1 + I_2.
	\end{aligned}
	\]
	The first term, \(I_1\), can be estimated directly using the Lipschitz property of \(\Pi\):
	\[
		I_1 \leq L_{\Pi}(T)(1+W) \Vert \gamma-\bar{\gamma} \Vert_{\ast,t}.
	\]
	To estimate \(I_2\), we employ the Kantorovich-Rubinstein duality formula for the \(1\)-Wasserstein distance, which ensures that for any two probability measures \(\mu,\nu \in \mathcal{P}(\mathcal{C}^d_t)\), it holds that
	\[
		W_{1}(\mu,\nu) = \sup_{\psi \in \mathcal{L}^d_1} \left\vert \int_{\mathcal{C}^d_t} \psi(\gamma) \,d(\mu-\nu)(\gamma) \right\vert,
	\]
	where \(\mathcal{L}^d_1 = \{ \psi: \mathcal{C}^d_t \to \mathbb{R}^d \text{ such that } \vert \psi(\gamma) -\psi(\bar{\gamma}) \vert \leq \Vert \gamma - \bar{\gamma} \Vert_{\ast, t}\} \). 
	Since the mapping \(\tilde{\gamma} \mapsto \psi(\tilde{\gamma}) := \int_{-W}^{W} \Pi_t(\bar{\gamma}_{i,[0,t]},\tilde{\gamma},w) \,dq_{ij,0}^N(w)\) satisfies the Lipschitz bound
	\[
		\vert \psi(\tilde{\gamma})-\psi(\hat{\gamma})\vert \leq L_{\Pi}(T)(1+W) \Vert \tilde{\gamma} - \hat{\gamma} \Vert_{\ast,t},
	\]
	we deduce that
	\[
	\begin{aligned}
		I_2 &\leq L_{\Pi}(T)(1+W) \frac{1}{N} \sum_{i=1}^{N} \sum_{j=1}^{N} W_{1}(\mu^{j}_{[0,t]},\nu^{j}_{[0,t]}) 
		\leq L_{\Pi}(T)(1+W) \sum_{j=1}^{N} W_{p}(\mu^{j},\nu^{j}) \\ 
		&\leq L_{\Pi}(T)(1+W) N W_{p}(\mu,\nu).
	\end{aligned}
	\]
	Finally, the required boundedness of \(b(t,0,\delta_0)\) is an immediate consequence of the boundedness of \(\Phi\) and \(K\). The measurability of \(b_i(\cdot,\gamma,\mu)\) follows directly from the continuity of the flow map \(\Phi\) and the kernel \(K\).
\end{proof}

\begin{remark}[Independent projection] \label{remark:independent-projection}
	If the family of initial states \((X_{i,0}^N)_{1\leq i \leq N}\) are independent, and also independent 
    from the family of initial random weights \((w_{ij,0}^N)_{1\leq i,j\leq N}\), then
    the solution \((\bar{X}_i^N)_{1\leq i\leq N}\) to the intermediate system
	\eqref{eq:multi-agent-discrete-independent-projection-solved-weights} forms an independent family of processes
    and each interaction term for \(i\neq j\) can be expressed as 
	\[
		\begin{aligned}
			\int_{\mathcal{C}^d_t}\int_{-W}^{W} \Phi_t(\bar X_{i,[0,t]}^N,\gamma,w)\, &dq_{ij,0}^N(w)\,K(\bar X_i^N(t),\gamma(t))\,d\bar \mu^{N,j}_{[0,t]}(\gamma)= \\  
			&= \mathbb{E}_i^N\left[\Phi_t(\bar X_{i,[0,t]}^N,\bar X_{j,[0,t]}^N,Nw_{ij,0}^N)\,K(\bar X_i^N(t),\bar X_j^N(t))\right],
		\end{aligned}
	\]
	where \(\mathbb{E}_i^N:= \mathbb{E}[\,\cdot \mid \bar{\mathcal{F}}^N_i(t)]\) is the conditional expectation to the
	element \(\bar {\mathcal{F}}_i^N(t):=\sigma(\{\bar X_i^N(s):\,s\in [0,t]\})\) of the natural filtration generated
	by the \(i\)-th particle up to time \(t\). This demonstrates that the intermediate system
	\eqref{eq:multi-agent-discrete-independent-projection-solved-weights} can be equivalently formulated as:
    \begin{equation}\label{eq:multi-agent-discrete-independent-projection-preparatory}
		\begin{aligned}
			d\bar X_i^N(t) &= \frac{1}{N}\sum_{j=1}^N \mathbb{E}_i^N\left[\Phi_t(\bar X_{i,[0,t]}^N,\bar X_{j,[0,t]}^N,Nw_{ij,0}^N)\,K(\bar X_i^N(t),\bar X_j^N(t))\right]\,dt\\
            &\quad + \frac{1}{N}f_t(\bar{X}_{i,[0,t]}^N,\bar{\mu}^{N,i}_{[0,t]},q^N_{ii,0})dt +  \sqrt{2\nu}\,dW_i^N(t),\\
			\bar X_i^N(0) &= X_{i,0}^N,
		\end{aligned}
    \end{equation}
    where \(f_t: \mathcal{C}_t^d\times \mathcal{P}(\mathcal{C}_t^d) \times \mathcal{P}([-W,W])\to \mathbb{R}^d\) is given for every \(t \in [0,T]\) by 
    \begin{equation} \label{eq:self-interaction}
    f_t(\gamma,\mu,q)=\int_{\mathcal{C}^d_t}\int_{-W}^{W}\Phi_t(\gamma,\tilde{\gamma},w) K(\gamma(t),\tilde{\gamma}(t)) \, dq(w) \, d\mu(\tilde{\gamma})- \int_{-W}^{W} \Phi_t(\gamma,\gamma,w)K(\gamma(t),\gamma(t)) \,dq(w).
    \end{equation}
	For this reason, we refer to the system \eqref{eq:multi-agent-discrete-independent-projection-solved-weights}
	as the \emph{independent projection} of the system \eqref{eq:multi-agent-discrete-closed-states}. This aligns with the terminology introduced in \cite{Lacker-26} for It{\^o} diffusions with local coefficients, an approach that seamlessly extends to our path-dependent setting.
\end{remark}

\begin{remark}[Reformulation as a coupled system for states and weights] \label{remark:reformulation-states-weights}
	Although we introduced the independent projection \eqref{eq:multi-agent-discrete-independent-projection-solved-weights}
	as a closed system of path-dependent McKean-Vlasov SDEs strictly for the particle states,
	we can equivalently reformulate it as a coupled system of SDEs for both the states 
        of the particles and the weights. Since the weights were integrated out during the independent projection,
	their evolution is now captured by the transport of their initial laws \((q^N_{ij,0})_{1\leq i,j\leq N}\). 
	
	To see this, for any \(i,j\in \llbracket 1,N \rrbracket\) and \(t\in [0,T]\), we define the map 
	\(\bar{q}^N_{ij}:[0,t]\times \mathcal{C}^d_t \times \Omega \to \mathcal{P}(\mathbb{R})\),
	\((s,\gamma,\omega)\mapsto \bar{q}^N_{ij,\gamma}(s)\) given by 
	\begin{equation} \label{eq:evolution-prob-graphon}
		\bar q_{ij,\gamma}^N(s) := \Phi_s(\bar X_{i,[0,s]}^N,\gamma_{[0,s]},\cdot)_{\#}q_{ij,0}^N, 
	\end{equation}
	where as usual we do not explicitly write the dependency on \(\omega \in \Omega\) in the notation, which is
	implicit through the process \(\bar X_i^N\).
	Since \(\Phi\) is the flow map of the ODE driven by the plasticity
	function \(\Gamma\), for every realization of \(\bar{X}^N_{i,[0,t]}\) and every
	\(\gamma\in\mathcal{C}^d_t\), the map \(s\in [0,t]\mapsto \bar{q}_{ij,\gamma}^N(s) \in \mathcal{P}(\mathbb{R})\)
	solves the following transport equation in distributional sense:
	\[
		\partial_s\bar q_{ij,\gamma}^N(s,w) + \partial_w \left(\Gamma(\bar X_i^N(s),\gamma(s),w)\bar{q}_{ij,\gamma}^N(s,w)\right) = 0, \quad s \in [0,t], w \in \mathbb{R},
	\]
	with the initial condition \(\bar q_{ij,\gamma}^N(0) = q_{ij,0}^N\).
	Altogether, this implies that the independent projection 
	\eqref{eq:multi-agent-discrete-independent-projection-solved-weights} can be formulated as the following coupled system
	for \((\bar{X}_i^N,\bar{q}_{ij}^N)_{1\leq i,j\leq N}\)
	\begin{equation}\label{eq:multi-agent-discrete-independent-projection}
		\begin{aligned}
			&d\bar X_i^N(t) = \frac{1}{N}\sum_{j=1}^N \int_{\mathcal{C}^d_t} \int_{\mathbb{R}} w\, \bar{q}^N_{ij,\gamma}(t,dw)\,K(\bar X_i^N(t),\gamma(t))\,d\bar \mu_{[0,t]}^{N,j}(\gamma)\,dt + \sqrt{2\nu}\,dW_i^N(t),\quad t\in [0,T],\\
			&\partial_s\bar q_{ij,\gamma}^N(s,w)+\partial_w\left(\Gamma(\bar X_i^N(s),\gamma(s),w)\bar{q}_{ij,\gamma}^N(s,w)\right)=0,\quad s\in [0,t], \,\gamma\in \mathcal{C}^d_t, \\
			&\bar X_i^N(0) = X_{i,0}^N, \quad \bar{q}_{ij,\gamma}^N(0) = q_{ij,0}^N, \quad \bar \mu^{N,j}_{[0,t]} := \mathrm{Law}(\bar X_{j,[0,t]}^N)\in \mathcal{P}(\mathcal{C}^d_t).
		\end{aligned}
	\end{equation}
	In fact, both formulations \eqref{eq:multi-agent-discrete-independent-projection-solved-weights} and
	\eqref{eq:multi-agent-discrete-independent-projection} are strictly equivalent, as the linear growth
	and boundedness properties of \(\Gamma\) ensure that:
	\[
		\int_{0}^{T} \int_{\mathbb{R}} \vert \Gamma(\bar{X}_i^N(t),\gamma(t),w) \vert \, \bar{q}_{ij,\gamma}^N(t,dw) \, dt < \infty.
	\]
	This integrability guarantees that the unique narrowly continuous solution
	\(s\in [0,t] \mapsto \bar{q}_{ij,\gamma}^N(s) \in \mathcal{P}(\mathbb{R})\) of the transport equation
	\(\eqref{eq:multi-agent-discrete-independent-projection}_2\) 
	(for every realization of \(\bar{X}^N_{i,[0,t]}\) and every \(\gamma\in\mathcal{C}^d_t\)) is exactly given
	by the pushforward of the initial condition through the flow map associated with the characteristic ODE of
	\(\Gamma\), following \cite[Theorem 8.2.1]{AGS-05}. Thus,
	\(\bar q_{ij,\gamma}^N(s) = \Phi_s(\bar X_{i,[0,s]}^N,\gamma_{[0,s]},\cdot)_{\#}q_{ij,0}^N\).
\end{remark}

We are now in position to compare the original system \eqref{eq:multi-agent-discrete} with its independent projection
\eqref{eq:multi-agent-discrete-independent-projection-solved-weights}. 
Under the weight scaling assumption \eqref{eq:hypothesis-weight}, we will show that the trajectories of
the two systems are asymptotically close as \(N \to \infty\). Specifically, we establish a quantitative
error bound of order \(1/\sqrt{N}\), reminiscent of classical propagation of chaos results for exchangeable
systems \cite{MJ-98,S-91}.

\begin{lemma}[Propagation of independence] \label{lemma:propagation-of-independence}
	Fix \(p\in [1,2]\) and \(q\in [1,\infty]\). Under Assumptions \ref{assump:interaction-kernel} and
	\ref{assump:plasticity-function} on \(K\) and \(\Gamma\), let \((X_1^N,\ldots,X_N^N)\) and 
	\((\bar{X}_1^N,\ldots,\bar{X}_N^N)\) denote the unique strong solutions to systems \eqref{eq:multi-agent-discrete} and
	\eqref{eq:multi-agent-discrete-independent-projection-solved-weights}, respectively. Assume that both systems
	are initialized with the same independent family of initial states and weights
	\((X_{i,0}^N,w_{ij,0}^N)_{1\leq i,j \leq N}\), with \(\mathbb{E}\vert X_{i,0}^N\vert^p < \infty\) and the weights
	satisfying Assumption \ref{assump:weights}. 
	Then, for every \(t\in [0,T]\), the following estimate holds
	\begin{equation}
		\left(\frac{1}{N} \sum_{i=1}^{N} \left( \mathbb{E}\Vert X_i^N -\bar{X}_i^N \Vert^p_{\ast,t}\right)^{q/p} \right)^{1/q} \leq \frac{C(t)}{\sqrt{N}}, 
	\end{equation}
	where the constant \(C(t)\) is given by
	\[
		C(t) := 10t B_K B_{\Phi}(t)(1+W) e^{2tL_{\Pi}(t)(1+W)}.
	\]
\end{lemma}
\begin{proof}
	We will present the proof for \(1 \leq q < \infty\). The case \(q=\infty\) follows identical steps, replacing the normalized sums over the index \(i\) with the maximum over \(i\). 
	
	Let us recall the integral equations satisfied by \(X_i^N\) and \(\bar{X}_i^N\):
	\[
	\begin{aligned}
		X_i^N(t) &= X_{i,0}^N + \sum_{j=1}^{N} \int_{0}^{t} w^N_{ij}(s)K(X_i^N(s),X_j^N(s)) \,ds + \sqrt{2\nu} W_i^N(t), \\
		\bar{X}_i^N(t) &= X_{i,0}^N + \sum_{j=1}^{N} \int_{0}^{t} \mathbb{E}^N_i[\bar{w}^N_{ij}(s)K(\bar{X}_i^N(s),\bar{X}_j^N(s))] \,ds + \frac{1}{N}\int_{0}^{t}f_s(\bar{X}_{i,[0,s]}^N,\bar{\mu}^{N,i}_{[0,s]},q^N_{ii,0})ds+\sqrt{2\nu} W_i^N(t), 
	\end{aligned}
	\]  
	where we use the abbreviations \(w^N_{ij}(s) = \frac{1}{N}\Phi(s,X_i^N,X_j^N,Nw_{ij,0}^N)\) and
	\(\bar{w}^N_{ij}(s) = \frac{1}{N}\Phi(s,\bar{X}_i^N,\bar{X}_j^N,Nw_{ij,0}^N)\). As established in
	Remark \ref{remark:independent-projection}, \(\mathbb{E}_i^N\) denotes the conditional expectation 
	with respect to the natural filtration of the \(i\)-th particle and the self-interaction term is given through the function \(f_s\) defined
    in \eqref{eq:self-interaction}. Taking the difference between the two equations,
	applying the triangle inequality, and taking the supremum over time up to \(t\), we obtain
	\[
		\begin{aligned}
		\Vert X_i^N-\bar{X}_i^N\Vert_{\ast,t} &\leq \int_{0}^{t}\sum_{j=1}^{N}\vert w^N_{ij}(s)K(X_i^N(s),X_j^N(s)) -\bar{w}^N_{ij}(s)K(\bar{X}_i^N(s),\bar{X}_j^N(s)) \vert \,ds \\ 
					   	&\quad+ \int_{0}^{t} \left\vert \sum_{j=1}^{N}\bar{w}^N_{ij}(s)K(\bar{X}_i^N(s),\bar{X}_j^N(s)) - \mathbb{E}^N_i\left[\bar{w}^N_{ij}(s)K(\bar{X}_i^N(s),\bar{X}_j^N(s))\right] \right\vert \,ds \\
                        &\quad + \frac{1}{N}\int_{0}^t \vert  f_s(\bar{X}_{i,[0,s]},\bar{\mu}^{N,i}_{[0,s]},q^N_{ii,0}) \vert ds  \\
					   	&:= \int_{0}^{t}A_i(s)\,ds + \int_{0}^{t}B_i(s)\,ds + \int_{0}^{t}C_i(s)\,ds,
		\end{aligned}
	\]
	with
	\[
		\begin{aligned}
			A_i(s) &:= \sum_{j=1}^{N}\vert w^N_{ij}(s)K(X_i^N(s),X_j^N(s)) -\bar{w}^N_{ij}(s)K(\bar{X}_i^N(s),\bar{X}_j^N(s)) \vert, \\
			B_i(s) &:= \left\vert \sum_{j=1}^{N}\bar{w}^N_{ij}(s)K(\bar{X}_i^N(s),\bar{X}_j^N(s)) - \mathbb{E}^N_i\left[\bar{w}^N_{ij}(s)K(\bar{X}_i^N(s),\bar{X}_j^N(s))\right] \right\vert,
		\end{aligned}
	\]
    and the last term that can be bounded as follows
    \[
    C_i(s):= \frac{1}{N}\vert f_s(\bar{X}^N_{i,[0,s]},\bar{\mu}^{N,i}_{[0,s]},q_{ii,0}^N)\vert \leq \frac{2B_KB_{\Phi}(s)(1+W)}{N}.
    \]
	Taking the \(L^p(\Omega,\mathcal{C}_T^d)\) norm, applying Minkowski's integral inequality, 
	and then taking the weighted \(\ell_q\) norm over the index \(i\), we find 
        \begin{align}
		&\left(\frac{1}{N}\sum_{i=1}^{N} \left(\mathbb{E} \Vert X_i^N-\bar{X}_i^N\Vert^p_{\ast,t}\right)^{q/p}\right)^{1/q}\label{eq:propagation-independence-starting-inequality}\\
        &\qquad\leq\int_{0}^{t} \left( \frac{1}{N} \sum_{i=1}^{N}\left(\mathbb{E} \vert A_i(s) \vert^p\right)^{q/p}\right)^{1/q} \,ds + \int_{0}^{t} \left( \frac{1}{N} \sum_{i=1}^{N}\left(\mathbb{E} \vert B_i(s) \vert^p\right)^{q/p}\right)^{1/q} \,ds+\frac{2tB_KB_{\Phi}(t)(1+W)}{N}.\nonumber
        \end{align}
	We will now estimate each of the two terms on the right-hand side separately.
	
	\medskip

	\noindent \(\diamond\) \textsc{Step 1}: Control of \(A_i(s)\). \\
	Recalling our notation and the Lipschitz property derived for the mapping \(\Pi\) in
	Proposition \ref{prop:properties-Phi-Pi}, together with the scaling assumption \eqref{eq:hypothesis-weight} on the
	initial random weights, we obtain:
	\[
		\begin{aligned}
			\vert A_i(s)\vert &= \frac{1}{N}\sum_{j=1}^{N} \vert \Pi(s,X_i^N,X_j^N,Nw_{ij,0}^N)-\Pi(s,\bar{X}_i^N,\bar{X}_j^N,Nw_{ij,0}^N)\vert \\ 
				     	&\leq L_{\Pi}(s)(1+ W)\left(\Vert X_i^N- \bar{X}_i^N\Vert_{\ast,s} + \frac{1}{N}\sum_{j=1}^{N} \Vert X_j^N- \bar{X}_j^N \Vert_{\ast,s} \right).  
		\end{aligned}
	\]
	Taking the \(L^p(\Omega,\mathbb{R}^d)\) norm and applying Minkowski's inequality yields
	\[
		\left(\mathbb{E} \vert A_i(s)\vert^p\right)^{1/p} \leq  L_{\Pi}(s)(1+ W) \left[ \left( \mathbb{E}  \Vert X_i^N- \bar{X}_i^N\Vert_{\ast,s}^p\right)^{1/p} + \frac{1}{N}\sum_{j=1}^{N} \left(\mathbb{E} \Vert X_j^N- \bar{X}_j^N \Vert_{\ast,s}^p\right)^{1/p}\right].
	\]
	Next, computing the weighted \(\ell_q\) norm over the index \(i\) and using Minkowski's inequality once again, we deduce
	\begin{equation} \label{eq:propagation-independence-control-A}
		\left(\frac{1}{N}\sum_{i=1}^{N} \left(\mathbb{E} \vert A_i(s)\vert^p\right)^{q/p} \right)^{1/q} \leq 2L_{\Pi}(s)(1+ W)\left(\frac{1}{N} \sum_{i=1}^{N} \left(\mathbb{E} \Vert X_i^N- \bar{X}_i^N\Vert_{\ast,s}^p\right)^{q/p}\right)^{1/q}.
	\end{equation}

	\medskip

	\noindent \(\diamond\) \textsc{Step 2}: Control of \(B_i(s)\). \\ 
	Since \(p\in [1,2]\), we can bound the \(p\)-th moment by the second moment using Jensen's inequality, which we now proceed to estimate. We define, for every \(i,j \in \llbracket 1,N \rrbracket\) and \(s \in [0,t]\),  
	\[
		x_{ij}(s) = \bar{w}^N_{ij}(s)K(\bar{X}_i^N(s),\bar{X}_j^N(s)) - \mathbb{E}^N_i\left[\bar{w}^N_{ij}(s)K(\bar{X}_i^N(s),\bar{X}_j^N(s))\right],
	\]
	which allows us to express \(B_i(s) = \sum_{j=1}^{N} x_{ij}(s)\). Expanding the square gives
	\[
		\mathbb{E}\vert B_i(s)\vert^2 = \mathbb{E}\left\vert \sum_{j=1}^{N} x_{ij}(s) \right\vert^2 = \sum_{j=1}^{N} \mathbb{E}\vert x_{ij}(s)\vert^2 + \sum_{\substack{j=1 \\ j\neq i}}^{N}\sum_{\substack{k=1 \\ k\neq j}}^{N} \mathbb{E}(x_{ij}(s)x_{ik}(s)) +
        \sum_{\substack{k=1 \\ k\neq i}}^{N} \mathbb{E}(x_{ii}(s)x_{ik}(s))
	\]
	This expansion resembles the classical variance control in the propagation of chaos literature, where the cross-terms
	are expected to vanish. Indeed, for distinct indices \(i \neq j \neq k\), the random variables \((\bar{X}_{i,[0,s]}^N,\bar{X}_{j,[0,s]}^N,\bar{X}_{k,[0,s]}^N)\) are independent. Thus, \(x_{ij}(s)\) and \(x_{ik}(s)\) are conditionally independent given \(\bar{\mathcal{F}}_i^N(s)\). Using the law of total expectation, \(\mathbb{E}=\mathbb{E}\mathbb{E}^N_i\), we obtain:
	\[
		\mathbb{E}(x_{ij}(s)x_{ik}(s)) = \mathbb{E}\mathbb{E}^N_i(x_{ij}(s)x_{ik}(s)) = \mathbb{E}(\mathbb{E}^N_i(x_{ij}(s))\mathbb{E}^N_i(x_{ik}(s))) = 0.
	\]
	Furthermore, the linear growth of the flow map \(\Phi\) established in \eqref{eq:bound-Phi} ensures that the initial weight scaling propagates in time
	\[
		\vert N\bar{w}^N_{ij}(s) \vert = \vert \Phi(s,\bar{X}_i^N,\bar{X}^N_j,Nw_{ij,0}^N) \vert \leq B_{\Phi}(s)(1+W).
	\]
	This, combined with the boundedness of the kernel \(K\), implies that for every \(i,j \in \llbracket1,N \rrbracket\), 

    \[
\vert x_{ij}(s) \vert \leq \frac{2B_KB_{\Phi}(s)(1+W)}{N},
    \]
    thus
	\[
		\begin{aligned}
			\mathbb{E}\vert B_i(s)\vert^2 &= \sum_{j=1}^{N} \mathbb{E} \vert x_{ij}\vert^2 + 
        \sum_{\substack{k=1 \\ k\neq i}}^{N} \mathbb{E}(x_{ii}(s)x_{ik}(s)) 
			 \leq \frac{8 B^2_K B^2_{\Phi}(s)(1+W)^2}{N}.
		\end{aligned}
	\]
	Consequently, the \(p\)-th moment of \(B_i(s)\) scales as \(N^{-1/2}\), and the corresponding weighted \(\ell_q\) norm admits the bound
	\begin{equation} \label{eq:propagation-independence-control-B}
		\left(\frac{1}{N} \sum_{i=1}^{N} \left(\mathbb{E}\vert B_i(s)\vert^p\right)^{q/p}\right)^{1/q} \leq \frac{8B_K B_{\Phi}(s)(1+W)}{\sqrt{N}}.
	\end{equation}
	Finally, combining the estimates \eqref{eq:propagation-independence-control-A}
	and \eqref{eq:propagation-independence-control-B} for \(A_i\) and \(B_i\) within
	\eqref{eq:propagation-independence-starting-inequality}, an application of Gr\"onwall's inequality
	yields the desired estimate. 
\end{proof}

In order to pass to the limit as \(N\to\infty\), the traditional approach in the mean-field limit literature is to derive the
partial differential equation (typically a Fokker-Planck equation) satisfied by the laws of the independent processes
\(\bar{X}_i^N(t)\), and subsequently pass to the limit within this PDE framework. However, the non-Markovian nature
of the system \eqref{eq:multi-agent-discrete-independent-projection-solved-weights} makes it impossible to derive a closed-form
PDE for the laws of \(\bar{X}_i^N(t)\). We therefore adopt an alternative strategy, which consists of reformulating
the intermediate system \eqref{eq:multi-agent-discrete-independent-projection-solved-weights} directly in the SDE form.

This reformulation is heavily inspired by the theory of graph limits \cite{Lovasz-12} and builds upon methodologies first
adapted for continuum limits in \cite{M-14,M-14-ARMA}. The core idea is to embed the discrete system into a continuous domain
by replacing the discrete indices \(i,j\in \llbracket 1,N \rrbracket\) with continuous labels \(\xi,\xi' \in [0,1]\). This
procedure yields continuous objects that encode both the particle states and the underlying network structure,
providing the correct functional setting to rigorously pass to the limit as \(N \to \infty\).

\begin{definition}[Reformulation of the independent projection]
	For every \(N\in \mathbb{N}\), we define the piecewise-constant extensions of the state processes
	and the initial weight laws as
	\begin{equation*}
		\begin{aligned}
		\bar X^N(t,\xi) &:= \sum_{i=1}^N \bar X_i^N(t)\,\mathds{1}_{I_i^N}(\xi), \quad t\in [0,T],\,\xi\in [0,1],\\
		q^{N,\xi,\xi'}_0 &:= \sum_{i,j=1}^N q^N_{ij,0}\,\mathds{1}_{I_i^N\times I_j^N}(\xi,\xi'), \quad \xi,\xi'\in [0,1].
		\end{aligned}
	\end{equation*}
	With this notation, the independent projection system \eqref{eq:multi-agent-discrete-independent-projection-solved-weights}
	can be equivalently rewritten as the following system of path-dependent McKean-Vlasov SDEs parameterized by \(\xi \in [0,1]\):
	\begin{align}
		d\bar X^N(t,\xi) &= \int_0^1\int_{\mathcal{C}^d_t} \int_{\mathbb{R}} \Phi_t(\bar{X}^N(\cdot,\xi)_{[0,t]},\gamma,w) K(\bar X^N(t,\xi),\gamma(t)) \, dq^{N,\xi,\xi'}_0(w)\,d\bar \mu_{[0,t]}^{N,\xi'}(\gamma)\,d\xi' dt + \sqrt{2\nu}\,dW^N(t,\xi),\nonumber\\
		\bar X^N(0,\xi) &= X^N_{0}(\xi), \quad \bar \mu^{N,\xi}_{[0,t]} := \mathrm{Law}(\bar X^N(\cdot,\xi)_{[0,t]}) \in \mathcal{P}(\mathcal{C}^d_t).\label{eq:multi-agent-macroscopic-intermediate-solved-weights}
	\end{align}
	The initial states are analogously set as
	\begin{equation*}
		X_0^N(\xi) := \sum_{i=1}^N X_{i,0}^N\,\mathds{1}_{I_i^N}(\xi), \quad \xi\in [0,1],
	\end{equation*}
	and the driving noise \(W^N(t,\xi)\) is defined piecewise by
	\[
		W^N(t,\xi) := \sum_{i=1}^N W_i^N(t)\, \mathds{1}_{I_i^N}(\xi),\quad t \in [0,T],\,\xi\in [0,1].
	\]
\end{definition}
\begin{remark} \label{remark:reformulation-states-weights-probability-graphon}
	Following Remark \ref{remark:reformulation-states-weights}, if we define for every \(t\in [0,T]\) the map
	\(\bar q^{N}:[0,t]\times [0,1]^2\times \mathcal{C}_t^d \times \Omega \to \mathcal{P}(\mathbb{R})\) by
	\[
		\bar q_\gamma^{N,\xi,\xi'}(s) := \sum_{i,j=1}^N \bar{q}^N_{ij,\gamma}(s)\,\mathds{1}_{I_i^N\times I_j^N}(\xi,\xi') ,\quad s\in [0,t],\,\xi,\xi'\in [0,1], \, \gamma \in \mathcal{C}^d_t, 
	\]
	where \(\bar{q}^N_{ij,\gamma}(s)\) is defined as in \eqref{eq:evolution-prob-graphon}, we can reformulate 
	\eqref{eq:multi-agent-macroscopic-intermediate-solved-weights} as the following coupled system  
	\begin{equation}\label{eq:multi-agent-macroscopic-intermediate}
	\begin{aligned}
		&d\bar X^N(t,\xi) = \int_0^1\int_{\mathcal{C}^d_t} \int_{\mathbb{R}} w \,\bar q^{N,\xi,\xi'}_\gamma(t,dw)\,K(\bar X^N(t,\xi),\gamma(t))\,d\bar \mu_{[0,t]}^{N,\xi'}(\gamma)\,d\xi' dt + \sqrt{2\nu}\,dW^N(t,\xi),\quad t\in [0,T],\\
		&\partial_s \bar{q}^{N,\xi,\xi'}_{\gamma}(s,w) + \partial_w\left(\Gamma(\bar X^N(s,\xi),\gamma(s),w)\bar q^{N,\xi,\xi'}_{\gamma}(s,w)\right)=0,\quad s\in [0,t],\,\gamma\in \mathcal{C}^d_t,\\
		&\bar X^N(0,\xi) = X_0^N(\xi),\quad \bar q_\gamma^{N,\xi,\xi'}(0) = q_0^{N,\xi,\xi'}, \quad \bar \mu_{[0,t]}^{N,\xi'} := \mathrm{Law}(\bar{X}^N(\cdot,\xi')_{[0,t]}).
	\end{aligned}
	\end{equation}
\end{remark}

Following the reformulation of the independent projection in 
\eqref{eq:multi-agent-macroscopic-intermediate}, the system is now in a suitable form to formally
pass to the limit as \(N\to \infty\) and identify the macroscopic equation \eqref{eq:multi-agent-macroscopic-SDE}. To make
this mean-field limit rigorous, we dedicate the next two sections to study first the well-posedness of
\eqref{eq:multi-agent-macroscopic-intermediate}, and subsequently establishing
its stability with respect to the initial data and the underlying probability-graphon.

\section{Well-posedness of the path-dependent McKean-Vlasov system} \label{sec:well-posedness}
In this section, we study the well-posedness of the McKean-Vlasov SDE system 
\eqref{eq:multi-agent-macroscopic-intermediate-solved-weights} arising from the 
continuum-label reformulation of the independent projection. This equation falls into the following general
class of path-dependent McKean-Vlasov SDEs:
\begin{equation} \label{eq:fibered-McKean-Vlasov-SDE}
	\begin{aligned}
		dX(t,\xi) &= b_{\xi}(t,X(\cdot,\xi),\mu)dt + \sigma_{\xi}(t,X(\cdot,\xi),\mu)dW(t,\xi), \\ 
		X(0,\xi) &= X_0(\xi), \quad \mu = (\mathrm{Law}(X(\cdot,\xi)))_{\xi \in [0,1]}.
	\end{aligned}
\end{equation}
Here, the drift \(b\) and the diffusion \(\sigma\) are given by the functionals
\[
	\begin{aligned}
		b&:[0,1]\times[0,T]\times \mathcal{C}_T^d\times \mathcal{P}_{p,q,\nu}(\mathcal{C}_T^d\times [0,1]) \to \mathbb{R}^d, \\
		\sigma&:[0,1]\times[0,T]\times \mathcal{C}_T^d\times \mathcal{P}_{p,q,\nu}(\mathcal{C}_T^d\times [0,1]) \to \mathbb{R}^{d\times d}, 
	\end{aligned}
\]
and \((W(\cdot,\xi))_{\xi \in [0,1]}\) is a family of Wiener processes. As we can see, the coupling in the system
is through the fibered pathwise law of the process \(X\), {\it i.e.}, the family of probability measures
\((\mathrm{Law}(X(\cdot,\xi)))_{\xi \in [0,1]}\).
\begin{definition}
	Let \((\Omega,\mathcal{F},\mathbb{P})\) be a complete probability space supporting a family of \(d\)-dimensional
	Wiener processes \((W(\cdot,\xi))_{\xi \in [0,1]}\) and a family of initial random variables
	\((X_0(\xi))_{\xi \in [0,1]}\) taking values in \(\mathbb{R}^d\). A family of stochastic processes
	\(X=(X(\cdot,\xi))_{\xi \in [0,1]}\) is called a \emph{strong solution} to \eqref{eq:fibered-McKean-Vlasov-SDE} if
	the following conditions hold: 
	\begin{enumerate}[label=(\roman*)]
		\item For almost every \(\xi \in [0,1]\), \(\mu^{\xi}=\mathrm{Law}(X(\cdot,\xi))\).  
		\item For almost every \(\xi \in [0,1]\), the process \(X(\cdot,\xi)\) has continuous paths, is adapted to the filtration generated by the Wiener process \(W(\cdot,\xi)\) and the initial condition \(X_0(\xi)\), and satisfies the following integral equation \(\mathbb{P}\)-a.s.:
		\[
			X(t,\xi) = X_0(\xi) + \int_{0}^{t} b_{\xi}(s,X(\cdot,\xi),\mu) \,ds + \int_{0}^{t} \sigma_{\xi}(s,X(\cdot,\xi),\mu) \,dW(s,\xi), \quad t \in [0,T].
		\]  
	\end{enumerate}
	We say that \emph{strong uniqueness} holds if any two strong solutions \(X\) and \(X'\), defined on the same probability space, driven by the same family of Wiener processes, and starting from the same initial conditions, satisfy \(X(\cdot,\xi)=X'(\cdot,\xi)\) \(\mathbb{P}\)-a.s.\ for almost every \(\xi \in [0,1]\).

	We say that \emph{uniqueness in law} holds if for any two strong solutions \(X\) and \(X'\), possibly defined on different probability spaces and driven by different Wiener processes but sharing the same fibered law for their initial conditions, one has that \(\mathrm{Law}(X(\cdot,\xi))=\mathrm{Law}(X'(\cdot,\xi))\) for almost every \(\xi \in [0,1]\).
\end{definition}
Under the following standard assumptions on the coefficients \(b\) and \(\sigma\), we can establish the well-posedness of the system \eqref{eq:fibered-McKean-Vlasov-SDE}.

\begin{assumption} \label{assump:fibered-McKean-Vlasov-SDE-multiplicative-noise}
	The functionals \(b\) and \(\sigma\) satisfy the following conditions:
	\begin{enumerate}[label=(\roman*)]
		\item \textbf{(Measurability):} For every fixed \((\gamma,\mu)\in \mathcal{C}_T^d\times \mathcal{P}_{p,q,\nu}(\mathcal{C}^d_T\times [0,1])\), the maps \((\xi,t)\in [0,1]\times [0,T]\mapsto b_{\xi}(t,\gamma,\mu)\in \mathbb{R}^d\) and \((\xi,t)\in [0,1]\times [0,T]\mapsto \sigma_{\xi}(t,\gamma,\mu)\in \mathbb{R}^{d\times d}\) are measurable.
		\item \textbf{(Lipschitz continuity):} There exists a constant \(L>0\) such that for almost every \(\xi \in [0,1]\), for every \(t \in [0,T]\), for every \(\gamma,\bar{\gamma} \in \mathcal{C}_T^d\), and for every \(\mu,\bar{\mu} \in \mathcal{P}_{p,q,\nu}(\mathcal{C}^d_T\times [0,1])\), it holds that
		\[
			\vert b_{\xi}(t,\gamma,\mu)-b_{\xi}(t,\bar{\gamma},\bar{\mu})\vert + \vert \sigma_{\xi}(t,\gamma,\mu)-\sigma_{\xi}(t,\bar{\gamma},\bar{\mu}) \vert \leq L \left( \Vert \gamma-\bar{\gamma} \Vert_{\ast,t} + d_{W_p,q}(\mu_{[0,t]},\bar{\mu}_{[0,t]}) \right).
		\]
		\item \textbf{(Boundedness):} There exists a constant \(C>0\) such that for almost every \(\xi \in [0,1]\) and for every \(t \in [0,T]\), it holds that
		\[
			\vert b_{\xi}(t,0,\delta_0)\vert + \vert \sigma_{\xi}(t,0,\delta_0) \vert \leq C.
		\]
	\end{enumerate}
\end{assumption}
\begin{proposition}[Well-posedness of the McKean-Vlasov system] \label{prop:well-posedness-fibered-McKean-Vlasov-SDE-multiplicative-noise}
	Suppose that the coefficients \(b\) and \(\sigma\) satisfy Assumption 
	\ref{assump:fibered-McKean-Vlasov-SDE-multiplicative-noise} for some \((p,q)\in [2,\infty)\times [1,\infty)\).
	Then, for any family of initial conditions \((X_0(\xi))_{\xi \in [0,1]}\) such that \(X_0(\xi)\)
	is independent of \(W(\cdot,\xi)\) and
	\((\mathrm{Law}(X_0(\xi)))_{\xi \in [0,1]} \in \mathcal{P}_{p,q,\nu}(\mathbb{R}^d\times [0,1])\),
	there exists a unique strong solution \(X\) to \eqref{eq:fibered-McKean-Vlasov-SDE} with
	\((\mathrm{Law}(X(\cdot,\xi)))_{\xi \in [0,1]} \in \mathcal{P}_{p,q,\nu}(\mathcal{C}^d_T\times [0,1])\).
	Moreover, uniqueness in law holds.
\end{proposition}
\begin{proof}
	For every fixed \(\mu \in \mathcal{P}_{p,q,\nu}(\mathcal{C}^d_T\times [0,1])\), consider the following system of SDEs:
	\[
		\begin{aligned}
			dX(t,\xi) &= b_{\xi}(t,X(\cdot,\xi),\mu) \,dt + \sigma_{\xi}(t,X(\cdot,\xi),\mu) \,dW(t,\xi), \\ 
			X(0,\xi) &= X_0(\xi).
		\end{aligned}
	\]
	Since the fibered law \(\mu\) is fixed, this system is decoupled and consists of a family of classical path-dependent SDEs.
	By Assumption \ref{assump:fibered-McKean-Vlasov-SDE-multiplicative-noise} on the coefficients \(b\) and \(\sigma\),
	the well-posedness of this decoupled system is guaranteed by Theorem \ref{th:well-posedness-path-dependent-SDEs}.
	We will denote by \(X^\mu(\cdot,\xi)\) the unique strong solution to this system. 

	We first assume the following claim: the map \(\Psi: \mathcal{P}_{p,q,\nu}(\mathcal{C}^d_T\times [0,1]) \to \mathcal{P}_{p,q,\nu}(\mathcal{C}^d_T\times [0,1]),\mu \mapsto \Psi(\mu) := (\mathrm{Law}(X^\mu(\cdot,\xi)))_{\xi \in [0,1]}\) is well-defined. 
	Accepting this claim for the moment, we can show that \(\Psi\) admits a unique fixed point. Take 
	\(\mu,\bar{\mu} \in \mathcal{P}_{p,q,\nu}(\mathcal{C}^d_T\times [0,1])\), and let 
	\(X^\mu(\cdot,\xi)\) and \(X^{\bar{\mu}}(\cdot,\xi)\) denote the unique strong solutions corresponding
	to \(\mu\) and \(\bar{\mu}\), respectively. For almost every \(\xi \in [0,1]\) and every \(t \in [0,T]\), the following
	estimate holds:
	\[
		\begin{aligned}
			\Vert X^\mu(\cdot,\xi)-X^{\bar{\mu}}(\cdot,\xi)\Vert_{\ast,t} &\leq \int_{0}^{t} \vert b_{\xi}(s,X^\mu(\cdot,\xi),\mu)-b_{\xi}(s,X^{\bar{\mu}}(\cdot,\xi),\bar{\mu}) \vert \,ds \\  
			&\quad+ \sup_{s\in [0,t]} \left\vert \int_{0}^{s} (\sigma_{\xi}(\tau,X^\mu(\cdot,\xi),\mu)-\sigma_{\xi}(\tau,X^{\bar{\mu}}(\cdot,\xi),\bar{\mu})) \,dW(\tau,\xi) \right\vert.  
    		\end{aligned}
	\]
	Taking the \(L^p(\Omega,\mathcal{C}_T^d)\) norm, applying Minkowski's integral inequality combined with the Burkholder-Davis-Gundy (BDG)
	inequality \eqref{eq:BDG-p-geq-2} for the stochastic integral, and using the Lipschitz properties of \(b\) and \(\sigma\) 
	from Assumption \ref{assump:fibered-McKean-Vlasov-SDE-multiplicative-noise}-(ii), we deduce that for some constant
	\(C(t) > 0\):
	\begin{align*}
		&\left(\mathbb{E}\left[ \Vert X^\mu(\cdot,\xi)-X^{\bar{\mu}}(\cdot,\xi)\Vert_{\ast,t}^p\right]\right)^{1/p}\\
        &\qquad \leq C(t)\left( \int_0^t \left(\mathbb{E}\left[ \Vert X^\mu(\cdot,\xi)-X^{\bar{\mu}}(\cdot,\xi)\Vert^p_{\ast,s}\right]\right)^{1/p} \,ds + \int_{0}^{t} d_{W_p,q}(\mu_{[0,s]},\bar{\mu}_{[0,s]}) \,ds \right).
	\end{align*}
	Thus, by Grönwall's inequality and subsequently taking the \(L^q\) norm with respect to \(\xi\), we find that
	\[
		d_{W_p,q}(\Psi(\mu)_{[0,t]},\Psi(\bar{\mu})_{[0,t]}) \leq C(T) \int_0^t d_{W_p,q}(\mu_{[0,s]},\bar{\mu}_{[0,s]}) \,ds.
	\]
	Iterating this estimate yields that for every \(k \geq 1\):
	\[
		d_{W_p,q}(\Psi^k(\mu)_{[0,t]},\Psi^k(\bar{\mu})_{[0,t]}) \leq \frac{(C(T) t)^k}{k!} d_{W_p,q}(\mu_{[0,t]},\bar{\mu}_{[0,t]}),
	\]
	which demonstrates that for \(k\) sufficiently large, \(\Psi^k\) is a strict contraction. Therefore, \(\Psi\) has a unique fixed point \(\mu^\ast \in \mathcal{P}_{p,q,\nu}(\mathcal{C}^d_T\times [0,1])\), and the corresponding process
	\(X^{\mu^\ast}\) is the unique strong solution to the original system \eqref{eq:fibered-McKean-Vlasov-SDE}. 

	To establish uniqueness in law, it suffices to observe that for every fixed
	\(\mu \in \mathcal{P}_{p,q,\nu}(\mathcal{C}^d_T\times [0,1])\), the fibered law is identical for any two strong
	solutions \(X^\mu\) and \(\bar{X}^\mu\) defined on possibly different probability spaces, provided that their initial
	data \(X_0(\xi)\) and \(\bar{X}_0(\xi)\) share the same law ({\it i.e.}, \(\mathrm{Law}(X_0(\xi))=\mathrm{Law}(\bar{X}_0(\xi))\)).
	Since we are dealing with path-dependent SDEs with Lipschitz coefficients, strong existence and pathwise uniqueness hold.
	By the classical Yamada-Watanabe theorem \cite{YW-71}, this implies uniqueness in law. 

	Finally, we are left to verify the claim that \(\Psi\) is well-defined. First, we will show that for every fixed \(\mu\),
	the map \(\xi \in [0,1] \mapsto \mathrm{Law}(X^\mu(\cdot,\xi)) \in \mathcal{P}_p(\mathcal{C}^d_T)\) is measurable. For every
	\(\xi \in [0,1]\), \(X^\mu(\cdot,\xi)\) is the limit in \(L^p(\Omega; \mathcal{C}^d_T)\) of the Picard iteration scheme:
	\(X^\mu(\cdot,\xi) = \lim_{n\to \infty} X^{\mu,n}(\cdot,\xi)\), where \(X^{\mu,0}(\cdot,\xi) = X_0(\xi)\) and for every
	\(n \geq 1\):
	\[ 
	 	X^{\mu,n}(t,\xi) = X_0(\xi) + \int_{0}^{t} b_{\xi}(s,X^{\mu,n-1}(\cdot,\xi),\mu) \,ds + \int_{0}^{t} \sigma_{\xi}(s,X^{\mu,n-1}(\cdot,\xi),\mu) \,dW(s,\xi), \quad t \in [0,T].
 	\]
 	Therefore, it is sufficient to check that for every \(n \geq 0\), the map 
	\(\xi \in [0,1] \mapsto \mathrm{Law}(X^{\mu,n}(\cdot,\xi)) \in \mathcal{P}_p(\mathcal{C}^d_T)\) is measurable. We will
	prove by induction the stronger condition that for every \(n\geq 0\), the map 
	\(\xi \in [0,1] \mapsto \mathrm{Law}(X^{\mu,n}(\cdot,\xi),W(\cdot,\xi)) \in \mathcal{P}_p(\mathcal{C}^d_T \times \mathcal{C}^d_T)\) is measurable. 

	For \(n = 0\), this holds true since \(\xi \in [0,1] \mapsto \mathrm{Law}(X_0(\xi)) \in \mathcal{P}_p(\mathbb{R}^d)\)
	is measurable, and the joint law of the initial data and the Wiener process is measurable because they are independent
	and the pathwise law of the Wiener process does not depend on \(\xi\). Assuming this holds for all \(k=0,\ldots,n-1\),
	we now prove it for \(k=n\). Because the family of finite-dimensional cylinder sets forms a \(\pi\)-system generating
	the Borel \(\sigma\)-algebra of \(\mathcal{C}^d_T\times \mathcal{C}^d_T\), it is enough to verify that for every finite
	sequence of times \(0 \leq t_1 < t_2 < \ldots < t_k \leq T\), the map 
	\[ 
		\xi \in [0,1] \mapsto \mathrm{Law}(X^{\mu,n}(t_1,\xi),W(t_1,\xi),\ldots,X^{\mu,n}(t_k,\xi),W(t_k,\xi)) \in \mathcal{P}_p((\mathbb{R}^d\times \mathbb{R}^d)^k)
	\]
 	is measurable.  

	For every \(t\in[0,T]\) and \(\xi \in [0,1]\), \(X^{\mu,n}(t,\xi)\) is the limit in probability of the
	discretization \(X^{\mu,n}_{\delta}(t,\xi)\) as \(\delta \to 0\), defined by:
	\begin{equation}\label{eq:measurability-X_delta}
    \begin{aligned} 
		X^{\mu,n}_{\delta}(t,\xi) = X_0(\xi) &+ \sum_{i=1}^{N_{\delta,t}-1}b_{\xi}(i\delta,X^{\mu,n-1}(\cdot,\xi),\mu) \delta\\
        &+ \sum_{i=1}^{N_{\delta,t}-1}\sigma_{\xi}(i\delta,X^{\mu,n-1}(\cdot,\xi),\mu) (W(\delta (i+1),\xi)-W(\delta i,\xi)),
	\end{aligned}
    \end{equation}
	with \(N_{\delta,t} = \lfloor \frac{t}{\delta} \rfloor\). Thus, it suffices to prove that for every \(\delta >0\),
	the map 
		\[
			\xi \in [0,1]\mapsto \textrm{Law}(X^{\mu,n}_{\delta}(t_1,\xi),W(t_1,\xi),\ldots,X^{\mu,n}_{\delta}(t_k,\xi),W(t_k,\xi)) \in \mathcal{P}_p((\mathbb{R}^d\times\mathbb{R}^d)^k),
		\]
	is measurable. By the structure of 
	\eqref{eq:measurability-X_delta}, we can write 
	\[
		X^{\mu,n}_{\delta}(t,\xi) = h_{\delta,t}(\xi,X^{\mu,n-1}(0,\xi),X^{\mu,n-1}(\cdot,\xi)_{[0,\delta]},\ldots,X^{\mu,n-1}(\cdot,\xi)_{[0,(N_{\delta,t}-1)\delta]},W(0,\xi),\ldots, W(N_{\delta,t}\delta,\xi))
	\]
	for some deterministic functional
	\[
		h_{\delta,t}:[0,1]\times \mathbb{R}^d \times \mathcal{C}_{\delta}\times \mathcal{C}_{2\delta}\times \ldots \times \mathcal{C}_{(N_{\delta,t}-1)\delta}\times (\mathbb{R}^d)^{N_{\delta,t}+1} \to \mathbb{R}^d. 
	\]
	It is clear that \(h_{\delta,t}(\xi,\cdot,\ldots,\cdot)\) is continuous and
	\(h_{\delta,t}(\cdot,x,\gamma_1,\ldots,\gamma_{N_{\delta,t}-1},w_1,\ldots,w_{N_{\delta,t}+1})\) is measurable by the
	properties assumed for \(b_{\xi}\) and \(\sigma_{\xi}\). Consequently, \(h_{\delta,t}\) is a Carathéodory function
	and thus jointly measurable. Applying Proposition \ref{prop:Measurability-Integral-Caratheodory-Function},
	we conclude that the desired map is measurable by the induction hypothesis. 

	Finally, we must check that \(\Psi(\mu)\in \mathcal{P}_{p,q,\nu}(\mathcal{C}^d_T\times [0,1])\), which requires: 
	\[
		\left(\int_0^1 \left(\mathbb{E} \Vert X^\mu(\cdot,\xi)\Vert^p_{\ast,T}\right)^{q/p} \,d\xi\right)^{1/q} < \infty.
	\]
	This is satisfied since the Lipschitz and boundedness properties in Assumption
	\ref{assump:fibered-McKean-Vlasov-SDE-multiplicative-noise} imply the linear growth bound
	\[
		\vert b_{\xi}(t,\gamma,\mu)\vert + \vert \sigma_{\xi}(t,\gamma,\mu) \vert  \leq (L+C) \left(1+\Vert \gamma\Vert_{\ast,t}+d_{W_p,q}(\mu_{[0,t]},\delta_0)\right).
	\]
	Combining Jensen's inequality with the BDG inequality \eqref{eq:BDG-p-geq-2} for the stochastic integral, we deduce that for some constant \(C(T) > 0\): 
	\[
		\left(\mathbb{E} \Vert X^\mu(\cdot,\xi) \Vert_{\ast,T}^p\right)^{1/p} \leq C(T)\left(1 + \int_0^T \left(\mathbb{E} \Vert X^\mu(\cdot,\xi) \Vert_{\ast,t}^p\right)^{1/p} \,dt + d_{W_p,q}(\mu,\delta_0) \right).
	\]
	Taking the \(L^q\) norm with respect to \(\xi\) and applying Minkowski's inequality followed by Grönwall's
	inequality yields the necessary uniform bound, allowing us to conclude that
	\(\Psi(\mu) \in \mathcal{P}_{p,q,\nu}(\mathcal{C}^d_T\times [0,1])\).   
\end{proof}
\begin{remark}[On the measurability of the solutions to the path-dependent McKean-Vlasov SDE]
	It is important to emphasize that we do not claim any joint measurability for the map 
	\((\xi,\omega) \in [0,1] \times \Omega \mapsto X(\cdot,\xi)(\omega)\in \mathcal{C}_T^d\). In fact, there is a well-known theoretical obstruction
	\cite[Proposition 1]{Sun-98}: on a product probability space, any jointly measurable family of almost
	surely pairwise independent random variables is necessarily composed of constant random variables. 
	Consequently, any uncountable family of independent Wiener processes, despite being the natural choice for
	modeling idiosyncratic noise, cannot be jointly measurable because its paths are non-trivial.
	
	Instead, we only require the measurability of the map \(\xi \in [0,1]\mapsto \mathrm{Law}(X(\cdot,\xi))\in \mathcal{P}_p(\mathcal{C}^d_T)\). For a given measurable functional \(\Phi:\mathcal{C}_T^d\to \mathbb{R}^d\) and a solution 
	\(X\) to \eqref{eq:fibered-McKean-Vlasov-SDE}, this ensures that quantities of the form 
	\[
		\int_{0}^{1} \mathbb{E}\left[\Phi(X(\cdot,\xi))\right] \,d\xi = \int_{0}^{1} \int_{\mathcal{C}_T^d} \Phi(\gamma) \,d(\mathrm{Law}(X(\cdot,\xi)))(\gamma) \,d\xi
	\]
	are well-defined; however, we cannot apply Fubini's theorem to exchange the order of the expectation and the integral
	over \(\xi\). 
	
	For our purposes, this level of measurability is entirely sufficient, as the coefficients in our equation 
	\eqref{eq:multi-agent-macroscopic-intermediate-solved-weights} depend exclusively on integrals with respect
	to the fibered laws. This approach aligns with the strategy employed in \cite{BCW-23}. To handle more general
	equations that require true pathwise integration over \(\xi\), one typically bypasses this measurability issue
	by constructing a Fubini extension \cite{Sun-06}, as seen in \cite{ACL-22,CDP-25,CT-24-arxiv}.
\end{remark}

With the preceding well-posedness result at hand, we are now in  position to establish the well-posedness of the
macroscopic system \eqref{eq:multi-agent-macroscopic-SDE}. 

\begin{definition}[Strong solution to \eqref{eq:multi-agent-macroscopic-SDE}] \label{def:strong-solution}
	Let \((\Omega,\mathcal{F},\mathbb{P})\) be a complete probability space supporting a family of \(d\)-dimensional Wiener processes \((W(\cdot,\xi))_{\xi \in [0,1]}\) and a family of initial random variables \((X_0(\xi))_{\xi \in [0,1]}\) , along with a deterministic initial fibered probability-graphon \(q_0 \in \mathcal{P}_{\nu\otimes \nu}([-W,W]\times [0,1]^2)\). We shall say that a pair \((X(\cdot,\xi),q_\gamma^{\xi,\xi'})_{\xi,\xi'\in [0,1]}\) is a \emph{strong solution} of \eqref{eq:multi-agent-macroscopic-SDE} if the family of continuous stochastic processes \(X=(X(\cdot,\xi))_{\xi \in [0,1]}\) and the map \(q:[0,1]^2\times [0,T]\times \mathcal{C}_T^d\times \Omega \to \mathcal{P}(\mathbb{R})\) verify:
	\begin{enumerate}[label=(\roman*)]
		\item For almost every \(\xi \in [0,1]\), the pathwise law of the state process satisfies \(\mu^{\xi} = \mathrm{Law}(X(\cdot,\xi)) \in \mathcal{P}(\mathcal{C}_T^d)\), and the map \(\xi \mapsto \mu^{\xi}\) is measurable.
		\item For almost every \(\xi \in [0,1]\), the process \(X(\cdot,\xi)\) is adapted to the filtration generated by the Wiener process \(W(\cdot,\xi)\) and the initial condition \(X_0(\xi)\), and satisfies the following integral equation \(\mathbb{P}\)-a.s.:
		\[
			\begin{aligned}
            &\quad X(t,\xi) = X_0(\xi)\\
            &\qquad+ \int_{0}^{t} \int_0^1\int_{\mathcal{C}^d_s} \int_{\mathbb{R}} w \,q^{\xi,\xi'}_\gamma(s,dw)\,K(X(s,\xi),\gamma(s))\,d\mu_{[0,s]}^{\xi'}(\gamma)\,d\xi' \,ds + \sqrt{2\nu}W(t,\xi), \quad t \in [0,T].
            \end{aligned}
		\]
		\item \(\mathbb{P}\)-a.s. and for almost every \(\xi,\xi' \in [0,1]\), the map 
		\[
		(t,\gamma) \in [0,T]\times \mathcal{C}_T^d \mapsto q^{\xi,\xi'}_\gamma(t) \in \mathcal{P}(\mathbb{R})
		\]
		is narrowly continuous. Moreover, for any fixed path \(\gamma \in \mathcal{C}_T^d\), the time-dependent measure \(t \in [0,T] \mapsto q^{\xi,\xi'}_\gamma(t) \in \mathcal{P}(\mathbb{R})\) solves the transport equation \(\eqref{eq:multi-agent-macroscopic-SDE}_2\) in the distributional sense, with initial condition \(q^{\xi,\xi'}_{\gamma}(0) = q_0^{\xi,\xi'}\).
		\item \(\mathbb{P}\)-a.s., for every \(t\in [0,T]\) and every \(\gamma\in \mathcal{C}_T^d\), the map \((\xi,\xi') \in [0,1]^2 \mapsto q^{\xi,\xi'}_\gamma(t) \in \mathcal{P}(\mathbb{R})\) is measurable.
	\end{enumerate}
\end{definition}
The assumptions in the previous definition guarantee that the integral defining the strong solution is well-defined. Furthermore,
the narrow continuity requirement ensures, as established in Remark \ref{remark:reformulation-states-weights-probability-graphon},
that the weight component of every strong solution \((X(\cdot,\xi),q_\gamma^{\xi,\xi'})_{\xi,\xi'\in [0,1]}\) to \eqref{eq:multi-agent-macroscopic-SDE} is explicitly given by
\[
	q^{\xi,\xi'}_{\gamma}(t)=\Phi_t(X(\cdot,\xi)_{[0,t]},\gamma,\cdot)_{\#} q_{0}^{\xi,\xi'}.
\]
Hence, the processes \(X\) determine a strong solution of the path-dependent McKean-Vlasov SDE:
\begin{align} 
		dX(t,\xi) &= \int_0^1\int_{\mathcal{C}^d_t} \int_{-W}^{W} \Phi_t(X(\cdot,\xi)_{[0,t]},\gamma,w) \,dq^{\xi,\xi'}_0(w)\,K(X(t,\xi),\gamma(t))\,d\mu_{[0,t]}^{\xi'}(\gamma)\,d\xi' \,dt + \sqrt{2\nu}\,dW(t,\xi),\nonumber\\
		X(0,\xi) &= X_{0}(\xi), \quad \mu^{\xi} := \mathrm{Law}(X(\cdot,\xi)) \in \mathcal{P}(\mathcal{C}^d_T).\label{eq:multi-agent-macroscopic-solved-weights}
\end{align}
We note that this equation is a particular instance of \eqref{eq:fibered-McKean-Vlasov-SDE}, where, for a fixed initial probability-graphon
\(q_0\in \mathcal{P}_{\nu\otimes \nu}([-W,W]\times [0,1]^2)\), the drift coefficient \(b^{q_0}:[0,1]\times [0,T]\times \mathcal{C}_T^d\times \mathcal{P}_{\nu}(\mathcal{C}_T^d\times [0,1])\to \mathbb{R}^d\) is explicitly given by:
\begin{equation} \label{eq:drift-macroscopic-system}
	b^{q_0}_{\xi}(t,\gamma,\mu) = \int_0^1\int_{\mathcal{C}^d_t} \int_{-W}^{W} \Phi_t(\gamma_{[0,t]},\bar{\gamma},w) \,dq_0^{\xi,\xi'}(w)\,K(\gamma(t),\bar{\gamma}(t))\,d\mu_{[0,t]}^{\xi'}(\bar{\gamma})\,d\xi'.
\end{equation}

\begin{proposition}\label{prop:well-posedness-multi-agent-macroscopic-solved-weights}
	Under Assumptions \ref{assump:interaction-kernel} and \ref{assump:plasticity-function} on \(K\) and \(\Gamma\)
	respectively, for any initial probability-graphon \(q_0 \in \mathcal{P}_{\nu\otimes \nu}([-W,W]\times[0,1]^2)\),
	and any family of initial conditions \((X_0(\xi))_{\xi \in [0,1]}\) such that \(X_0(\xi)\) is independent of the
	Wiener process \(W(\cdot,\xi)\) and 
	\((\mathrm{Law}(X_0(\xi)))_{\xi \in [0,1]} \in \mathcal{P}_{p,q,\nu}(\mathbb{R}^d\times [0,1])\) with \(p\in [2,\infty)\)
	and \(q\in [1,\infty)\), there exists a unique strong solution \(X\) to
	\eqref{eq:multi-agent-macroscopic-solved-weights} with \((\mathrm{Law}(X(\cdot,\xi)))_{\xi \in [0,1]} \in \mathcal{P}_{p,q,\nu}(\mathcal{C}_T^d\times [0,1])\).
\end{proposition}
\begin{proof}
	The drift coefficient defined in \eqref{eq:drift-macroscopic-system} satisfies Assumption
	\ref{assump:fibered-McKean-Vlasov-SDE-multiplicative-noise}-(i) due to the measurability of the probability-graphon \(q_0\) 
	and the continuity of the flow map \(\Phi_t\). To verify the Lipschitz property, we apply the triangle inequality and the
	Lipschitz continuity of \(\Pi\) established in Proposition \ref{prop:properties-Phi-Pi}. This ensures that the map
	\[
		\tilde{\gamma} \in \mathcal{C}^d_t \mapsto \int_{-W}^{W} \Pi_t(\gamma_{[0,t]},\tilde{\gamma},w) \,dq_0^{\xi,\xi'}(w) \in \mathbb{R}^d
	\]
	is Lipschitz continuous, allowing us to deduce the following estimate for any paths \(\gamma, \bar{\gamma} \in \mathcal{C}_T^d\) and any fibered measures \(\mu, \bar{\mu} \in \mathcal{P}_{p,q,\nu}(\mathcal{C}_T^d \times [0,1])\):
	\begin{equation} \label{eq:lipschitz-continuity-b_xi_q}
		\begin{aligned}
			\vert b_{\xi}^{q_0}(t,\gamma,\mu) - b_{\xi}^{q_0}(t,\bar{\gamma},\bar{\mu}) \vert 
			&\leq \left\vert \int_0^1 \int_{\mathcal{C}^d_t} \int_{-W}^{W} (\Pi_t(\gamma_{[0,t]},\tilde{\gamma},w) - \Pi_t(\bar{\gamma}_{[0,t]},\tilde{\gamma},w)) \,dq^{\xi,\xi'}_0(w)\,d\mu_{[0,t]}^{\xi'}(\tilde{\gamma})\,d\xi'\right\vert \\
			&\quad + \left\vert  \int_0^1 \int_{\mathcal{C}^d_t} \int_{-W}^{W} \Pi_t(\bar{\gamma}_{[0,t]},\tilde{\gamma},w) \,dq^{\xi,\xi'}_0(w)\,d(\mu_{[0,t]}^{\xi'} - \bar{\mu}_{[0,t]}^{\xi'})(\tilde{\gamma})\,d\xi' \right\vert \\
			&\leq L_{\Pi}(t)(1+W) \Vert \gamma - \bar{\gamma} \Vert_{\ast,t} + L_{\Pi}(t)(1+W) \int_0^1 W_{p}(\mu_{[0,t]}^{\xi'},\bar{\mu}_{[0,t]}^{\xi'}) \,d\xi' \\
			&\leq L_{\Pi}(t)(1+W) \left( \Vert \gamma - \bar{\gamma} \Vert_{\ast,t} + d_{W_p,q}(\mu_{[0,t]},\bar{\mu}_{[0,t]}) \right).
		\end{aligned}
	\end{equation}
	The boundedness condition in Assumption \ref{assump:fibered-McKean-Vlasov-SDE-multiplicative-noise}-(iii) follows
	from a similar argument using the uniform boundedness of \(\Pi\). Therefore, by Proposition
	\ref{prop:well-posedness-fibered-McKean-Vlasov-SDE-multiplicative-noise}, the well-posedness of the system is
	rigorously established.
\end{proof}

\section{Stability of the path-dependent McKean-Vlasov system} \label{sec:stability-estimate}
In this section, we derive a stability estimate for the macroscopic McKean-Vlasov SDE system
\eqref{eq:multi-agent-macroscopic-SDE} with respect to perturbations in both the initial data and the
underlying probability-graphon. This result plays a crucial role in the proof of the main mean-field limit theorem presented in the
next section.
\begin{proposition}[Stability estimate for the path-dependent McKean-Vlasov system] \label{prop:stability-estimate-multi-agent-macroscopic-solved-weights}
	Fix \((p,q)\in [2,\infty)\times [1,\infty)\). Let \(\mu_0,\bar{\mu}_0 \in \mathcal{P}_{p,q,\nu}(\mathbb{R}^d\times [0,1])\) be any two initial fibered
	probability measures, and let \(q_0,\bar{q}_0 \in \mathcal{P}_{\nu\otimes \nu}([-W,W]\times [0,1]^2)\) be
	any two initial probability-graphons. Let \(\mu,\bar{\mu}\in \mathcal{P}_{p,q,\nu}(\mathcal{C}^d_T\times [0,1])\)
	denote the unique fibered pathwise laws associated with the strong solutions of \eqref{eq:multi-agent-macroscopic-solved-weights}
	starting from initial data distributed according to \(\mu_0\) and \(\bar{\mu}_0\), and driven by the probability-graphons \(q_0\) and \(\bar{q}_0\), respectively. 
	
	Then, for every \(0 < \varepsilon \leq 1\), there exists a constant \(\kappa(\varepsilon,t) > 0\) such that 
	\begin{equation} \label{eq:stability-estimate}
		d_{W_p,q}(\mu_{[0,t]},\bar{\mu}_{[0,t]}) \leq \tilde{C}(t)\left(d_{W_p,q}(\mu_0,\bar{\mu}_0) + \kappa(\varepsilon,t)^{1/r}d_{\text{BL},\square}(q_0,\bar{q}_0)^{1/r} + \varepsilon^{1/r}\right), \quad t \in [0,T], 
	\end{equation}  
	where \(r = \max(p,q)\) and the constant \(\tilde{C}(t)\) is given by 
	\begin{equation} \label{eq:constant-stability-estimate}
		\tilde{C}(t) = e^{2L_{\Pi}(t)(1+W)t} \left( 1 + 4t\left[2B_{\Pi}(t)(1+W) + 1\right] \right).
	\end{equation}
\end{proposition}
\begin{proof}
	Since for every \(\xi \in [0,1]\) the set of optimal couplings with respect to the \(p\)-Wasserstein distance
	between \(\mu_0^{\xi}\) and \(\bar{\mu}_0^{\xi}\) is non-empty, we can find a family of couplings
	\((\nu_0^{\xi})_{\xi \in [0,1]}\) such that \(\nu_0^{\xi}\in \Gamma_o(\mu_0^{\xi},\bar{\mu}_0^{\xi})\) for almost
	every \(\xi \in [0,1]\). Up to a standard enlargement of the underlying probability space
	\((\Omega,\mathcal{F},\mathbb{P})\), we can support a family of random variables
	\((X_0(\xi),\bar{X}_0(\xi))_{\xi \in [0,1]}\) and Wiener processes \((W(\cdot,\xi))_{\xi \in [0,1]}\) such
	that \(\mathrm{Law}(X_0(\xi),\bar{X}_0(\xi))=\nu_0^{\xi}\), with \(W(\cdot,\xi)\) being independent of
	\((X_0(\xi),\bar{X}_0(\xi))\) for almost every \(\xi \in [0,1]\). 
	
	To construct a coupling between \(\mu^{\xi}\) and \(\bar{\mu}^{\xi}\), we consider the strong solutions
	\(X(\cdot,\xi)\) and \(\bar{X}(\cdot,\xi)\) of \eqref{eq:multi-agent-macroscopic-solved-weights} associated with the
	initial data \(X_0(\xi)\) and \(\bar{X}_0(\xi)\), and the probability-graphons \(q_0\) and \(\bar{q}_0\),
	respectively, and driven by the same family of Wiener processes \((W(\cdot,\xi))_{\xi \in [0,1]}\), so they
	satisfy the following equations:
	\begin{equation} \label{eq:stability-estimate-proof-equations}
		\begin{aligned}
			X(t,\xi) &= X_0(\xi)+ \int_{0}^{t}b^{q_0}_{\xi}(s,X(\cdot,\xi),\mu) \,ds + \sqrt{2\nu}\,W(t,\xi),\\
			\bar{X}(t,\xi) &= \bar{X}_0(\xi)+ \int_{0}^{t}b^{\bar{q}_0}_{\xi}(s,\bar{X}(\cdot,\xi),\bar{\mu}) \,ds + \sqrt{2\nu}\,W(t,\xi),
		\end{aligned}
	\end{equation} 
	where \(b^{q_0}_{\xi}\) and \(b^{\bar{q}_0}_{\xi}\) are the drift coefficients defined in
	\eqref{eq:drift-macroscopic-system}. Since \(\mathrm{Law}(X(\cdot,\xi),\bar{X}(\cdot,\xi))\) is a valid
	coupling between \(\mu^{\xi}\) and \(\bar{\mu}^{\xi}\), and the initial data form an optimal coupling,
	we deduce the following estimate:
	\begin{equation}\label{eq:stability-estimate-auxiliary}
    \begin{aligned}
			&\left(\mathbb{E}\left[\Vert X(\cdot,\xi)-\bar{X}(\cdot,\xi) \Vert_{\ast,t}^p\right]\right)^{1/p} 
			\\
            &\qquad\leq W_{p}(\mu_0^{\xi},\bar{\mu}_0^{\xi}) + \int_{0}^{t} \left(\mathbb{E}\left[\vert b^{q_0}_{\xi}(s,X(\cdot,\xi),\mu)-b^{\bar{q}_0}_{\xi}(s,\bar{X}(\cdot,\xi),\bar{\mu}) \vert^p\right]\right)^{1/p} \,ds.
    \end{aligned}
	\end{equation}
	By the triangle inequality, the drift difference splits into two parts
	\[
		\begin{aligned}
    			&\int_{0}^{t} (\mathbb{E}[\vert b^{q_0}_{\xi}(s,X(\cdot,\xi),\mu)-b^{\bar{q}_0}_{\xi}(s,\bar{X}(\cdot,\xi),\bar{\mu}) \vert^p])^{1/p} ds \\
			&\leq  \int_{0}^{t} (\mathbb{E}[\vert b^{q_0}_{\xi}(s,X(\cdot,\xi),\mu)-b^{q_0}_{\xi}(s,\bar{X}(\cdot,\xi),\bar{\mu}) \vert^p])^{1/p} ds + \int_{0}^{t} (\mathbb{E}[\vert b^{q_0}_{\xi}(s,\bar{X}(\cdot,\xi),\bar{\mu})-b^{\bar{q}_0}_{\xi}(s,\bar{X}(\cdot,\xi),\bar{\mu}) \vert^p])^{1/p} ds \\ 
    			&:= \int_{0}^{t}I_1(s,\xi) + \int_{0}^{t}I_2(s,\xi) ds.
		\end{aligned}
	\]
	The first term, \(I_1(s,\xi)\), can be controlled using the Lipschitz property \eqref{eq:lipschitz-continuity-b_xi_q}
	of the drift coefficient:
	\[
		\int_{0}^{t} I_1(s,\xi) \,ds \leq L_{\Pi}(t)(1+W)\int_{0}^{t} \left( \left(\mathbb{E}\left[\Vert X(\cdot,\xi)-\bar{X}(\cdot,\xi) \Vert_{\ast,s}^p\right]\right)^{1/p} + d_{W_p,q}(\mu_{[0,s]},\bar{\mu}_{[0,s]}) \right) \,ds.
	\]
	Plugging this estimate back into \eqref{eq:stability-estimate-auxiliary}, taking the \(L^q\) norm in \(\xi\), and
	applying Minkowski's inequality, we obtain:
	\[
		\begin{aligned}
			&\left(\int_{0}^{1} \mathbb{E} [\Vert X(\cdot,\xi)-\bar{X}(\cdot,\xi) \Vert^p_{\ast,t}]^{q/p} d\xi\right)^{1/q} 
			\leq d_{W_p,q}(\mu_0,\bar{\mu}_0) \\
			&\qquad + 2L_{\Pi}(t)(1+W)\int_{0}^{t} \left(\int_{0}^{1} \mathbb{E} [\Vert X(\cdot,\xi)-\bar{X}(\cdot,\xi) \Vert^p_{\ast,s}]^{q/p} d\xi\right)^{1/q} ds + \int_{0}^{t} \Vert I_2(s,\cdot) \Vert_{L^q([0,1])} ds.
		\end{aligned}
	\]
	Applying Gr\"onwall's inequality, and noting that \(d_{W_p,q}(\mu_{[0,t]},\bar{\mu}_{[0,t]})\) is bounded by the left-hand side of the previous estimate, we obtain
	\begin{equation} \label{eq:estimate-previous-I2}
		d_{W_p,q}(\mu_{[0,t]},\bar{\mu}_{[0,t]}) \leq e^{2L_{\Pi}(t)(1+W)t}\left(d_{W_p,q}(\mu_0,\bar{\mu}_0) + \int_{0}^{t} \Vert I_2(s,\cdot) \Vert_{L^q([0,1])} \,ds\right).
	\end{equation}
	Now we control the term involving \(I_2(s,\xi)\). Expanding its definition yields
	\[
		\begin{aligned}
			\Vert I_2(s,\cdot)\Vert_{L^q([0,1])} &= \left(\int_{0}^{1} \left( \mathbb{E} \left\vert \int_{0}^{1}\int_{\mathcal{C}_T^d} \int_{-W}^{W} \Pi(s,\bar{X}(\cdot,\xi),\tilde{\gamma},w)\,d(q_0^{\xi,\xi'}-\bar{q}_0^{\xi,\xi'})(w) \,d\bar{\mu}^{\xi'}(\tilde{\gamma})\,d\xi' \right\vert^p \right)^{q/p}d\xi \right)^{1/q} \\ 
			&= \left(\int_{0}^{1} \left( \int_{\mathcal{C}_T^d} \left\vert \int_{0}^{1}\int_{\mathcal{C}_T^d} \int_{-W}^{W} \Pi(s,\gamma,\tilde{\gamma},w)\,d(q_0^{\xi,\xi'}-\bar{q}_0^{\xi,\xi'})(w) \,d\bar{\mu}^{\xi'}(\tilde{\gamma})\,d\xi' \right\vert^p \,d\bar{\mu}^{\xi}(\gamma)  \right)^{q/p}d\xi \right)^{1/q}.
		\end{aligned}
	\]
	Since the probability measure \(\int_{0}^{1}\bar{\mu}^{\xi}\,d\xi \in \mathcal{P}(\mathcal{C}_T^d)\) is tight,
	for every \(\varepsilon > 0\), there exists a compact set \(S_{\varepsilon}\subset \mathcal{C}_T^d\) such that  
	\begin{equation} \label{eq:tightness-mu-xi}
		\int_{0}^{1} \bar{\mu}^{\xi}(\mathcal{C}_T^d \setminus S_{\varepsilon}) \,d\xi < \varepsilon.
	\end{equation}
	We define the compact set \(K_{\varepsilon} = [0,T]\times S_{\varepsilon}\times S_{\varepsilon}\times [-W,W]\), and consider the function \(\Pi_{\varepsilon}:[0,T]\times \mathcal{C}_T^d\times \mathcal{C}_T^d \times [-W,W]\to \mathbb{R}^d\) given by  
	\[
		\Pi_{\varepsilon} = \Pi \mathds{1}_{K_{\varepsilon}},
	\]  
	which is continuous on \(K_{\varepsilon}\). It is straightforward to verify that the following family of functions
	forms a subalgebra of \(C(K_{\varepsilon},\mathbb{R})\) that separates points and contains the constant functions:
	\[
	\mathcal{A}_{\varepsilon} = \left\{ \sum_{k=1}^{N} T_k(t)F_k(\gamma)G_k(\tilde{\gamma})H_k(w) : T_k\in \mathcal{C}_T, F_k,G_k\in C(S_{\varepsilon},\mathbb{R}), H_k\in \text{BL}_1([-W,W]) \right\}. 
	\]
	Thus, by the Stone-Weierstrass theorem, \(\mathcal{A}_{\varepsilon}^d\) is dense in \(C(K_{\varepsilon},\mathbb{R}^d)\). Consequently, there exists \(N(\varepsilon)\in \mathbb{N}\) and functions \((T^j_k)_{1\leq k\leq N(\varepsilon)}\subset \mathcal{C}_T\), \((F^j_k,G^j_k)_{1\leq k\leq N(\varepsilon)}\subset C(S_{\varepsilon},\mathbb{R})\), and \((H^j_k)_{1\leq k \leq N(\varepsilon)} \subset \text{BL}_1([-W,W])\) such that \(\tilde{\Pi}_{\varepsilon}=(\tilde{\Pi}^1_{\varepsilon},\ldots, \tilde{\Pi}^d_{\varepsilon}) \in \mathcal{A}^d_{\varepsilon}\) satisfies the uniform bound
	\begin{equation} \label{eq:stone-Weierstrass-bound}
	\sup_{(t,\gamma,\tilde{\gamma},w) \in K_{\varepsilon}} \vert \Pi(t,\gamma,\tilde{\gamma},w) - \tilde{\Pi}_{\varepsilon}(t,\gamma,\tilde{\gamma},w) \vert \leq \varepsilon,
	\end{equation}
	where 
	\[
		\tilde{\Pi}^j_{\varepsilon}(t,\gamma,\tilde{\gamma},w) = \sum_{k=1}^{N(\varepsilon)} T^j_k(t)F^j_k(\gamma)G^j_k(\tilde{\gamma})H^j_k(w), \quad j=1,\ldots,d.
	\]
	Using Minkowski's inequality, we split the integral for \(I_2\) over the compact sets and their complements
	\[
		\begin{aligned}
			\Vert I_2(s,\cdot)\Vert_{L^q([0,1])} &\leq \left(\int_{0}^{1} \left( \int_{\mathcal{C}_T^d\backslash S_{\varepsilon}} \left\vert \int_{0}^{1}\int_{\mathcal{C}_T^d} \int_{-W}^{W} \Pi(s,\gamma,\tilde{\gamma},w)d(q_0^{\xi,\xi'}-\bar{q}_0^{\xi,\xi'})(w) d\bar{\mu}^{\xi'}(\tilde{\gamma})d\xi' \right\vert^p d\bar{\mu}^{\xi}(\gamma)  \right)^{q/p}d\xi \right)^{1/q} \\ 
			&+\left(\int_{0}^{1} \left( \int_{S_{\varepsilon}} \left\vert \int_{0}^{1}\int_{\mathcal{C}_T^d \backslash S_{\varepsilon}} \int_{-W}^{W} \Pi(s,\gamma,\tilde{\gamma},w)d(q_0^{\xi,\xi'}-\bar{q}_0^{\xi,\xi'})(w) d\bar{\mu}^{\xi'}(\tilde{\gamma})d\xi' \right\vert^p d\bar{\mu}^{\xi}(\gamma)  \right)^{q/p}d\xi \right)^{1/q} \\ 
			&+\left(\int_{0}^{1} \left( \int_{S_{\varepsilon}} \left\vert \int_{0}^{1}\int_{S_{\varepsilon}} \int_{-W}^{W} \Pi(s,\gamma,\tilde{\gamma},w)d(q_0^{\xi,\xi'}-\bar{q}_0^{\xi,\xi'})(w) d\bar{\mu}^{\xi'}(\tilde{\gamma})d\xi' \right\vert^p d\bar{\mu}^{\xi}(\gamma)  \right)^{q/p}d\xi \right)^{1/q} \\
			:=& I_{2,1} + I_{2,2} + I_{2,3}.
		\end{aligned}
	\]
	We can control the first term, \(I_{2,1}\), using the boundedness of \(\Pi\) in the following way:
	\[
		I_{2,1} \leq 2B_{\Pi}(s)(1+W) \left(\int_{0}^{1} \left( \bar{\mu}^{\xi}(\mathcal{C}_T^d\backslash S_{\varepsilon}) \right)^{q/p} d\xi \right)^{1/q}.
	\]
	If \(q\geq p\), since \(\bar{\mu}^{\xi}\) is a probability measure, 
	\(\left( \bar{\mu}^{\xi}(\mathcal{C}_T^d\backslash S_{\varepsilon}) \right)^{q/p}\leq \bar{\mu}^{\xi}(\mathcal{C}_T^d\backslash S_{\varepsilon})\),
	and we can apply the tightness property \eqref{eq:tightness-mu-xi} to obtain
	\(I_2 \leq 2B_{\Pi}(s)(1+W)\varepsilon^{1/q}\). If \(q < p\), since the function \(x \mapsto x^{q/p}\) is concave,
	we can apply Jensen's inequality to obtain \(I_2 \leq 2B_{\Pi}(s)(1+W)\varepsilon^{1/p}\). In either case, for
	\( 0<\varepsilon\leq 1 \), we have the following bound for \(I_{2,1}\):
	\[
		I_{2,1} \leq 2B_{\Pi}(s)(1+W)\varepsilon^{1/r},
	\]
	where \(r = \max(p,q)\).  
	A similar estimate holds for \(I_{2,2}\), the only difference is that the tightness property is used in the
	variable \(\tilde{\gamma}\) instead of \(\gamma\). This ensures that the exponents balance out, and we obtain
	\[
		I_{2,2} \leq 2B_{\Pi}(s)(1+W)\varepsilon.
	\]
	Since \(r = \max(p,q)\), the remaining term \(I_{2,3}\) can be bounded using Minkowski's inequality and the boundedness of \(\Pi\):
	\[
		\begin{aligned}
			I_{2,3} &\leq \left(\int_{0}^{1}  \int_{S_{\varepsilon}} \left\vert \int_{0}^{1}\int_{S_{\varepsilon}} \int_{-W}^{W} \Pi(s,\gamma,\tilde{\gamma},w)\,d(q_0^{\xi,\xi'}-\bar{q}_0^{\xi,\xi'})(w) \,d\bar{\mu}^{\xi'}(\tilde{\gamma})\,d\xi' \right\vert^r \,d\bar{\mu}^{\xi}(\gamma) \,d\xi \right)^{1/r} \\
			&\leq \left[2B_{\Pi}(s)(1+W)\right]^{(r-1)/r} \left(\int_{0}^{1}  \int_{S_{\varepsilon}} \left\vert \int_{0}^{1}\int_{S_{\varepsilon}} \int_{-W}^{W} \Pi(s,\gamma,\tilde{\gamma},w)\,d(q_0^{\xi,\xi'}-\bar{q}_0^{\xi,\xi'})(w) \,d\bar{\mu}^{\xi'}(\tilde{\gamma})\,d\xi' \right\vert \,d\bar{\mu}^{\xi}(\gamma) \,d\xi \right)^{1/r} \\ 
			&:= \left[2B_{\Pi}(s)(1+W)\right]^{(r-1)/r} I_3^{1/r}.
		\end{aligned}
	\]  
	By the triangle inequality, the integral \(I_3\) is bounded by: 
	\[
		\begin{aligned}
			I_3 &\leq \int_{0}^{1}  \int_{S_{\varepsilon}} \left\vert \int_{0}^{1}\int_{S_{\varepsilon}} \int_{-W}^{W} \left[\Pi(s,\gamma,\tilde{\gamma},w)-\tilde{\Pi}_{\varepsilon}(s,\gamma,\tilde{\gamma},w)\right]d(q_0^{\xi,\xi'}-\bar{q}_0^{\xi,\xi'})(w) \,d\bar{\mu}^{\xi'}(\tilde{\gamma}) \,d\xi' \right\vert \,d\bar{\mu}^{\xi}(\gamma) \,d\xi \\ 
			&\quad+ \int_{0}^{1}  \int_{S_{\varepsilon}} \left\vert \int_{0}^{1}\int_{S_{\varepsilon}} \int_{-W}^{W} \tilde{\Pi}_{\varepsilon}(s,\gamma,\tilde{\gamma},w)\,d(q_0^{\xi,\xi'}-\bar{q}_0^{\xi,\xi'})(w) \,d\bar{\mu}^{\xi'}(\tilde{\gamma}) \,d\xi' \right\vert \,d\bar{\mu}^{\xi}(\gamma) \,d\xi \\ 
			&:= I_{3,1} + I_{3,2}.
		\end{aligned}
	\]
	The first term, \(I_{3,1}\), is controlled directly by the uniform approximation \eqref{eq:stone-Weierstrass-bound}:
	\[
		I_{3,1} \leq 2\varepsilon.
	\]
	For the second term, \(I_{3,2}\), the separation of variables in the Stone-Weierstrass approximation allows us to
	recover the bounded-Lipschitz cut distance between the probability-graphons through the adjacency operator in
	\eqref{eq:cut-distance-prob-graphons-operator-norm}:
	\[
		\begin{aligned}
			I_{3,2} &\leq \sum_{k=1}^{N(\varepsilon)} \Vert F_k \Vert_{\ast,S_{\varepsilon}} \Vert T_k \Vert_{\ast,s} \int_{0}^{1}  \left\vert \int_{0}^{1} \int_{S_{\varepsilon}} G_k(\tilde{\gamma})\,d\bar{\mu}^{\xi}(\tilde{\gamma})  \int_{-W}^{W} H_k(w)\,d(q_0^{\xi,\xi'}-\bar{q}_0^{\xi,\xi'})(w) \,d\xi' \right\vert  d\xi \\ 
			&\leq \underbrace{4\sum_{k=1}^{N(\varepsilon)}\Vert G_k \Vert_{\ast,S_{\varepsilon}} \Vert T_k \Vert_{\ast,s} \Vert F_k \Vert_{\ast,S_{\varepsilon}}}_{:=\kappa(\varepsilon,s)} d_{\text{BL},\square}(q_0,\bar{q}_0).
		\end{aligned}
	\]
	Combining these bounds, we deduce that for every \(0 < \varepsilon \leq 1\), the term \(I_2\) satisfies:
	\[
		\begin{aligned}
			\Vert I_{2}(s,\cdot)\Vert_{L^q([0,1])} &\leq 4B_{\Pi}(s)(1+W)\varepsilon^{1/r} + \left[2B_{\Pi}(s)(1+W)\right]^{(r-1)/r}\left(2\varepsilon+4\kappa(\varepsilon,s)d_{\text{BL},\square}(q_0,\bar{q}_0)\right)^{1/r} \\
			&\leq 4B_{\Pi}(s)(1+W)\varepsilon^{1/r} + \left[2B_{\Pi}(s)(1+W)+1\right]\left(2\varepsilon^{1/r} + 4\kappa(\varepsilon,s)^{1/r}d_{\text{BL},\square}(q_0,\bar{q}_0)^{1/r}\right) \\
			&\leq \left[8B_{\Pi}(s)(1+W) + 2\right]\varepsilon^{1/r} + 4\left[2B_{\Pi}(s)(1+W)+1\right]\kappa(\varepsilon,s)^{1/r}d_{\text{BL},\square}(q_0,\bar{q}_0)^{1/r} \\
			&\leq 4\left[2B_{\Pi}(s)(1+W) + 1\right]\left(\varepsilon^{1/r} + \kappa(\varepsilon,s)^{1/r}d_{\text{BL},\square}(q_0,\bar{q}_0)^{1/r}\right).
		\end{aligned}
	\]
	Finally, substituting this bound back into \eqref{eq:estimate-previous-I2} establishes the desired stability
	estimate \eqref{eq:stability-estimate}.
\end{proof}

\section{Proof of main Theorem \ref{theo:main}} \label{sec:main-theorem}
In this section, we prove our main result, Theorem \ref{theo:main}, which we now restate in its precise formulation.

\begin{theorem} \label{theo:main-complete}
	Assume that \(K\) and \(\Gamma\) satisfy Assumptions \ref{assump:interaction-kernel} and \ref{assump:plasticity-function}.
	Consider the solution \((X_i^{N},w^N_{ij})_{1\leq i,j\leq N}\) to the system of SDEs \eqref{eq:multi-agent-discrete},
	starting from independent initial states and weights \((X^N_{i,0},w^N_{ij,0})_{1\leq i,j\leq N}\) satisfying
	Assumptions \ref{assump:weights} and \ref{assump:initial-data}. Define the associated (random) extended empirical measure \(\mu^N_{[0,t]} \in \mathcal{P}_{2,2,\nu}(\mathcal{C}_t^d\times [0,1])\) and empirical probability-graphon \(q^N_{t} \in \mathcal{P}_{\nu \otimes \nu}(\mathbb{R}\times [0,1]^2)\) as follows:  
	\[
		\begin{aligned}
			\mu^{N,\xi}_{[0,t]} &= \sum_{i=1}^{N} \mathds{1}_{I_i^N}(\xi) \delta_{X_{i,[0,t]}^N}, \quad &&\xi \in [0,1], \quad t \in [0,T], \\ 
			q^{N,\xi,\xi'}_{t} &= \sum_{i,j=1}^{N} \mathds{1}_{I_i^N\times I_j^N}(\xi,\xi') \delta_{Nw_{ij}^N(t)}, \quad &&(\xi,\xi') \in [0,1]^2, \quad t \in [0,T].
		\end{aligned}
	\]
    
	Then, there exist a subsequence \(N_k\to \infty\), a sequence of measure-preserving maps \(\varphi_k \in S_{[0,1]}\),
	a limiting initial probability-graphon \(q_0\in \mathcal{P}_{\nu\otimes\nu}([-W,W]\times [0,1]^2)\), and a limiting
	initial fibered probability measure \(\mu_0\in \mathcal{P}_{2,2,\nu}(\mathbb{R}^d\times [0,1])\) such that the following
	holds. 
	
	Set any initial states $(X_0(\xi))_{\xi\in [0,1]}$ distributed according to \((\mu_0^\xi)_{\xi\in [0,1]}\) and any family of Wiener processes $(W(\cdot,\xi))_{\xi\in [0,1]}$ and, by virtue of Proposition \ref{prop:well-posedness-multi-agent-macroscopic-solved-weights}, let \((X(\cdot,\xi),q_{\gamma}^{\xi,\xi'})_{\xi,\xi'\in[0,1]}\) be the unique strong solution to the continuum system of McKean-Vlasov SDEs
	\eqref{eq:multi-agent-macroscopic-SDE} whose initial datum $X_0$ and its initial probability-graphon is given by \(q_0\). Denote by \(\mu_{[0,T]}\in \mathcal{P}_{2,2,\nu}(\mathcal{C}_T^d \times [0,1])\) the fibered pathwise law of the solution \(X\), and define the probability-graphon
	\(q_t\in \mathcal{P}_{\nu \otimes \nu}(\mathbb{R} \times [0,1]^2)\) as: 
	\begin{equation} \label{eq:limiting-probability-graphon}
		q^{\xi,\xi'}_t= \int_{\mathcal{C}_t^d} \mathbb{E}[q^{\xi,\xi'}_{\gamma}(t)] \, d\mu^{\xi'}_{[0,t]}(\gamma), \quad (\xi,\xi') \in [0,1]^2, \quad t \in [0,T].
	\end{equation}
	Then, as \(k \to \infty\), we have the following convergence for the extended empirical measure:
	\[
		\begin{aligned}
			d_{W_2,2}(\mathbb{E}\mu^{N_k,\varphi_{k}}_{[0,T]},\mu_{[0,T]}) &\to 0,  \\ 
			\mathbb{E} W_2(\mu^{N_k,\varphi_{k}}_{[0,T]},\mu_{[0,T]}) &\to 0,
		\end{aligned}
	\]
	and for the empirical probability-graphon:
	\[
		\begin{aligned}
			\sup_{t\in [0,T]}d_{\text{BL},\square}(\mathbb{E}q^{N_k,\varphi_{k}}_t,q_t) &\to 0, \\ 
			\sup_{t\in [0,T]}\mathbb{E}d_{\text{BL}}(q^{N_k,\varphi_{k}}_t,q_t) &\to 0.
		\end{aligned}
	\]
\end{theorem}

\begin{remark}[On the definition of the limiting probability-graphon] \label{remark:reformulation-prob-graphon}
	The above map \(q:[0,T]\times [0,1]^2\times \mathcal{C}_T^d\times \Omega \to \mathcal{P}(\mathbb{R})\) is the component of macroscopic equation \eqref{eq:multi-agent-macroscopic-SDE} governing the weight distribution. As noted in Remark
	\ref{remark:reformulation-states-weights-probability-graphon}, this can be expressed in terms of the flow map \(\Phi\) as follows
	\[
		q^{\xi,\xi'}_{\gamma}(t)= \Phi_t(X(\cdot,\xi)_{[0,t]},\gamma,\cdot)_{\#} q_0^{\xi,\xi'}.
	\]
	Taking the expectation with respect to the underlying randomness (that is, $X(\cdot,\xi)_{[0,t]}$) yields the map \(\mathbb{E}q:[0,T]\times [0,1]^2\times \mathcal{C}_T^d \to \mathcal{P}(\mathbb{R})\), given explicitly by
	\[
		\mathbb{E}q^{\xi,\xi'}_{\gamma}(t)= \Phi_t(\cdot,\gamma,\cdot)_{\#}(\mu^{\xi}_{[0,t]}\otimes q_0^{\xi,\xi'}).
	\]
	Finally, integrating with respect to the pathwise law \(\mu^{\xi'}_{[0,t]}\) of the solution, we recover the
	deterministic time-dependent probability-graphon \(q_t: [0,1]^2 \to \mathcal{P}(\mathbb{R})\) defined as
	\begin{equation} \label{eq:limiting-probability-graphon-flow-map-formulation}
		q^{\xi,\xi'}_t= \int_{\mathcal{C}_t^d} \mathbb{E}[q^{\xi,\xi'}_{\gamma}(t)] \, d\mu^{\xi'}_{[0,t]}(\gamma) = \Phi_t(\cdot,\cdot,\cdot)_{\#}(\mu^{\xi}_{[0,t]}\otimes \mu^{\xi'}_{[0,t]}\otimes q_0^{\xi,\xi'}), \quad t \in [0,T].
	\end{equation}
\end{remark}

Given the flow map representation \eqref{eq:limiting-probability-graphon-flow-map-formulation},
we can deduce the convergence of the empirical probability-graphon from the convergence
of the extended empirical measure using the following tensorization bound.

\begin{lemma} \label{lemma:tensorization-cut-distance}
	For any \(\mu,\bar{\mu} \in \mathcal{P}_{\nu}(\mathcal{X}\times [0,1])\) 
	and \(q,\bar{q} \in \mathcal{P}_{\nu \otimes \nu}(\mathcal{Y}\times [0,1]^2)\), define the probability-graphon 
	\(\mu\otimes \mu \otimes q \in \mathcal{P}_{\nu\otimes \nu}((\mathcal{X}\times\mathcal{X}\times\mathcal{Y})\times [0,1]^2)\) as 
	\[
		(\mu\otimes \mu \otimes q)^{\xi,\xi'}= \mu^{\xi} \otimes \mu^{\xi'} \otimes q^{\xi,\xi'}, \quad (\xi,\xi') \in [0,1]^2,
	\]    
	and similarly for \(\bar{\mu}\otimes \bar{\mu} \otimes \bar{q}\). Then, for every \(\varepsilon >0\) there exists
	\(\tilde{\kappa}(\varepsilon)\in \mathbb{N}\) such that 
	\[
		d_{\text{BL},\square}(\mu\otimes \mu \otimes q,\bar{\mu}\otimes \bar{\mu} \otimes \bar{q}) \leq 2 d_{\text{BL},1}(\mu,\bar{\mu}) +  \tilde{\kappa}(\varepsilon)^2 d_{\text{BL},\square}(q,\bar{q}) + \varepsilon.
	\]  
\end{lemma}
\begin{proof}
	By the triangle inequality, we can decompose the cut distance as follows:
	\[
		\begin{aligned}
			d_{\text{BL},\square}(\mu\otimes \mu \otimes q,\bar{\mu}\otimes \bar{\mu} \otimes \bar{q}) &\leq d_{\text{BL},\square}(\mu\otimes \mu \otimes q,\bar{\mu}\otimes \mu \otimes q) \\ 
			&+ d_{\text{BL},\square}(\bar{\mu}\otimes \mu \otimes q,\bar{\mu}\otimes \bar{\mu} \otimes q) \\ 
			&+ d_{\text{BL},\square}(\bar{\mu}\otimes \bar{\mu} \otimes q,\bar{\mu}\otimes \bar{\mu} \otimes \bar{q}) \\ 
			&:= I_1 + I_2 + I_3.
		\end{aligned}
	\]
	The term \(I_1\) can be estimated by transitioning to the \(d_{\text{BL},1}\) distance and expanding its definition: 
	\[
		\begin{aligned}
			I_1 &\leq d_{\text{BL},1}(\mu\otimes \mu \otimes q,\bar{\mu}\otimes \mu \otimes q) \\ 
		    	&=\int_{0}^{1} \int_{0}^{1} \sup_{\Phi \in \text{BL}_1} \left\vert \int_{\mathcal{X}\times \mathcal{X}\times \mathcal{Y}} \Phi(x,x',y) \,d(\mu^{\xi}-\bar{\mu}^{\xi})(x) \,d\mu^{\xi'}(x') \,dq^{\xi,\xi'}(y) \right\vert d\xi \,d\xi'.
		\end{aligned}
	\]
	To evaluate the inner integrals over \(x'\) and \(y\), we define the test function
	\[
		\Psi(x) := \int_{\mathcal{X}\times\mathcal{Y}} \Phi(x,x',y) \,d\mu^{\xi'}(x') \,dq^{\xi,\xi'}(y).
	\]
	It is straightforward to verify that since \(\Phi \in \text{BL}_1(\mathcal{X}\times \mathcal{X}\times \mathcal{Y})\),
	then \(\Psi \in \text{BL}_1(\mathcal{X})\), which allows us to bound the inner integral by
	\[
		\begin{aligned}
			I_1 &\leq \int_{0}^{1} \int_{0}^{1} \sup_{\Psi \in \text{BL}_1(\mathcal{X})} \left\vert \int_{\mathcal{X}} \Psi(x) \,d(\mu^{\xi}-\bar{\mu}^{\xi})(x) \right\vert d\xi \,d\xi' \\
		    	&= \int_{0}^{1} d_{\text{BL}}(\mu^{\xi}, \bar{\mu}^{\xi}) \,d\xi = d_{\text{BL},1}(\mu,\bar{\mu}).
		\end{aligned}
	\]
	An analogous estimate holds for \(I_2\) regarding the second coordinate, yielding:
	\[
		I_2 \leq d_{\text{BL},1}(\mu,\bar{\mu}).
	\]
	For the last term \(I_3\), fix \(\varepsilon >0\). By the weak regularity lemma for fibered probability measures 
	\ref{lemma:regularity-fibered-probability-measures}, there exists \(\tilde{\kappa}(\varepsilon)\in \mathbb{N}\), a partition
	of the interval \([0,1]\) into \(\tilde{\kappa}(\varepsilon)\) measurable sets \(\{A_1,\ldots,A_{\tilde{\kappa}(\varepsilon)}\}\),
	and probability measures \(\mu_i \in \mathcal{P}(\mathcal{X})\), such that the fibered probability measure
	\(\tilde{\mu}\) defined as
	\[
		\tilde{\mu}^{\xi}=\sum_{i=1}^{\tilde{\kappa}(\varepsilon)} \mathds{1}_{A_i}(\xi) \mu_i, \quad \xi \in [0,1],
	\]
	approximates \(\bar{\mu}\) up to an error of \(\varepsilon/4\):
	\[
		d_{\text{BL},1}(\bar{\mu},\tilde{\mu}) \leq \frac{\varepsilon}{4}.
	\]
	Using this approximation and the triangle inequality, \(I_3\) can be expanded and bounded as follows:
	\[ 
		\begin{aligned}
			I_3 &= d_{\text{BL},\square}(\bar{\mu}\otimes \bar{\mu} \otimes q,\bar{\mu}\otimes \bar{\mu} \otimes \bar{q}) \\
			&\leq d_{\text{BL},\square}(\bar{\mu}\otimes \bar{\mu} \otimes q,\tilde{\mu}\otimes \tilde{\mu} \otimes q)  
			+ d_{\text{BL},\square}(\tilde{\mu}\otimes \tilde{\mu} \otimes q,\tilde{\mu}\otimes \tilde{\mu} \otimes \bar{q})+ d_{\text{BL},\square}(\tilde{\mu}\otimes \tilde{\mu} \otimes \bar{q},\bar{\mu}\otimes \bar{\mu} \otimes \bar{q}) \\ 
			&\leq 2d_{\text{BL},1}(\bar{\mu},\tilde{\mu}) + d_{\text{BL},\square}(\tilde{\mu}\otimes \tilde{\mu} \otimes q,\tilde{\mu}\otimes \tilde{\mu} \otimes \bar{q}) + 2d_{\text{BL},1}(\bar{\mu},\tilde{\mu}) \\
			&\leq \varepsilon + d_{\text{BL},\square}(\tilde{\mu}\otimes \tilde{\mu} \otimes q,\tilde{\mu}\otimes \tilde{\mu} \otimes \bar{q}).
		\end{aligned}
	\]
	To estimate this remaining term, take arbitrary measurable sets \(S,T \subset [0,1]\) and a test function \(\Phi \in \text{BL}_1(\mathcal{X}\times \mathcal{X}\times\mathcal{Y})\). Substituting the explicit form of \(\tilde{\mu}\), we have: 
	\[
		\begin{aligned}
			\int_{S} \int_{T} &\int_{\mathcal{X}\times \mathcal{X}\times \mathcal{Y}} \Phi(x,x',y) \,d\tilde{\mu}^{\xi}(x) \,d\tilde{\mu}^{\xi'}(x') \,d(q^{\xi,\xi'}-\bar{q}^{\xi,\xi'})(y) \,d\xi \,d\xi'  \\ 
			&= \sum_{i,j=1}^{\tilde{\kappa}(\varepsilon)} \int_{S\cap A_i} \int_{T\cap A_j} \int_{\mathcal{Y}} \left( \int_{\mathcal{X}\times \mathcal{X}} \Phi(x,x',y) \,d\mu_i(x) \,d\mu_j(x') \right) d(q^{\xi,\xi'}-\bar{q}^{\xi,\xi'})(y) \,d\xi \,d\xi' \\ 
			&\leq \tilde{\kappa}(\varepsilon)^2 d_{\text{BL},\square}(q,\bar{q}),
		\end{aligned}
	\]
	where we have used the fact that the inner integral \(g_{ij}(y):=\int_{\mathcal{X}\times \mathcal{X}} \Phi(x,x',y) \,d\mu_i(x) \,d\mu_j(x')\)
	defines a function \(g_{ij} \in \text{BL}_1(\mathcal{Y})\), allowing us to bound each of the
	\(\tilde{\kappa}(\varepsilon)^2\) terms individually by the cut metric over the subsets \(S \cap A_i\) and \(T \cap A_j\).
	Taking the supremum over all test functions \(\Phi\) and sets \(S,T\), we obtain:
	\[
		d_{\text{BL},\square}(\tilde{\mu}\otimes \tilde{\mu} \otimes q,\tilde{\mu}\otimes \tilde{\mu} \otimes \bar{q}) \leq \tilde{\kappa}(\varepsilon)^2 d_{\text{BL},\square}(q,\bar{q}).
	\]
	Combining these estimates, we conclude that 
	\[
		I_3 \leq \varepsilon + \tilde{\kappa}(\varepsilon)^2 d_{\text{BL},\square}(q,\bar{q}),
	\]
	which, together with the bounds on \(I_1\) and \(I_2\), leads to the desired result.
\end{proof}

\begin{proof}[Proof of Theorem \ref{theo:main-complete}]
	Consider the solution \((\bar{X}_i^N)_{1\leq i\leq N}\) of the intermediate system
	\eqref{eq:multi-agent-discrete-independent-projection-solved-weights} with the same initial data
	\((X_{i,0}^N)_{1\leq i\leq N}\) as the solution \((X_i^N)_{1\leq i\leq N}\) of the original system \eqref{eq:multi-agent-discrete}, and define
	\begin{equation} \label{eq:weights-intermediate-system}
		N\bar{w}_{ij}^N(t) := \Phi_t(\bar{X}^N_{i,[0,t]},\bar{X}^N_{j,[0,t]},Nw_{ij,0}^N).
	\end{equation}
	Let \(\bar{\mu}^N\) and \(\bar{q}^N\) be the extended empirical measure and empirical probability-graphon associated
	with such solution to the intermediate system:
	\[
		\begin{aligned}
			\bar{\mu}^{N,\xi}_{[0,t]} &= \sum_{i=1}^{N} \mathds{1}_{I_i^N}(\xi) \delta_{\bar{X}_{i,[0,t]}^N}, \quad &&\xi \in [0,1], \quad t \in [0,T], \\ 
			\bar{q}^{N,\xi,\xi'}_t &= \sum_{i,j=1}^{N} \mathds{1}_{I_i^N\times I_j^N}(\xi,\xi') \delta_{N\bar{w}_{ij}^N(t)}, \quad &&(\xi,\xi') \in [0,1]^2, \quad t \in [0,T].
		\end{aligned}
	\]

	\(\diamond\) \textsc{Step 1}: Definition of the limiting \(q_0\) and \(\mu_0\).\\
	The sequence of initial fibered laws 
	\((\mathbb{E}\mu_0^N)_{N\in\mathbb{N}}\subset \mathcal{P}_{2,2,\nu}(\mathbb{R}^d\times [0,1])\) takes the following form:
	\[
		\mathbb{E}\mu_0^{N,\xi}=\sum_{i=1}^{N} \mathds{1}_{I_i^N}(\xi) \mathrm{Law}(X_{i,0}^N), \quad \xi \in [0,1].
	\]
	Assumption \ref{assump:initial-data} on the initial data ensures that 
	\[
		\sup_{N\in \mathbb{N}} \int_{0}^{1} \int_{\mathbb{R}^d} \vert x \vert^{2+\delta} d(\mathbb{E}\mu_0^{N,\xi})(x) d\xi < \infty,
	\]
	thus, by Proposition \ref{prop:higher-order-moments-implies-compactness}, the sequence \((\mathbb{E}\mu_0^N)_{N\in \mathbb{N}}\) is relatively compact in \((\widetilde{\mathcal{P}}_{2,2,\nu}(\mathbb{R}^d\times [0,1]),\delta_{W_2,2})\). 	
	On the other hand, the sequence of initial probability-graphons \((\mathbb{E}q_0^N)_{N\in\mathbb{N}}\subset \mathcal{P}_{\nu\otimes\nu}(\mathbb{R}\times [0,1]^2)\) takes the following form:
	\[
		\mathbb{E}q_0^{N,\xi,\xi'}=\sum_{i,j=1}^{N} \mathds{1}_{I_i^N\times I_j^N}(\xi,\xi') \mathrm{Law}(Nw_{ij,0}^N), \quad (\xi,\xi') \in [0,1]^2.
	\]
	Assumption \ref{assump:weights} on the weights ensures that the sequence takes values in 
	\(\widetilde{\mathcal{P}}_{\nu\otimes\nu}([-W,W]\times [0,1]^2)\), which by Theorem \ref{thm:compactness-probability-graphons}
	we know is a compact space with respect to the unlabeled cut distance \(\delta_{\text{BL},\square}\). Altogether,
	by Theorem \ref{thm:compactness-node-edge-probability-graphons},
	we can ensure the existence of a subsequence \(N_k\to \infty\), measure-preserving maps \(\varphi_k \in S_{[0,1]}\),
	a fibered probability measure \(\mu_0\in \mathcal{P}_{2,2,\nu}(\mathbb{R}^d\times [0,1])\) and a probability-graphon
	\(q_0\in \mathcal{P}_{\nu\otimes\nu}([-W,W]\times [0,1]^2)\) such that
	\begin{equation} \label{eq:convergence-initial-fib-measure-and-prob-graphon}
		\begin{aligned}
			d_{W_2,2}(\mathbb{E}\mu_0^{N_k,\varphi_k},\mu_0) &\to 0, \\ 
			d_{\text{BL},\square}(\mathbb{E}q_0^{N_k,\varphi_k},q_0) &\to 0.
		\end{aligned}
	\end{equation}
	By virtue of Proposition \ref{prop:well-posedness-multi-agent-macroscopic-solved-weights}, we can define \(\mu \in \mathcal{P}_{2,2,\nu}(\mathcal{C}_T^d\times [0,1])\) to be the unique fibered pathwise law
	of the solution of \eqref{eq:multi-agent-macroscopic-SDE} with initial data distributed as \(\mu_0\) and initial
	probability-graphon \(q_0\). We also define \(q_t\in \mathcal{P}_{\nu \otimes \nu}(\mathbb{R}\times [0,1]^2)\) as
	in \eqref{eq:limiting-probability-graphon}, which can be equivalently expressed, as we have seen in Remark \ref{remark:reformulation-prob-graphon}, in terms of the flow map \(\Phi\) as
	\[
		q_t^{\xi,\xi'}= \Phi_t(\cdot,\cdot,\cdot)_{\#} (\mu_{[0,t]}^{\xi} \otimes \mu_{[0,t]}^{\xi'} \otimes q_0^{\xi,\xi'}), \quad (\xi,\xi') \in [0,1]^2, \quad t \in [0,T].
	\]

	\medskip

	\(\diamond\) \textsc{Step 2}: Comparing \(\mathbb{E}\mu^N\) with \(\mathbb{E}\bar{\mu}^N\), and \(\mathbb{E}q^N\) with \(\mathbb{E}\bar{q}^N\). \\
	By the propagation of independence stated in Lemma \ref{lemma:propagation-of-independence}, we have that 
	\begin{equation} \label{eq:estimate-propagation-independence-empirical-measure}
		d_{W_2,2}(\mathbb{E}\mu^N_{[0,T]},\mathbb{E}\bar{\mu}^N_{[0,T]}) \leq \left(\frac{1}{N} \displaystyle\sum_{i=1}^{N} \left( \mathbb{E} \Vert X_i^N -\bar{X}_i^N \Vert^2_{\ast,T}\right) \right)^{1/2} \leq \frac{C(T)}{\sqrt{N}}. 
	\end{equation}
	Since the weights \(Nw_{ij}^N\) and \(N\bar{w}_{ij}^N\) are given by 
	\[
		Nw_{ij}^N(t)=\Phi_t(X_{i,[0,t]}^N,X_{j,[0,t]}^N,Nw^N_{ij,0}), \quad N\bar{w}_{ij}^N(t)=\Phi_t(\bar{X}_{i,[0,t]}^N,\bar{X}_{j,[0,t]}^N,N w^N_{ij,0}),
	\]
	and \(\Phi\) is Lipschitz continuous with constant \(L_{\Phi}(t) = e^{L_{\Gamma}t}(1+tL_{\Gamma})\), we have that:
	\begin{equation} \label{eq:estimate-propagation-independence-probability-graphon}
		\begin{aligned}
			d_{\text{BL},\square}(\mathbb{E}q^N_t,\mathbb{E}\bar{q}^N_t) &\leq \left(\frac{1}{N^2} \sum_{i,j=1}^{N} \mathbb{E} \vert Nw_{ij}^N(t) - N\bar{w}_{ij}^N(t) \vert^2 \right)^{1/2} \\
			&\leq \left(\frac{1}{N^2} \sum_{i,j=1}^{N} 2L_{\Phi}(t)^2 \left( \mathbb{E}\Vert X_i^N - \bar{X}_i^N \Vert_{\ast,t}^2 + \mathbb{E}\Vert X_j^N - \bar{X}_j^N \Vert_{\ast,t}^2 \right) \right)^{1/2} \\
			&= 2L_{\Phi}(t) \left(\frac{1}{N} \sum_{i=1}^{N} \mathbb{E}\Vert X_i^N - \bar{X}_i^N \Vert_{\ast,t}^2 \right)^{1/2} \\
			&\leq \frac{2L_{\Phi}(T)C(T)}{\sqrt{N}}.
		\end{aligned}
	\end{equation}

	\medskip

	\(\diamond\) \textsc{Step 3}: Comparing \(\mathbb{E}\bar{\mu}^{N}\) with \(\mu\), and \(\mathbb{E}\bar{q}^{N}\) with \(q\).\\
	Since \(\mathbb{E}\bar{\mu}^N\) is the fibered pathwise law of the solution of \eqref{eq:multi-agent-discrete-independent-projection-preparatory} with initial data distributed as \(\mathbb{E}\mu_0^N\), and initial probability-graphon \(\mathbb{E}q_0^N\), we have that \(\mathbb{E}\bar{\mu}^{N_k,\varphi_k}\) is the unique fibered pathwise law of \eqref{eq:multi-agent-discrete-independent-projection-preparatory} with initial data distributed as \(\mathbb{E}\mu_0^{N_k,\varphi_k}\) and initial probability-graphon \(\mathbb{E}q_0^{N_k,\varphi_k}\). Thus, by the stability estimate stated in Proposition \ref{prop:stability-estimate-multi-agent-macroscopic-solved-weights}, we have that for every \(0<\varepsilon \leq 1\), 
	\begin{equation} \label{eq:stability-estimate-empirical-measure}
		d_{W_2,2}(\mathbb{E}\bar{\mu}^{N_k,\varphi_k}_{[0,T]},\mu_{[0,T]}) \leq \tilde{C}(T)\left(d_{W_2,2}(\mathbb{E}\mu_0^{N_k,\varphi_k},\mu_0) +\kappa(\varepsilon,T)^{1/2}d_{\text{BL},\square}(\mathbb{E}q_0^{N_k,\varphi_k},q_0)^{1/2}+ \varepsilon^{1/2} \right).
	\end{equation}
	For the probability-graphons, since \(\mathbb{E}\bar{q}_t^{N_k,\varphi_k}\) and \(q_t\) are given by
	\[
		\begin{aligned}
			q_t^{\xi,\xi'} &= \Phi_t(\cdot,\cdot,\cdot)_{\#} (\mu_{[0,t]}^{\xi} \otimes \mu_{[0,t]}^{\xi'} \otimes q_0^{\xi,\xi'}),  \\ 
			\mathbb{E}\bar{q}_t^{N_k,\varphi_k,\xi,\xi'} &= \Phi_t(\cdot,\cdot,\cdot)_{\#} (\mathbb{E}\bar{\mu}^{N_k,\varphi_k,\xi}_{[0,t]} \otimes \mathbb{E}\bar{\mu}^{N_k,\varphi_k,\xi'}_{[0,t]} \otimes \mathbb{E}q_0^{N_k,\varphi_k,\xi,\xi'}),
		\end{aligned}
	\] 
	and \(\Phi\) is Lipschitz continuous and bounded, we can estimate the distance between \(\mathbb{E}\bar{q}_t^{N_k,\varphi_k}\) and \(q_t\) by
	\[
		d_{\text{BL},\square}(\mathbb{E}\bar{q}^{N_k,\varphi_k}_t,q_t)\leq \text{BL}(\Phi_t) d_{\text{BL},\square}(\mu_{[0,t]}\otimes \mu_{[0,t]} \otimes q_0,\mathbb{E}\bar{\mu}^{N_k,\varphi_k}_{[0,t]} \otimes \mathbb{E}\bar{\mu}^{N_k,\varphi_k}_{[0,t]} \otimes \mathbb{E}q_0^{N_k,\varphi_k}),
	\]
	where \(\text{BL}(\Phi_t):=\max\{B_{\Phi}(t)(1+W),L_{\Phi}(t)\}\) is the bounded-Lipschitz constant of the flow
	map \(\Phi\) at time \(t\). By Lemma \ref{lemma:tensorization-cut-distance}, for every \(0<\varepsilon \leq 1\)
	there exists \(\tilde{\kappa}(\varepsilon,t)\in \mathbb{N}\) such that:
	\begin{equation} \label{eq:stability-estimate-empirical-probability-graphon}
		d_{\text{BL},\square}(\mathbb{E}\bar{q}^{N_k,\varphi_k}_t,q_t) \leq \text{BL}(\Phi_t)\left( 2 d_{\text{BL},1}(\mathbb{E}\bar{\mu}^{N_k,\varphi_k}_{[0,t]},\mu_{[0,t]}) +  \tilde{\kappa}(\varepsilon,t)^2 d_{\text{BL},\square}(\mathbb{E}q_0^{N_k,\varphi_k},q_0) + \varepsilon \right).
	\end{equation}

	\medskip

	\(\diamond\) \textsc{Step 4}: Comparing \(\mathbb{E}\mu^{N}\) with \(\mu\) and \(\mathbb{E}q^{N}\) with \(q\). \\
	By the triangle inequality, using the estimates \eqref{eq:estimate-propagation-independence-empirical-measure} and \eqref{eq:stability-estimate-empirical-measure} obtained in \textsc{Step} 2 and \textsc{Step} 3, we can bound the distance between \(\mathbb{E}\mu^{N_k,\varphi_k}\) and \(\mu\) for every \(0<\varepsilon \leq 1\) as follows:
	\[
		\begin{aligned}
			d_{W_2,2}(\mathbb{E}\mu^{N_k,\varphi_k}_{[0,T]},\mu_{[0,T]}) &\leq d_{W_2,2}(\mathbb{E}\mu^{N_k,\varphi_k}_{[0,T]},\mathbb{E}\bar{\mu}^{N_k,\varphi_k}_{[0,T]}) + d_{W_2,2}(\mathbb{E}\bar{\mu}^{N_k,\varphi_k}_{[0,T]},\mu_{[0,T]}) \\ 
			&\leq \frac{C(T)}{\sqrt{N_k}}+ \tilde{C}(T)\left(d_{W_2,2}(\mathbb{E}\mu_0^{N_k,\varphi_k},\mu_0) +\kappa(\varepsilon,T)^{1/2}d_{\text{BL},\square}(\mathbb{E}q_0^{N_k,\varphi_k},q_0)^{1/2}+ \varepsilon^{1/2} \right).
		\end{aligned}
	\]
	If we take the limit as \(k \to \infty\) in the previous estimate and use the convergence of \(\mathbb{E}\mu_0^{N_k,\varphi_k}\) to \(\mu_0\) and \(\mathbb{E}q_0^{N_k,\varphi_k}\) to \(q_0\) from \eqref{eq:convergence-initial-fib-measure-and-prob-graphon}, we get that
	\[
		\limsup_{k\to \infty} d_{W_2,2}(\mathbb{E}\mu^{N_k,\varphi_k}_{[0,T]},\mu_{[0,T]}) \leq \tilde{C}(T)\varepsilon^{1/2}.
	\]
	Since this holds for every \(0<\varepsilon \leq 1\), we conclude that
	\begin{equation} \label{eq:mean-convergence-empirical-measure}
		\lim_{k\to \infty} d_{W_2,2}(\mathbb{E}\mu^{N_k,\varphi_k}_{[0,T]},\mu_{[0,T]}) = 0.
	\end{equation}
	Analogous reasoning holds for the probability-graphons, where we use estimates
	\eqref{eq:estimate-propagation-independence-probability-graphon}
	and \eqref{eq:stability-estimate-empirical-probability-graphon} to obtain:
	\[
		\begin{aligned}
			&\sup_{t\in [0,T]}d_{\text{BL},\square}(\mathbb{E}q^{N_k,\varphi_k}_t,q_t)\\
            &\qquad \leq \sup_{t\in [0,T]} d_{\text{BL},\square}(\mathbb{E}q^{N_k,\varphi_k}_t,\mathbb{E}\bar{q}_t^{N_k,\varphi_k}) + \sup_{t\in [0,T]} d_{\text{BL},\square}(\mathbb{E}\bar{q}_t^{N_k,\varphi_k},q_t)  \\
			&\qquad\leq \frac{2L_{\Phi}(T)C(T)}{\sqrt{N_k}} + \text{BL}(\Phi_T)\left( 2 d_{\text{BL},1}(\mathbb{E}\bar{\mu}^{N_k,\varphi_k}_{[0,T]},\mu_{[0,T]}) +  \tilde{\kappa}(\varepsilon,T)^2 d_{\text{BL},\square}(\mathbb{E}q_0^{N_k,\varphi_k},q_0) + \varepsilon \right),
		\end{aligned}
	\]
	where in the second inequality we have used the fact that the constant \(\tilde{\kappa}(\varepsilon,t)\)
	is non-decreasing in \(t\). To see this, recall from the proof of Lemma \ref{lemma:tensorization-cut-distance}
	that \(\tilde{\kappa}(\varepsilon,t)\) is defined as the minimum number of sets in a partition of \([0,1]\)
	required to approximate the fibered probability measure \(\mu_{[0,t]}\) up to an error of \(\varepsilon/4\)
	in the \(d_{\text{BL},1}\) distance. Because the pathwise supremum norm is non-decreasing in time,
	the bounded-Lipschitz distance between any two measures on the path space is also non-decreasing in time.
	Therefore, if a piecewise-constant fibered measure \(\nu_{[0,T]}\) approximates \(\mu_{[0,T]}\) with error \(\varepsilon/4\) then 
	\[
	    d_{\text{BL},1}(\mu_{[0,t]},\nu_{[0,t]}) \leq d_{\text{BL},1}(\mu_{[0,T]},\nu_{[0,T]}) < \frac{\varepsilon}{4},
	\]
	and therefore we can use the exact same partition to approximate \(\mu_{[0,t]}\) for any \(t\in [0,T]\). This guarantees
	that \(\tilde{\kappa}(\varepsilon,t)\leq \tilde{\kappa}(\varepsilon,T)\).
	
	Using \eqref{eq:convergence-initial-fib-measure-and-prob-graphon} and \eqref{eq:mean-convergence-empirical-measure} leads to 
	\[
		\limsup_{k\to \infty} \sup_{t\in [0,T]} d_{\text{BL},\square}(\mathbb{E}q^{N_k,\varphi_k}_t,q_t) \leq \text{BL}(\Phi_T) \varepsilon,
	\]
	and since this holds for every \(0<\varepsilon \leq 1\), we conclude that
	\begin{equation} \label{eq:mean-convergence-probability-graphon}
		\lim_{k\to \infty} \sup_{t\in [0,T]} d_{\text{BL},\square}(\mathbb{E}q^{N_k,\varphi_k}_t,q_t) = 0.
	\end{equation}

	\medskip

	\(\diamond\) \textsc{Step 5}: Second mode of convergence for the extended empirical measure.\\
	We will prove now that the sequence of random measures \((\mu^{N_k,\varphi_k}_{[0,T]})_{k\in \mathbb{N}}\), with 
	\[
	\mu^{N_k,\varphi_k}_{[0,T]}:\Omega \to \mathcal{P}_{2,2,\nu}(\mathcal{C}^d_T\times [0,1]),
	\]
	converges to the deterministic fibered probability
	measure \(\mu_{[0,T]}\in \mathcal{P}_{2,2,\nu}(\mathcal{C}^d_T\times [0,1])\) in the sense that
	\(\mathbb{E} W_2(\mu^{N_k,\varphi_k}_{[0,T]},\mu_{[0,T]}) \to 0\) as \(k \to \infty\).
	Using the equivalence of convergence in \(W_2\) as narrow convergence plus uniform integrability of
	second moments, one can prove, with the same techniques as in Proposition \ref{prop:characterization-convergence-d_Wp_q},
	that \(\mathbb{E} W_2(\mu^{N_k,\varphi_k}_{[0,T]},\mu_{[0,T]}) \to 0\) as \(k \to \infty\) if and only if
	\begin{equation} \label{eq:narrow-convergence-extended-empirical-measure}
		\mathbb{E} d_{\text{BL}}(\mu^{N_k,\varphi_k}_{[0,T]},\mu_{[0,T]}) \to 0,
	\end{equation}
	and  
	\begin{equation} \label{eq:uniformly-int-moments-extended-empirical-measure}
		\lim_{R\to \infty}\sup_{k\in \mathbb{N}} \mathbb{E} \int_{0}^{1} \int_{ \Vert \gamma \Vert_{\ast,T} > R } \Vert \gamma \Vert_{\ast,T}^2 d\mu^{N_k,\varphi_k,\xi}_{[0,T]}(\gamma)\,d\xi = 0.
	\end{equation}
	We will prove first \eqref{eq:narrow-convergence-extended-empirical-measure}. To do so, we take a countable convergence determining class \((f_m)_{m\in \mathbb{N}}\) of functions in \(\text{BL}_1(\mathcal{C}^d_T\times [0,1])\) for the narrow topology in \(\mathcal{P}(\mathcal{C}^d_T\times [0,1])\), and consider the following distance in this space \cite[Remark 5.1.1]{AGS-05}:
	\[
		d(\mu,\bar{\mu})=\sum_{m=1}^{\infty} 2^{-m} \left\vert \int_{\mathcal{C}^d_T\times [0,1]} f_m(\gamma,\xi) d(\mu-\bar{\mu}) (\gamma,\xi)\right\vert.
	\]  
	The previous distance also metrizes the narrow topology, thus, to prove that \(\mathbb{E} d_{\text{BL}}(\mu^{N_k,\varphi_k}_{[0,T]},\mu_{[0,T]}) \to 0\) as \(k \to \infty\), we will prove that \(\mathbb{E} d(\mu^{N_k,\varphi_k}_{[0,T]},\mu_{[0,T]}) \to 0\), which is equivalent by the dominated convergence theorem since \(d\) and \(d_{\text{BL}}\) both induce the same narrow topolgy. By the triangle inequality
	\[
		\mathbb{E}d(\mu^{N_k,\varphi_k}_{[0,T]},\mu_{[0,T]}) \leq \mathbb{E}d(\mu^{N_k,\varphi_k}_{[0,T]},\bar{\mu}^{N_k,\varphi_k}_{[0,T]}) +\mathbb{E} d(\bar{\mu}^{N_k,\varphi_k}_{[0,T]},\mathbb{E} \bar{\mu}^{N_k,\varphi_k}_{[0,T]})+ d(\mathbb{E} \bar{\mu}^{N_k,\varphi_k}_{[0,T]},\mu_{[0,T]}).
	\]
	The first term goes to zero as \(k \to \infty\) since 
	\[
		\begin{aligned}
			\mathbb{E}d(\mu^{N_k,\varphi_k}_{[0,T]},\bar{\mu}^{N_k,\varphi_k}_{[0,T]}) &= \mathbb{E} \sum_{m=1}^{\infty} 2^{-m} \left\vert \int_{0}^{1} \int_{\mathcal{C}^d_T} f_m(\gamma,\xi) d(\mu^{N_k,\varphi_k,\xi}_{[0,T]}-\bar{\mu}^{N_k,\varphi_k,\xi}_{[0,T]}) (\gamma) d\xi\right\vert \\ 
			&\leq \sum_{m=1}^{\infty} 2^{-m} \mathbb{E}\left\vert \sum_{j=1}^{N_k} \int_{I_j^{N_k}}  f_m(X^{N_k}_{j,[0,T]},\xi) -f_m(\bar{X}^{N_k}_{j,[0,T]},\xi)d\xi \right\vert \\ 
			&\leq \frac{2}{N_k} \sum_{j=1}^{N_k} \mathbb{E} \Vert X_j^{N_k} -\bar{X}_j^{N_k} \Vert_{\ast,T} \leq \frac{2C(T)}{\sqrt{N_k}}.
		\end{aligned}
	\]	
	The third term also goes to zero as \(k \to \infty\), since we have proved \eqref{eq:mean-convergence-empirical-measure}, and the narrow topology in \(\mathcal{P}(\mathcal{C}^d_T\times [0,1])\) is weaker than the topology induced by the distance \(d_{W_2,2}\). For the second term we have 
	\[
		\begin{aligned}
			\mathbb{E} d(\bar{\mu}^{N_k,\varphi_k}_{[0,T]},\mathbb{E} \bar{\mu}^{N_k,\varphi_k}_{[0,T]})=&\mathbb{E} \sum_{m=1}^{\infty} 2^{-m} \left\vert \int_{0}^{1} \int_{\mathcal{C}^d_T} f_m(\gamma,\xi) d(\bar{\mu}^{N_k,\varphi_k,\xi}_{[0,T]}-\mathbb{E}\bar{\mu}^{N_k,\varphi_k,\xi}_{[0,T]}) (\gamma) d\xi\right\vert \\
			=&  \sum_{m=1}^{\infty} 2^{-m}\mathbb{E} \left\vert \sum_{j=1}^{N_k} \int_{I_j^{N_k}}  f_m(\bar{X}^{N_k}_{j,[0,T]},\xi)d\xi -\int_{I_j^{N_k}}\mathbb{E}f_m(\bar{X}^{N_k}_{j,[0,T]},\xi)d\xi \right\vert. 
		\end{aligned}
	\]
	 We define 
	\[
		x_j^m := \int_{I_j^{N_k}}  f_m(\bar{X}^{N_k}_{j,[0,T]},\xi)d\xi -\int_{I_j^{N_k}}\mathbb{E}f_m(\bar{X}^{N_k}_{j,[0,T]},\xi)d\xi.
	\]
	Since \(f_m\) is bounded by \(1\), we have \(\vert x_j^m \vert \leq 2/N_k\). Expanding the expectation of the squared sum yields 
	
	\[
		\mathbb{E} \left\vert \sum_{j=1}^{N_k} x_j^m \right\vert^2 = \sum_{j=1}^{N_k} \mathbb{E} \vert x_j^m\vert^2 +\sum_{\substack{j,r=1 \\ j \neq r}}^{N_k} \mathbb{E}(x_j^m x_r^m) = \sum_{j=1}^{N_k} \mathbb{E} \vert x_j^m\vert^2 \leq \sum_{j=1}^{N_k} \frac{4}{N^2_k}= \frac{4}{N_k},
	\] 
	where we used \(\mathbb{E}(x_{j}^m x_{r}^m)=0\) for \(j\neq r\) since \((\bar{X}_1^{N_k},\ldots,\bar{X}_{N_k}^{N_k})\) 
	are independent. By Jensen's inequality, we have that 
	\[
		\mathbb{E}d(\bar{\mu}^{N_k,\varphi_k}_{[0,T]},\mathbb{E}\bar{\mu}^{N_k,\varphi_k}_{[0,T]})\leq \frac{2}{\sqrt{N_k}},
	\]
	which is enough to conclude that \(\mathbb{E}d(\mu^{N_k,\varphi_k}_{[0,T]},\mu_{[0,T]}) \to 0\).

	To upgrade the narrow convergence to convergence in \(W_2\), we need to prove
	\eqref{eq:uniformly-int-moments-extended-empirical-measure}. To do so, we are going to see that the bound on the
	moments of the initial data given in Assumption \ref{assump:initial-data} is propagated in time by the system.
	More precisely, we will show that  
	\[
		\sup_{k\in \mathbb{N}} \mathbb{E} \int_{0}^{1} \int_{\mathcal{C}_T^d} \Vert \gamma \Vert_{\ast,T}^{2+\delta} d\mu^{N_k,\varphi_k,\xi}_{[0,T]}(\gamma)\,d\xi = \sup_{k\in \mathbb{N}} \frac{1}{N_k} \sum_{j=1}^{N_k} \mathbb{E} \Vert X_j^{N_k} \Vert_{\ast,T}^{2+\delta} < \infty,
	\]
	which is sufficient to conclude that \eqref{eq:uniformly-int-moments-extended-empirical-measure} holds. To obtain this
	bound, we use the growth condition on the map \(\Pi\) defined in Proposition \ref{prop:properties-Phi-Pi},
	which by \eqref{eq:bound-Pi} and \eqref{eq:lipschitz-property-Pi} ensures that 
	\[
		\vert \Pi_t(\gamma,\bar{\gamma},w)\vert \leq (1+\vert w \vert)(L_{\Pi}(t)+B_{\Pi}(t)) (1+\Vert \gamma \Vert_{\ast,t}).
	\]
	Because the initial scaled weights \(N_k w^{N_k}_{ij,0}\) are uniformly bounded by \(W\) (due to Assumption \ref{assump:weights}), this linear growth condition ensures that, since \(X_i^{N_k}\) satisfies 
	\[
		X_i^{N_k}(s) = X^{N_k}_{i,0} + \frac{1}{N_k} \sum_{j=1}^{N_k} \int_{0}^{s} \Pi_\tau(X^{N_k}_{i,[0,\tau]},X^{N_k}_{j,[0,\tau]},N_k w^{N_k}_{ij,0}) d\tau + \sqrt{2\nu}W_i^{N_k}(s),
	\] 
	there exists a constant \(\kappa(t) > 0\) such that the supremum norm of the path satisfies
	\[
		\Vert X_i^{N_k} \Vert_{\ast,t}  \leq \vert X^{N_k}_{i,0} \vert + \kappa(t) \int_{0}^{t} (1+\Vert X^{N_k}_{i} \Vert_{\ast,s}) ds +  \sqrt{2\nu} \Vert W_i^{N_k} \Vert_{\ast,t}.
	\]
	Raising both sides to the power \(2+\delta\), taking the expectation, and applying standard estimates, we obtain:
	\[
		\frac{1}{N_k}\sum_{i=1}^{N_k} \mathbb{E} \Vert X_i^{N_k} \Vert^{2+\delta}_{\ast,t}  \leq \tilde{\kappa}(t) \left( 1 + \frac{1}{N_k} \sum_{i=1}^{N_k} \mathbb{E} \vert X^{N_k}_{i,0} \vert^{2+\delta} + \int_{0}^{t} \frac{1}{N_k} \sum_{i=1}^{N_k} \mathbb{E} \Vert X^{N_k}_{i} \Vert_{\ast,s}^{2+\delta} ds +\frac{1}{N_k} \sum_{i=1}^{N_k}  \mathbb{E}\Vert W_i^{N_k} \Vert_{\ast,t}^{2+\delta}  \right),
	\]
	where \(\tilde{\kappa}(t)\) is a new constant. By hypothesis, the initial \(2+\delta\) moments are uniformly bounded
	with respect to \(N_k\), and \(\mathbb{E}\Vert W_i^{N_k} \Vert_{\ast,t}^{2+\delta}\) is finite and independent of
	\(i\) and \(N_k\). Therefore, an application of Gronwall's lemma concludes the proof of
	\eqref{eq:uniformly-int-moments-extended-empirical-measure}.

	\medskip

	\(\diamond\) \textsc{Step 6}: Second mode of convergence for the empirical probability-graphon. \\
	Again, we need to take a countable convergence determining class \((f_m)_{m\in \mathbb{N}}\) in \(\text{BL}_1(\mathbb{R}\times [0,1]^2)\) for the narrow topology in \(\mathcal{P}(\mathbb{R}\times [0,1]^2)\), and consider the following distance in this space \cite[Remark 5.1.1]{AGS-05}:
	\[
		d(q,\bar{q})=\sum_{m=1}^{\infty} 2^{-m} \left\vert \int_{\mathbb{R}\times [0,1]^2} f_m(w,\xi,\xi') d(q-\bar{q}) (w,\xi,\xi')\right\vert.
	\]
	The previous distance also metrizes the narrow topology, thus, to prove that
	\( \sup_{t\in [0,T]} \mathbb{E} d_{\text{BL}}(q^{N_k,\varphi_k}_t,q_t) \to 0\) as \(k \to \infty\), we will prove that
	\(\sup_{t\in [0,T]} \mathbb{E} d(q^{N_k,\varphi_k}_t,q_t) \to 0\). By the triangle inequality we have that
	\[
		\sup_{t\in[0,T]}\mathbb{E}d(q_t^{N_k,\varphi_k},q_t) \leq  \sup_{t\in[0,T]}\mathbb{E}d(q_t^{N_k,\varphi_k},\bar{q}_t^{N_k,\varphi_k}) + \sup_{t\in[0,T]}\mathbb{E} d(\bar{q}_t^{N_k,\varphi_k},\mathbb{E} \bar{q}_t^{N_k,\varphi_k}) + \sup_{t\in[0,T]}d(\mathbb{E} \bar{q}_t^{N_k,\varphi_k},q_t).
	\]
	The first term goes to zero as \(k \to \infty\) by \eqref{eq:estimate-propagation-independence-probability-graphon}, since
	\[
		\begin{aligned}
			&\sup_{t\in[0,T]}\mathbb{E}d(q_t^{N_k,\varphi_k},\bar{q}_t^{N_k,\varphi_k}) \\
            &\qquad\leq \sup_{t\in[0,T]}\sum_{m=1}^{\infty} 2^{-m} \mathbb{E}\left\vert \sum_{i,j=1}^{N_k} \int_{I_i^{N_k}\times I_j^{N_k}}  f_m(N_k w^{N_k}_{ij}(t),\xi,\xi') - f_m(N_k \bar{w}^{N_k}_{ij}(t),\xi,\xi') d\xi d\xi' \right\vert \\ 
			&\qquad\leq \sup_{t\in[0,T]}\frac{2}{N_k^2} \sum_{i,j=1}^{N_k} \mathbb{E} \vert N_k w^{N_k}_{ij}(t) - N_k\bar{w}^{N_k}_{ij}(t)\vert \leq \frac{4L_{\Phi}(T)C(T)}{\sqrt{N_k}}.
		\end{aligned}
	\]

	The third term also goes to zero as \(k \to \infty\), since by Corollary
	\ref{cor:d-BL-cut-stronger-narrow-topology} the distance \(d_{\text{BL},\square}\) induces a stronger
	topology than the narrow topology in \(\mathcal{P}_{\nu\otimes\nu}(\mathbb{R}\times [0,1]^2)\), and
	we have proved \eqref{eq:mean-convergence-probability-graphon}.	

	For the second term, we have 
	\[
    \begin{aligned}
		&\mathbb{E} d(\bar{q}_t^{N_k,\varphi_k},\mathbb{E} \bar{q}_t^{N_k,\varphi_k}) \\ 
        &\qquad=   \sum_{m=1}^{\infty} 2^{-m} \mathbb{E} \left\vert \sum_{i,j=1}^{N_k} \int_{I_i^{N_k}\times I_j^{N_k}}  f_m(N_k \bar{w}^{N_k}_{ij}(t),\xi,\xi')d\xi d\xi' - \int_{I_i^{N_k}\times I_j^{N_k}}\mathbb{E}f_m(N_k \bar{w}^{N_k}_{ij}(t),\xi,\xi')d\xi d\xi' \right\vert.
    \end{aligned}
	\]
	We define
	\[
		x_{ij}^m := \int_{I_i^{N_k}\times I_j^{N_k}}  f_m(N_k \bar{w}^{N_k}_{ij}(t),\xi,\xi')d\xi d\xi' - \int_{I_i^{N_k}\times I_j^{N_k}}\mathbb{E}f_m(N_k \bar{w}^{N_k}_{ij}(t),\xi,\xi')d\xi d\xi'.
	\]
	By the boundedness of \(f_m\), we have \(\vert x^m_{ij} \vert\leq 2/N_k^2\). Since the weights \(N_k\bar{w}^{N_k}_{ij}(t)\) take the form given in \eqref{eq:weights-intermediate-system}, the random variables \(x_{ij}^m\) and \(x_{rs}^m\) are independent whenever their defining index sets are disjoint, {\it i.e.}, \(\{i,j\} \cap \{r,s\} = \emptyset\). Therefore, \(\mathbb{E}(x_{ij}^m x_{rs}^m) = \mathbb{E}(x_{ij}^m)\mathbb{E}(x_{rs}^m) = 0\) for all disjoint pairs.

	Expanding the expectation of the squared sum, we group it into overlapping and non-overlapping index sets as follows:
	\[
		\mathbb{E} \left\vert \sum_{i,j=1}^{N_k} x_{ij}^m \right\vert^2 = \sum_{\{i,j\} \cap \{r,s\} \neq \emptyset} \mathbb{E}(x_{ij}^m x_{rs}^m) + \sum_{\{i,j\} \cap \{r,s\} = \emptyset} \mathbb{E}(x_{ij}^m x_{rs}^m).
	\]
	The second sum evaluates to exactly zero. For the first sum, the number of index pairs \(((i,j), (r,s))\) that share at least one index is bounded by \(4N_k^3\). Thus:
	\[
		\sum_{\{i,j\} \cap \{r,s\} \neq \emptyset} \mathbb{E}(x_{ij}^m x_{rs}^m) \leq 4N_k^3 \left( \frac{4}{N_k^4} \right) = \frac{16}{N_k}.
	\]
	This implies that   
	\[
		\sup_{t\in [0,T]}\mathbb{E}d(\bar{q}_t^{N_k,\varphi_k},\mathbb{E}\bar{q}_t^{N_k,\varphi_k}) \leq  \frac{4}{\sqrt{N_k}},
	\]
	which is enough to conclude that \(\sup_{t\in [0,T]}\mathbb{E}d(q_t^{N_k,\varphi_k},q_t) \to 0\) as \(k \to \infty\).
\end{proof}

\begin{remark}[On the choice of the family of standard Wiener processes]
	As one can see in the statement of Theorem \ref{theo:main-complete}, we are allowed to take any initial data $(X_0(\xi))_{\xi\in [0,1]}$ distributed according to $(\mu_0^\xi)_{\xi\in [0,1]}$ and any family of Wierner processes $(W(\cdot,\xi))_{\xi\in [0,1]}$. Indeed, the associated deterministic quantities \(\mu\) and \(q_t\) associated to the associated unique strong solution $(X(\cdot,\xi),q_\gamma^{\xi,\xi'})_{\xi,\xi'\in [0,1]}$ of the system of McKean-Vlasov SDEs \eqref{eq:multi-agent-macroscopic-SDE} are all identical by the uniqueness in law property proved in Proposition \ref{prop:well-posedness-fibered-McKean-Vlasov-SDE-multiplicative-noise}, regardless of the precise realization of the initial data and Wiener processes.
	
	In particular, if we consider the strong solution associated Wiener processes \((W(\cdot,\xi))_{\xi \in [0,1]}\) and initial data $(X_0(\xi))_{\xi\in [0,1]}$ that are essentially pairwise independent ({\it i.e.}, for a.e.\ \(\xi,\xi' \in [0,1]\), the processes \(W(\cdot,\xi)\) and \(W(\cdot,\xi')\), and random variables \(X_0(\xi')\) and \(X_0(\xi')\) are mutually independent), then the structure of the limiting probability-graphon \(q_t\) becomes more intuitive. Specifically, the resulting family of state processes \((X(\cdot,\xi))_{\xi \in [0,1]}\) is also
	essentially pairwise independent. 
	After a standard enlargement of the probability space, we can introduce a family of random variables
	\((w_0(\xi,\xi'))_{(\xi,\xi')\in [0,1]^2}\) such that \(w_0(\xi,\xi')\sim q_0^{\xi,\xi'}\)
	for almost every \((\xi,\xi') \in [0,1]^2\), independent of the processes \(X(\cdot,\xi)\) and
	\(X(\cdot,\xi')\). The limiting probability-graphon \(q_t^{\xi,\xi'}\) then simply represents
	the law of the dynamical random graphon: 
	\[
		w(t,\xi,\xi')=\Phi_t(X(\cdot,\xi)_{[0,t]},X(\cdot,\xi')_{[0,t]},w_0(\xi,\xi')).
	\]
	However, this random graphon and its law are hidden in the macroscopic equation \eqref{eq:multi-agent-macroscopic-SDE}. To make it appear, note that, as it happens for the intermediate system, the interaction term can be written in probabilistic terms as a suitable conditional expectation with respect to the history of the target particle:
	\[
	\int_{0}^{1} \int_{\mathcal{C}_t^d} \int_{\mathbb{R}} w\, q_{\gamma}^{\xi,\xi'}(t,dw) K(X(t,\xi),\gamma(t))  d\mu_{[0,t]}^{\xi'}(\gamma) d\xi'=\int_{0}^{1} \mathbb{E}_{\xi}\left[w(t,\xi,\xi')K(X(t,\xi),X(t,\xi'))\right] d\xi',
	\]
	where \(\mathbb{E}_{\xi}:=\mathbb{E}[\,\cdot\,\vert \mathcal{F}_{\xi}(t)]\) is the expectation conditioned to the filtration 
	\(\mathcal{F}_{\xi}(t):=\sigma(\{X(s,\xi):\,s\in [0,t]\})\). 
\end{remark}

\section{Special cases and some extensions} \label{sec:conclusions}
In this section we showcase several special cases of the model \eqref{eq:multi-agent-discrete} that are covered by the theory developed in our main Theorem \ref{theo:main} (extended in Theorem \ref{theo:main-complete}) of this work. Those special cases lead to simpler versions of our limit equation \eqref{eq:multi-agent-macroscopic-SDE} which have been already formulated and studied in previous literature. We also introduce natural extensions extensions of our main result, where the machinery introduced in this work still work, but we have not address explicitly for conciseness. We end the section with some future directions of research that can be pursued based
on the results of this paper. 

\subsection{Special cases covered by our work} \label{subsec:special-cases}
We start by highlighting some special cases of the original model that are covered by our work and have been previously considered in the literature.

\medskip

\(\diamond\) \textbf{Static weights}.\\
When the plasticity rule is identically zero, i.e., \(\Gamma \equiv 0\), the weights remain static and the
path-dependence structure coming from the flow map \(\Phi_t\) disappears. In this case, the limiting process
reduces to the solution of a classical McKean-Vlasov SDE in which the interaction term depends only on the
current state of the agents, recovering the Markovian structure:
\begin{equation} \label{eq:static-weights-limit-process}
\begin{aligned}
&dX(t,\xi) = \int_0^1 \int_{\mathbb{R}^d} w_0(\xi,\xi') K(X(t,\xi),y) \,d\mu_t^{\xi'}(y) \,d\xi' dt + \sqrt{2\nu}\,dW(t,\xi),\\
&X(0,\xi)=X_0(\xi),  \quad w_0(\xi,\xi'):=\int_{\mathbb{R}} w \,q^{\xi,\xi'}_0(dw),
\end{aligned}
\end{equation}
where \(\mu_t^{\xi'}:=\mathrm{Law}(X(t,\xi'))\). Because the underlying state process is now strictly Markovian,
standard arguments yield that the marginal laws \(\mu_t^\xi\) satisfy the corresponding non-linear Fokker-Planck
equations. This allows the full macroscopic limit to be cleanly reformulated as a system of deterministic PDEs:
\begin{equation} \label{eq:static-weights-macroscopic-PDE}
\begin{aligned}
&\partial_t \mu_t^{\xi}(x) + \divop_x \left( \mu_t^{\xi}(x) \int_0^1 \int_{\mathbb{R}^d} w_0(\xi,\xi') K(x,y) \,d\mu_t^{\xi'}(y) \,d\xi' \right) = \nu \Delta_x \mu_t^{\xi}(x),\\
&\mu^{\xi}(0)= \mu_0^{\xi}, \quad w_0(\xi,\xi'):=\int_{\mathbb{R}} w \,q^{\xi,\xi'}_0(dw).
\end{aligned}
\end{equation}
A key advantage of the probability-graphon framework is that it accommodates a general class of random initial
weights \((w_{ij,0}^N)_{1\leq i,j\leq N}\) to model heterogeneous interactions between agents.
The graph limit is then described by a probability-graphon \(q^{\xi,\xi'}_0\) that encodes the distribution
of the initial weights between macroscopic labels \(\xi\) and \(\xi'\) in the limit \(N \to \infty\). As evident
from equations \eqref{eq:static-weights-limit-process} and \eqref{eq:static-weights-macroscopic-PDE}, the limiting dynamics
of the agents are completely determined by the expectation of the probability-graphon \(q^{\xi,\xi'}_0\), which we denote
by \(w_0(\xi,\xi')\). This generalizes previous works on random interacting systems, such as
\cite{CDG-20,DGL-16,Lucon-20}, which focus on specific random graph models like Erdős-Rényi networks and \(W\)-random graphs.
In contrast, our approach allows for a significantly broader class of initial network topologies,
requiring only that the weights are uniformly bounded almost surely, as stated in Assumption \ref{assump:weights}.

\medskip

\(\diamond\) \textbf{Affine plasticity rule.} \\
Assume that the plasticity function \(\Gamma\) is affine in the weight variable, {\it i.e.},
\begin{equation} \label{eq:plasticity-rule-affine}
\Gamma(x,y,w)=\Gamma_0(x,y)+\Gamma_1(x,y)w.
\end{equation}
This case has been considered in \cite{CD-25-arxiv, GKX-25, T-25-arxiv, Zhou-25-arxiv}. We note that, in this case, we do not need to know the full probability-graphon \(q_\gamma^{\xi,\xi'}\)  in order to characterize the limit equation \eqref{eq:multi-agent-macroscopic-SDE}. Indeed, such limit equation can be closed in terms of the average of the probability-graphon \(q^{\xi,\xi'}_{\gamma}\), namely,
\[
w_{\gamma}(t,\xi,\xi')=\int_{\mathbb{R}} w\, q_{\gamma}^{\xi,\xi'}(t,dw).
\]
To see this, recall that $q_\gamma^{\xi,\xi'}(t,\cdot)=\Phi_t(X(t,\xi),\gamma(t),\cdot)_{\#}q_0^{\xi,\xi'}$ and then we can compute the time evolution:
\[
\begin{aligned}
\frac{d}{dt} w_{\gamma}(t,\xi,\xi') &= \frac{d}{dt} \int_{\mathbb{R}} w\, q^{\xi,\xi'}_{\gamma}(t,dw)= \frac{d}{dt} \int_{-W}^{W} \Phi_t(X(\cdot,\xi)_{[0,t]},\gamma,\tilde{w}) dq^{\xi,\xi'}_{0}(\tilde{w}) \\ 
    &= \int_{-W}^{W} \Gamma\Big(X(t,\xi),\gamma(t),\Phi_t(X(\cdot,\xi)_{[0,t]},\gamma,\tilde{w})\Big) \,dq^{\xi,\xi'}_{0}(\tilde{w}) \\ 
    &= \int_{-W}^{W} \Big( \Gamma_0(X(t,\xi),\gamma(t)) + \Gamma_1(X(t,\xi),\gamma(t))\Phi_t(X(\cdot,\xi)_{[0,t]},\gamma,\tilde{w}) \Big)\,dq^{\xi,\xi'}_{0}(\tilde{w}) \\    &=\int_{\mathbb{R}}\Big(\Gamma_0(X(t,\xi),\gamma(t)) + \Gamma_1(X(t,\xi),\gamma(t)),w\Big)\,q_\gamma^{\xi,\xi'}(t,dw)\\
	&= \Gamma_0(X(t,\xi),\gamma(t)) + \Gamma_1(X(t,\xi),\gamma(t))\,w_{\gamma}(t,\xi,\xi').
\end{aligned}
\]
We remark that the affine structure of \(\Gamma\) is crucial in the last step in order to write the right hand side above in terms of $w_\gamma(t,\xi,\xi')$ only. Consequently, for affine plasticity function the mean-field limit equation \eqref{eq:multi-agent-macroscopic-SDE} reduces to a simpler system
where the transport PDE is replaced by an ODE:
\[
	\begin{aligned}
	&d X(t,\xi)=\int_0^1\int_{\mathcal{C}^d_t} w_{\gamma}(t,\xi,\xi')\,K(X(t,\xi),\gamma(t))\,d\mu_{[0,t]}^{\xi'}(\gamma)\,d\xi' dt+ \sqrt{2\nu}\,dW(t,\xi), \\
	&\partial_t w_{\gamma}(t,\xi,\xi') = \Gamma_0(X(t,\xi),\gamma(t)) + \Gamma_1(X(t,\xi),\gamma(t)) w_{\gamma}(t,\xi,\xi'),\\
	&X(0,\xi)=X_0(\xi),\quad w_{\gamma}(0,\xi,\xi')=\int_{\mathbb{R}} w \, q^{\xi,\xi'}_0(dw).
	\end{aligned}
\]
	This reduction clarifies the strict necessity of probability-graphon \(q_\gamma^{\xi,\xi'}\) in the general case of non-linear plasticity function \(\Gamma\), and not only its average \(w_\gamma^{\xi,\xi'}\). For instance, assume that \(\Gamma\) depends on \(w\) nonlinearly as a polynomial of degree larger than one. Then, the evolution of the first order moment \(w_{\gamma}\) will inherently depend
	on higher-order
	moments of \(q^{\xi,\xi'}_{\gamma}\). The evolution of these higher-order moments will, in turn,
	depend on even higher moments, leading to the classical moment hierarchy. Therefore, for general non-linear
	plasticity rules, it is impossible to close the equations around a finite set of moments. This makes the probability measure \(q^{\xi,\xi'}_{\gamma}\) an indispensable structural component
	of the mean-field limit.

	\medskip 

\(\diamond\) \textbf{Noiseless setting: the fixed-point formulation}. \\
In the absence of noise (\(\nu = 0\)), the law of the limiting process \(X(t,\xi)\) solving the McKean-Vlasov SDE \eqref{eq:multi-agent-macroscopic-SDE} can be characterized as the unique fixed point of the characteristic flow map associated with the drift field generated by the interactions. More precisely, recall that the process \(X(t,\xi)\) is defined as the solution to the following integral equation:
\[
X(t,\xi) = X_0(\xi) + \int_{0}^{t} b^{q_0}_{\mu}(s, X(\cdot,\xi), \xi) \,ds, \quad t \in [0,T],
\]
with the drift field \(b^{q_0}_{\mu}\) defined in \eqref{eq:drift-macroscopic-system} and \(\mu^{\xi}:=\mathrm{Law}(X(\cdot,\xi))\). 

For every fibered measure \(\mu \in \mathcal{P}_{\nu}(\mathcal{C}_T^d \times [0,1])\), initial probability-graphon \(q_0 \in \mathcal{P}_{\nu \otimes \nu}([-W,W] \times [0,1]^2)\), initial state \(x_0 \in \mathbb{R}^d\), and \(\xi \in [0,1]\), we define the characteristic flow map \(\Psi^{q_0}_{\mu,x_0}(\xi) \in \mathcal{C}_T^d\) associated with the drift field \(b^{q_0}_{\mu}\) as the unique solution to the following integral equation:
\[
\Psi^{q_0}_{\mu,x_0}(t,\xi) = x_0 + \int_{0}^{t} b^{q_0}_{\mu}(s, \Psi^{q_0}_{\mu,x_0}(\cdot,\xi), \xi) \,ds.
\]
The limiting process can then be recast as \(X(\cdot,\xi)=\Psi^{q_0}_{\mu,X_0(\xi)}(\cdot,\xi)\), and its law is simply expressed via the push-forward:
\begin{equation}\label{eq:fixed-point-formulation}
\mu^{\xi} = \Psi^{q_0}_{\mu,\cdot}(\xi)_{\#} \mathrm{Law}(X_0(\xi))=\Psi^{q_0}_{\mu,\cdot}(\xi)_{\#} \mu_0^\xi.
\end{equation}
This demonstrates that, in this noiseless setting, the macroscopic limit of the system can be fully characterized as a 
fixed point of the characteristic flow map associated with the drift field \(b^{q_0}_{\mu}\); that is, \(\mu \in \mathcal{P}(\mathcal{C}_T^d)\) is the unique
measure satisfying the fixed-point relation \eqref{eq:fixed-point-formulation}. This perspective traces back to the seminal
works \cite{BH-77,Dobrushin-79,N-84} on classical deterministic mean-field systems. It has recently been revisited for
non-exchangeable particle systems on static networks \cite{CM-19,K-VM-18}, and has also been adopted in the study of co-evolving
networks \cite{GKX-25,T-25-arxiv}.

The necessity of characterizing the macroscopic limit as a fixed point on the full path space, rather than as the solution to 
a PDE, stems intrinsically from the non-Markovian nature of the co-evolution. Since the flow map \(\Phi_t\) inherently depends
on the historical trajectories of the interacting particles, the characteristic drift field \(b^{q_0}_{\mu}(t, \gamma, \xi)\) 
requires knowledge of the family of path measures \((\mu_{[0,t]}^{\xi})_{\xi \in [0,1]}\). Consequently, one cannot obtain a
closed, local-in-time evolution PDE for the instantaneous marginals \(t \mapsto \mu_t^\xi\). The evolution of the
spatial density at time \(t\) is inextricably coupled to the entire joint distribution of the paths up to time \(t\).

\medskip

\(\diamond\) \textbf{Noiseless setting: the continuum limit.} \\
Building directly upon the noiseless setting (\(\nu = 0\)) discussed above, the macroscopic dynamics admits a special class of solutions that explicitly bridge the measure-valued framework provided by the mean-field limit with other classical continuum limits. Specifically,
if the initial datum \(\mu_0^\xi = \delta_{x_0(\xi)}\) is deterministic and perfectly concentrated at a map
\(x_0:[0,1]\to \mathbb{R}^d\), then the absence of noise ensures that the fibered pathwise laws $\mu^\xi_{[0,t]}$ remain concentrated on deterministic trajectories \(x: [0,T]\times [0,1] \to \mathbb{R}^d\). Specifically, using the above fixed-point formulation \eqref{eq:fixed-point-formulation}, we obtain that \(\mu^{\xi}_{[0,t]} = \delta_{x(\cdot,\xi)_{[0,t]}}\). By substituting this Dirac ansatz, the integration over the path space \(\mathcal{C}_t^d\) evaluates trivially and leads to
\begin{align*}
b_\xi^{q_0}(t,\gamma,\mu)&=\int_0^1\int_{-W}^W\Phi_t(\gamma_{[0,t]},x(\cdot,\xi')_{[0,t]},w)\,dq_0^{\xi,\xi'}(w)\,K(\gamma(t),x(t,\xi'))\,d\xi'\\
&=\int_0^1 \int_{\mathbb{R}} w\,q_{\gamma}^{\xi,\xi'}(t,dw)\,K(\gamma(t),x(t,\xi'))\,d\xi'\\
&=\int_0^1 w_{\gamma}(t,\xi,\xi')K(\gamma(t),x(t,\xi'))\,d\xi'.
\end{align*}
Since $x(t,\xi)=X^{q_0,\mu}_\xi(t,x_0(\xi))$, and therefore $\partial_t x(t,\xi)=b_\xi^{q_0}(t,x(t,\xi),\mu)$, we observe that in the above expressions the probability-graphon \(q^{\xi,\xi'}_{\gamma}(t,dw)\) and its average \(w_\gamma(t,\xi,\xi')\) lose their dependence on the abstract path
variable \(\gamma\), as they are evaluated strictly along the deterministic trajectories $\gamma=x(\cdot,\xi)$. Dropping the redundant
path subscript, the full macroscopic limit rigorously reduces to a coupled, fully deterministic integro-differential system:
\begin{equation}\label{eq:continuum-limit}
\begin{aligned}
&\partial_t x(t,\xi) = \int_0^1 w(t,\xi,\xi') \,K(x(t,\xi),x(t,\xi')) \,d\xi', \quad && t \in [0,T], \\
&\partial_t q^{\xi,\xi'}(t,w) + \partial_w\Big(\Gamma(x(t,\xi), x(t,\xi'), w)q^{\xi,\xi'}(t,w)\Big) = 0, \quad && t \in [0,T], \\
&x(0,\xi)=x_0(\xi), \quad q^{\xi,\xi'}(0)=q^{\xi,\xi'}_0, \quad w(t,\xi,\xi'):=\int_{\mathbb{R}} w \,q^{\xi,\xi'}(t,dw).
\end{aligned}
\end{equation}

Again, as discussed before, for general non-linear plasticity the system \eqref{eq:continuum-limit} cannot be closed in terms of the average graphon $w(t,\xi,\xi')$ and the full probability-graphon $q^{\xi,\xi'}(t,w)$ is required.

A simpler version of this continuum limit \eqref{eq:continuum-limit} was previously derived in \cite{GKX-23}, though their analysis was restricted precisely to the specific case
of affine plasticity rules, as in \eqref{eq:plasticity-rule-affine}. As discussed before, in such affine regimes, the transport PDE for the probability-graphon $q^{\xi,\xi'}(t,w)$ can be bypassed entirely as the system closes exactly around the expected 
average graphon \(w(t,\xi,\xi')\):
\begin{equation*}
\begin{aligned}
&\partial_t x(t,\xi) = \int_0^1 w(t,\xi,\xi') \,K(x(t,\xi),x(t,\xi')) \,d\xi', \quad && t \in [0,T], \\
&\partial_t w(t,\xi,\xi') = \Gamma_0(x(t,\xi), x(t,\xi')) + \Gamma_1(x(t,\xi), x(t,\xi'))w(t,\xi,\xi'), \quad && t \in [0,T], \\
&x(0,\xi)=x_0(\xi), \quad w(0,\xi,\xi')=\int_{\mathbb{R}} w \,q^{\xi,\xi'}_0(dw).
\end{aligned}
\end{equation*}

The framework was later extended in \cite{Throm-24} to accommodate non-linear plasticity and still using only graphons and not probability-graphons. However, the derivation of the continuum limit assumed well-prepared initial weights coming from special step-function discretizations of the desired limiting graphon and converging in $L^2$ strong norms. We however do not prepare our initial weights and our argument is based on general compactness arguments of probability-graphons. From \eqref{eq:continuum-limit} we can recover the simpler version of the continuum limit introduced in \cite{Throm-24} by restricting our initial probability-graphon to a graphon, namely $q_0^{\xi,\xi'}(w)=\delta_{w_0(\xi,\xi')}(w)$. Therefore, the transport equation propagates this Dirac shape of the probability-graphon $q^{\xi,\xi'}(t,w)=\delta_{w(\xi,\xi')}(w)$ and it is easy to see that \eqref{eq:continuum-limit} reduces to 
\begin{equation*}
\begin{aligned}
&\partial_t x(t,\xi) = \int_0^1 w(t,\xi,\xi') \,K(x(t,\xi),x(t,\xi')) \,d\xi', \quad && t \in [0,T], \\
&\partial_t w(t,\xi,\xi') = \Gamma(x(t,\xi), x(t,\xi'),w(t,\xi,\xi')), \quad && t \in [0,T], \\
&x(0,\xi)=x_0(\xi), \quad w(0,\xi,\xi')=w_0(\xi,\xi').
\end{aligned}
\end{equation*}

\subsection{Extensions covered by our work} Now we present some extensions of the original model that can be handled with the techniques developed in this paper.

\medskip

\(\diamond\) \textbf{Non-additive noise in the particles' states}\\  
The techniques presented in this work can be naturally extended to more general dynamics of the agents' states, where the noise is non-additive
in the state equations. A natural extension is to incorporate an additional network structure in the multi-agent system describing how the interactions between the various agents modulate the noise intensity experienced by each agent. This is similar to the framework considered in \cite{KP-26}, which introduced two different static directed
graphs (one modulating the deterministic part of the interactions and the other modulating the stochastic part of the interactions) represented via digraph measures.
Our methodology allows these dual network topologies to be fully adaptive. 

More precisely, we can consider a system where the pairwise interactions driving the deterministic drift and the stochastic diffusion are governed by two separate sets of evolving weights, \(w_{ij}^N(t)\) and \(\tilde{w}_{ij}^N(t)\):
\begin{equation}\label{eq:multi-agent-discrete-extension-non-additive-noise}
\begin{aligned}
&dX_i^N(t)=\sum_{j=1}^N w_{ij}^N(t)\,K(X_i^N(t),X_j^N(t))\,dt+ \sum_{j=1}^{N} \tilde{w}_{ij}^N(t) \tilde{K}(X_i^N(t),X_j^N(t))\,dW_i^N(t),\\
&\frac{d}{dt}w_{ij}^N(t)=\frac{1}{N}\Gamma(X_i^N(t),X_j^N(t),Nw_{ij}^N(t)), \quad \frac{d}{dt}\tilde{w}_{ij}^N(t)=\frac{1}{N}\tilde{\Gamma}(X_i^N(t),X_j^N(t),N\tilde{w}_{ij}^N(t)),\\
& X_i^N(0)=X_{i,0}^N, \quad  w_{ij}^N(0)=w_{ij,0}^N, \quad \tilde{w}_{ij}^N(0)=\tilde{w}_{ij,0}^N.
\end{aligned}
\end{equation}

From a modeling perspective, this decoupling allows for two different connection topologies to independently modulate
the deterministic alignment of states and the correlated exposure to random fluctuations. With assumptions on \(\tilde{K}\)
and \(\tilde{\Gamma}\) analogous to those imposed on \(K\) and \(\Gamma\), the mean-field limit can be rigorously established.
The compactness argument used in Theorem \ref{theo:main-complete} to obtain the limiting network structure and the initial
fibered measure, which relied entirely on Theorem \ref{thm:compactness-node-edge-probability-graphons} extends
in a straightforward manner to accommodate the additional network structure. We simply enlarge the underlying mathematical
space from pairs to triplets, tracking the empirical extended measure alongside the two distinct empirical probability-graphons
\((q_t^N, \tilde{q}_t^N)\). Consequently, these empirical processes converge to deterministic limits
\(\mu_{[0,t]}^{\xi}\), \(q_t^{\xi,\xi'}\), and \(\tilde{q}_t^{\xi,\xi'}\), given by:
\[
q_t^{\xi,\xi'}=\Phi_t(\cdot,\cdot,\cdot)_{\#}(\mu_{[0,t]}^{\xi} \otimes \mu_{[0,t]}^{\xi'} \otimes q_0^{\xi,\xi'}), \quad \tilde{q}_t^{\xi,\xi'}=\tilde{\Phi}_t(\cdot,\cdot,\cdot)_{\#} (\mu_{[0,t]}^{\xi} \otimes \mu_{[0,t]}^{\xi'} \otimes \tilde{q}_0^{\xi,\xi'}),
\]
where \(\Phi\) and \(\tilde{\Phi}\) are the deterministic flow maps associated to $\Gamma$ and $\tilde \Gamma$ respectively, defined as in Proposition
\ref{prop:properties-Phi-Pi}. The limiting fibered measure \(\mu\) then uniquely corresponds to the law of the strong solution \(X\) of the following McKean-Vlasov SDE with non-additive noise:
$$
\begin{aligned}
&d X(t,\xi)=\int_0^1\int_{\mathcal{C}^d_t} w_{\gamma}(t,\xi,\xi')  \,K(X(t,\xi),\gamma(t))\,d\mu_{[0,t]}^{\xi'}(\gamma)\,d\xi' dt\\
&\qquad\qquad+ \int_0^1\int_{\mathcal{C}^d_t} \tilde w_{\gamma}(t,\xi,\xi')  \,\tilde K(X(t,\xi),\gamma(t))\,d\mu_{[0,t]}^{\xi'}(\gamma)\,d\xi' dW(t,\xi),\quad t\in [0,T],\,\xi\in [0,1],\\
\end{aligned}
$$
where again we define the average graphons
$$w_{\gamma}(t,\xi,\xi')=\int_{\mathbb{R}} w \, q^{\xi,\xi'}_{\gamma}(t,dw),\quad \tilde w_{\gamma}(t,\xi,\xi')=\int_{\mathbb{R}} w \, \tilde q^{\xi,\xi'}_{\gamma}(t,dw),$$
and where the random probability-graphons $q_\gamma^{\xi,\xi'}$ and $\tilde q_\gamma^{\xi,\xi'}$ satisfy
$$
\begin{aligned}
&\partial_s q^{\xi,\xi'}_{\gamma}(s,w)+\partial_w(\Gamma(X(s,\xi),\gamma(s),w)q^{\xi,\xi'}_{\gamma}(s,w))=0,\quad s\in [0,t],\,\xi,\xi'\in [0,1],\,\gamma\in \mathcal{C}^d_t,\\
&\partial_s \tilde q^{\xi,\xi'}_{\gamma}(s,w)+\partial_w(\tilde \Gamma(X(s,\xi),\gamma(s),w)\tilde q^{\xi,\xi'}_{\gamma}(s,w))=0,\quad s\in [0,t],\,\xi,\xi'\in [0,1],\,\gamma\in \mathcal{C}^d_t.
\end{aligned}
$$
As it was proved in \eqref{eq:limiting-probability-graphon} in Theorem \ref{theo:main-complete}, the random probability-graphons $q_\gamma^{\xi,\xi'}$ and $\tilde q_\gamma^{\xi,\xi'}$ and the deterministic limiting probability-graphons $q^{\xi,\xi'}$ and $\tilde q^{\xi,\xi'}$ are related by
$$
q^{\xi,\xi'}_t= \int_{\mathcal{C}_t^d} \mathbb{E}[q^{\xi,\xi'}_{\gamma}(t)] \, d\mu^{\xi'}_{[0,t]}(\gamma), \quad \tilde q^{\xi,\xi'}_t= \int_{\mathcal{C}_t^d} \mathbb{E}[\tilde q^{\xi,\xi'}_{\gamma}(t)] \, d\mu^{\xi'}_{[0,t]}(\gamma).
$$
The rigorous justification of this limit relies on the exact same methodology developed throughout this paper.
Because the stochastic fluctuations are no longer additive, the only technical modification required in the proofs
is the systematic use of the Burkholder-Davis-Gundy (BDG) inequality to bound the suprema of the stochastic integrals
when establishing the necessary estimates.

\medskip

\(\diamond\) \textbf{Additive bounded noise in the particles' weights}\\
A particularly interesting case is the incorporation of an additive bounded noise in the dynamics of the weights. We start by introducing our precise notion of bounded noise.

\begin{definition}[Bounded noise]\label{def:bounded-noise}
Consider a filtered probability space \((\Omega, \mathcal{F}, \{\mathcal{F}_t\}_{t\in[0,T]},\mathbb{P})\). A bounded noise consist in an adapted stochastic process $R:[0,T]\times \Omega\longrightarrow \mathbb{R}$ verifying the following properties:
\begin{enumerate}[label=(\roman*)]
    \item $R$ has continuous paths.
    \item $R(0)=0$ almost surely.
    \item We have $\mathbb{E}[R(t))]=0$ for every $0\leq t\leq T$.
    \item There exists some $M>0$ such that $\mathbb{P}\left[\max_{0\leq t\leq T}|R(t)|\leq M\right]=1$.
\end{enumerate}
We shall denote the pathwise law of such bounded noise by $\mathcal{R}_{[0,T]}={\rm Law} R_{[0,T]}\in \mathcal{P}(\mathcal{C}_T)$.
\end{definition}

We note that contrarily to standard Wiener processes, bounded noises have uniformly bounded paths by definition. These bounded noises share some properties with standard Wiener processes ({\it e.g.}, (i)-(iii) in Definition \ref{def:bounded-noise}), but they are however not necessarily Gaussian, and they can be colored. While Brownian motion is mathematically convenient, it allows arbitrarily large fluctuations, which may be physically unrealistic in some applications where external perturbations originate from finite-energy sources, bounded environmental influences, or constrained microscopic interactions. Bounded noises have been introduced in Statistical Physics as alternatives to Brownian motion when the fluctuations acting on a system possess both a finite correlation time and a bounded amplitude. We refer to monograph \cite{DO-13} for a systematic treatment of bounded noises and their applications in Statistical Physics, Biology and Engineering, and also to \cite{DDF-20, LS-84} for some rigorous mathematical foundations. In the following we shall showcase some well known examples of bounded noises that can be considered in this extension.

\begin{example}[Sine-Wiener noise]\label{example:bounded-noises-1}
For any standard 1D Wiener process $B$, any $M>0$ and $\tau>0$, the sine-Wiener noise is defined as the following stochastic process:
$$R(t)=M\sin\left(\sqrt{\frac{2}{\tau}}\,B(t)\right).$$
The resulting process has continuous trajectories and bounded amplitude, while retaining a stochastic structure inherited from Brownian motion. This construction has proved a useful example of bounded noise in statistical physics and biology as it leads to changes in phase transitions that appear to be enhanced by such bounded noise \cite{BC-05,CW-04}.
\end{example}

\begin{example}[Doering-Cai-Lin noise]\label{example:bounded-noises-2}
For any standard 1D Wiener process $B$, any $M>0$ and $\theta,\delta>0$, the Doering-Cai-Lin noise is defined as the solution to the following SDE:
$$dR(t)=-\frac{1}{\theta}R(t)dt+\frac{\sqrt{M^2-R(t)^2}}{\sqrt{\theta\delta}}dB(t).$$
It was designed to provide a physically realistic alternative to Ornstein–Uhlenbeck noise while preserving strict amplitude bound, and it was studied in the works \cite{CL-96,D-87} in connection with noise-induced transitions and stochastic transport phenomena. In this case confinement is achieved dynamically as the combination of a confining drift field and a state-dependent diffusion coefficient vanishing at the boundary such that as the process approaches the boundary, stochastic fluctuations become progressively weaker and eventually disappear, preventing further excursions.
\end{example}

\begin{example}[Tsallis–Stariolo–Borland noise]\label{example:bounded-noises-3}
For any standard 1D Wiener process $B$, any $M>0$, $\theta>0$,  and $0<q<1$, the Tsallis–Stariolo–Borland noise is defined as the solution to the following SDE:
$$dR(t)=-\frac{M^2}{\theta}\frac{R(t)}{M^2-R(t)^2}dt+M\sqrt{\frac{1-q}{\theta}}dB(t).$$
The type of bounded noise emerged from the theory of non-extensive statistical mechanics developed by Tsallis and collaborators \cite{B-98,TB-96}. Unlike the previous type of noise, the diffusion coefficient remains nonzero near the boundaries. Instead, the drift becomes singular at the boundary, generating an infinitely strong force directed toward the interior of the interval. The boundaries therefore act as infinitely repulsive barriers that cannot be crossed. This mechanism leads to bounded trajectories and stationary distributions related to Tsallis $q$-statistics.
\end{example}

\begin{example}[Reflected Wiener process]\label{example:bounded-noises-4}
For any standard 1D Wiener process B and any $M>0$, we define the reflected Brownian motion on the interval $[-M,M]$ as the process $R$ which solves the so-called Skorokhod problem on $[-W,W]$, see \cite{S-61,LS-84}. Specifically, such problem consists in finding the unique pair of adapted processes $(R,K)$ verifying $R(t)=B(t)+K(t)$ such that $K$ has continuous paths, $R(t)\in [-W,W]$ almost surely, and also $K$ verifies
$$K(t)=\int_0^t n(R(s))\,d|K|_s,\quad |K|_t=\int_0^t\mathds{1}_{R(s)\in \{\pm M\}}\,d|K|_s.$$
Above $|K|_s$ represents the total variation of $K$ on the time interval $[0,s]$ and $n(\pm M)=\mp M$. From the above, $K$ is has bounded variation and the measure $d|K|_s$ is only supported at points at which $R(s)$ touches the boundary $\{\pm M\}$. Thus, the reflection term $K(t)$ pushes the trajectory inward whenever it reaches one of the endpoints.
\end{example}

More precisely, we shall can consider the following extension of the original multi-agent system \eqref{eq:multi-agent-discrete}, which accommodates bounded noises as defined above on the weights dynamics:
\begin{equation}\label{eq:multi-agent-discrete-extension-noise-weights}
\begin{aligned}
&dX_i^N(t)=\sum_{j=1}^N w_{ij}^N(t)\,K(X_i^N(t),X_j^N(t))\,dt+ \sqrt{2\nu}\,dW_i^N(t),\\
&dw_{ij}^N(t)=\frac{1}{N}\Gamma(X_i^N(t),X_j^N(t),Nw_{ij}^N(t))\,dt + \frac{\sqrt{2\sigma}}{N}\,dR_{ij}^N(t),\\
&X_i^N(0)=X_{i,0}^N, \quad w_{ij}^N(0)=w_{ij,0}^N,
\end{aligned}
\end{equation} 
where \((R_{ij}^N)_{1\leq i,j\leq N}\) is a family of bounded noises as in Definition \ref{def:bounded-noise}. We shall assume that those bounded noises are strictly independent of the standard Wiener processes \((W_i^N)_{1\leq i\leq N}\) on the states dynamics and the initial conditions \((X_{i,0}^N,w^N_{ij,0})_{1\leq i,j\leq N}\). We recall that propagating the uniform bounds on the initial weights $Nw_{ij,0}^N$ in Assumption \ref{assump:weights} to the evolved weights $Nw_{ij}^N(t)$ is a fundamental technical step required in Lemma \ref{lemma:propagation-of-independence} to derive propagation of independence via classical coupling methods. For this reason, we shall assume the following additional condition on the bounded noises $R_{ij}^N$ which will prove sufficient to propagate such uniform bound of weights on finite time intervals.

\begin{assumption}[On the bounded noises of weights]\label{assump:noise-weights}
The bounded noises \((R_{ij}^N)_{1\leq i,j\leq N}\) are taken as in Definition \ref{def:bounded-noise}, independent and identically distributed, and such that there exists \(M>0\) verifying  the following uniform bound holds almost surely:
	\[
		\sup_{N\in \mathbb{N}} \sup_{1\leq i,j\leq N} \Vert R_{ij}^N \Vert_{\ast,T} \leq M.
	\]
\end{assumption}

As in the case without noise, a key step in analyzing this system is exploiting the flow structure of the weight equations.
Since the noise enters additively in the equation for the weights, the equation can be understood in the integral sense and realization by realization. More specifically, for each fixed realization \(\omega \in \Omega\), the weight \(w_{ij}^N(t,\omega)\) is the unique solution to the integral equation:
\[
	w_{ij}^N(t,\omega) = w_{ij,0}^N(\omega) + \frac{1}{N}\int_{0}^{t} \Gamma(X_i^N(s,\omega),X_j^N(s,\omega),Nw_{ij}^N(s,\omega)) ds + \frac{\sqrt{2\sigma}}{N} R_{ij}^N(t,\omega), \quad t \in [0,T].
\]
This means that as in the noiseless case, we can still effectively define a continuous flow map \(\Phi:[0,T]\times \mathcal{C}^d_T \times \mathcal{C}^d_T \times \mathcal{C}_{T,0} \times \mathbb{R} \to \mathbb{R}\) that governs the weight dynamics, where $\mathcal{C}_{T,0}$ is the space of continuous paths $\tilde \gamma\in \mathcal{C}_T$ with $\tilde\gamma(0)=0$. Specifically, for any continuous paths \(\gamma,\bar{\gamma} \in \mathcal{C}^d_T\), noise path \(\tilde{\gamma} \in \mathcal{C}_{T,0}\), and initial weight \(w_0 \in \mathbb{R}\), we define \(\Phi(t,\gamma,\bar{\gamma},\tilde{\gamma},w_0)\) as the unique solution to the following integral equation:
\begin{equation}\label{eq:flowmap-noise-weights}
	\Phi(t,\gamma,\bar{\gamma},\tilde{\gamma},w_0) = w_0 + \int_{0}^{t} \Gamma(\gamma(s),\bar{\gamma}(s),\Phi(s,\gamma,\bar{\gamma},\tilde{\gamma},w_0)) ds + \sqrt{2\sigma} \tilde{\gamma}(t), \quad t \in [0,T].
\end{equation}
This flow map perfectly retains its non-anticipative structure and its Lipschitz continuity if \(\Gamma\) satisfies Assumption \ref{assump:plasticity-function}. We note that the flow map allows us to express the evolved weights explicitly in terms of its initial value and also the particle trajectories and the noise paths, that is,
\[
	Nw_{ij}^N(t) = \Phi_t(X_{i,[0,t]}^N,X_{j,[0,t]}^N,R_{ij,[0,t]}^N,Nw_{ij,0}^N).
\]
Substituting this representation back into the original dynamics ensures that every solution
\((X_i^N,w^N_{ij})_{1\leq i,j\leq N}\) to the fully coupled system \eqref{eq:multi-agent-discrete-extension-noise-weights}
is uniquely associated with a solution \((X^N_{i})_{1\leq i \leq N}\) to the following closed, path-dependent system of SDEs:
\[
\begin{aligned}
&dX_i^N(t)= \frac{1}{N} \sum_{j=1}^N \Phi_t(X^N_{i,[0,t]},X^N_{j,[0,t]}, R_{ij,[0,t]}^N ,Nw_{ij,0}^N)K(X^N_i(t),X^N_j(t))dt + \sqrt{2\nu}dW_i^N(t),\\
& X_i^N(0)=X_{i,0}^N.
\end{aligned}
\] 
Its associated intermediate system corresponding to the independent projection ({\it cf.} \eqref{eq:multi-agent-macroscopic-intermediate-solved-weights}) entails a new type of averaging over the noise paths. By the independence assumptions on the underlying processes, this averaging can be expressed again as a conditional expectation with respect to the history of the target particle which leads to:
\[
\begin{aligned}
d\bar X_i^N(t) &= \sum_{j=1}^N\int_{\mathcal{C}_t^d}\int_{\mathcal{C}_{t,0}}\int_{-W}^{W} \Phi_t(\bar{X}^N_{i,[0,t]},\gamma, \tilde{\gamma},w)K(\bar X_i^N(t),\gamma(t))\,dq^N_{ij,0}(w) d\mathcal{R}_{[0,t]}(\tilde{\gamma}) d\bar{\mu}^{N,j}_{[0,t]}(\gamma)\,dt + \sqrt{2\nu}\,dW_i^N(t), \\
\bar X_i^N(0) &= X_{i,0}^N, \quad \bar \mu^{N,i}_{[0,t]} := \mathrm{Law}(\bar X_{i,[0,t]}^N) \in \mathcal{P}(\mathcal{C}^d_t),\quad q_{ij,0}^N={\rm Law}(Nw_{ij,0}^N)\in \mathcal{P}([-W,W]).
\end{aligned}
\]
Passing to the mean-field limit (\(N \to \infty\)) introduces a critical technical challenge arising from the linear growth of the flow map with respect to the noise variable. Specifically, we have a control
\begin{equation}\label{eq:bound-Phi-extension}
	\vert \Phi(t,\gamma,\bar{\gamma},\tilde{\gamma},w)\vert \leq B_{\Phi}(t)\left(1+ \vert w\vert + \sqrt{2\sigma}\Vert \tilde{\gamma}\Vert_{\ast,t} \right).
\end{equation}
This implies that if Assumption \ref{assump:noise-weights} did not hold, then \(R_{ij}^N\) would not be uniformly bounded and, the measure $\mathcal{R}_{[0,t]}$ would not have bounded support. Hence, the uniform bound imposed on the initial rescaled weights \(Nw_{ij,0}^N\) in Assumption \ref{assump:weights} would fail to propagate in time. Because we rely heavily on this uniform bound for the propagation of independence in Lemma \ref{lemma:propagation-of-independence}, we do not consider cases where the weight noise is driven by unbounded processes, such as standard Wiener processes.

To rigorously perform the mean-field limit, we must therefore impose the structural Assumption \ref{assump:noise-weights} on the additive noise. Under such assumption, the mean-field limit remains analytically tractable. On the one hand the (random) empirical probability-graphon $q^N$ will converge to a deterministic limit given by the following pushforward measure:
\[
q_t^{\xi,\xi'} = \Phi_t(\cdot,\cdot,\cdot,\cdot)_{\#}(\mu_{[0,t]}^{\xi} \otimes \mu_{[0,t]}^{\xi'} \otimes \mathcal{R}_{[0,t]}\otimes q_0^{\xi,\xi'}), \quad \xi,\xi' \in [0,1], \quad t \in [0,T],
\]
and the generalized empirical measures $\mu^N$ will converge to the fibered pathwise law \(\mu\) of the unique strong solution to the resulting McKean-Vlasov SDE: 
\begin{equation}\label{eq:multi-agent-macroscopic-SDE-noise-weights}
\begin{aligned}
dX(t,\xi)&=  \int_{0}^{1} \int_{\mathcal{C}_t^d} \int_{\mathcal{C}_{t,0}} w_{\gamma,\tilde\gamma}(t,\xi,\xi')K(X(t,\xi),\gamma(t))\,d\mathcal{R}_{[0,t]}(\tilde{\gamma})  d\mu_{[0,t]}^{\xi'}(\gamma) \,d\xi' dt + \sqrt{2\nu}dW(t,\xi),\\ 
q^{\xi,\xi'}_{\gamma,\tilde{\gamma}}(t) &:= \Phi_t(X(\cdot,\xi)_{[0,t]},\gamma,\tilde{\gamma},\cdot)_{\#} q_0^{\xi,\xi'},\\
X(0,\xi)&=X_0(\xi), \quad \mu^{\xi}_{[0,t]}:=\mathrm{Law}(X(\cdot,\xi)_{[0,t]}),\quad w_{\gamma,\tilde{\gamma}}(t,\xi,\xi')=\int_\mathbb{R}wq_\gamma^{\xi,\xi'}(t,dw).
\end{aligned}
\end{equation}
As it was proved in \eqref{eq:limiting-probability-graphon} in Theorem \ref{theo:main-complete}, the random probability-graphons $q_{\gamma,\tilde\gamma}^{\xi,\xi'}$ and the deterministic limiting probability-graphons $q^{\xi,\xi'}$ are related by
$$
q^{\xi,\xi'}_t= \int_{\mathcal{C}_t^d} \mathbb{E}[q^{\xi,\xi'}_{\gamma,\tilde\gamma}(t)]\,d\mathcal{R}_{[0,t]}(\tilde{\gamma}) \, d\mu^{\xi'}_{[0,t]}(\gamma).
$$

It is crucial to note that, unlike the case of noiseless weight dynamics, when bounded noise is added then the time-dependent random probability measure \(q^{\xi,\xi'}_{\gamma,\tilde{\gamma}}(t) := \Phi_t(X(\cdot,\xi)_{[0,t]},\gamma,\tilde{\gamma},\cdot)_{\#} q_0^{\xi,\xi'}\) does no longer solve a transport equation similar to \(\eqref{eq:multi-agent-macroscopic-SDE}_2\). Indeed, the flow map \(\Phi_t\) now incorporates the stochastic driving path \(\tilde{\gamma}\) and therefore it ceases to be the characteristic flow of a classical ordinary differential equation. Since $\Gamma$ is nonlinear and therefore $\Phi$ nonlinear too, then the effect of bounded noise appears in a complicated averaged and nonlinear way into the macroscopic dynamics. In next extension we introduce a special case where the role of the bounded noise becomes explicit.

\medskip

\(\diamond\) \textbf{Additive bounded noise in the particles' weights and affine plasticity.}\\
A remarkable simplification of the above occurs in the specific case where the plasticity function \(\Gamma\) takes again the affine form 
of \eqref{eq:plasticity-rule-affine}.
In this regime, the equation \eqref{eq:flowmap-noise-weights} of the flow map \(\Phi\) can be explicitly solved using Duhamel formula, and the solution decouples entirely into a deterministic component and a stochastic component that is linear with respect to the noise path \(\tilde{\gamma}\). More specifically, we have
\[
	\Phi_t(\gamma, \bar{\gamma}, \tilde{\gamma}, w_0) = \Phi^{\text{det}}_t(\gamma, \bar{\gamma}, \tilde\gamma,w_0) + \Phi^{\text{stoc}}(\gamma, \bar{\gamma}, \tilde{\gamma},w_0),
\]
where the deterministic and stochastic parts take the forms:
\[
\begin{aligned}
\Phi^{\text{det}}_t(\gamma, \bar{\gamma}, \tilde\gamma,w_0) &= w_0 \exp\left( \int_{0}^{t} \Gamma_1(\gamma(r), \bar{\gamma}(r)) dr \right)  + \int_{0}^{t} \exp\left( \int_{s}^{t} \Gamma_1(\gamma(r), \bar{\gamma}(r)) dr \right) \Gamma_0(\gamma(s), \bar{\gamma}(s)) ds,\\
\Phi^{\text{stoc}}_t(\gamma, \bar{\gamma}, \tilde{\gamma},w_0)&= \sqrt{2\sigma} \left(\tilde{\gamma}(t) + \int_{0}^{t} \exp\left( \int_{s}^{t} \Gamma_1(\gamma(r), \bar{\gamma}(r)) dr \right) \Gamma_1(\gamma(s), \bar{\gamma}(s)) \tilde{\gamma}(s) ds\right).
\end{aligned}
\]
Note that the deterministic part $\Phi^{\text{det}}_t$ corresponds to the flow map of the noiseless ODE for the weights, and when bounded noise is added then the flow maps $\Phi$ is perturbed by the stochastic part $\Phi^{\text{stoc}}_t$. Based on this split, we can reformulate the right hand side of \eqref{eq:multi-agent-macroscopic-SDE-noise-weights} as follows
\begin{equation}\label{eq:multi-agent-macroscopic-SDE-noise-weights-RHS}
\begin{aligned}
&\int_{0}^{1} \int_{\mathcal{C}_t^d} \int_{\mathcal{C}_{t,0}} w_{\gamma,\tilde\gamma}(t,\xi,\xi')K(X(t,\xi),\gamma(t))\,d\mathcal{R}_{[0,t]}(\tilde{\gamma})  d\mu_{[0,t]}^{\xi'}(\gamma) \,d\xi'\\
&\qquad=\int_{0}^{1} \int_{\mathcal{C}_t^d} \int_{\mathcal{C}_{t,0}} w_{\gamma,\tilde\gamma}^{\text{det}}(t,\xi,\xi')K(X(t,\xi),\gamma(t))\,d\mathcal{R}_{[0,t]}(\tilde{\gamma})  d\mu_{[0,t]}^{\xi'}(\gamma) \,d\xi'\\
&\qquad +\int_{0}^{1} \int_{\mathcal{C}_t^d} \int_{\mathcal{C}_{t,0}} w_{\gamma,\tilde\gamma}^{\text{stoc}}(t,\xi,\xi')K(X(t,\xi),\gamma(t))\,d\mathcal{R}_{[0,t]}(\tilde{\gamma})  d\mu_{[0,t]}^{\xi'}(\gamma) \,d\xi',
\end{aligned}
\end{equation}
where we define the determinist and stochastic part of the random probability-graphons 
$$q_{\gamma,\tilde{\gamma}}^{\text{det},\xi,\xi'}=\Phi_t^{\text{det}}(X(\cdot,\xi)_{[0,t]},\gamma,\tilde{\gamma,\cdot})_{\#}q_0^{\xi,\xi'},\quad q_{\gamma,\tilde{\gamma}}^{\text{stoc},\xi,\xi'}=\Phi_t^{\text{stoc}}(X(\cdot,\xi)_{[0,t]},\gamma,\tilde{\gamma,\cdot})_{\#}q_0^{\xi,\xi'},$$
whose average random graphons are given by
$$w_{\gamma,\tilde{\gamma}}^{\text{det}}(t,\xi,\xi')=\int_{\mathbb{R}}w\,q_{\gamma,\tilde\gamma}^{\text{det},\xi,\xi'}(t,dw),\quad w_{\gamma,\tilde{\gamma}}^{\text{stoc}}(t,\xi,\xi')=\int_{\mathbb{R}}w\,q_{\gamma,\tilde\gamma}^{\text{stoc},\xi,\xi'}(t,dw).$$
On the one hand, note that $\Phi^{\text{det}}$ does not depend on the noise path \(\tilde{\gamma}\). Therefore, the first term of the right hand side of \eqref{eq:multi-agent-macroscopic-SDE-noise-weights-RHS} takes the same form as the noisless case:
$$\int_{0}^{1} \int_{\mathcal{C}_t^d} w_{\gamma,\tilde\gamma}^{\text{det}}(t,\xi,\xi')K(X(t,\xi),\gamma(t))\,d\mu_{[0,t]}^{\xi'}(\gamma) \,d\xi'.$$
Similarly, we can compute the averaging over the noise in the second term of the right hand side of \eqref{eq:multi-agent-macroscopic-SDE-noise-weights-RHS}. Specifically, by Fubini's theorem we have
$$\int_0^1\int_{\mathcal{C}_t^d}\int_{-W}^W\left(\int_{\mathcal{C}_{t,0}}\Phi_t^{\rm stoc}(X(\cdot,\xi)_{[0,t]},\bar\gamma,\tilde{\gamma},w)\,d\mathcal{R}_{[0,t]}(\tilde\gamma)\right)K(X(t,\xi),\bar{\gamma}(t))\,dq_0^{\xi,\xi'}(w)\,d\mu_{[0,t]}^{\xi'}(\bar \gamma)\,d\xi'.$$
Crucially, we also remark that the affine plasticity assumption makes $\Phi^{\text{stoc}}_t$ act linearly on the noise path variable \(\tilde{\gamma}\) and therefore the above integral against the noise measure $\mathcal{R}_{[0,t]}$ must vanish by the zero-mean assumption in Definition \ref{def:bounded-noise} (satisfied by all the Examples \ref{example:bounded-noises-1}- \ref{example:bounded-noises-4}). Indeed, note that
$$\int_{\mathcal{C}_{t,0}}\tilde \gamma(s)\,d\mathcal{R}_{[0,t]}(\tilde{\gamma})=\mathbb{E}[R_{ij}^N(s)]=0,$$
for all $0\leq s\leq t$ by Assumption \ref{assump:noise-weights}. Then, the entire stochastic contribution
strictly vanishes upon integration. Consequently, despite the presence of additive noise at the microscopic level, the interaction term in the macroscopic limit equation becomes completely unaffected by the noise in the weights. The system reduces exactly to the mean-field equation of the deterministic-weight framework, rigorously confirming \cite[Conjecture 2.1]{Zhou-25-arxiv}, at least under bounded noise conditions. This reveals a powerful self-averaging effect of the network: when the plasticity rule is affine, the zero-mean symmetric fluctuations of the pairwise interactions perfectly cancel out in the infinite-particle limit.

\appendix
\section{Path-dependent McKean-Vlasov SDEs}\label{appendix:SDEs}

In this section, we recall basic notions of stochastic processes and state well-posedness results for McKean-Vlasov SDEs
with path-dependent coefficients. Throughout this section, \((\Omega, \mathcal{F}, \{\mathcal{F}_t\}_{t\in[0,T]},\mathbb{P})\)
will be a filtered probability space.

\begin{definition}
A stochastic process \(X:[0,T]\times \Omega \to \mathbb{R}^d\) is:
\begin{itemize}
	\item \emph{measurable} if it is \(\mathcal{B}([0,T])\otimes \mathcal{F}\)-measurable,
	\item \(\{\mathcal{F}_t\}_{t\in[0,T]}\)-adapted (or simply adapted, when the filtration is understood) if \(X(t,\cdot)\) is \(\mathcal{F}_t\)-measurable for every \(t\in [0,T]\),
	\item a process with \emph{continuous paths} if \(X(\cdot,\omega) \in \mathcal{C}_T^d\) for \(\mathbb{P}\)-almost every \(\omega \in \Omega\).
\end{itemize}
We define \(L^2_a([0,T]\times \Omega)\) as the class of processes \(X:[0,T]\times \Omega \to \mathbb{R}^d\) satisfying the following conditions:
\begin{enumerate}
	\item \(X\) is measurable,
	\item \(X\) is \(\{\mathcal{F}_t\}_{t\in[0,T]}\)-adapted,
	\item \(\mathbb{E}\int_{0}^{T} \vert X(t)\vert^2 dt < \infty\).
\end{enumerate}
\end{definition}

\begin{remark}
	For every \(X \in L^2_a([0,T]\times \Omega)\), the Itô integral \(I(t)=\int_{0}^{t}X(s)dW(s)\) 
	is a well-defined process in \(L^2_a([0,T]\times \Omega)\), and it admits a modification with continuous paths.
\end{remark}

When a process has continuous paths, it can be viewed as a measurable map \(X:\Omega \to \mathcal{C}_T^d\), where on the space \(\mathcal{C}_T^d\) we consider the \(\sigma\)-algebra \(\mathcal{H}_T\) defined below.

\begin{definition}
	The filtration \(\{\mathcal{H}_t\}_{t\in [0,T]}\) on the space \(\mathcal{C}_T^d\) is defined as
	\[
		\mathcal{H}_t=\sigma(\{e_s:s \in [0,t]\}), \quad t \in [0,T],
	\]
	where \(e_s:\mathcal{C}_T^d\to \mathbb{R}^d\) is the evaluation map \(\gamma\mapsto \gamma(s)\).
\end{definition}

The pair \((\mathcal{C}_T^d,\mathcal{H}_T)\) is thus a measurable space, and it is easy to verify that \(\mathcal{H}_T\)
coincides with the Borel \(\sigma\)-algebra of \((\mathcal{C}_T^d, \Vert \cdot \Vert_{\ast,T})\). This yields the
following standard equivalences:
\begin{itemize}
	\item \(X:[0,T]\times \Omega \to \mathbb{R}^d\) is measurable if and only if \(X:\Omega\to \mathcal{C}_T^d\) is measurable.
	\item \(X:[0,T]\times \Omega \to \mathbb{R}^d\) is \(\{\mathcal{F}_t\}_{t\in [0,T]}\)-adapted if and only if \(X:(\Omega,\mathcal{F}_t) \to \mathcal{C}_T^d\) is measurable for all \(t\in [0,T]\).
\end{itemize}

Thus, a process \(X:[0,T]\times \Omega\to \mathbb{R}^d\) that is measurable and has continuous paths can also be interpreted
as a random variable taking values in \((\mathcal{C}_T^d,\mathcal{H}_T)\). The pathwise law of \(X\) is simply
\(\textrm{Law}(X)=X_{\#}\mathbb{P}\in\mathcal{P}(\mathcal{C}_T^d)\). 

We define \(L^2_a(\Omega,\mathcal{C}_T^d)\) as the set of stochastic processes with continuous paths such that
\[
	\mathbb{E} \Vert X \Vert^2_{\ast,T} < \infty,
\]
which corresponds to processes in \(L^2_a([0,T]\times \Omega)\) whose pathwise laws belong to the Wasserstein space
\(\mathcal{P}_2(\mathcal{C}_T^d)\). This space will be used in the fixed-point argument of Theorem
\ref{th:well-posedness-path-dependent-SDEs}, as it forms a complete normed space under the norm
\[
	\Vert X \Vert_{L^2_a(\Omega,\mathcal{C}_T)}=\left(\mathbb{E} \Vert X \Vert^2_{\ast,T}\right)^{1/2}.
\]

\begin{definition}[Cylinder sets]
	A cylinder set in \(\mathcal{C}_T^d\) is a set of the form
	\[
		\pi^{-1}_{t_1,\ldots,t_n}(B)=\{\gamma \in \mathcal{C}_T^d : (\gamma(t_1),\ldots,\gamma(t_n))\in B\},
	\]
	where \(0\leq t_1<t_2<\ldots < t_n\leq T\), \(B \in \mathcal{B}((\mathbb{R}^{d})^n)\), and \(\pi_{t_1,\ldots,t_n}: \mathcal{C}_T^d\to (\mathbb{R}^d)^n\) is the evaluation map. We denote the collection of all cylinder sets in \(\mathcal{C}_T^d\) by \(\textrm{Cyl}\).
\end{definition}

It is a well-known result that the family of cylinder sets forms a \(\pi\)-system generating the \(\sigma\)-algebra
\(\mathcal{H}_T\). By Dynkin's \(\pi\)-\(\lambda\) theorem, if two probability measures agree on a \(\pi\)-system,
they must agree on the \(\sigma\)-algebra generated by that \(\pi\)-system. Consequently, the pathwise law of a process
is uniquely determined by its finite-dimensional distributions.

\begin{definition}[Finite-dimensional distributions]
	Let \(X:[0,T]\times \Omega\to \mathbb{R}^d\) be a stochastic process. The set of finite-dimensional distributions of \(X\) is the family of probability measures
	\[
		\{ \mu_{\bar{t}}: \bar{t}=(t_1,\ldots,t_n) \in [0,T]^n, n \in \mathbb{N} \},
	\]
	where \(\mu_{(t_1,\ldots,t_n)}(B)=\mathbb{P}((X(t_1),\ldots,X(t_n))\in B)\) for \(B \in \mathcal{B}((\mathbb{R}^d)^n)\).
\end{definition}

In this paper, we consider a family of continuous stochastic processes \((X(\xi))_{\xi \in [0,1]}\) and rely heavily
on the measurability of the map \([0,1]\ni \xi \mapsto \textrm{Law}(X(\xi)) \in \mathcal{P}(\mathcal{C}^d_T)\).
The following propositions provide useful equivalent conditions to establish this measurability in practice.
\begin{proposition}
	The following conditions are equivalent:
	\begin{enumerate}
		\item[(i)] The map \([0,1]\to \mathcal{P}(\mathcal{C}^d_T)\), given by \(\xi \mapsto P^{\xi}\), is measurable
		when \(\mathcal{P}(\mathcal{C}^d_T)\) is endowed with the narrow topology.
		\item[(ii)] For every \(0\leq t_1\leq \ldots\leq t_n\leq T\) with \(n\in \mathbb{N}\), the map
		\[
			\begin{aligned}
				[0,1] &\longrightarrow \mathcal{P}((\mathbb{R}^d)^n) \\
				\xi &\longmapsto \pi_{t_1,\ldots,t_n \#}  P^{\xi},
			\end{aligned}
		\]
		is measurable.
	\end{enumerate}
\end{proposition}
\begin{proof}
	Condition \((i)\) is equivalent to the measurability of the map \(\xi \mapsto P^{\xi}(B) \in \mathbb{R}\) for every 
	\(B\in \mathcal{B}(\mathcal{C}^d_T)\). Given this equivalence, the fact that \((i)\) implies \((ii)\) follows
	immediately, since \((ii)\) simply restricts this requirement to cylinder sets. 
	
	To prove that \((ii)\) implies \((i)\), we rely on Dynkin's \(\pi\)-\(\lambda\) theorem. Define the class
	\[
		\mathcal{L}=\{ B \in \mathcal{B}(\mathcal{C}^d_T): \xi\mapsto P^{\xi}(B) \text{ is measurable} \}.
	\]
	It is straightforward to verify that \(\mathcal{L}\) is a \(\lambda\)-system: \(\mathcal{C}^d_T\in \mathcal{L}\) since
	\(P^{\xi}(\mathcal{C}^d_T)=1\); if \(B \in \mathcal{L}\), then \(B^c \in \mathcal{L}\) because
	\(\xi \mapsto P^{\xi}(B^c)=1-P^{\xi}(B)\) is measurable; and if \(\{B_n\}_{n\in\mathbb{N}}\) is a sequence of
	disjoint sets in \(\mathcal{L}\), then
	\[
		\xi \mapsto P^{\xi}\left(\bigcup_{n=1}^{\infty} B_n\right) = \sum_{n=1}^{\infty} P^{\xi}(B_n)
	\]
	is measurable as the limit of measurable functions. Condition \((ii)\) ensures that the family of cylinder sets,
	which forms a \(\pi\)-system, is contained in \(\mathcal{L}\). By Dynkin's theorem, 
	\(\sigma(\textrm{Cyl})\subset \mathcal{L}\), and since \(\sigma(\textrm{Cyl})=\mathcal{B}(\mathcal{C}^d_T)\),
	the proof is complete.
\end{proof}

The following proposition will be useful for the proof of Proposition 
\ref{prop:well-posedness-fibered-McKean-Vlasov-SDE-multiplicative-noise}, as it allows us to verify the required
measurability conditions. It is a classical result in measure theory (see, for instance,
\cite[Theorem 10.7.2]{Bogachev-07}).

\begin{proposition} \label{prop:Measurability-Integral-Caratheodory-Function}
	Let \((\mathcal{X},d)\) be a Polish space. For every \(\mu \in \mathcal{P}_{\nu}(\mathcal{X}\times [0,1])\) and every bounded, jointly measurable function \(\varphi:[0,1]\times \mathcal{X}\to \mathbb{R}^d\), the map
	\[
		\xi \in [0,1] \mapsto \int_{\mathcal{X}} \varphi(\xi,x) d\mu^{\xi}(x) \in \mathbb{R}^d
	\]
	is Borel measurable.
\end{proposition}

Finally, we recall the Burkholder-Davis-Gundy (BDG) inequality, a fundamental estimate in stochastic analysis \cite[Theorem 20.12]{Kallenberg-21}.
\begin{theorem}[Burkholder-Davis-Gundy]
	For any \(1\leq p<\infty\), there exist two positive constants \(c_p\) and \(C_p\) such that for any continuous local martingale \(M\) with \(M_0=0\),
	\[
		c_p \mathbb{E}\left[ [M]_t^{p/2} \right] \leq \mathbb{E}\Vert M \Vert^p_{\ast, t}  \leq C_p \mathbb{E} \left[[M]_{t}^{p/2}\right],
	\]
	where \([M]_t\) denotes the quadratic variation of \(M\) at time \(t\).
\end{theorem}
In the case of the Itô integral, \(M(t)=\int_{0}^{t} X(s) dW(s)\), the quadratic variation is explicitly given by \([M]_t=\int_{0}^{t} \vert X(s) \vert^2 ds\). Therefore, for any process \(X\in L^2_a([0,T]\times \Omega)\), we obtain the inequalities:
\begin{equation} \label{eq:BDG-inequality-Ito-integral}
	c_p \mathbb{E}\left[ \left( \int_{0}^{t} \vert X(s) \vert^2 ds \right)^{p/2} \right] \leq \mathbb{E}\left[ \sup_{s \in [0,t]} \left\vert \int_{0}^{s} X(\tau) dW(\tau) \right\vert^p \right] \leq C_p \mathbb{E}\left[ \left( \int_{0}^{t} \vert X(s) \vert^2 ds \right)^{p/2} \right].
\end{equation}
When \(p\geq 2\), the function \(x\mapsto x^{p/2}\) is convex. Applying Jensen's inequality (or equivalently Hölder's inequality) to the time integral on the right-hand side of \eqref{eq:BDG-inequality-Ito-integral} yields the highly useful bound:
\begin{equation} \label{eq:BDG-p-geq-2}
	\mathbb{E}\left[ \sup_{s \in [0,t]} \left\vert \int_{0}^{s} X(\tau) dW(\tau) \right\vert^p \right] \leq C_p t^{\frac{p}{2}-1} \mathbb{E}\left[ \int_{0}^{t} \vert X(s) \vert^p ds \right].
\end{equation}

We now address the well-posedness of the path-dependent SDE governing the particle dynamics
\eqref{eq:multi-agent-discrete-closed-states}. Unlike standard Markovian formulations, the coefficients
in our model depend on the entire history of the path and incorporate random parameters that represent
the initial network configuration. This motivates the study of the following general class of path-dependent SDEs:
\begin{equation} \label{eq:path-dependent-SDE}
	\begin{aligned}
		dX(t)&=b(t,X,\lambda)dt + \sigma(t,X,\bar{\lambda})dW(t), \quad t \in [0,T], \\
		X(0)&=X_0.
	\end{aligned}
\end{equation}
Here, the drift and diffusion coefficients \(b\) and \(\sigma\) are path-dependent functionals
\[
	\begin{aligned}
		&b:[0,T]\times \mathcal{C}_T^d\times E\to \mathbb{R}^d,\\ 
		&\sigma:[0,T]\times \mathcal{C}_T^d \times E\to \mathbb{R}^{d\times d}.
	\end{aligned}
\]
The inputs consist of a \(d\)-dimensional standard Wiener process \(W\) defined on a filtered probability space 
\((\Omega,\mathcal{F},\{\mathcal{F}_t\}_{t\in [0,T]},\mathbb{P})\), an \(\mathcal{F}_0\)-measurable initial condition
\(X_0:\Omega\to \mathbb{R}^d\), and two \(\mathcal{F}_0\)-measurable random parameters
\(\lambda,\bar{\lambda}:\Omega \to E\) taking values in a Banach space \((E,\Vert \cdot \Vert_E)\). 

We first provide the precise definitions of strong and weak solutions for the path-dependent SDE \eqref{eq:path-dependent-SDE}.

\begin{definition}[Strong solution and uniqueness]
	Let \((\Omega,\mathcal{F},\{\mathcal{F}_t\}_{t\in [0,T]},\mathbb{P})\) be a filtered probability space equipped with a
	\(d\)-dimensional Wiener process \(W\) and \(\mathcal{F}_0\)-measurable random variables \(X_0, \lambda, \bar{\lambda}\).
	A stochastic process \(X\) is called a \emph{strong solution} if \(X\) is \(\{\mathcal{F}_t\}\)-adapted, has continuous
	paths, and satisfies the following integral equation \(\mathbb{P}\)-a.s.:
	\begin{equation} \label{eq:integral-equation-path-dependent-SDE}
		X(t)=X_0 + \int_{0}^{t} b(s,X,\lambda) ds + \int_{0}^{t} \sigma(s,X,\bar{\lambda}) dW(s),\quad t\in [0,T].
	\end{equation}
	\emph{Strong uniqueness} holds for equation \eqref{eq:path-dependent-SDE} if, for any two strong solutions \(X\) and \(Y\) (defined on the same probability space with the same \(W, X_0, \lambda, \bar{\lambda}\)), their paths are indistinguishable; that is,
	\[
		\mathbb{P}\big( \Vert X-Y \Vert_{\ast,T} = 0\big)=1.
	\]
\end{definition}

\begin{definition}[Weak solution and uniqueness in law]
	\emph{Weak existence} holds for the path-dependent SDE \eqref{eq:path-dependent-SDE} if, for any prescribed joint probability law \(\mu \in \mathcal{P}(\mathbb{R}^d\times E \times E)\) governing \((X_0,\lambda,\bar{\lambda})\), there exists a probabilistic setup consisting of:
	\begin{enumerate}
		\item A filtered probability space \((\Omega,\mathcal{F}, \{\mathcal{F}_t\}_{t\in[0,T]}, \mathbb{P})\) equipped with a Wiener process \(W\);
		\item Random variables \(X_0:\Omega\to \mathbb{R}^d\) and \(\lambda,\bar{\lambda}:\Omega\to E\);
		\item An \(\{\mathcal{F}_t\}\)-adapted process \(X\) with continuous paths;
	\end{enumerate}
	such that the joint law of \((X_0,\lambda,\bar{\lambda})\) is exactly \(\mu\), and \(X\) satisfies the integral equation \eqref{eq:integral-equation-path-dependent-SDE} \(\mathbb{P}\)-a.s.
	
	\emph{Uniqueness in law} holds if any two weak solutions starting with the same initial joint law \(\mu\) induce the same pathwise law for \(X\) on \(\mathcal{C}_T^d\).
\end{definition}

The following lemma guarantees that composing a measurable and adapted process with a measurable and non-anticipating
functional yields a process that remains measurable and adapted. The proof is standard and is omitted.
\begin{lemma}\label{prop:Measurability-Composition}
	Suppose that \(\sigma:[0,T]\times \mathcal{C}_T^d\times E \to \mathbb{R}^{d\times d}\) is jointly measurable with respect
	to the product Borel \(\sigma\)-algebra on \([0,T] \times \mathcal{C}_T^d \times E\), and that for every 
	\(t\in [0,T]\) and \(\lambda \in E\), the map \(\sigma(t,\cdot,\lambda):\mathcal{C}_T^d \to \mathbb{R}^{d\times d}\) is
	\(\mathcal{H}_t\)-measurable. Then, for every measurable and adapted process \(X\) with continuous paths and every
	\(\mathcal{F}_0\)-measurable random variable \(\lambda:\Omega\to E\), the composite process
	\([0,T]\times \Omega \ni (t,\omega) \mapsto \sigma(t,X(\omega),\lambda(\omega))\) is measurable and adapted.     
\end{lemma}

The requirement that \(\sigma(t,\cdot,\lambda)\) is \(\mathcal{H}_t\)-measurable formalizes the notion that the functional
depends only on the history of the path up to time \(t\). This is made explicit by the following equivalent definition.

\begin{definition}
	A functional \(\sigma:[0,T]\times \mathcal{C}_T^d\to \mathbb{R}^d\) is \emph{non-anticipating} if there exists
	a family of functionals \((\sigma_t)_{t\in[0,T]}\) with \(\sigma_t: \mathcal{C}_t^d\to \mathbb{R}^d\) such that
	\(\sigma(t,\gamma)=\sigma_t(\gamma_{[0,t]})\) for every \(t\in [0,T]\) and \(\gamma \in \mathcal{C}_T^d\).
\end{definition}

\begin{remark}
	Every functional \(\sigma:[0,T]\times \mathcal{C}_T^d \to \mathbb{R}^d\) adapted to the natural filtration
	\(\{\mathcal{H}_t\}_{t\in[0,T]}\) is non-anticipating. This can be readily deduced using the Doob-Dynkin lemma
	\cite[Lemma 1.14]{Kallenberg-21}. Furthermore, under the pathwise Lipschitz condition stated below in Assumption
	\ref{assump:path-dependent-SDEs}(i), the coefficients \(b\) and \(\sigma\) are automatically non-anticipating.
	For instance, fixing \(\lambda \in E\) and taking any \(\gamma \in \mathcal{C}_T^d\), the Lipschitz bound gives
	\(\vert b(t,\gamma,\lambda) - b(t,\gamma_{\cdot \wedge t},\lambda)\vert = 0\), implying
	\(b(t,\gamma,\lambda)=b(t,\gamma_{\cdot \wedge t},\lambda)\). The continuity in the path variable combined with
	joint measurability in \((t,\lambda)\) ensures that \(b\) and \(\sigma\) satisfy the Carathéodory property.
	Consequently, Lemma \ref{prop:Measurability-Composition} can be directly applied to establish the adaptedness
	of the processes appearing in the SDE \eqref{eq:path-dependent-SDE}.
\end{remark}

To guarantee the existence of a unique strong solution, we impose standard Lipschitz and linear growth conditions. Since the
coefficients are path-dependent, these bounds are formulated using the supremum norm \(\Vert \cdot \Vert_{\ast,t}\).

\begin{assumption} \label{assump:path-dependent-SDEs}
	The drift and diffusion coefficients \(b\) and \(\sigma\) satisfy the following conditions:
	\begin{itemize}
		\item[(i)] \textbf{(Pathwise Lipschitz Continuity)} There exists a continuous map \(L:E\to \mathbb{R}_+\) such that  
		\[	
			\left\vert \sigma(t,\gamma,\lambda) - \sigma(t,\bar{\gamma},\lambda) \right\vert + \left\vert b(t,\gamma,\lambda) - b(t,\bar{\gamma},\lambda) \right\vert \leq L(\lambda)\Vert \gamma-\bar{\gamma}\Vert_{\ast,t}, 
		\]
		for all \(t\in[0,T]\), \(\gamma,\bar{\gamma} \in \mathcal{C}_T^d\), and \(\lambda \in E\).  
		\item[(ii)] \textbf{(Linear Growth)} There exists a continuous map \(C:E\to\mathbb{R}_+\) such that 
		\[
			\vert b(t,0,\lambda) \vert + \vert \sigma(t,0,\lambda) \vert \leq C(\lambda),
		\]
		for all \(t \in [0,T]\) and \(\lambda \in E\).
		\item[(iii)] \textbf{(Measurability)} The maps 
		\((t,\gamma,\lambda) \mapsto b(t,\gamma,\lambda)\) and \((t,\gamma,\lambda) \mapsto \sigma(t,\gamma,\lambda)\)
		are jointly measurable with respect to the product Borel \(\sigma\)-algebra on \([0,T] \times \mathcal{C}_T^d \times E\).	
	\end{itemize}
\end{assumption}

\begin{theorem} \label{th:well-posedness-path-dependent-SDEs}
	Suppose that Assumption \ref{assump:path-dependent-SDEs} holds. Let \(p\in [2,\infty)\). For any initial condition
	\(X_0 \in L^p(\Omega,\mathbb{R}^d)\) and any pair of bounded random parameters
	\(\lambda,\bar{\lambda}\in L^{\infty}(\Omega,E)\), the path-dependent SDE \eqref{eq:path-dependent-SDE} admits
	a unique strong solution \(X \in L^p(\Omega,\mathcal{C}_T^d)\).
\end{theorem}

\begin{proof}
	We proceed via a standard Picard iteration argument. Define the operator \(A:L^p_a(\Omega,\mathcal{C}_T^d)\rightarrow L^p_a(\Omega,\mathcal{C}_T^d)\) by
	\[
		AX(t)=X_0+\int_{0}^{t}b(s,X,\lambda)ds +\int_{0}^{t}\sigma(s,X,\bar{\lambda})dW(s), \quad t\in [0,T].
	\]
	By Assumption \ref{assump:path-dependent-SDEs} and Lemma \ref{prop:Measurability-Composition}, the process \(AX\) is measurable and adapted. By choosing a continuous modification of the Itô integral, \(AX\) has continuous paths \(\mathbb{P}\)-a.s.
	
	Since the parameters \(\lambda, \bar{\lambda} \in L^\infty(\Omega, E)\) are bounded, the continuous maps \(L\) and \(C\) ensure that the random variables \(L(\lambda), L(\bar{\lambda})\) and \(C(\lambda), C(\bar{\lambda})\) are bounded almost surely by deterministic constants. Therefore, for all \(t \in [0,T]\):
	\[
		\vert b(t,X,\lambda) \vert^p + \vert \sigma(t,X,\bar{\lambda}) \vert^p \leq C \left(1+\Vert X\Vert_{\ast,t}^p\right),
	\]
	where \(C > 0\) denotes a generic constant depending only on \(T\) and \(p\), whose exact value may change from line to line. 
	
	Using this growth bound, the elementary inequality \((a+b+c)^p\leq 3^{p-1}(a^p+b+c^p)\), and the BDG inequality \eqref{eq:BDG-p-geq-2}, we obtain:   
	\[
	\begin{aligned}
		\mathbb{E}\Vert AX \Vert_{\ast,t}^p &\leq 3^{p-1}\left(\mathbb{E}\vert X_0\vert^p + \mathbb{E} \left[ \sup_{s\in[0,t]} \left\vert \int_{0}^{s}b(\tau,X,\lambda) d\tau \right\vert^p \right] + \mathbb{E} \left[ \sup_{s\in[0,t]} \left\vert \int_{0}^{s} \sigma(\tau,X,\bar{\lambda}) dW(\tau) \right\vert^p \right] \right) \\  
		&\leq 3^{p-1}\mathbb{E}\vert X_0\vert^p + C \int_0^t \mathbb{E}\vert b(\tau,X,\lambda)\vert^p d\tau + C \int_0^t \mathbb{E}\vert \sigma(\tau,X,\bar{\lambda})\vert^p d\tau \\
		&\leq C \left( 1 + \mathbb{E}\vert X_0\vert^p + \int_0^t \mathbb{E}\Vert X \Vert_{\ast,\tau}^p d\tau \right).
	\end{aligned}
	\]  
	Since \(X \in L^p_a(\Omega,\mathcal{C}_T^d)\), the right-hand side is finite, confirming that
	\(AX \in L^p_a(\Omega,\mathcal{C}_T^d)\). 
	
	To prove existence, we define the Picard sequence \((X^n)_{n\geq 0}\) in \(L^p_a(\Omega,\mathcal{C}_T^d)\) as:
	\[
	\begin{cases}
		X^0(t) =X_0, \quad t \in [0,T], \\[1ex]
		X^{n+1}(t) = AX^n(t), \quad t \in [0,T].
	\end{cases}
	\]
	To demonstrate that this sequence is Cauchy in \(L^p_a(\Omega,\mathcal{C}_T^d)\), we use the pathwise Lipschitz condition alongside Hölder's inequality and the BDG inequality to estimate the difference:
	\[
		\mathbb{E}\Vert X^{n+1}-X^{n}\Vert_{\ast,t}^p \leq C \int_{0}^{t} \mathbb{E} \Vert X^{n}-X^{n-1}\Vert_{\ast,s}^p ds.
	\]  
	Iterating this inequality yields:
	\[
		\mathbb{E}\Vert X^{n+1}-X^{n}\Vert_{\ast,T}^p\leq \frac{(CT)^n}{n!} \mathbb{E}\Vert X^1-X^0\Vert_{\ast,T}^p.
	\]
	Since the series \(\sum \left(\frac{(CT)^n}{n!}\right)^{1/p}\) is summable, the sequence \((X^n)_{n\geq 0}\) is Cauchy in the complete space \(L^p_a(\Omega,\mathcal{C}_T^d)\). Its limit \(X\) is a fixed point of \(A\), thus satisfying \eqref{eq:path-dependent-SDE}.

	For strong uniqueness, suppose that \(X\) and \(Y\) are two strong solutions. Applying the inequality \((a+b)^p\leq 2^{p-1}(a^p+b^p)\), the Lipschitz continuity of the coefficients, and the BDG inequality, we obtain:
	\[
	\begin{aligned}
		\mathbb{E} \Vert X-Y \Vert^p_{\ast,t} &\leq 2^{p-1}\mathbb{E}\left[ \sup_{s\in[0,t]} \left\vert \int_{0}^{s} \left(b(\tau,X,\lambda)-b(\tau,Y,\lambda)\right) d\tau\right\vert^p \right] \\
		&\quad + 2^{p-1}\mathbb{E}\left[\sup_{s\in[0,t]} \left\vert \int_{0}^{s} \left(\sigma(\tau,X,\bar{\lambda})-\sigma(\tau,Y,\bar{\lambda})\right) dW(\tau) \right\vert^p\right] \\
		&\leq C \int_{0}^{t} \mathbb{E} \Vert X - Y \Vert_{\ast,s}^p ds.
	\end{aligned}
	\]
	Gronwall's inequality immediately forces \(\mathbb{E} \Vert X-Y\Vert^p_{\ast,T}=0\), concluding the proof of strong uniqueness.     
\end{proof}

In the specific case of additive noise in which the diffusion matrix \(\sigma\) depends only on time and the parameter
\(\bar{\lambda}\), but not on the path \(X\), we can relax the integrability requirement on the initial condition.
The proof follows the exact logic of Theorem \ref{th:well-posedness-path-dependent-SDEs}, with the simplification that
the stochastic integral cancels out when bounding the difference \(X^{n+1}-X^n\). The contraction estimate relies entirely
on the deterministic drift integral, which removes the need for the BDG inequality and permits \(p \ge 1\). 

\begin{theorem} \label{th:well-posedness-path-dependent-SDEs-additive-noise}
	Suppose that Assumption \ref{assump:path-dependent-SDEs} holds, but assume the diffusion coefficient is independent of the state path ({\it i.e.}, \(\sigma \equiv \sigma(t, \bar{\lambda})\)). Let \(p \in [1,\infty)\). For any initial condition \(X_0 \in L^p(\Omega,\mathbb{R}^d)\) and any pair of bounded random parameters \(\lambda,\bar{\lambda}\in L^{\infty}(\Omega,E)\), the path-dependent SDE \eqref{eq:path-dependent-SDE} admits a unique strong solution \(X \in L^p(\Omega,\mathcal{C}_T^d)\).
\end{theorem}

The well-posedness of the intermediate system \eqref{eq:multi-agent-discrete-independent-projection-solved-weights} will now
be addressed. We consider the following path-dependent McKean-Vlasov SDE: 
\begin{equation} \label{eq:path-dependent-McKean-Vlasov}
	\begin{aligned}
		dX(t) &=b(t,X,\mu) dt + \sigma(t,X,\mu)dW(t), \\
		X(0) &= X_0, \quad \mu = \textrm{Law} (X) \in \mathcal{P}(\mathcal{C}_T^d),
	\end{aligned}
\end{equation}
where \(b\) and \(\sigma\) are path-dependent coefficients
\[
	\begin{aligned}
		&b:[0,T]\times \mathcal{C}_T^d\times \mathcal{P}_p\left(\mathcal{C}_T^d\right)\to \mathbb{R}^{d}, \\
		&\sigma:[0,T]\times \mathcal{C}_T^d\times \mathcal{P}_p\left(\mathcal{C}_T^d\right)\to \mathbb{R}^{d\times d},
	\end{aligned}
\]
and the inputs are a \(d\)-dimensional standard Wiener process \(W\) defined on a filtered probability space \((\Omega,\mathcal{F},\{\mathcal{F}_t\}_{t\in [0,T]},\mathbb{P})\) and an \(\mathcal{F}_0\)-measurable initial condition \(X_0:\Omega\to \mathbb{R}^d\). Under the following conditions on the coefficients \(b\) and \(\sigma\), we will prove strong existence and uniqueness for the McKean-Vlasov SDE \eqref{eq:path-dependent-McKean-Vlasov}. 

\begin{assumption}
	\label{assump:well-posedness-non-Markovian-McKean-Vlasov}
	 The drift and diffusion coefficients \(b\) and \(\sigma\) satisfy the following properties:
	\begin{itemize}
		\item[(i)] \textbf{(Pathwise Lipschitz Continuity)} For every \(t\in[0,T]\), \(\gamma,\bar{\gamma} \in \mathcal{C}_T^d\), and \(\mu,\nu \in \mathcal{P}_p(\mathcal{C}_T^d)\), there exists an \(L>0\) such that  
		\[	
			\left\vert \sigma(t,\gamma,\mu) - \sigma(t,\bar{\gamma},\nu) \right\vert + \left\vert b(t,\gamma,\mu) - b(t,\bar{\gamma},\nu) \right\vert \leq L \left(\Vert \gamma-\bar{\gamma}\Vert_{\ast,t}+ W_{p}(\mu_{[0,t]},\nu_{[0,t]})\right).
		\]
		\item[(ii)] \textbf{(Boundedness)} There exists a \(C>0\) such that 
		\[
			\vert b(t,0,\delta_0) \vert + \vert \sigma(t,0,\delta_0) \vert \leq C,
		\]
		for all \(t \in [0,T]\), where \(0\) denotes the constant zero path in \(\mathcal{C}_T^d\) and \(\delta_0\) is the Dirac delta measure at this zero path.
		\item[(iii)] \textbf{(Measurability)} For all \(\gamma \in \mathcal{C}_T^d\) and \(\mu \in \mathcal{P}_p(\mathcal{C}_T^d) \), the maps 
		\[
			t \mapsto b(t,\gamma,\mu) \quad \text{and} \quad t \mapsto \sigma(t,\gamma,\mu)   
		\] 
		are Borel measurable on \([0,T]\).
	\end{itemize}
\end{assumption}

\begin{remark}
	The pathwise Lipschitz property in Assumption \ref{assump:well-posedness-non-Markovian-McKean-Vlasov}(i)
	ensures that the functionals \(b\) and \(\sigma\) are non-anticipating, {\it i.e.}, 
	\[
		b(t,\gamma,\mu)=b(t,\gamma_{\cdot \wedge t}, \mu_{\cdot \wedge t}),\quad	\sigma(t,\gamma,\mu)=\sigma(t,\gamma_{\cdot \wedge t}, \mu_{\cdot \wedge t}),
	\]  
	where \(\mu_{\cdot \wedge t}= (\cdot \wedge t)_{\#} \mu\) is the pushforward measure under the stopping map: 
	\[
		\begin{aligned}
		\cdot \wedge t: \mathcal{C}^d_T &\longrightarrow \mathcal{C}^d_T \\
	            \gamma &\longmapsto \gamma_{\cdot \wedge t}.
		\end{aligned}
	\]  
	This ensures that for every fixed \(\mu \in \mathcal{P}_p(\mathcal{C}_T^d)\), the maps \(b\) and \(\sigma\) are adapted
	to the natural filtration \(\{\mathcal{H}_t\}_{t\in[0,T]}\). Since they are continuous in the path variable and measurable
	in time, they satisfy the Carathéodory property and are jointly measurable in \((t,\gamma)\). This allows us to apply
	Lemma \ref{prop:Measurability-Composition} in the proof of Theorem \ref{th:well-posedness-non-Markovian-McKean-Vlasov}
	to establish the measurability and adaptedness of the composite processes \(b(t,X,\mu)\) and \(\sigma(t,X,\mu)\).  
\end{remark}

\begin{theorem} \label{th:well-posedness-non-Markovian-McKean-Vlasov}
	Suppose that the functionals \(b\) and \(\sigma\) satisfy Assumption \ref{assump:well-posedness-non-Markovian-McKean-Vlasov}
	for some \(p\in [2,\infty)\). Then, for every initial condition \(X_0 \in L^p(\Omega,\mathbb{R}^d)\), there exists a
	unique strong solution \(X\in L^p(\Omega,\mathcal{C}_T^d)\) to the McKean-Vlasov SDE
	\eqref{eq:path-dependent-McKean-Vlasov}.
\end{theorem}

\begin{proof}
	The proof relies on a standard fixed-point argument, which is a straightforward extension of the classical Markovian case.
	Given \(\mu \in \mathcal{P}_p(\mathcal{C}_T^d)\), we first consider the decoupled path-dependent Itô diffusion: 
	\begin{equation} \label{eq:Vlasov-McKeanFixingP}
		\begin{aligned}
			dX^\mu(t) &=b(t,X^\mu,\mu) dt + \sigma(t,X^\mu,\mu)dW(t), \\
			X^\mu(0) &= X_0. 
		\end{aligned}
	\end{equation}
	This equation is well-posed since the functionals \(b(\cdot,\cdot, \mu)\) and \(\sigma(\cdot,\cdot, \mu)\) satisfy
	the conditions of Theorem \ref{th:well-posedness-path-dependent-SDEs}. Therefore, there exists a unique strong
	solution \(X^\mu\in L^p(\Omega,\mathcal{C}_T^d)\) to \eqref{eq:Vlasov-McKeanFixingP} for each
	\(\mu \in \mathcal{P}_p(\mathcal{C}_T^d)\). This allows us to define the map:  
	\[
		\begin{aligned}
		\Psi: \mathcal{P}_p(\mathcal{C}^d_T) &\longrightarrow \mathcal{P}_p(\mathcal{C}^d_T) \\
	                                       \mu&\longmapsto \textrm{Law}(X^\mu).
		\end{aligned}
	\]
	To prove that \(\Psi\) is a contraction under the \(p\)-Wasserstein metric, fix any
	\(\mu, \nu \in \mathcal{P}_p(\mathcal{C}_T^d)\) and let \(X^\mu, X^\nu\) be the solutions
	of \eqref{eq:Vlasov-McKeanFixingP} corresponding to \(\mu\) and \(\nu\), respectively. Since the joint law
	\(\textrm{Law}(X_{[0,t]}^\mu,X_{[0,t]}^\nu)\) constitutes a valid coupling between \(\Psi(\mu)_{[0,t]}\) and
	\(\Psi(\nu)_{[0,t]}\), the \(p\)-Wasserstein distance is bounded by:
	\[
		W^p_{p}(\Psi(\mu)_{[0,t]},\Psi(\nu)_{[0,t]}) \leq \mathbb{E} \Vert X^\mu-X^\nu\Vert_{\ast,t}^p.
	\]  
	Using the inequality \((a+b)^p\leq 2^{p-1}(a^p+b^p)\), Hölder's inequality, the BDG inequality \eqref{eq:BDG-p-geq-2},
	and the Lipschitz property of the coefficients, we obtain the estimate:
	\[
		\begin{aligned} 
			\mathbb{E} \Vert X^\mu-X^\nu\Vert_{\ast,t}^p &\leq (2t)^{p-1} \mathbb{E} \int_{0}^{t} \vert b(s,X^\mu,\mu)-b(s,X^\nu,\nu)\vert^p ds \\
			&\quad + C_p 2^{p-1}t^{\frac{p}{2}-1} \mathbb{E} \int_{0}^{t} \vert \sigma(s,X^\mu,\mu)-\sigma(s,X^\nu,\nu) \vert^p ds \\ 
			&\leq C \left(\int_{0}^{t}\mathbb{E}\Vert X^\mu-X^\nu\Vert^p_{\ast,s} ds + \int_{0}^{t}W^p_{p}(\mu_{[0,s]},\nu_{[0,s]}) ds\right),
		\end{aligned}
	\] 
	where \(C > 0\) denotes a generic constant depending only on \(T\) and \(p\), whose exact value may change from line to line.
	
	By Gronwall's inequality, we deduce that
	\[
		W^p_{p}(\Psi(\mu)_{[0,t]},\Psi(\nu)_{[0,t]})\leq C \int_{0}^{t}W^p_{p}(\mu_{[0,s]},\nu_{[0,s]}) ds.
	\]
	Iterating this inequality \(k\) times yields: 
	\[
		W^p_{p}(\Psi^k(\mu),\Psi^k(\nu))\leq \frac{(CT)^k}{k!} W^p_{p}(\mu,\nu).
	\]
	For \(k\) sufficiently large, the factor \(\frac{(CT)^k}{k!}\) is strictly less than 1, making \(\Psi^k\)
	a contraction. This ensures that \(\Psi\) has a unique fixed point \(\mu^*\), and the corresponding process
	\(X^{\mu^*}\) is the unique strong solution to the McKean-Vlasov SDE.   
\end{proof}

As in the case of path-dependent SDEs with additive noise, if the diffusion coefficient \(\sigma\) depends only on time
(and not on the path \(\gamma\) nor on the measure \(\mu\)), we can relax the integrability condition on the initial
condition to allow for \(L^p\) inputs with \(p\in [1,\infty)\). 

\begin{theorem} \label{th:well-posedness-non-Markovian-McKean-Vlasov-additive-noise}
	Suppose that Assumption \ref{assump:well-posedness-non-Markovian-McKean-Vlasov} holds with \(p\in [1,\infty)\), but assume
	that the diffusion coefficient is independent of both the state path and the measure ({\it i.e.}, \(\sigma \equiv \sigma(t)\)).
	Then, for any initial condition \(X_0 \in L^p(\Omega,\mathbb{R}^d)\), there exists a unique strong solution
	\(X \in L^p(\Omega,\mathcal{C}_T^d)\) to the path-dependent McKean-Vlasov SDE \eqref{eq:path-dependent-McKean-Vlasov}.
\end{theorem}
\bibliographystyle{amsplain} 
\bibliography{biblio}

\end{document}